\documentclass[11pt]{amsart}
\usepackage{amsmath}
\usepackage[psamsfonts]{amssymb}
\usepackage[all]{xy}
\usepackage{graphicx}
\usepackage[T1]{fontenc}
\usepackage{enumitem}
\usepackage[bitstream-charter]{mathdesign}
\usepackage{hyperref}
\usepackage{enumitem}
\makeatletter
\def\namedlabel#1#2{\begingroup
	#2%
	\def\@currentlabel{#2}%
	\phantomsection\label{#1}\endgroup
}
\makeatother
\usepackage{amscd}
\usepackage{verbatim}
\usepackage[english]{babel}
\usepackage{xcolor}
\usepackage{tikz}
\usetikzlibrary{shapes.geometric, arrows}
\tikzstyle{box} = [rectangle, minimum width=3cm, text centered, text width=3cm, draw=black]
\tikzstyle{arrow} = [thick,->,>=stealth]
\let\oldtocsection=\tocsection

\let\oldtocsubsection=\tocsubsection

\renewcommand{\tocsection}[2]{\hspace{0em}\oldtocsection{#1}{#2}}
\renewcommand{\tocsubsection}[2]{\hspace{1em}\oldtocsubsection{#1}{#2}}

\usepackage{mathtools}
\usepackage{mathdots}
\theoremstyle{definition}
\newtheorem{theorem}{Theorem}[section]
\newtheorem{prop}[theorem]{Proposition}
\newtheorem{lemma}[theorem]{Lemma}
\newtheorem{cor}[theorem]{Corollary}

\newtheorem{ex}[theorem]{Example}
\theoremstyle{remark}
\newtheorem{dfn}[theorem]{Definition}
\newtheorem{remark}[theorem]{Remark}
\newtheorem{claim}[theorem]{Claim}
\numberwithin{equation}{section}

\def\co{\colon\thinspace}

\def\ep{\epsilon}

\def\R{\mathbb{R}}
\def\Z{\mathbb{Z}}
\def\N{\mathbb{N}}
\def\C{\mathbb{C}}

\def\uph{\underline{\phi}}

\DeclareMathOperator{\Img}{Im}
\begin{document}
	\title{Filtered cospans and interlevel persistence with boundary conditions}
	\author{Michael Usher}
	\address{Department of Mathematics, University of Georgia, 
		Athens, GA 30602}
	\email{usher@uga.edu}

	\maketitle 

\begin{abstract}
	We develop the notion of a ``filtered cospan'' as an algebraic object that stands in the same relation to interlevel persistence modules as filtered chain complexes stand with respect to sublevel persistence modules.  This relation is expressed via a functor from a category of filtered cospans to a category of persistence modules that arise in Bauer-Botnan-Fluhr's study of relative interlevel set homology. We associate a filtered cospan to a Morse function $f:\mathbb{X}\to [-\Lambda,\Lambda]$ such that $\partial \mathbb{X}$ is the union of the regular level sets $f^{-1}(\{\pm\Lambda\})$; this allows us to capture the interlevel persistence of such a function in terms of data associated to Morse chain complexes. Similar filtered cospans are associated to simplicial and singular chain complexes, and isomorphism theorems are proven relating these to each other and to relative interlevel set homology. Filtered cospans can be decomposed, under modest hypotheses, into certain standard elementary summands, giving rise to a notion of persistence diagram for filtered cospans that is amenable to computation.  An isometry theorem connects interleavings of filtered cospans to matchings between these persistence diagrams.
\end{abstract}

\section{Introduction}
The standard framework for classical Vietoris-Rips or sublevel persistent homology constructs from geometric data a persistence module indexed by $\R$---that is, a collection of vector spaces $V_t$ for all $t\in \R$ together with appropriately compatible maps $V_s\to V_t$ for $s\leq t$---and then extracts quantitative data from the persistence module (in the form of a barcode or persistence diagram) which classifies it.  While classification theorems such as the one in \cite{CB} require only the presence of a persistence module structure satisfying modest tameness hypotheses, the standard computational procedures going back in various contexts to \cite{Bar}, \cite{ELZ}, \cite{ZC} exploit the fact that this persistence module is usually obtained by taking the filtered homologies of a filtered (or, in the terminology of this paper, ``ascending'') chain complex.   An ascending chain complex is given by data $(C_{\uparrow *},\partial,\ell)$ where $(C_{\uparrow *},\partial)$ is a chain complex and $\ell\co \C_{\uparrow *}\to\R\cup\{-\infty\}$ is a certain type of function that records the filtration data (see Definition \ref{ascdef}).  Even if, as is generally the case in practical situations, the vector spaces $V_t$ change only at a discrete set of values $t\in \{t_0,\ldots,t_N\}$,  it stands to reason that it would be more efficient to compute persistent homology by manipulating the single object $(C_{\uparrow *},\partial,\ell)$  rather than by recording and analyzing the full persistence module $V_{t_0}\to V_{t_1}\to\cdots\to V_{t_N}$, which often contains much redundant information when $N$ is large.  The standard algorithms as in \cite{ZC} corroborate this intuition.

For a continuous function $f\co \mathbb{X}\to\R$ on a topological space $\mathbb{X}$, a more refined object than the sublevel persistence alluded to in the previous paragraph is the \emph{interlevel} persistence \cite{CDM}, which encompasses the relations between the homologies $H_*(f^{-1}(I))$ as $I$ varies through the partially ordered set of closed intervals in $\R$.   As was already shown in \cite{CDM}, under modest hypotheses the interlevel persistence carries equivalent information to the extended persistence of \cite{CEH09}.  There are various viewpoints under which the interlevel persistence can be derived from a type of persistence module structure which admits decompositions into standard irreducible summands from which one can derive a version of a persistence diagram.  For example, \cite{BEMP} and \cite{CDM} formulate interlevel persistence in terms of zigzag modules, while \cite{BGO} and \cite{BBF20} describe it in terms of special classes of two-dimensional persistence modules which satisfy  properties corresponding to the Mayer-Vietoris sequence.  Our preferred formulation throughout this paper is the elegant construction in \cite{BBF20} of the ``relative interlevel set homology'' as a certain kind of persistence module indexed by the strip \begin{equation} \label{mdef}\mathbb{M}=\{(x,y)\in\R^2|-2\Lambda\leq x+y\leq 2\Lambda\}\end{equation} with the partial order $\preceq$ given by $(x,y)\preceq (x',y')\Leftrightarrow \left(x\geq x' \mbox{ and }y\leq y'\right)$.  Here $\Lambda$ is some fixed positive number, which in \cite{BBF20} is usually taken to be $\frac{\pi}{2}$. This single persistence module captures all relative homologies $H_k(f^{-1}(I),f^{-1}(C))$ where $I$ is a closed interval, $C$ is the complement of an open interval, and $k\in \mathbb{Z}$, and incorporates information from connecting homomorphisms in Mayer-Vietoris sequences.  

The specific $\mathbb{M}$-indexed persistence modules considered in \cite{BBF20} decompose as direct sums of ``block modules'' $B^v$ indexed by elements  $v\in \mathbb{M}\setminus\partial\mathbb{M}$, yielding a persistence diagram for relative interlevel set homology that is a multiset of elements of $\mathbb{M}\setminus\partial\mathbb{M}$.

The present paper introduces the notion of a ``filtered cospan'' which is intended to stand in the same relation to the $\mathbb{M}$-indexed persistence modules of \cite{BBF20} as filtered chain complexes stand with respect to more classical $\R$-indexed persistence modules.  A filtered cospan (see Definition \ref{codfn}) amounts to a diagram \begin{equation}\label{gencospan} \xymatrix{ (C^{\downarrow}_{ *},\partial^{\downarrow},\ell^{\downarrow}) \ar[rd]^{\psi^{\downarrow}} & \\ & (D_*,\partial) \\ (C_{\uparrow *},\partial_{\uparrow},\ell_{\uparrow}) \ar[ru]_{\psi_{\uparrow}}   }\end{equation} where $(C_{\uparrow *},\partial_{\uparrow},\ell_{\uparrow})$ is an ascending chain complex, $(C^{\downarrow}_{ *},\partial^{\downarrow},\ell^{\downarrow})$ is a descending chain complex (like an ascending chain complex but with a descending filtration rather than an ascending one), $(D_*,\partial)$ is another chain complex, and the maps are chain maps.  Filtered cospans form a category $\mathsf{HCO}(\kappa)$ (where $\kappa$ is the field of coefficients of the chain complexes) which we intend to be analogous to the homotopy category of filtered chain complexes; like the latter, $\mathsf{HCO}(\kappa)$ arises as the zeroth cohomology category of a dg-category.  To simplify the description of the relation to \cite{BBF20}, it is convenient to restrict to the full subcategory $\mathsf{HCO}^{\Lambda}(\kappa)$ of ``$\Lambda$-bounded'' filtered cospans where $\Lambda$ is the parameter in (\ref{mdef}) (see Definition \ref{bounded-dfn}); geometrically, this corresponds to considering functions $\mathbb{X}\to(-\Lambda,\Lambda)$ instead of $\mathbb{X}\to\R$.  Among our results and constructions are: 
\begin{itemize}\item[(i)] A classification theorem (Corollary \ref{decompexists}) which presents any suitably well-behaved filtered cospan as a direct sum of certain standard elementary summands.
	\item[(ii)] A functor $H_0\circ\mathcal{F}$ from $\mathsf{HCO}^{\Lambda}(\kappa)$ to the category $\mathsf{Vect}_{\kappa}^{\mathbb{M}}$ of $\mathbb{M}$-indexed persistence modules, which maps the collection of standard elementary summands from (i) bijectively to the collection of block modules $B^v$ from \cite{BBF20}, as $v$ varies through $\mathbb{M}\setminus\partial\mathbb{M}$. (See Section \ref{fsect} and Proposition \ref{blockclass}.)
	\item[(iii)] An isometry theorem (Theorem \ref{isometry}) equating a suitably-defined interleaving distance between two objects $\mathcal{C}$ and $\mathcal{X}$ of $\mathsf{HCO}^{\Lambda}(\kappa)$ with a suitably-defined bottleneck distance between the persistence diagrams of the $\mathbb{M}$-indexed persistence modules $H_0\circ\mathcal{F}(\mathcal{C})$ and $H_0\circ\mathcal{F}(\mathcal{X})$.
\end{itemize}

Given the algebraic data as in (\ref{gencospan}) associated to a filtered cospan $\mathcal{C}$, Section \ref{cospan-decomp} shows in fairly concrete terms how to determine the decomposition in (i).  While we leave to future work (including work in progress by the author's student Cameron Thomas) the matter of systematically describing a particularly efficient algorithm for this, the description in Section \ref{cospan-decomp} should make  clear that it can be achieved by a sequence of the type of Gaussian elimination-based methods that appear in classical persistence algorithms like that in \cite{ZC}, along with intermediate steps of similar or lower complexity.  Several examples of the computation of the decomposition are given in Section \ref{calcsect}. Just as sublevel persistence is normally computed by taking as input a filtered chain complex rather than the persistence module given by the sequence of inclusion-induced maps on its filtered homology groups, our filtered cospans are intended to provide a more efficient input to the calculation of interlevel persistence than a zigzag module or an $\mathbb{M}$-indexed persistence module. 

If $\mathbb{X}$ is the geometric realization of a simplicial complex and $f$ is a simplexwise linear function on $\mathbb{X}$ which takes values in $(-\Lambda,\Lambda)$, a special case of Example \ref{simpdef} associates to $f$ a filtered cospan in which each of the three chain complexes $(C_{\uparrow *},\partial_{\uparrow}),(C^{\downarrow}_{*},\partial^{\downarrow}),(D_*,\partial)$ are   equal to the simplicial chain complex of $\mathbb{X}$, and the chain maps $\psi_{\uparrow}$ and $\psi^{\downarrow}$ in (\ref{gencospan}) are the identity.  In this algebraic situation the decomposition of the filtered cospan can be computed using the algorithm in \cite[Section 6]{CEH09}.  However, we will also have occasion to consider cases where chain complexes in (\ref{gencospan}) are different from each other and the maps $\psi_{\uparrow}$ and $\psi^{\downarrow}$ do not induce isomorphisms on homology.  These cases seem to be bit outside the purview of extended persistence; correspondingly, their decompositions include summands of a type that do not arise in that context.  We discuss the origin of such cases now.

\subsection{Interlevel persistence with boundary conditions}

Our general motivating object of study will be a continuous function $f\co \mathbb{X}\to [-\Lambda,\Lambda]$ where $\Lambda>0$.   The topological space $\mathbb{X}$ (generally assumed compact), the parameter $\Lambda$, and the sets  $\partial^-\mathbb{X}:=f^{-1}(\{-\Lambda\})$ and $\partial^+\mathbb{X}:=f^{-1}(\{\Lambda\})$  should each be regarded as background input (given before choosing $f$); thus, when $f$ is treated as variable (for example in a stability theorem) the variation should be within the class of functions that satisfy the boundary conditions $f|_{\partial^-\mathbb{X}}= -\Lambda$ and $f|_{\partial^+\mathbb{X}}=\Lambda$ for prespecified subsets $\partial^{\pm}\mathbb{X}$ of $\mathbb{X}$.  It is entirely permissible for $\partial^-\mathbb{X}$ and/or $\partial^+\mathbb{X}$ to be empty; if both are empty then we are studying functions $f\co \mathbb{X}\to (-\Lambda,\Lambda)$, which is equivalent to studying functions $\varphi^{-1}\circ f\co \mathbb{X}\to\R$ where $\varphi$ is the reader's favorite homeomorphism $\R\to (-\Lambda,\Lambda)$.   

In Example \ref{pinsing}, we associate  to the function $f\co \mathbb{X}\to [-\Lambda,\Lambda]$ as above the ``pinned singular filtered cospan'' $\mathcal{S}_{\partial}(X,f;\kappa)$.  (The word ``pinned'' is meant to evoke the constraint that $f|_{\partial^{\pm}\mathbb{X}}=\pm\Lambda$.)	This is an object of the category $\mathsf{HCO}^{\Lambda}(\kappa)$ and is constructed from subquotients of the singular chain complex of $\mathbb{X}$ with coefficients in the field $\kappa$.   Theorem \ref{H0iso} asserts that, under modest tameness hypotheses on $f$, the image of $\mathcal{S}_{\partial}(X,f;\kappa)$ under the functor $H_0\circ\mathcal{F}\co \mathsf{HCO}^{\Lambda}(\kappa)\to \mathsf{Vect}_{\kappa}^{\mathbb{M}}$ is isomorphic to the relative interlevel set homology (as in \cite{BBF20}, but see Section \ref{rishsect} for the details of our conventions) of the function $f|_{\mathbb{X}\setminus\partial\mathbb{X}}$, where $\partial\mathbb{X}=\partial^-\mathbb{X}\cup\partial^+\mathbb{X}$.  

In the case that $\mathbb{X}$ is the geometric realization of a finite simplcial complex $X$ and that $f\co \mathbb{X}\to [-\Lambda,\Lambda]$ is simplexwise linear with $f^{-1}(\partial\mathbb{X})$ corresponding to a subcomplex of $X$, Example \ref{simpdef} constructs the ``pinned simplicial filtered cospan'' $\mathcal{C}_{\partial}(X,f;\kappa)$ in terms of subquotients of the simplicial chain complex of $X$.   Proposition \ref{simpsing} shows that $\mathcal{C}_{\partial}(X,f;\kappa)$ is isomorphic in the category $\mathsf{HCO}^{\Lambda}(\kappa)$ to $\mathcal{S}_{\partial}(\mathbb{X},f;\kappa)$.  As the isomorphism type of $\mathcal{C}_{\partial}(X,f;\kappa)$ can be computed combinatorially, this together with Theorem \ref{H0iso} yields a way of computing the relative interlevel set homology of $f|_{\mathbb{X}\setminus\partial\mathbb{X}}$.

To be sure, since $\partial\mathbb{X}=f^{-1}(\{-\Lambda,\Lambda\})$ and $f(\mathbb{X})\subset [-\Lambda,\Lambda]$, one could also determine the relative interlevel set homology of  $f|_{\mathbb{X}\setminus\partial\mathbb{X}}$ by using the procedure in \cite{CEH09} to compute the extended persistence of $f$.  A modest advantage of the approach using $\mathcal{C}_{\partial}(X,f;\kappa)$ is that its computation can be carried out while disregarding any simplices that are contained in $\partial \mathbb{X}$.

The case which most heavily motivated this paper is where $\mathbb{X}$ is a smooth compact manifold with boundary $\partial\mathbb{X}=\partial^-\mathbb{X}\cup\partial^+\mathbb{X}$, and the function $f\co \mathbb{X}\to [-\Lambda,\Lambda]$ with $f^{-1}(\{\pm\Lambda\})=\partial^{\pm}\mathbb{X}$ is Morse and has $\Lambda$ and $-\Lambda$ as regular values. (Figure \ref{k1n2fig} illustrates some examples.)  Such Morse functions on cobordisms have long been fruitful objects of study in topology, as for example in \cite{Milhcob}. When $\partial^-\mathbb{X}=\varnothing$, it is well-known (\cite[Theorem 3.5]{Mil}) that $\mathbb{X}$ has the homotopy type of a cell complex formed from the descending manifolds of the critical points of $f$. The associated cellular chain complex is naturally filtered by the function values of the critical points, and the sublevel persistence of $f$ can then be computed as the filtered homology persistence module of this ascending chain complex, which is denoted in this paper as $CM_*(f,v;\kappa)$, with $v$ denoting the gradient-like vector field used to construct the descending manifolds of the critical points. If $\partial^-\mathbb{X}\neq\varnothing$, one can still form the ascending Morse complex $CM_*(f,v;\kappa)$, but now its filtered homologies are isomorphic to the relative homologies of the pairs $(f^{-1}([-\Lambda,t]),\partial^-\mathbb{X})$ as $t$ varies through $[-\Lambda,\Lambda]$.  

We introduce the Morse filtered cospan $\mathcal{M}(\mathbb{X},f,v,;\kappa)$ in Example \ref{morseex} as a way to capture the interlevel persistence of $f$ in terms of data associated to the Morse complexes of $f$ and of $-f$.  In the case that $\partial\mathbb{X}=\varnothing$ this is equivalent to a special case of constructions in \cite{U23} and can also be understood in terms of extended persistence; see \cite[Section 3.1]{U23} for a discussion.  However, the setup of \cite{U23}, while well-adapted to the Novikov-theoretic situations that are the focus of that paper, does not apply to the case that $\partial\mathbb{X}\neq \varnothing$ since in this case the Morse homologies of $f$ and $-f$ are generally non-isomorphic (as they are isomorphic to $H_*(\mathbb{X},\partial^-\mathbb{X};\kappa)$ and $H_*(\mathbb{X},\partial^+\mathbb{X};\kappa)$, respectively).

Proposition \ref{morseiso} shows that      $\mathcal{M}(\mathbb{X},f,v;\kappa)$ is isomorphic in the category $\mathsf{HCO}^{\Lambda}(\kappa)$ to the pinned singular filtered cospan $\mathcal{S}_{\partial}(\mathbb{X},f;\kappa)$; hence it is connected to the relative interlevel set homology of $f$ by Theorem \ref{H0iso}.  While, in general, the isomorphism type of $\mathcal{M}(\mathbb{X},f,v;\kappa)$ cannot be computed combinatorially since it is based in part on the flow of the gradient-like vector field $v$, we compute many examples in Section \ref{calcsect}.  In particular, in Section \ref{elemsect} we give a thorough discussion of the case that $f\co \mathbb{X}\to[-\Lambda,\Lambda]$ has only one critical point, so that $\mathbb{X}$ is the cobordism associated to surgery on a sphere in $\partial^-\mathbb{X}$. 

The cospans $\mathcal{S}_{\partial}(X,f;\kappa)$, $\mathcal{C}_{\partial}(X,f;\kappa)$, and $\mathcal{M}(\mathbb{X},f,v;\kappa)$ discussed in this section have in common that, when one passes to (unfiltered) homology from the diagram (\ref{gencospan}), one obtains, up to isomorphism, \[ \xymatrix{H_*(\mathbb{X},\partial^+\mathbb{X};\kappa) \ar[rd] & \\ & H_*(\mathbb{X},\partial\mathbb{X};\kappa) \\ H_*(\mathbb{X},\partial^-\mathbb{X};\kappa) \ar[ru] &    }   \] where the maps are induced by inclusion of pairs.  The only case in which both of these maps are isomorphisms in all gradings is when $\partial\mathbb{X}=\varnothing$.  Our general decomposition theorem, Corollary \ref{decompexists}, involves eight types of standard elementary summands, denoted $(\uparrow_{a}^{b})_k$, $(\downarrow^a_b)_k$, $(\uparrow_{a}^{\infty})_k$, $(\downarrow^a_{-\infty})_k$, $(\nearrow_a)_k$, $(\searrow^a)_k$, $(>_a^b)_k$, and $\square_k$, where $a$ and $b$ denote filtration parameters and $k$ refers to grading.  If the maps  $\psi_{\uparrow}$ and $\psi^{\downarrow}$ in (\ref{gencospan}) induce isomorphisms on homology, then the only summands that arise are of the form $(\uparrow_{a}^{b})_k$, $(\downarrow^a_b)_k$, and  $(>_a^b)_k$, and these correspond to the ``ordinary,'' ``relative,'' and ``extended'' subdiagrams  in the language of extended persistence.  Under our functor $H_0\circ\mathcal{F}\co \mathsf{HCO}^{\Lambda}(\kappa)\to \mathsf{Vect}_{\kappa}^{\mathbb{M}}$, the summands $(\uparrow_{a}^{b})_k$, $(\downarrow^a_b)_k$, and  $(>_a^b)_k$, are (modulo changing the value of $k$) sent to block modules $B^v$ where $v$ lies in the interior of one of the regions that are denoted $T^kL,T^kA$ and $T^kS$ starting in Section \ref{stripfig}; see also \cite[Figure 1.1]{BBF20}.  Elements $v\in \mathbb{M}\setminus\partial\mathbb{M}$ 
that do \emph{not} lie in the interiors of any of the regions $T^kL,T^kA$ and $T^kS$ are precisely those lying on the perimeter of a square $T^kS$.  In the context of \cite{BBF20}, these points are seen as non-realizable, since in our language \cite{BBF20} concerns the case that $\psi_{\uparrow}$ and $\psi^{\downarrow}$ both induce isomorphisms on homology.  But by relaxing that condition, these points can be realized: the four open edges of the square $T^kS$ are accounted for by summands of type $(\uparrow_{a}^{\infty})_{-k}$, $(\downarrow^a_{-\infty})_{-k}$, $(\nearrow_a)_{-k}$, $(\searrow^a)_{-k}$, respectively, and the two vertices of $T^kS$ that do not lie on $\partial\mathbb{M}$ correspond to $\square_{-k}$ and $\square_{-k+1}$.  See Figure \ref{blockfig}.  We will see in Section \ref{elemsect} that the case of a Morse function with one critical point is already rich enough to yield examples of decompositions involving any of the summands $(\uparrow_{a}^{\infty})_k$, $(\downarrow^a_{-\infty})_k$, $(\nearrow_a)_k$, $(\searrow^a)_k$, $(>_a^b)_k$, and $\square_k$;  the remaining summands $(\uparrow_a^b)_k$ and $(\downarrow^a_b)_k$ first arise for Morse functions with two critical points as in Example \ref{cubic}.

\subsection{Cospan basics}
We now begin formulating some of the notions discussed above more precisely.
	
	\begin{dfn}\label{ascdef}
		 An \textbf{ascending chain complex} over a field $\kappa$ is a triple $\mathcal{C}_{\uparrow}=(C_{\uparrow*},\partial_{\uparrow},\ell_{\uparrow})$ where $(C_{\uparrow*},\partial_{\uparrow})$ is a chain complex of $\kappa$-vector spaces, and $\ell_{\uparrow}\co C_*\to  \R\cup\{-\infty\}$ is a function satisfying:
		\begin{itemize} \item $\ell_{\uparrow}(x)=-\infty$ if and only if $x=0$;
			\item $\ell_{\uparrow}\left(\sum_{i=1}^{m}c_ix_i\right)\leq \max\{\ell_{\uparrow}(x_i)|c_i\neq 0\}$ for all $x_1,\ldots,x_m\in C_{\uparrow*}$ and $c_1,\ldots,c_m\in\R$, with equality if $x_i\in C_{\uparrow k_i}$ for distinct $k_1,\ldots,k_m$.
			\item $\ell_{\uparrow}(\partial_{\uparrow}x)\leq \ell_{\uparrow}(x)$ for all $x\in C_*$. 
		\end{itemize}
	\end{dfn}
	
	In this situation, we have a nested family of subcomplexes $C_{\uparrow}^{\leq t}:=\{x\in C_{\uparrow *}|\ell_{\uparrow}(x)\leq t\}$, with $C_{\uparrow}^{\leq s}\subset C_{\uparrow }^{\leq t}$ for $s\leq t$.
	The symbol $\uparrow$, obviously, is meant to emphasize the ascending nature of this family of subcomplexes; however, we acknowledge (and hereby warn the reader of) the confusion that might result from the fact that this family of subcomplexes is determined by the \emph{sub}level sets of $\ell_{\uparrow}$, not the superlevel sets.  
	
	There is a companion notion yielding subcomplexes that are associated to superlevel sets of a function:
	
	\begin{dfn}\label{descdef}
		 A \textbf{descending chain complex} over a field $\kappa$ is a triple $\mathcal{C}^{\downarrow}=(C^{\downarrow}_{*},\partial^{\downarrow},\ell^{\downarrow})$ where $(C^{\downarrow},\partial^{\downarrow})$ is a chain complex of $\kappa$-vector spaces and $\ell^{\downarrow}\co C^{\downarrow}_{*}\to \R\cup\{+\infty\}$ has the property that $(C^{\downarrow}_{*},\partial^{\downarrow},-\ell^{\downarrow})$ is an ascending chain complex in the sense of Definition \ref{ascdef}.
	\end{dfn}
	
	Thus for a descending chain complex $(C^{\downarrow}_{*},\partial^{\downarrow},\ell^{\downarrow})$ we have $\ell^{\downarrow}\left(\sum_ic_ix_i\right)\geq \min\{\ell^{\downarrow}(x_i)|c_i\neq 0\}$ and $\ell^{\downarrow}(\partial^{\downarrow}x)\geq \ell^{\downarrow}(x)$, and there is a decreasing family of subcomplexes $C^{\downarrow}_{\geq t}=\{x\in C^{\downarrow}_{*}|\ell^{\downarrow}(x)\geq t\}$.

As suggested above, our key algebraic objects are defined as follows:
	\begin{dfn}\label{codfn}
	Let $\kappa$ be a field.  A \textbf{filtered cospan} over $\kappa$ consists of the following data:
	\begin{itemize} \item an ascending chain complex over $\kappa$, $\mathcal{C}_{\uparrow}=(C_{\uparrow *},\partial_{\uparrow},\ell_{\uparrow})$;
		\item a descending chain complex over $\kappa$, $\mathcal{C}^{\downarrow}=(C^{\downarrow}_{*},\partial^{\downarrow},\ell^{\downarrow})$;
		\item another chain complex of $\kappa$-vector spaces, $(D_*,\partial_D)$; and
		\item chain maps $\psi_{\uparrow}\co C_{\uparrow *}\to D_*$ and $\psi^{\downarrow}\co C^{\downarrow}_{*}\to D_*$.	\end{itemize}		
\end{dfn}

The relation to the setup of \cite{BBF20} simplifies notationally if we fix a number $\Lambda>0$ and assume that the nontrivial values of our filtration functions have absolute value bounded by $\Lambda$. 

\begin{dfn}\label{bounded-dfn} Let $\Lambda$ be a positive number.  An ascending chain complex $\mathcal{C}_{\uparrow}=(C_{\uparrow *},\partial_{\uparrow},\ell_{\uparrow})$ (resp., a descending chain complex $\mathcal{C}^{\downarrow}=(C^{\downarrow}_{*},\partial^{\downarrow},\ell^{\downarrow})$) is said to be $\Lambda$-\textbf{bounded} if $-\Lambda<\ell(x)<\Lambda$ for all $x\in C_{\uparrow *}\setminus\{0\}$ (resp. for all $x\in C^{\downarrow}_{*}\setminus\{0\}$).
	
A filtered cospan $\mathcal{C}=(\mathcal{C}_{\uparrow},\mathcal{C}^{\downarrow},(D_*,\partial_D))$ is said to be $\Lambda$-bounded if both the ascending chain complex $\mathcal{C}_{\uparrow}$ and the descending chain complex $\mathcal{C}^{\downarrow}$ are $\Lambda$-bounded.
\end{dfn}

	Thus in this case we have $C_{\uparrow *}=\cup_{t<\Lambda}C_{\uparrow}^{\leq t}$, $C_{*}^{\downarrow}=\cup_{t>-\Lambda} C^{\downarrow}_{\geq t}$, and $C_{\uparrow}^{\leq -\Lambda}=C^{\downarrow}_{\geq \Lambda}=\{0\}$.

	There is no real loss of generality in restricting attention to $\Lambda$-bounded ascending or descending chain complexes, since if we fix an increasing bijection $\varphi\co \R\to(-\Lambda,\Lambda)$ then an arbitrary ascending chain complex   $(C_{\uparrow *},\partial_{\uparrow},\ell_{\uparrow})$ will become $\Lambda$-bounded upon replacing $\ell_{\uparrow}|_{C_{\uparrow *}\setminus\{0\}}$ by $\varphi\circ\ell_{\uparrow}|_{C_{\uparrow *}\setminus\{0\}}$, and similarly for descending chain complexes.  Something like this conversion appears repeatedly in \cite{BBF20}, with $\Lambda=\frac{\pi}{2}$ and $\varphi$ equal to the arctangent function.
		Throughout the rest of the paper, the positive real parameter $\Lambda$ should be regarded as fixed.

	In Section \ref{catsect}, we will construct a category $\mathsf{HCO}(\kappa)$ whose objects are filtered cospans over $\kappa$. We denote by $\mathsf{HCO}^{\Lambda}(\kappa)$ the full subcategory of $\mathsf{HCO}(\kappa)$ whose objects are the $\Lambda$-bounded filtered cospans. 
	
	The definition of the morphisms in $\mathsf{HCO}(\kappa)$ (and hence also in $\mathsf{HCO}^{\Lambda}(\kappa)$) is deferred to Section \ref{catsect}.
 However, we record here the following useful characterization of which filtered cospans are isomorphic to each other in $\mathsf{HCO}(\kappa)$.
	
	\begin{theorem}\label{introiso}
		Let $\mathcal{C}=\left((C_{\uparrow *},\partial_{\uparrow}^C,\ell_{\uparrow}^C),(C^{\downarrow}_{*},\partial^{\downarrow}_C,\ell^{\downarrow}_C),(D_*,\partial_D),\psi_{\uparrow},\psi^{\downarrow}\right)$  and \\ $\mathcal{X}=\left((X_{\uparrow *},\partial_{\uparrow}^X,\ell_{\uparrow}^{X}),(X^{\downarrow}_{*},\partial^{\downarrow}_{X},\ell^{\downarrow}_{X}),(Y_*,\partial_Y),\phi_{\uparrow},\phi^{\downarrow}\right)$ be two objects of $\mathsf{HCO}(\kappa)$.  Then $\mathcal{C}$ and $\mathcal{X}$ are isomorphic in $\mathsf{HCO}(\kappa)$ if and only if there exists a diagram  \[ \xymatrix{ C^{\downarrow}_{*} \ar[rrr]^{\alpha^{\downarrow}} \ar[rd]_{\psi^{\downarrow}} & & & X^{\downarrow}_{*}\ar[rd]^{\phi^{\downarrow}} & \\ & D_* \ar[rrr]^{\alpha} & & & Y_* \\ C_{\uparrow*} \ar[ru]^{\psi_{\uparrow}} \ar[rrr]_{\alpha_{\uparrow}}  & & & X_{\uparrow*} \ar[ru]_{\phi_{\uparrow}} & }  \] where $\alpha_{\uparrow}$ and $\alpha^{\downarrow}$ are filtered homotopy equivalences, $\alpha$ is a homotopy equivalence, and the two parallelograms commute up to chain homotopy.
	\end{theorem}
	\begin{proof} This is the first part of Corollary \ref{hcoiso}.\end{proof}

We now give the details of the definitions of the singular, simplicial, and Morse filtered cospans mentioned earlier in the introduction.

\begin{ex}\label{singular}
	Let $f\co \mathbb{X}\to  \R$ be a continuous function.  The \textbf{singular filtered cospan} $\mathcal{S}(\mathbb{X},f;\kappa)$ associated to $f$ has all three chain complexes $(C_{\uparrow *},\partial_{\uparrow})$, $(C^{\downarrow}_{*},\partial^{\downarrow})$, and $(D_*,\partial_D)$ equal to the singular chain complex $S_*(\mathbb{X};\kappa)$ of $\mathbb{X}$ with coefficients in $\kappa$.  The chain maps $\psi_{\uparrow}\co C_{\uparrow *}\to D_*$ and $\psi^{\downarrow}\co C^{\downarrow}_{*}\to D_*$ are each taken to be the identity.  
	
	Writing a general element $c$ of $S_k(\mathbb{X})$ as $c=\sum_{i=1}^{n}  r_i\sigma_i$ for $r_i\in \kappa\setminus \{0\}$ and $\sigma_i\co \Delta^k\to \mathbb{X}$ (with the various $\sigma_i$ distinct from one another), we define the filtration functions $\ell_{\uparrow}\co S_*(\mathbb{X};\kappa)\to \R\cup\{-\infty\}$ and $\ell^{\downarrow}\co S_*(\mathbb{X};\kappa)\to \R\cup \{+\infty\}$ by \[ \ell_{\uparrow}(c)=\max\left\{t\in \R\left|(\forall i)\left(\max_{\Delta^k}(f\circ\sigma_i)\leq t\right)\right.\right\},\quad \ell^{\downarrow}(c)=\min\left\{t\in\R\left|(\forall i)\left(\min_{\Delta^k}(f\circ\sigma_i)\geq t\right)\right.\right\} \] (with the maximum of the empty set considered to be $-\infty$, and the minimum of the empty set considered to be $+\infty$). 
	We thus have $C_{\uparrow }^{\leq t}=S_*(f^{-1}((-\infty,t]);\kappa)$ and $C^{\downarrow}_{\geq t}=S_*(f^{-1}([t,\infty));\kappa)$. If $f$ has image in the open interval $(-\Lambda,\Lambda)$ then $\mathcal{S}(\mathbb{X},f;\kappa)$ is $\Lambda$-bounded.
\end{ex}

\begin{ex}  \label{pinsing}
	Now consider a continuous map $f\co \mathbb{X}\to [-\Lambda,\Lambda]$.  Denote \[ \partial^+\mathbb{X}=f^{-1}(\{\Lambda\})\qquad \partial^-\mathbb{X}=f^{-1}(\{-\Lambda\}) \qquad \partial \mathbb{X}=\partial^+\mathbb{X}\cup \partial^-\mathbb{X}.\]  We allow one or both of $\partial^+\mathbb{X}$ and $\partial^-\mathbb{X}$ to be empty; if both of them are empty this example will reduce to Example \ref{singular}.
	
	The \textbf{pinned singular filtered cospan} $\mathcal{S}_{\partial}(\mathbb{X},f;\kappa)$, which is an object of $\mathsf{HCO}^{\Lambda}(\kappa)$, is defined as follows.  The three chain complexes in the definition will be the relative singular chain complexes \[ 
	C_{\uparrow *}=S_*(\mathbb{X}\setminus \partial^+\mathbb{X},\partial^-\mathbb{X};\kappa),\qquad C^{\downarrow}_{*}=S_*(\mathbb{X}\setminus\partial^-\mathbb{X},\partial^+\mathbb{X};\kappa),\qquad D_*=S_*(\mathbb{X},\partial \mathbb{X};\kappa).\]  The maps $\psi_{\uparrow}\co C_{\uparrow *}\to D_*$  and $\psi^{\downarrow}\co C^{\downarrow}_{*}\to D_*$ are those induced by the inclusions of pairs $(\mathbb{X}\setminus \partial^{\pm}\mathbb{X},\partial^{\mp}\mathbb{X})\to (\mathbb{X},\partial\mathbb{X})$; in other words, they are the restrictions to $S_*(\mathbb{X}\setminus\partial^{\pm}\mathbb{X},\partial^{\mp}\mathbb{X};\kappa)$ of the projections $\pi^{\mp}\co \frac{S_*(\mathbb{X};\kappa)}{S_*(\partial^{\mp}\mathbb{X};\kappa)}\to \frac{S_*(\mathbb{X};\kappa)}{S_*(\partial\mathbb{X};\kappa)}$.
	
	A general element of $C_{\uparrow k}$ can then be represented uniquely as a singular chain $c=\sum_{i=1}^{n}r_i\sigma_i$ with $r_i\in \kappa\setminus\{0\}$ and with the $\sigma_i$ being distinct maps $\Delta^k\to \mathbb{X}\setminus \partial^+\mathbb{X}$ such that no image of any $\sigma_i$ is completely contained in $\partial^-\mathbb{X}$.  As in Example \ref{singular} we define in this case $\ell_{\uparrow}(c)=\max\left\{t \left|(\forall i)\left(\max_{\Delta^k}(f\circ\sigma_i)\leq t\right)\right.\right\}$.
	Similarly, a general element of $C^{\downarrow}_{k}$ can be represented as $c=\sum_ir_i\sigma_i$ with the $\sigma_i$ having image disjoint from $\partial^-\mathbb{X}$ and not contained in $\partial^+\mathbb{X}$, and in this case we put $\ell^{\downarrow}(c)=\min\left\{t\left|(\forall i)\left(\min_{\Delta^k}(f\circ\sigma_i)\geq t\right)\right.\right\}$.  Thus, for $t\in [-\Lambda,\Lambda)$, $C_{\uparrow}^{\leq t}=S_*(f^{-1}([-\Lambda,t]),\partial^-\mathbb{X};\kappa)$, and for $t\in (-\Lambda,\Lambda]$, $C^{\downarrow}_{\geq t}=S_*(f^{-1}([t,\Lambda]),\partial^+\mathbb{X};\kappa)$.
\end{ex}

As mentioned earlier, Theorem \ref{H0iso} implies, under appropriate tameness hypotheses on the function $f$, that the interlevel persistence of $f|_{\mathbb{X}\setminus\partial\mathbb{X}}$ can be read off from the isomorphism classification of $\mathcal{S}_{\partial}(\mathbb{X},f;\kappa)$.

\begin{ex}\label{simpdef}
	Let $X$ be a finite abstract simplicial complex with geometric realization $\mathbb{X}$ in some $\R^d$, say with the vertices identified with points $v_0,\ldots,v_N\in \R^d$. (In particular, we have chosen a total order on the vertices.) Suppose that $f\co \mathbb{X}\to \R$ is a simplexwise linear function, with image contained in $[-\Lambda,\Lambda]$. Thus $f$ is determined by its values on the vertices $v_i$.  As in Example \ref{pinsing}, let $\partial^{\pm}\mathbb{X}=f^{-1}(\{\pm\Lambda\})$; for either choice of $+$ or $-$, the faces of $X$ whose vertices all lie in $\partial^{\pm}\mathbb{X}$ form a (possibly empty) subcomplex $\partial^{\pm}X$ of the original abstract simplicial complex $X$.  Write $\partial \mathbb{X}=\partial^{+}\mathbb{X}\cup\partial^-\mathbb{X}$ and $\partial X=\partial^+X\cup\partial^-X$. (These are both disjoint unions, and $\partial X$ is, like $\partial^-X$ and $\partial^+X$, thus a subcomplex of $X$.) 
	Let $X^-$ be the full subcomplex of $X$ consisting of the simplices in $X$ having no vertices in $\partial^+X$, and likewise let $X^+$ be the full subcomplex of $X$ generated by the vertices in $X\setminus \partial^-X$.

	One has an inclusion $\iota$ from the $\kappa$-coefficient simplicial chain complex $C_*(X;\kappa)$ to $S_*(\mathbb{X};\kappa)$, sending a generator of $C_*(X;\kappa)$ that corresponds to a $k$-simplex spanned by, say, $v_{i_0},\ldots,v_{i_k}$ with $i_0<\cdots<i_k$ to the generator for $S_*(\mathbb{X};\kappa)$ given by the map \begin{align*} \Delta^k &\to \mathbb{X} \\ (t_0,\ldots,t_k) &\mapsto \sum_j t_jv_{i_j}.  \end{align*}  This inclusion restricts to inclusions $C_*(X^-;\kappa)\to S_*(\mathbb{X}\setminus \partial \mathbb{X}^+;\kappa)$, $C_*(X^+;\kappa)\to S_*(\mathbb{X}\setminus \partial \mathbb{X}^-;\kappa)$, $C_*(\partial^-X;\kappa)\to S_*(\partial^-\mathbb{X};\kappa)$, and $C^*(\partial^+X;\kappa)\to S_*(\partial^+\mathbb{X};\kappa)$.  
	
	We form the \text{pinned simplicial filtered cospan} $\mathcal{C}_{\partial}(X,f;\kappa)$ as follows.  We take \[ C_{\uparrow *}=\frac{C_{*}(X^-;\kappa)}{C_*(\partial^-X;\kappa)},\qquad C^{\downarrow}_{*}=\frac{C_*(X^+;\kappa)}{C_*(\partial^+X;\kappa)},\qquad D_*=\frac{C_*(X;\kappa)}{C_*(\partial X;\kappa)}. \] The maps $\psi_{\uparrow}\co C_{\uparrow *}\to D_*$ and $\psi^{\downarrow}\co C^{\downarrow}_{*}\to D_*$ are the obvious ones induced by inclusions.  
	
	The inclusion $\iota$ induces maps $\iota_{\uparrow}\co C_{\uparrow *}\to S_*(\mathbb{X}\setminus\partial^+\mathbb{X},\partial^-\mathbb{X};\kappa)$ and $\iota^{\downarrow}\co C^{\downarrow}_{*}\to S_*(\mathbb{X}\setminus \partial^-\mathbb{X},\partial^+\mathbb{X}:\kappa)$, and the filtration functions $\ell_{\uparrow},\ell^{\downarrow}$ are taken to be equal to the pullbacks via $\iota_{\uparrow}$ and $\iota^{\downarrow}$ of the corresponding functions on the singular filtered cospan $\mathcal{S}(\mathbb{X},\partial\mathbb{X},f;\kappa)$ from Example \ref{pinsing}.
	Thus $C_{\uparrow}^{\leq t}$ is the subcomplex of $\frac{C_*(X^-:\kappa)}{C_*(\partial^-X;\kappa)}$ spanned the images under the quotient projection of just those simplices whose vertices $v_i$ all have $f(v_i)\leq t$, and similarly for $C^{\downarrow}_{\geq t}$.  
	
	In Proposition \ref{simpsing} we show, based on Theorem \ref{introiso}, that the inclusion $\iota$ gives rise to an isomorphism in the category $\mathsf{HCO}^{\Lambda}(\kappa)$ between $\mathcal{C}_{\partial}(X,f;\kappa)$ and $\mathcal{S}_{\partial}(\mathbb{X},f;\kappa)$. Together with Theorem \ref{H0iso}, this allows the interlevel persistence of $f|_{\mathbb{X}\setminus\partial\mathbb{X}}$ to be computed from the isomorphism type of $\mathcal{C}_{\partial}(X,f;\kappa)$, which in turn can be determined algorithmically following the outline at the start of Section \ref{calcsect}.
\end{ex}

\begin{ex}\label{morseex}
	Let $\mathbb{X}$ be a compact smooth manifold with boundary $\partial \mathbb{X}$ decomposed into disjoint closed subsets as $\partial \mathbb{X}=\partial^-\mathbb{X}\cup\partial^+\mathbb{X}$.  Consider Morse functions $f\co \mathbb{X}\to [-\Lambda,\Lambda]$ with $-\Lambda$ and $\Lambda$ being regular values of $f$, such that $f^{-1}(\{-\Lambda\})=\partial^-\mathbb{X}$ and $f^{-1}(\{\Lambda\})=\partial^+\mathbb{X}$.  We allow the possibility that $\partial^-\mathbb{X}$ and/or $\partial^-\mathbb{X}$ is empty.  

	A generic choice of gradient-like vector field $v$ for $f$, together (if $2\neq 0$ in $\kappa$) with an arbitrary choice of orientations for the descending manifolds of $v$, gives rise to a Morse chain complex $CM_*(f,v;\kappa)$ over the field $\kappa$; $CM_k(f,v;\kappa)$ is freely generated by the index-$k$ critical points for $f$, and the matrix elements for the boundary operator count isolated unparametrized trajectories for $-v$ that are asymptotic at their ends to critical points of index differing by one.  For the details of the definition we follow the conventions of \cite[Chapter 6]{Paj}, where this complex is denoted $\mathcal{M}_*(f,v)$.  The complex $CM_*(f,v;\kappa)$ becomes an ascending chain complex by putting $\ell_{\uparrow}\left(\sum_i r_ip_i\right)=\max\{f(p_i)|r_i\neq 0\}$ for distinct critical points $p_i$ and for $r_i\in \kappa$.  We can likewise form $CM_*(-f,-v;\kappa)$ which becomes a descending complex by putting $\ell^{\downarrow}\left(\sum_ir_ip_i\right)=\min \{f(p_i)|r_i\neq 0\}$.
	
	We form the \textbf{Morse filtered cospan} of $f$, denoted $\mathcal{M}(\mathbb{X},f,v;\kappa)$, as follows: \begin{itemize} 
		\item $C_{\uparrow *}=CM_*(f,v;\kappa)$;
		\item $C^{\downarrow}_{*}=CM_*(-f,-v;\kappa)$;
		\item $D_*=S_*(\mathbb{X},\partial\mathbb{X};\kappa)$;
		\item $\psi_{\uparrow}\co CM_*(f,v;\kappa)\to S_*(\mathbb{X},\partial\mathbb{X})$ is the composition of the chain homotopy equivalence $\mathcal{E}(f,v)\co CM_*(f,v;\kappa)\to S_*(\mathbb{X},\partial^-\mathbb{X};\kappa)$ from \cite[p. 218]{Paj} with the map $ S_*(\mathbb{X},\partial^-\mathbb{X};\kappa)\to S_*(\mathbb{X},\partial\mathbb{X};\kappa)$ induced by inclusion of pairs.
		\item $\psi^{\downarrow}\co CM_*(-f,-v;\kappa)\to S_*(\mathbb{X},\partial\mathbb{X};\kappa)$ is the composition of the chain homotopy equivalence $\mathcal{E}(-f,-v)\co CM_*(-f,-v;\kappa)\to S_*(\mathbb{X},\partial^+\mathbb{X};\kappa)$ from \cite[p. 218]{Paj} with the map $ S_*(\mathbb{X},\partial^+\mathbb{X};\kappa)\to S_*(\mathbb{X},\partial\mathbb{X};\kappa)$ induced by inclusion of pairs.
\end{itemize}
The maps $\mathcal{E}(f,v)$ and $\mathcal{E}(-f,-v)$ are specified only up to homotopy in \cite{Paj}, but Proposition \ref{htopicpsi} shows that changing $\psi_{\uparrow}$ and $\psi^{\downarrow}$ within their homotopy classes does not affect isomorphism classes in $\mathsf{HCO}^{\Lambda}(\kappa)$, so the above specifies $\mathcal{M}(X,f,v;\kappa)$ uniquely up to isomorphism.  In fact, we will show in Proposition \ref{morseiso} that $\mathcal{M}(X,f,v;\kappa)$ is isomorphic to the pinned singular cospan $\mathcal{S}_{\partial}(\mathbb{X},f;\kappa)$.  

An appealing feature of the Morse filtered cospan is that the filtered complexes $C_{\uparrow *}$ and $C^{\downarrow}_{*}$  are often rather small.  On the other hand, in the definition given above $D_*=S_*(\mathbb{X},\partial\mathbb{X};\kappa)$ is quite large; however, it follows from Theorem \ref{introiso} that $D_*$ can be replaced, without changing the isomorphism type of the filtered cospan, by any complex that is homotopy equivalent to $D_*$, if one composes $\psi_{\uparrow}$ and $\psi^{\downarrow}$ with the homotopy equivalence.  So we could equally well use, in the role of $D_*$, smaller complexes such as the singular complex of a relative triangulation of $(\mathbb{X},\partial\mathbb{X})$, or the Morse complex of a Morse function $g\co \mathbb{X}\to [0,\infty)$ having $g^{-1}(\{0\})=\partial \mathbb{X}$.  

In the case that $\partial^-\mathbb{X}=\varnothing$, the map $\psi^{\downarrow}\co CM_*(-f,-v;\kappa)\to S_*(\mathbb{X},\partial\mathbb{X};\kappa)=S_*(\mathbb{X},\partial^+\mathbb{X};\kappa)$ is a homotopy equivalence. So in this case the Morse filtered cospan is isomorphic to a filtered cospan of the form \[ \xymatrix{ CM_*(-f,-v;\kappa)\ar@{=}[rd] & \\ &  CM_*(-f,-v;\kappa) \\ CM_*(f,v;\kappa)\ar[ru] & } \] where the only nontrivial arrow induces on homology a map which can be identified with the canonical map $H_*(\mathbb{X};\kappa)\to H_*(\mathbb{X},\partial\mathbb{X};\kappa)$. 
\end{ex}

\subsection{Organization of the rest of the paper, and further discussion}

The forthcoming Section \ref{catsect} completes the definition of the category $\mathsf{HCO}(\kappa)$, obtaining it as the zeroth cohomology category of a dg-category $\mathsf{CO}(\kappa)$, and then, based upon this definition, establishes the criterion for isomorphism in $\mathsf{HCO}(\kappa)$ stated above as Theorem \ref{introiso}.

The main result of Section \ref{decompsect} is Corollary \ref{decompexists}, which asserts that any ``admissible'' filtered cospan is isomorphic in $\mathsf{HCO}(\kappa)$ to a direct sum of a collection of the ``standard elementary summands'' enumerated in Definition \ref{sesdef}.  The admissibility hypothesis (Definition \ref{admdef}) is satisfied for instance if the ascending and descending chain complexes $C_{\uparrow *}$ and $C^{\downarrow}_{*}$ in the filtered cospan are finite-dimensional in each degree, as is the case for the pinned simplicial and Morse filtered cospans considered in this paper.  A natural source of filtered cospans \emph{not} satisfying this property would arise from the Morse complexes of Morse functions $f\co \mathbb{X}\to \R$ having critical value sets that are discrete and bounded below but unbounded above; these occur, for example, in the classical Morse theory of geodesics and in symplectic homology. Natural  filtered cospans can be constructed in this situation which are not admissible but (using appropriate truncations of the descending chain complex $C^{\downarrow}_{*}$) are inverse limits of admissible filtered cospans.  We leave the development of this to future work.     

Section \ref{quasisect} establishes the isomorphisms (Propositions \ref{simpsing} and \ref{morseiso}) in $\mathsf{HCO}(\kappa)$ between, respectively, the pinned simplicial and Morse filtered cospans and the pinned singular filtered cospan.  The algebraic ingredients for this are Theorem \ref{introiso} and Proposition \ref{quasihtopy}, the latter of which relates filtered homotopy equivalences (as appear in Theorem \ref{introiso}) to filtered quasi-isomorphisms.  The relation between the simplicial and singular filtered cospans then follows readily from the standard isomorphism between simplicial and singular homology.  The relation between the Morse and singular filtered cospans is obtained by a formally similar argument, but some technical ingredients concerning the isomorphism between Morse and singular homology are required for this; these are deferred to Appendix \ref{app}.

Section \ref{calcsect} provides several examples of the calculation of decompositions into standard elementary summands (as in Corollary \ref{decompexists}) of the pinned simplicial or Morse filtered cospans of various functions (and hence, by Propositions \ref{simpsing} and \ref{morseiso}, of their pinned singular filtered cospans). As mentioned earlier, this includes a general consideration in Section \ref{elemsect} of the Morse filtered cospans associated to Morse functions $f\co\mathbb{X}\to [-2,2]$ with only one critical point.  In this case, if $\dim \mathbb{X}=n$ and the critical point has index $k$, then $\mathbb{X}$ can be obtained from an $(n-1)$-dimensional manifold with boundary $Y$ such that $\partial Y\cong S^{k-1}\times S^{n-k-1}$, by gluing a neighborhood of the critical point to $Y\times [-2,2]$.  The upper and lower boundaries $\partial^+\mathbb{X}$ and $\partial^-\mathbb{X}$ are then obtained by filling  $Y$ by gluing, respectively, $D^k\times S^{n-k-1}$ or $S^{k-1}\times D^{n-k}$ to its boundary.  In Corollary \ref{morsesummary}, we show how to read off the isomorphism class of the Morse filtered cospan from whether the respective sphere factors of $\partial Y\cong S^{k-1}\times S^{n-k-1}$ are homologically essential in $Y$.

While Sections \ref{catsect} through \ref{calcsect} are entirely concerned with filtered cospans in their own right, in Section \ref{funsect} we build a bridge between filtered cospans and persistence modules.  This takes the form of a functor $\mathcal{F}$ from $\mathsf{HCO}^{\Lambda}(\kappa)$ to the category $\mathsf{K}(\kappa)^{\mathbb{M}}$ of $\mathbb{M}$-indexed persistence modules with values in the homotopy category of chain complexes $\mathsf{K}(\kappa)$.  For each $(s,t)\in \mathbb{M}$, we first construct in Section \ref{f0sec} a functor $\mathcal{F}_{(s,t)}^{0}$ from the dg-category $\mathsf{CO}^{\Lambda}(\kappa)$ to the dg-category $\mathsf{Ch}(\kappa)$ of chain complexes by using mapping cones associated to filtered subcomplexes, depending on $(s,t)$, of the constituent chain complexes of a filtered cospan.  In Section \ref{transf} we define transformations between these functors in various cases when $(s,t)\preceq (s',t')$ that become natural upon passing to the zeroth cohomology categories $\mathsf{HCO}^{\Lambda}(\kappa)$ and $\mathsf{K}(\kappa)$, and in Section \ref{fsect} we show how these assemble to give the functor $\mathcal{F}\co \mathsf{HCO}^{\Lambda}(\kappa)\to \mathsf{K}(\kappa)^{\mathbb{M}}$.

Denoting by $\mathsf{Vect}_{\kappa}$ the category of vector spaces over $\kappa$, one has a covariant functor $H_0\co \mathsf{K}(\kappa)\to \mathsf{Vect}_{\kappa}$ and a contravariant functor $H^0\co \mathsf{K}(\kappa)\to \mathsf{Vect}_{\kappa}$, given respectively by the zeroth homology and the zeroth cohomology of a chain complex.  The category $\mathsf{K}(\kappa)$ is triangulated, and $H_0$ and $H^0$ both have the general property of mapping distinguished triangles in $\mathsf{K}(\kappa)$ to exact sequences.  In \cite{BBF21}, it is noted that the relative interlevel set homology and cohomology of a function, as functors $\mathbb{M}\to\mathsf{Vect}_{\kappa}$ and $\mathbb{M}^{op}\to\mathsf{Vect}_{\kappa}$, are, in their language, ``homological'' and ``cohomological'' in the sense that certain sequences associated to rectangles in $\mathbb{M}$ are always exact.  Section \ref{dist} gives what could be considered an algebraic account of this, showing that for any filtered cospan $\mathcal{C}$, the functor $\mathcal{F}(\mathcal{C})\co \mathbb{M}\to\mathsf{K}(\kappa)$ has the property that certain triangles associated to rectangles in $\mathbb{M}$ are distinguished; the exactness of the sequences as in \cite{BBF21} then follows directly from the fact that $H_0$ and $H^0$ map distinguished triangles to exact sequences.  See Theorem \ref{distmain} and Corollary \ref{cohfunct}.

From Section \ref{elemfun} to the end of the paper, we generally focus on the functor $H_0\circ \mathcal{F}\co \mathsf{HCO}^{\Lambda}(\kappa)\to \mathsf{Vect}_{\kappa}^{\mathbb{M}}$, rather than $\mathcal{F}\co \mathsf{HCO}^{\Lambda}(\kappa)\to\mathsf{K}(\kappa)^{\mathbb{M}}$.  (Of course, we could equally well consider $H^0\circ \mathcal{F}\co \mathsf{HCO}^{\Lambda}(\kappa)\to \mathsf{Vect}_{\kappa}^{\mathbb{M}^{op}}$, giving rise to cohomology rather than homology, with similar results.)  In Section \ref{elemfun} we determine how $H_0\circ\mathcal{F}$ acts on the standard elementary summands of Section \ref{decompsect}, showing in Proposition \ref{blockclass} that $H_0\circ\mathcal{F}$ sets up an explicit bijection between these basic filtered cospans and the block modules $B^v$ of \cite{BBF21}, as $v$ varies through $\mathbb{M}\setminus\partial\mathbb{M}$.  The persistence diagram of a $\Lambda$-bounded admissible filtered cospan $\mathcal{C}$ is accordingly defined to be the submultiset of $\mathbb{M}\setminus\partial\mathbb{M}$ indexing the block modules that appear in the direct sum decomposition of $H_0\circ\mathcal{F}(\mathcal{C})$.  

Section \ref{rishsect} contains the proof of Theorem \ref{H0iso}, asserting that, under modest hypotheses on the function $f\co \mathbb{X}\to [-\Lambda,\Lambda]$, the image of the pinned singular filtered cospan $\mathcal{S}_{\partial}(\mathbb{X},f;\kappa)$ under $H_0\circ\mathcal{F}$ is isomorphic to the relative interlevel set homology, as in \cite{BBF20}, of $f|_{\mathbb{X}\setminus\partial\mathbb{X}}$.  We follow \cite{BBF20}---and not \cite{BBF21}---by basing the relative interlevel set homology on preimages of closed intervals rather than open ones (and also by using homology rather than cohomology).  If one prefers to work with open intervals, one could appropriately modify the functor $\mathcal{F}$, for example by replacing appearances of $\mathcal{C}^{\downarrow}_{\geq x}$ and $C_{\uparrow}^{\leq y}$ in the definition of $\mathcal{F}^0_{T^k(x,y)}$ in (\ref{mdecomp})  by $\mathcal{C}^{\downarrow}_{>x}$ and $C_{\uparrow}^{<y}$,  and by changing the definition of $S$ in Section \ref{stripsect} to $[-\Lambda,\Lambda)\times (-\Lambda,\Lambda]$ (instead of $(-\Lambda,\Lambda]\times [-\Lambda,\Lambda)$) and making corresponding adjustments to $L$ and $A$.  One could then prove an analogue of Theorem \ref{H0iso} for relative interlevel set homology (or, using $H^0\circ\mathcal{F}$ in place of $H_0\circ\mathcal{F}$, cohomology) based on open intervals instead of closed ones; indeed, the proof would be slightly easier and would not require a tameness hypothesis on $f$ because certain triads would automatically be excisive.   However, this author has a mild preference for closed interlevel persistence, at least for the well-behaved functions that this paper is largely concerned with, because it manifestly contains within it the level set persistence.  By combining Theorem \ref{H0iso} with the calculations in Section \ref{elemfun}, we show in Proposition \ref{barclass} how to read off the level set persistence barcode of $f|_{\mathbb{X}\setminus\partial\mathbb{X}}$ from the decomposition into standard elementary summands of the pinned singular filtered cospan of $f$ (equivalently, when defined, of the pinned simplicial or Morse filtered cospan).  This is helpful for visualizing the significance of the calculations in Section \ref{calcsect}.

Finally, Section \ref{stabsect} proves an algebraic isometry theorem.  For any $\Lambda$-bounded filtered cospan $\mathcal{C}$, we have a persistence diagram $\mathcal{D}(\mathcal{C})$, based upon the decomposition of $H_0\circ\mathcal{F}(\mathcal{C})$ into block modules, which is a multiset of elements of $\mathbb{M}\setminus\partial\mathbb{M}$.  A choice of increasing homeomorphism $\varphi\co [-\infty,\infty]\to [-\Lambda,\Lambda]$ gives rise to a family of translations of the interval $(-\Lambda,\Lambda)$ and consequently to flows both on the category $\mathsf{HCO}^{\Lambda}(\kappa)$ and on the strip $\mathbb{M}$.  The flow on $\mathbb{M}$ is similar to one described in \cite[Section 4]{BBF21}, but our description (see (\ref{mflow})) is simpler, partly as an artifact of our use of a different fundamental domain for the action of the glide-reflection $T\co\mathbb{M}\to\mathbb{M}$ that corresponds to shifting homological degree.  From these flows we define an interleaving distance $d_{\mathrm{int}}^{\mathcal{T}^{\varphi}}$ on $\mathsf{HCO}^{\Lambda}(\kappa)$, an extended metric $d_{\mathrm{int}}^{\mathcal{T}_{\mathbb{M}}^{\varphi}}$ on $\mathbb{M}$, and a bottleneck distance $d_{B,\mathcal{T}_{\mathbb{M}}^{\varphi}}^{\mathbb{M},\partial\mathbb{M}}$ on multisets of elements of $\mathbb{M}\setminus\partial\mathbb{M}$, and we prove:

\begin{theorem}\label{isometry}
	Let $\mathcal{C}$ and $\mathcal{X}$ be two $\Lambda$-bounded filtered cospans that are of finite type.  Then $d_{B,\mathcal{T}_{\mathbb{M}}^{\varphi}}^{\mathbb{M},\partial\mathbb{M}}(\mathcal{D}(\mathcal{X}),\mathcal{D}(\mathcal{C}))= d_{\mathrm{int}}^{\mathcal{T}^{\varphi}}(\mathcal{X},\mathcal{C})$.
\end{theorem} 
\begin{proof}
	One inequality (``stability'') is the content of Theorem \ref{mainstab}, and the other (``converse stability'') is established in Proposition \ref{convstab}.
\end{proof}

Here a finite-type filtered cospan is one that is isomorphic in $\mathsf{HCO}(\kappa)$ to one in which the chain complexes $C_{\uparrow *}$ and $C^{\downarrow}_{*}$ are finite-dimensional.  It is likely that the finite-type hypothesis can be weakened somewhat.

If $\mathcal{C}$ and $\mathcal{X}$ are the pinned singular filtered cospans associated to functions $f,g\co \mathbb{X}\to[-\Lambda,\Lambda]$, one has $d_{\mathrm{int}}^{\mathcal{T}^{\varphi}}(\mathcal{X},\mathcal{C})\leq \|\varphi^{-1}\circ f-\varphi^{-1}\circ g\|_{\infty}$, and so one obtains a stability result for the associated persistence diagrams in terms of $ \|\varphi^{-1}\circ f-\varphi^{-1}\circ g\|_{\infty}$ (Corollary \ref{fnstable}).

Our eight types of standard elementary summands are sent by $H_0\circ \mathcal{F}$ to block modules corresponding to eight types of regions in $\mathbb{M}\setminus \partial\mathbb{M}$ (images under a power of $T$ of the sets in (\ref{regionlist}), excluding $\mathfrak{D}\cap\partial\mathbb{M}$); by Proposition \ref{sepregions}, these regions lie an infinite distance away from one another according to the extended metric $d_{\mathrm{int}}^{\mathcal{T}_{\mathbb{M}}^{\varphi}}$, and only two of them (corresponding to summands of form $(\uparrow_a^b)_k$ and $(\downarrow^a_b)_k$, or to the ``ordinary'' and ``relative'' subdiagrams of extended persistence) lie a finite distance from $\partial\mathbb{M}$.  Hence, for any of the remaining six types of regions, the total multiplicity of points in a persistence diagram within that region is preserved under the matchings involved in the definition of the bottleneck distance $d_{B,\mathcal{T}_{\mathbb{M}}^{\varphi}}^{\mathbb{M},\partial\mathbb{M}}$.  In particular, for example, points corresponding to the extended subdiagram of extended persistence (corresponding in our notation to $T^kS^{\circ}$) must be matched to other such points, consistently with \cite[Remark 2.5]{BBF20}; in the context of level set persistence this corresponds to the statement that a closed interval in a continuous family of level set barcodes in degree $k$ can shrink to a singleton and then be replaced by an open interval in degree $k-1$, see the remark after Theorem 1.5 in \cite{BH}.  In terms of our persistence diagrams, this process corresponds a point in $T^kS^{\circ}$ moving from $T^{-k}S^-$ into $T^{-k}S^+$, where $S^-$ and $S^+$ are as defined shortly after the statement of Theorem \ref{H0iso}.  

\subsection*{Acknowledgements}  The author thanks Cameron Thomas for clarifying discussions.  This project was partly motivated by questions of Mihai Damian and Claude Viterbo at the 2023 Conference on Persistence Homology in Symplectic and Contact Topology in Albi.  The work is supported by the Simons Foundation through award MPS-TSM-8091.

\section{The category of filtered cospans}\label{catsect}

As indicated in the introduction, the category $\mathsf{HCO}(\kappa)$ is to have as its objects the filtered cospans over $\kappa$, as   in Definition \ref{codfn}.
Before stating what the morphisms in $\mathsf{HCO}(\kappa)$ are, we digress in order to introduce some background. 

Recall from \cite{BK} that a dg-category is a preadditive category $\mathcal{C}$ in which the morphism groups $\mathrm{Hom}_{\mathcal{C}}(X,Y)$ are each endowed with the structure of a cochain complex, say with coboundary operator $\delta_{X,Y}$, such that composition defines a cochain map $\mathrm{Hom}_{\mathcal{C}}(Y,Z)\otimes \mathrm{Hom}_{\mathcal{C}}(X,Y)\to \mathrm{Hom}_{\mathcal{C}}(X,Z)$, and such that each identity element $\mathrm{id}_X$ has $\delta_{X,X}(\mathrm{id}_X)=0$. 

A standard example of a dg-category is the category $\mathsf{Ch}(\kappa)$ of chain\footnote{We will generally take our \emph{objects} to be chain complexes, even while following the usual convention that the morphism spaces in a dg-category are cochain complexes.} complexes $\left(C_*=\oplus_{k\in \Z}C_k,\partial_C \right)$ over  the field $\kappa$, with the $\kappa$-vector space of degree-$m$ morphisms $\mathrm{Hom}_{\mathsf{Ch}(\kappa)}^{m}\left((C_*,\partial_C),(D_*,\partial_D)\right)$ consisting of all $\kappa$-linear maps $C_*\to D_*$ of homogeneous grading $-m$, \emph{i.e.}, \[ \mathrm{Hom}_{\mathsf{Ch}(\kappa)}^{m}\left((C_*,\partial_C),(D_*,\partial_D)\right)=\prod_{j\in \Z}\mathrm{Hom}_{\kappa}(C_j,D_{j-m}).\]  The differential $\delta_{C,D}\co \mathrm{Hom}_{\mathsf{Ch}(\kappa)}^{m}\left((C_*,\partial_C),(D_*,\partial_D)\right)\to \mathrm{Hom}_{\mathsf{Ch}(\kappa)}^{m+1}\left((C_*,\partial_C),(D_*,\partial_D)\right)$ is then given by \[ \delta_{C,D} f=\partial_D\circ f-(-1)^mf\circ\partial_C. \]
So the degree-zero elements $f$ satisfying $\delta_{C,D} f=0$ are precisely the chain maps $(C_*,\partial_C)\to (D_*,\partial_D)$, and the zeroth cohomology $H^0(\mathrm{Hom}_{\mathsf{Ch}(\kappa)}^{*}\left((C_*,\partial_C),(D_*,\partial_D)\right))$ is the $\kappa$-vector space of chain maps modulo chain homotopy.  Thus the zeroth cohomology category of $\mathsf{Ch}(\kappa)$ (\emph{i.e.}, the category with the same objects as $\mathsf{Ch}(\kappa)$ but with morphisms from $X$ to $Y$ given by the zeroth cohomology of the cochain complex $\mathrm{Hom}^*_{\mathsf{Ch}(\kappa)}(X,Y)$, with the obvious composition) is precisely the homotopy category of chain complexes $\mathsf{K}(\kappa)$.

Our category $\mathsf{HCO}(\kappa)$ will likewise be the zeroth cohomology of a dg-category, $\mathsf{CO}(\kappa)$.  So $\mathsf{CO}(\kappa)$  also has as its objects the filtered cospans over $\kappa$, any of which consists of data $\mathcal{C}=\left((C_{\uparrow *},\partial_{\uparrow}^C,\ell_{\uparrow}^C),(C^{\downarrow}_{*},\partial^{\downarrow}_C,\ell^{\downarrow}_C),(D_*,\partial_D),\psi_{\uparrow},\psi^{\downarrow}\right)$ as in Definition \ref{codfn}. 
A degree-$m$ morphism in $\mathsf{CO}(\kappa)$ from $\mathcal{C}$ to another filtered cospan \[\mathcal{X}=\left((X_{\uparrow *},\partial_{\uparrow}^X,\ell_{\uparrow}^{X}),(X^{\downarrow}_{*},\partial^{\downarrow}_{X},\ell^{\downarrow}_{X}),(Y_*,\partial_Y),\phi^{\downarrow},\phi_{\uparrow}\right)\] will be a quintuple \begin{align}\label{quint} \left(\alpha^{\downarrow},\alpha_{\uparrow},\alpha,K^{\downarrow},K_{\uparrow}\right)\in & \mathrm{Hom}_{\mathsf{Ch}(\kappa)}^m(C^{\downarrow}_{*},X^{\downarrow}_{*})\times \mathrm{Hom}_{\mathsf{Ch}(\kappa)}^m(C_{\uparrow*},X_{\uparrow*})\times\mathrm{Hom}_{\mathsf{Ch}(\kappa)}^m(D_{*},Y_{*})\\ & \quad \times \mathrm{Hom}_{\mathsf{Ch}(\kappa)}^{m-1}(C^{\downarrow}_{*},Y_{*})\times\mathrm{Hom}_{\mathsf{Ch}(\kappa)}^{m-1}(C_{\uparrow*},Y_{*}) \nonumber \end{align}
such that \begin{equation}\label{filtpres} \ell_{\uparrow}^{X}(\alpha_{\uparrow}c)\leq \ell_{\uparrow}^{C}(c)  \mbox{\,\, for all $c\in C_{\uparrow *}$, and\,\,\,}  \ell^{\downarrow}_{X}(\alpha^{\downarrow}c)\geq \ell^{\downarrow}_{C}(c) \quad \mbox{for all $c\in C^{\downarrow}_{*}$}\end{equation}

Thus we have a diagram \[ \xymatrix{ C^{\downarrow}_{*} \ar[rrr]^{\alpha^{\downarrow}}\ar[rrrrd]_{K^{\downarrow}} \ar[rd]_{\psi^{\downarrow}} & & & X^{\downarrow}_{*}\ar[rd]^{\phi^{\downarrow}} & \\ & D_* \ar[rrr]^{\alpha} & & & Y_* \\ C_{\uparrow*} \ar[ru]^{\psi_{\uparrow}} \ar[rrr]_{\alpha_{\uparrow}} \ar[rrrru]_{K_{\uparrow}} & & & X_{\uparrow*} \ar[ru]_{\phi_{\uparrow}} & },  \] with $\alpha^{\downarrow},\alpha,\alpha_{\uparrow}$ each affecting grading on the respective chain complexes by $-m$, and $K^{\downarrow}$ and $K_{\uparrow}$ affecting grading by $1-m$. As stated in (\ref{filtpres}), $\alpha^{\downarrow}$ and $\alpha_{\uparrow}$ are required to respect the (respectively, descending and ascending) filtrations on the appropriate complexes; the other three maps have codomain the unfiltered complex $(Y_*,\partial_Y)$ and so no filtration condition is imposed on them. Composition of morphisms is defined by 
\begin{align}\label{composition} & (\beta^{\downarrow},\beta_{\uparrow},\beta,L^{\downarrow},L_{\uparrow})\circ (\alpha^{\downarrow},\alpha_{\uparrow},\alpha,K^{\downarrow},K_{\uparrow})\\ & \quad =\left(\beta^{\downarrow}\circ\alpha^{\downarrow},\,\beta_{\uparrow}\circ\alpha_{\uparrow},\,\beta\circ\alpha,\,L^{\downarrow}\circ\alpha^{\downarrow}+\beta\circ K^{\downarrow},\,L_{\uparrow}\circ\alpha_{\uparrow}+\beta\circ K_{\uparrow}\right).  \nonumber\end{align}  Note that this operation has $(1_{C^{\downarrow}_{*}},1_{C_{\uparrow *}},1_{D_*},0,0)$ as the identity in $\mathrm{Hom}^{0}_{\mathsf{CO}(\kappa)}(\mathcal{C},\mathcal{C})$.

The dg-category structure is given by the differential $\delta_{\mathcal{C},\mathcal{X}}\co \mathrm{Hom}^{m}_{\mathsf{CO}(\kappa)}(\mathcal{C},\mathcal{X})\to\mathrm{Hom}^{m+1}_{\mathsf{CO}(\kappa)}(\mathcal{C},\mathcal{X})$  defined by the formula 

\begin{align}\label{diffl} &\delta_{\mathcal{C},\mathcal{X}} \left(\alpha^{\downarrow},\alpha_{\uparrow},\alpha,K^{\downarrow},K_{\uparrow}\right)=\left(-\partial^{\downarrow}_{X}\circ\alpha^{\downarrow}+(-1)^m\alpha^{\downarrow}\circ\partial^{\downarrow}_{C},\,-\partial_{\uparrow}^{X}\circ\alpha_{\uparrow}+(-1)^m\alpha_{\uparrow}\circ\partial_{\uparrow}^{C},\,\partial_Y\circ\alpha-(-1)^m\alpha\circ\partial_D,  \right. \\  & \quad \left. \phi^{\downarrow}\circ\alpha^{\downarrow}-(-1)^m\alpha\circ\psi^{\downarrow}+\partial_Y\circ K^{\downarrow}+(-1)^mK^{\downarrow}\circ\partial^{\downarrow}_{C}, \,
\phi_{\uparrow}\circ\alpha_{\uparrow}-(-1)^m\alpha\circ\psi_{\uparrow}+\partial_Y\circ K_{\uparrow}+(-1)^mK_{\uparrow}\circ\partial_{\uparrow}^{C}	
	 \right).  \nonumber \end{align} Thus an element $(\alpha^{\downarrow},\alpha_{\uparrow},\alpha,K^{\downarrow},K_{\uparrow})\in \mathrm{Hom}^{0}_{\mathsf{CO}(\kappa)}(\mathcal{C},\mathcal{X})$ belongs to the kernel of $\delta_{\mathcal{C},\mathcal{X}}$ if and only if $\alpha^{\downarrow},\alpha_{\uparrow},$ and $\alpha$ are each chain maps, $K^{\downarrow}$ is a chain homotopy from $\phi^{\downarrow}\circ\alpha^{\downarrow}$ to $\alpha\circ\psi^{\downarrow}$, and $K_{\uparrow}$ is a chain homotopy from $\phi_{\uparrow}\circ\alpha_{\uparrow}$ to $\alpha\circ\psi_{\uparrow}$.

\begin{remark}
The following may help put the formulas (\ref{composition}) and (\ref{diffl}) into context.
	Recall that the mapping cone $\mathrm{Cone}(f)$ of a chain map $f\co C_*\to D_*$  between two chain complexes of $\kappa$-vector spaces is the chain complex given as a graded vector space by $\mathrm{Cone}(f)_k=C_{k-1}\oplus D_k$, with boundary operator $\partial_{\mathrm{Cone}(f)}(c,d)=(-\partial_Cc,fc+\partial_Dd)$.  Analogously to \cite[2.9]{Drin}, one obtains a dg-functor $\mathcal{S}^1\co \mathsf{CO}(\kappa)\to \mathsf{Ch}(\kappa)$ which, at the level of objects, sends \[ \xymatrix{ C^{\downarrow}_{*}\ar[rd]^{\psi^{\downarrow}} & & \\ & D_* & \quad \mbox{to}\quad \mathrm{Cone}\left(-\psi^{\downarrow}+\psi_{\uparrow}\co C^{\downarrow}_{*}\oplus C_{\uparrow *}\to D_* \right). \\ C_{\uparrow *}\ar[ru]_{\psi_{\uparrow}} & &}. \]

This functor will appear again in Section \ref{transf}.  At the level of morphisms, we take $\mathcal{S}^1$ to send a degree-$m$ morphism $\left(\alpha^{\downarrow},\alpha_{\uparrow},\alpha,K^{\downarrow},K_{\uparrow}\right)\co \mathcal{C}\to\mathcal{X}$ as above to the map $\mathrm{Cone}(-\psi^{\downarrow}+\psi_{\uparrow})\to \mathrm{Cone}(-\phi^{\downarrow}+\phi_{\uparrow})$ given in block form (as a map $C^{\downarrow}_{*-1}\oplus C_{\uparrow*-1}\oplus D_*\to X^{\downarrow}_{*-m-1}\oplus X_{\uparrow*-m-1}\oplus X_{*-m}$) by \[ \left(\begin{array}{ccc} \alpha^{\downarrow} & 0 & 0 \\ 0 & \alpha_{\uparrow} & 0 \\ -K^{\downarrow} & K_{\uparrow} & \alpha 
\end{array}\right).\]

The formulas (\ref{composition}) and (\ref{diffl}) for the composition law and the differential in $\mathsf{CO}(\kappa)$ are designed precisely so that the following proposition will hold; its proof is a routine computation that we leave to the reader.

\begin{prop}\label{conefunctor}
	For morphisms $\mathfrak{a}=\left(\alpha^{\downarrow},\alpha_{\uparrow},\alpha,K^{\downarrow},K_{\uparrow}\right)\co \mathcal{C}\to\mathcal{X}$ and\\ $\mathfrak{b}=\left(\beta^{\downarrow},\beta_{\uparrow},\beta,L^{\downarrow},L_{\uparrow}\right)\co \mathcal{X}\to\mathcal{Z}$ in the dg-category $\mathsf{CO}(\kappa)$, we have \[ \mathcal{S}^1(\mathfrak{b}\circ\mathfrak{a})=\mathcal{S}^1(\mathfrak{b})\circ \mathcal{S}^1(\mathfrak{a})\quad \mbox{and}\quad \mathcal{S}^1(\delta_{\mathcal{C},\mathcal{X}}\mathfrak{a})=\delta_{\mathcal{S}^1(\mathcal{C}),\mathcal{S}^1(\mathcal{X})}\left(\mathcal{S}^1(\mathfrak{a})\right).\] Thus $\mathcal{S}^1\co\mathsf{CO}(\kappa)\to\mathsf{Ch}(\kappa)$ is indeed a dg-functor.	
\end{prop}	
\end{remark}

We now finally complete the definition of the categories $\mathsf{HCO}(\kappa)$ and $\mathsf{HCO}^{\Lambda}(\kappa)$:

\begin{dfn} $\mathsf{HCO}(\kappa)$ is the zeroth cohomology category of the dg-category $\mathsf{CO}(\kappa)$.  Thus $\mathsf{HCO}(\kappa)$ has the same objects as $\mathsf{CO}(\kappa)$, and, in terms of the notation above, \[ \mathrm{Hom}_{\mathsf{HCO}(\kappa)}(\mathcal{C},\mathcal{X})=H^0\left(\mathrm{Hom}^*_{\mathsf{CO}(\kappa)}(\mathcal{C},\mathcal{X}),\delta_{\mathcal{C},\mathcal{X}}\right).\] Similarly, $\mathsf{HCO}^{\Lambda}(\kappa)$ is the zeroth cohomology category of the dg-category $\mathsf{CO}^{\Lambda}(\kappa)$.
\end{dfn}

Since $\mathsf{CO}^{\Lambda}(\kappa)$ is a full subcategory of the dg-category $\mathsf{CO}(\kappa)$, one may equivalently describe $\mathsf{HCO}^{\Lambda}(\kappa)$ as the full subcategory of $\mathsf{HCO}(\kappa)$ that has the same objects as $\mathsf{CO}^{\Lambda}(\kappa)$ (\emph{i.e.}, the $\Lambda$-bounded filtered cospans).

\subsection{Isomorphisms in $\mathsf{HCO}(\kappa)$}

We now work toward the proof of Theorem \ref{introiso} which gives a necessary and sufficient criterion for two filtered cospans to be isomorphic in $\mathsf{HCO}(\kappa)$. A first step is the following.

\begin{prop}\label{htopicpsi}
	Let $\mathcal{C}=\left(\mathcal{C}_{\uparrow},\mathcal{C}^{\downarrow},(D_*,\partial_D),\psi_{\uparrow},\psi^{\downarrow}\right)$ and $ \mathcal{X}=\left(\mathcal{C}_{\uparrow},\mathcal{C}^{\downarrow},(D_*,\partial_D),\phi_{\uparrow},\phi^{\downarrow}\right)$ be two objects of $\mathsf{HCO}(\kappa)$ with the same chain complexes $\mathcal{C}_{\uparrow},\mathcal{C}^{\downarrow},(D_*,\partial_D)$, and assume that there are chain homotopies between the maps $\psi_{\uparrow}$ and $\phi_{\uparrow}$ and between the maps $\psi^{\downarrow}$ and $\phi^{\downarrow}$.  Then $\mathcal{C}$ and $\mathcal{X}$ are isomorphic in $\mathsf{HCO}(\kappa)$ by an isomorphism that is represented by an element $(\alpha^{\downarrow},\alpha_{\uparrow},\alpha,K^{\downarrow},K_{\uparrow})$  of $\mathrm{Hom}_{\mathsf{CO}(\kappa)}(\mathcal{C},\mathcal{X})$ such that $\alpha^{\downarrow},\alpha_{\uparrow},$ and $\alpha$ act as the identity on $C^{\downarrow}_{*},C_{\uparrow*},$ and $D_*$, respectively.
\end{prop}

\begin{proof}
	Assume that $\partial_D\circ K^{\downarrow}+K^{\downarrow}\circ\partial_{C}^{\downarrow}=\psi^{\downarrow}-\phi^{\downarrow}$ and that $\partial_D\circ K_{\uparrow}+K_{\uparrow}\circ\partial^{C}_{\uparrow}=\psi_{\uparrow}-\phi_{\uparrow}$.  One checks directly from (\ref{diffl}) that $\delta_{\mathcal{C},\mathcal{X}}(1_{C^{\downarrow}_{*}},1_{C_{\uparrow *}},1_{D_*},K^{\downarrow},K_{\uparrow})=0$ and that $\delta_{\mathcal{X},\mathcal{C}}(1_{C^{\downarrow}_{*}},1_{C_{\uparrow *}},1_{D_*},-K^{\downarrow},-K_{\uparrow})=0$, so that $(1_{C^{\downarrow}_{*}},1_{C_{\uparrow *}},1_{D_*},K^{\downarrow},K_{\uparrow})$ and $(1_{C^{\downarrow}_{*}},1_{C_{\uparrow *}},1_{D_*},-K^{\downarrow},-K_{\uparrow})$ induce morphisms in $\mathsf{HCO}(\kappa)$; these morphisms are inverse to each other according to (\ref{composition}).
\end{proof}

If $\mathcal{C}_{\uparrow}=(C_{\uparrow*},\partial_{\uparrow}^{C},\ell_{\uparrow}^{C})$ and $\mathcal{X}_{\uparrow}=(X_{\uparrow*},\partial_{\uparrow}^{X},\ell_{\uparrow}^{X}$) are ascending chain complexes, we define \textbf{filtered homotopy equivalence data} between $\mathcal{C}_{\uparrow}$ and $\mathcal{X}_{\uparrow}$ to consist of a tuple $(\alpha_{\uparrow},\beta_{\uparrow},K^{C_{\uparrow}},K^{X_{\uparrow}})$ where $\alpha_{\uparrow}\co C_{\uparrow*}\to X_{\uparrow*}$ and $\beta_{\uparrow}\co X_{\uparrow *}\to C_{\uparrow*}$ are chain maps satisfying $\ell^X_{\uparrow}\circ \alpha_{\uparrow}\leq \ell^C_{\uparrow}$ and $\ell^{C}_{\uparrow}\circ \beta_{\uparrow}\leq \ell_{\uparrow}^{X}$, and where $K^{C_{\uparrow}}\co C_{\uparrow * }\to C_{\uparrow *+1}$ and $K^{X_{\uparrow}}\co X_{\uparrow *}\to X_{\uparrow *+1}$ satisfy the filtration conditions $\ell_{\uparrow}^{C}\circ K^{C_{\uparrow}}\leq \ell_{\uparrow}^{C}$ and $\ell_{\uparrow}^{X}\circ K^{X_{\uparrow}}\leq \ell_{\uparrow}^{X}$ and the chain homotopy equations $\partial_{\uparrow}^{C}\circ K^{C_{\uparrow}}+K^{C_{\uparrow}}\circ \partial_{\uparrow}^{C}=1_{C_{\uparrow *}}-\beta_{\uparrow}\circ \alpha_{\uparrow}$, and $\partial_{\uparrow}^{X}\circ K^{X_{\uparrow}}+K^{X_{\uparrow}}\circ \partial_{\uparrow}^{X}=1_{X_{\uparrow *}}-\alpha_{\uparrow}\circ \beta_{\uparrow}$.   A filtered homotopy equivalence from $\mathcal{C}_{\uparrow}$ to $\mathcal{X}_{\uparrow }$ is by definition a map $\alpha_{\uparrow}\co C_{\uparrow *}\to X_{\uparrow *}$ that can be completed to filtered homotopy equivalence data $(\alpha_{\uparrow},\beta_{\uparrow},K^{C_{\uparrow}},K^{X_{\uparrow}})$.   Likewise one defines filtered homotopy equivalence data $(\alpha^{\downarrow},\beta^{\downarrow},K^{C^{\downarrow}},K^{X^{\downarrow}})$ between descending chain complexes (with the directions of the inequalities reversed to account for the filtrations being descending rather than ascending, \emph{e.g.} in place of $\ell^{C}_{\uparrow}\circ \beta_{\uparrow}\leq \ell_{\uparrow}^{X}$ one requires $\ell_{C}^{\downarrow}\circ \beta^{\downarrow} \geq \ell^{\downarrow}_{X}$).  

\begin{prop} \label{bounce}
	Let $\mathcal{C}=\left((C_{\uparrow *},\partial_{\uparrow}^{C},\ell_{\uparrow}^{C}),(C^{\downarrow}_{*},\partial^{\downarrow}_C,\ell^{\downarrow}_C),(D_*,\partial_D),\psi_{\uparrow},\psi^{\downarrow}\right)$ be an object of $\mathsf{HCO}(\kappa)$, and suppose we have an ascending chain complex $(X_{\uparrow *},\partial_{\uparrow}^{X},\ell_{\uparrow}^{X})$, a descending chain complex $(X^{\downarrow}_{*},\partial_{X}^{\downarrow},\ell_{X}^{\downarrow})$, and another chain complex $(Y_*,\partial_Y)$, as well as a diagram \[ \xymatrix{  C^{\downarrow}_{*} \ar@<.5ex>[rr]^{\alpha^{\downarrow}}\ar[rd]_{\psi^{\downarrow}} & & X^{\downarrow}_{*} \ar@<.5ex>[ll]^{\beta^{\downarrow}} & \\ & D_*\ar@<.5ex>[rr]^{\alpha} & & Y_* \ar@<.5ex>[ll]^{\beta} \\ C_{\uparrow *}\ar@<.5ex>[rr]^{\alpha_{\uparrow}}\ar[ru]^{\psi_{\uparrow}} & & X_{\uparrow *}\ar@<.5ex>[ll]^{\beta_{\uparrow}} &     }\] where $\alpha^{\downarrow}$ and $\beta^{\downarrow}$ are part of  filtered homotopy equivalence data $(\alpha^{\downarrow},\beta^{\downarrow},K^{C^{\downarrow}},K^{X^{\downarrow}})$, $\alpha_{\uparrow}$ and $\beta_{\uparrow}$ are part of filtered homotopy equivalence data $(\alpha_{\uparrow},\beta_{\uparrow},K^{C_{\uparrow}},K^{X_{\uparrow}})$, and the chain maps $\alpha$ and $\beta$ are (unfiltered) homotopy inverses to each other.
	Then, denoting \[\mathcal{X}:=\left((X_{\uparrow *},\partial_{\uparrow}^{X},\ell_{\uparrow}^{X}),(X^{\downarrow}_{*},\partial_{X}^{\downarrow},\ell_{X}^{\downarrow}),(Y_*,\partial_Y),\alpha\circ \psi^{\downarrow}\circ \beta^{\downarrow},\alpha\circ\psi_{\uparrow}\circ \beta_{\uparrow} \right),\] the element $\frak{a}=(\alpha^{\downarrow},\alpha_{\uparrow},\alpha,\alpha\circ\psi^{\downarrow}\circ K^{C^{\downarrow}},\alpha\circ\psi_{\uparrow}\circ K^{C_{\uparrow}})$ of $\mathrm{Hom}^0_{\mathsf{CO}(\kappa)}(\mathcal{C},\mathcal{X})$  descends to an isomorphism in the category $\mathsf{HCO}(\kappa)$.
\end{prop}
	
\begin{proof}
	Throughout the proof we omit the $\circ$ symbol to save space.  Let $K^D\co D_*\to D_{*+1}$ and $K^Y\co Y_*\to Y_{*+1}$ be such that $\partial_DK^D+K^D\partial_D=1_{D_*}-\beta\alpha$ and $\partial_YK^Y+K^Y\partial_Y=1_{Y_*}-\alpha\beta$.  Complementing the element $\frak{a}\in \mathrm{Hom}^{0}_{\mathsf{CO}(\kappa)}(\mathcal{C},\mathcal{X})$ defined at the end of the proposition, let \[ 	\frak{b} = \left(\beta^{\downarrow},\beta_{\uparrow},\beta,\,-K^D\psi^{\downarrow}\beta^{\downarrow},\,-K^D\psi_{\uparrow}\beta_{\uparrow}\right)\in \mathrm{Hom}_{\mathsf{CO}(\kappa)}^0(\mathcal{X},\mathcal{C}).\] 

	One may verify that $\delta_{\mathcal{C},\mathcal{X}}\frak{a}=0$ and that $\delta_{\mathcal{X},\mathcal{C}}\frak{b}=0$, so that $\frak{a}$ and $\frak{b}$ descend to elements of  $\mathrm{Hom}_{\mathsf{HCO}(\kappa)}(\mathcal{C},\mathcal{X})$ and $\mathrm{Hom}_{\mathsf{HCO}(\kappa)}(\mathcal{X},\mathcal{C})$, respectively.  
			
 	Moreover, if \[ \zeta=(-K^{C^{\downarrow}},\,-K^{C_{\uparrow}},\,K^D,\,K^D\psi^{\downarrow}K^{C^{\downarrow}},\,K^D\psi_{\uparrow}K^{C_{\uparrow}})\in \mathrm{Hom}^{-1}_{\mathsf{CO}(\kappa)}(\mathcal{C},\mathcal{C}), \] a calculation shows that \[ \delta_{\mathcal{C},\mathcal{C}}\zeta=1_{\mathcal{C}}-\frak{ba}.\] (Here $1_{\mathcal{C}}=(1_{C^{\downarrow}_{*}},1_{C_{\uparrow *}},1_{D_*},0,0)$ is the identity in $\mathrm{Hom}^0_{\mathsf{CO}(\kappa)}(\mathcal{C},\mathcal{C})$.)  Similarly, if we let $L^{\downarrow}=K^D\psi^{\downarrow}\beta^{\downarrow}-\psi^{\downarrow}K^{C^{\downarrow}}\beta^{\downarrow}+\psi^{\downarrow}\beta^{\downarrow}K^{X^{\downarrow}}$,  $L_{\uparrow}=K^D\psi_{\uparrow}\beta_{\uparrow}-\psi_{\uparrow}K^{C_{\uparrow}}\beta_{\uparrow}+\psi_{\uparrow}\beta_{\uparrow}K^{X_{\uparrow}}$, and \begin{align*} \eta &=\bigg( -(\alpha^{\downarrow}K^{C^{\downarrow}}\beta^{\downarrow}+(1_{X^{\downarrow}_{*}}-\alpha^{\downarrow}\beta^{\downarrow})K^{X^{\downarrow}}),\,\,-(\alpha_{\uparrow}K^{C_{\uparrow}}\beta_{\uparrow}+(1_{X_{\uparrow*}}-\alpha_{\uparrow}\beta_{\uparrow})K^{X_{\uparrow}}),\\ & \qquad\,\, \alpha K^{D}\beta+\left(\alpha K^D\beta+(1_{Y_*}-\alpha\beta)K^Y\right)(1_{Y_*}-\alpha\beta),\\ & \qquad\,\, \left((1_{Y_*}-\alpha\beta)K^Y\alpha-\alpha K^D(1_{D_*}-\beta\alpha) \right)L^{\downarrow}+\alpha\psi^{\downarrow}K^{C^{\downarrow}}(\beta^{\downarrow}K^{X^{\downarrow}}-K^{C^{\downarrow}}\beta^{\downarrow}), \\ & \qquad\,\,
 	\left((1_{Y_*}-\alpha\beta)K^Y\alpha-\alpha K^D(1_{D_*}-\beta\alpha) \right)L_{\uparrow}+\alpha\psi_{\uparrow}K^{C_{\uparrow}}(\beta_{\uparrow}K^{X_{\uparrow}}-K^{C_{\uparrow}}\beta_{\uparrow})\,\, \bigg ), \end{align*}	 	 

a lengthier, but straightforward, calculation shows that
the element $\eta\in \mathrm{Hom}^{-1}_{\mathsf{CO}(\kappa)}(\mathcal{X},\mathcal{X})$ satisfies \[ \delta_{\mathcal{X},\mathcal{X}}\eta=1_{\mathcal{X}}-\frak{ab}.\]
Thus the cohomology classes represented by $\frak{a}$ and $\frak{b}$ are inverse to each other in $\mathsf{HCO}(\kappa)$, and so are isomorphisms.

\end{proof}	

We can now prove the following more specific version of Theorem \ref{introiso}.
	
	\begin{cor}\label{hcoiso}
Let $\mathcal{C}=\left((C_{\uparrow *},\partial_{\uparrow}^C,\ell_{\uparrow}^C),(C^{\downarrow}_{*},\partial^{\downarrow}_C,\ell^{\downarrow}_C),(D_*,\partial_D),\psi_{\uparrow},\psi^{\downarrow}\right)$  and \\ $\mathcal{X}=\left((X_{\uparrow *},\partial_{\uparrow}^X,\ell_{\uparrow}^{X}),(X^{\downarrow}_{*},\partial^{\downarrow}_{X},\ell^{\downarrow}_{X}),(Y_*,\partial_Y),\phi_{\uparrow},\phi^{\downarrow}\right)$ be two objects of $\mathsf{HCO}(\kappa)$.  Then $\mathcal{C}$ and $\mathcal{X}$ are isomorphic in $\mathsf{HCO}(\kappa)$ if and only if there exists a diagram  \[ \xymatrix{ C^{\downarrow}_{*} \ar[rrr]^{\alpha^{\downarrow}} \ar[rd]_{\psi^{\downarrow}} & & & X^{\downarrow}_{*}\ar[rd]^{\phi^{\downarrow}} & \\ & D_* \ar[rrr]^{\alpha} & & & Y_* \\ C_{\uparrow*} \ar[ru]^{\psi_{\uparrow}} \ar[rrr]_{\alpha_{\uparrow}}  & & & X_{\uparrow*} \ar[ru]_{\phi_{\uparrow}} & }  \] where $\alpha_{\uparrow}$ and $\alpha^{\downarrow}$ are filtered homotopy equivalences, $\alpha$ is a homotopy equivalence, and the two parallelograms commute up to chain homotopy.  Moreover, in this case, such an isomorphism can be taken to be the cohomology class of some $\frak{a}\in \mathrm{Hom}_{\mathsf{CO}(\kappa)}^0(\mathcal{C},\mathcal{X})$ with $\delta_{\mathcal{C},\mathcal{X}}\frak{a}=0$ that can be written in the form $\frak{a}=(\alpha^{\downarrow},\alpha_{\uparrow},\alpha,K^{\downarrow},K_{\uparrow})$ for the given maps $\alpha^{\downarrow},\alpha_{\uparrow},\alpha$ and some maps $K^{\downarrow}\co C^{\downarrow}_{*}\to Y_{*+1}$ and $K_{\uparrow}\co C_{\uparrow *}\to Y_{*+1}$.	
	\end{cor}
	
	\begin{proof}
		The forward implication follows straightforwardly from the definitions: if $\frak{a}=(\alpha^{\downarrow},\alpha_{\uparrow},\alpha,K^{\downarrow},K_{\uparrow})\in \ker\delta_{\mathcal{C},\mathcal{X}}|_{\mathrm{Hom}^{0}_{\mathsf{CO}(\kappa)}(\mathcal{C},\mathcal{X})}$ descends to an isomorphism in $\mathsf{HCO}(\kappa)$, the fact that $\delta_{\mathcal{C},\mathcal{X}}\frak{a}=0$ implies that $\alpha^{\downarrow},\alpha_{\uparrow},$ and $\alpha$ are chain maps and that $K^{\downarrow}$ (resp. $K_{\uparrow}$) gives a homotopy between $\phi^{\downarrow}\circ \alpha^{\downarrow}$ and $\alpha\circ \psi^{\downarrow}$ (resp. between $\phi_{\uparrow}\circ \alpha_{\uparrow}$ and $\alpha\circ \psi_{\uparrow})$, and the definitions of composition in $\mathsf{CO}(\kappa)$ and of $\delta_{\mathcal{C},\mathcal{C}}$ imply that $\alpha^{\downarrow}$ and $\alpha_{\uparrow}$ are filtered homotopy equivalences and that $\alpha$ is a homotopy equivalence.  
		
		For the converse, given $\alpha^{\downarrow},\alpha_{\uparrow},\alpha$ as in the corollary, let $\beta^{\downarrow},\beta_{\uparrow},\beta$ be (filtered, in the first two cases) homotopy inverses as in Proposition \ref{bounce}.  That proposition implies the result in the special case that $\phi^{\downarrow}=\alpha\circ \psi^{\downarrow}\circ \beta^{\downarrow}$ and $\phi_{\uparrow}=\alpha\circ \psi_{\uparrow}\circ \beta_{\uparrow}$.  In  the general case, the hypothesis that $\phi^{\downarrow}\circ \alpha^{\downarrow}$ is chain homotopic to $\alpha\circ \psi^{\downarrow}$ together with the fact that $\alpha^{\downarrow}$ and $\beta^{\downarrow}$ are homotopy inverses implies that $\phi^{\downarrow}$ is chain homotopic to $\alpha\circ \psi^{\downarrow}\circ \beta^{\downarrow}$; similarly, $\phi_{\uparrow}$ is chain homotopic to $\alpha\circ\psi_{\uparrow}\circ \beta_{\uparrow}$.  
		
		Let $\mathcal{X}'$ be the object of $\mathsf{HCO}(\kappa)$ obtained from $\mathcal{X}$ by replacing $\phi_{\uparrow}$ and $\phi^{\downarrow}$ by 	$\alpha\circ\psi_{\uparrow}\circ \beta_{\uparrow}$ and  $\alpha\circ \psi^{\downarrow}\circ \beta^{\downarrow}$.  Then Proposition \ref{bounce} gives a map $\frak{a}_0\in \mathrm{Hom}_{\mathsf{CO}(\kappa)}(\mathcal{C},\mathcal{X}')$ that descends to an isomorphism in $\mathsf{HCO}(\kappa)$, and Proposition \ref{htopicpsi} gives a map $\frak{a}_1\in\mathsf{Hom}_{\mathsf{CO}(\kappa)}(\mathcal{X}',\mathcal{X})$ that descends to an isomorphism in $\mathsf{HCO}(\kappa)$.  Thus the cohomology class of $\frak{a}_1\circ\frak{a}_0$ gives an isomorphism in $\mathsf{HCO}(\kappa)$ between $\mathcal{C}$ and $\mathcal{X}$.  Moreover, since $\frak{a}_0$ can be taken to be of the form $(\alpha^{\downarrow},\alpha_{\uparrow},\alpha,*,*)$ and $\frak{a}_1$ can be taken to be of the form $(1_{X^{\downarrow}_{*}},1_{X_{\uparrow*}},1_{Y_*},*,*)$, the map $\frak{a}=\frak{a}_1\circ\frak{a}_0$ can be taken to have the form required by the last sentence of the corollary.
			
	\end{proof}

\section{Direct sum decompositions}\label{decompsect}

We now work towards establishing general conditions for a filtered cospan to decompose into standard elementary summands.  En route, we discuss direct sum decompositions for filtered vector spaces and ascending chain complexes.  Throughout, we use the language of ``orthogonality'' with respect to a filtration as in \cite{UZ}.  (This use of the term ``orthogonal'' is borrowed from non-Archimedean geometry, though since in the present context the relevant non-Archimedean valuation is the trivial one, the term may not appear as well-motivated geometrically.)

\subsection{Filtered vector spaces}

\begin{dfn}\label{filtvsdfn}
	A filtered vector space over a field $\kappa$ is a pair $(V,\ell)$ where $V$ is a vector space over $\kappa$ and $\ell\co V\to\R\cup\{-\infty\}$ is a function satisfying: 
	\begin{itemize} \item $\ell(v)=-\infty$ if and only if $v=0$;
		\item $\ell(cv)=\ell(v)$ for all $v\in V$ and $c\in\kappa\setminus\{0\}$ 
		\item $\ell(v+w)\leq\max\{\ell(v),\ell(w)\}$ for all $v,w\in V$.
	\end{itemize} 
\end{dfn}

\begin{remark}\label{strict}
	A standard exercise shows that, in a filtered vector space $(V,\ell)$, if $\ell(v)\neq \ell(w)$, then $\ell(v+w)=\max\{\ell(v),\ell(w)\}$.
\end{remark}

\begin{remark}\label{ellfinite}
	If $(V,\ell)$ is a filtered vector space then, for any $t\in\R$, $V_{\leq t}:=\ell^{-1}([-\infty,t])$ is a subspace of $V$, with $V^{\leq s}\leq V^{\leq t}$ when $s\leq t$.  So if $t_1<t_2<\cdots<t_k$ are elements of the image of $\ell$, we have a strictly increasing sequence of subspaces $V^{\leq t_1}<V^{\leq t_2}<\cdots<V^{\leq t_k}$.  In particular, if $V$ is finite-dimensional, the image of $\ell$ is a finite set.
\end{remark}

Filtered vector spaces form a category in which a morphism $(V_0,\ell_0)\to (V_1,\ell_1)$ is a linear map $A\co V_0\to V_1$ such that $\ell_1(Av_0)\leq \ell_0(v_0)$ for all $v_0\in V_0$.  This category admits direct sums: given filtered vector spaces $(V_{\alpha},\ell_{\alpha})$ (as $\alpha$ varies over some possibly-infinite index set), let $V=\oplus_{\alpha}V_{\alpha}$ and define $\ell\co V\to\R\cup\{-\infty\}$ by $\ell\left(\sum_{\alpha}c_{\alpha}v_{\alpha}\right)=\max_{\alpha}\{\ell_{\alpha}(v_{\alpha})|c_{\alpha}\neq 0\}$.

\begin{dfn} If $(V,\ell)$ is a filtered vector space and $S\subset V\setminus\{0\}$, we say that $S$ is $\ell$-\textbf{orthogonal} if we have \begin{equation}\label{orthc}\ell(\sum_{i=1}^{n}c_i v_i)=\max_i\{\ell(v_i)|c_i\neq 0\}\end{equation} for every choice of finitely many distinct elements $v_i$ from $S$ and $c_i\in \kappa$.\end{dfn}  
Equivalently, $S$ is $\ell$-orthogonal iff the span of $S$, with filtration given by the restriction of $\ell$, is the direct sum (as filtered vector spaces) of the (one-dimensional) spans of the various elements of $S$.  Note that if $S$ is $\ell$-orthogonal then it is linearly independent.

The following gives a criterion for checking that a basis of a finite-dimensional filtered vector space is orthogonal; we will find this convenient in Section \ref{calcsect}.

\begin{prop}\label{orthcrit} Let $(V,\ell)$ be a finite-dimensional filtered vector space and for each $t\in \R$ introduce the subspaces \[ V^{\leq t}=\{v\in V|\ell(v)\leq t\} \qquad V^{<t}=\{v\in V|\ell(v)<t\}.\]  Then:
	\begin{itemize}\item[(i)] A basis $\{v_1,\ldots,v_d\}$ for $V$ is $\ell$-orthogonal if and only if, for every $t$, the multiplicity of $t$ in the multiset $\{\ell(v_i)|i=1,\ldots,d\}$ is equal to $\dim\left(\frac{V^{\leq t}}{V^{<t}}\right)$.
		\item[(ii)]  If $\{v_1,\ldots,v_d\}$ is an $\ell$-orthogonal basis for $(V,\ell)$, and $\{w_1,\ldots,w_d\}$ is another basis for $V$, then $\{w_1,\ldots,w_d\}$ is $\ell$-orthogonal if and only if the multisets $\{\ell(v_i)|i=1,\ldots,d\}$ and $\{\ell(w_i)|i=1,\ldots,d\}$ are equal to each other. \end{itemize}
	\end{prop}

\begin{proof}
	For the forward implication of (i), just note that the orthogonality of $\{v_1,\ldots,v_d\}$ implies that, for any $t$, $V^{\leq t}$ is the span of those $v_i$ with $\ell(v_i)\leq t$, and $V^{<t}$ is the span of those $v_i$ with $\ell(v_i)<t$.  
	
	We turn to the backward implication of (i).  Since $V$ is finite-dimensional, $\ell(V\setminus\{0\})$ is a finite set, say $\{t_1,\ldots,t_k\}$ with $t_1<\cdots<t_k$.  Our hypothesis is equivalent to the statement that, for $j=1,\ldots,k$, \begin{equation}\label{itj}\#\{i|\ell(v_i)<t_j\}=\dim V^{<t_j}.\end{equation}  Supposing for contradiction that $\{v_1,\ldots,v_d\}$ were not orthogonal, there would be some nonzero linear combination $v=\sum_i a_iv_i$ with $\ell(v)<\max\{\ell(v_i)|a_i\neq 0\}$.  In that case, let $j_0$ be the maximal value such that there is some $i_0$ with $a_{i_0}\neq 0$ and $\ell(v_{i_0})=t_{j_0}$.  Then $\{v\}\cup \{v_i|\ell(v_i)<t_{j_0}\}$ would be a linearly independent set in $V^{<t_{j_0}}$, contradicting the $j=j_0$ version of (\ref{itj}).  This completes the proof of (i).
	
	(ii) is then clear by applying (i) both to $\{v_1,\ldots,v_d\}$ and to $\{w_1,\ldots,w_d\}$.     
\end{proof}

\begin{dfn}
	A filtered vector space $(V,\ell)$ satisfies the \textbf{best approximation property} if, for every proper subspace $W$ of $V$ and every $v\in V\setminus W$, there is $w_0\in W$ such that $\ell(v-w_0)\leq \ell(v-w)$ for all $w\in W$.
\end{dfn}

\begin{prop} \label{wellbest}
	A filtered vector space $(V,\ell)$ satisfies the best approximation property if and only if $\ell(V\setminus \{0\})$ is a well-ordered subset of $\mathbb{R}$.
\end{prop}

(In particular, by Remark \ref{ellfinite}, if $V$ is finite-dimensional then $(V,\ell)$ satisfies the best approximation property.)

\begin{proof}
	The backward implication is clear, since the best approximation property amounts to the statement that subsets of $\ell(V\setminus\{0\})$ of the form $\{\ell(v-w)|w\in W\}$ (where $W\leq V$ and $v\notin W$) have least elements.
	
	For the forward implication, let us suppose that $\ell(V\setminus\{0\})$ is not well-ordered and show that this allows one to construct a counterexample to the best approximation property.  That $\ell(V\setminus\{0\})$ is not well-ordered implies that there is a sequence $\{x_i\}_{i=1}^{\infty}$ in $V\setminus\{0\}$ such that $\ell(x_{i+1})<\ell(x_i)$ for all $i$.  Note that Remark \ref{strict}  implies that $\{x_i|i\geq 1\}$ is $\ell$-orthogonal and hence is a linearly independent set.  Put $v=x_1$ and $W=\mathrm{span}\{x_{i+1}-x_i|i\geq 1\}$.  Since $x_i=x_1+\sum_{j=1}^{i-1}(x_{j+1}-x_j)$, the set $L:=\{\ell(v-w)|w\in W\}$ contains each $\ell(x_i)$; conversely, since every element in $L$ is given by taking $\ell$ of a nonzero linear combination of the various $x_i$, it follows from Remark \ref{strict} that the only elements of $L$ are the $\ell(x_i)$.  So the fact that the $\ell(x_i)$ are strictly decreasing implies that $L$ does not contain its infimum, and the best approximation property is violated. 
\end{proof}

The following is very similar to a special case of \cite[Lemma 2.15]{UZ}; we give a proof for the reader's convenience.
\begin{prop}\label{project}
	Let $(V,\ell)$ be a filtered vector space, $S\subset V$ an $\ell$-orthogonal subset, and $W=\mathrm{span}(S)$.  Suppose that  $v\in V\setminus W$ and $w_0\in W$ have the property that $\ell(v-w_0)\leq \ell(v-w)$ for all $w\in W$.  Then $S\cup\{v-w_0\}$ is still an $\ell$-orthogonal set.
\end{prop}

\begin{proof}
	Let $W'=\mathrm{span}(S\cup\{v-w_0\})$, so that a general element $w'\in W'$ can be written uniquely as $w'=a(v-w_0)+w$ where $a\in\kappa$ and $w\in W$; in turn, we can write $w=\sum_{i=1}^{n}b_is_i$ for some $s_i\in S$ and $b_i\in\kappa$, and the orthogonality of $S$ implies that $\ell(w)=\max\{\ell(s_i)|b_i\neq 0\}$.
	
	Comparing with (\ref{orthc}), we are to show that  \begin{equation}\label{orthneed}
	\ell(w')=\left\{\begin{array}{ll} \max\left(\{\ell(v-w_0)\}\cup\{\ell(s_i)|b_i\neq 0\}\right) & \mbox{if }a\neq 0  \\  \max\{\ell(s_i)|b_i\neq 0\} & \mbox{if }a=0 \end{array}\right. 	
		\end{equation}
	If $a=0$ then $w'=w$ and the desired statement is immediate from the orthogonality of $S$, so assume from now on that $a\neq 0$.  If $\ell(v-w_0)<\ell(w)$, then Remark \ref{strict} shows that $\ell(w')=\ell(w)=  \max\{\ell(s_i)|b_i\neq 0\}$, which equals the the first line of the right hand side of (\ref{orthneed}) due to the assumption that $\ell(v-w_0)<\ell(w)$.
	
	In the remaining case that $a\neq 0$ and $\ell(v-w_0)\geq \ell(w)$, we have \[ \ell(w')=\ell\left(\frac{1}{a}w'\right)=\ell\left(v-\left(w_0-\frac{1}{a}w\right)\right)\geq \ell(v-w_0),\] where the first equality uses the second bullet in Definition \ref{filtvsdfn} and the second uses the assumption on $w_0$ in the proposition. So if $\ell(v-w_0)\geq \ell(w)$ then $\ell(w')$ is  greater than or equal to the right-hand side of (\ref{orthneed}); on the other hand, the reverse inequality is immediate from the third bullet in Definition \ref{filtvsdfn}.  This proves (\ref{orthneed}) in all cases, and hence proves the proposition.	

\end{proof}

\begin{prop}\label{orthbasis}
	Suppose that the filtered vector space $(V,\ell)$ satisfies the best approximation property. Then $V$ has an $\ell$-orthogonal subset which is a basis for $V$.  Moreover, if $S$ is an $\ell$-orthogonal basis for a subspace $W$ of $V$, there is an $\ell$-orthogonal basis for $V$ that contains $S$.  
\end{prop}

\begin{proof}
	The empty set is an $\ell$-orthogonal basis for the trivial subspace $\{0\}$ of $V$, so the second sentence follows from the third, which we now prove.  Given the orthogonal basis $S$ for $W$, let $\mathcal{S}$ denote the collection of orthogonal subsets $S'$ of $V$ such that $S\subset S'$, ordered by inclusion.  The union of a totally ordered family of $\ell$-orthogonal subsets of $V$ is still $\ell$-orthogonal, so Zorn's Lemma shows that $\mathcal{S}$ has a maximal element. It follows from Proposition \ref{project} that an $\ell$-orthogonal subset that did not span $V$ could not be maximal in $\mathcal{S}$, so our maximal element is the desired $\ell$-orthogonal basis for $V$.    
\end{proof}

\begin{dfn}
	If $(V,\ell)$ is a filtered vector space with subspace $W$, an $\ell$-\textbf{orthogonal complement} of $W$ is a subspace $X\leq V$ such that $W\oplus X=V$ and $\ell(w+x)=\max\{\ell(w),\ell(x)\}$ for all $w\in W$ and $x\in X$.
\end{dfn}

(Equivalently, $V$ is the direct sum of $W$ and $X$ in the category of \emph{filtered} vector spaces.)

\begin{cor}\label{orthcomp}
	If the filtered vector space $(V,\ell)$ satisfies the best approximation property then every subspace of $V$ has an $\ell$-orthogonal complement.
\end{cor}

\begin{proof} Let $W\leq V$ be any subspace.  This subspace inherits the best approximation property from $V$, so $W$ has an $\ell$-orthogonal basis, say $S$, by Proposition \ref{orthbasis}. Moreover, another application of Proposition \ref{orthbasis} gives an $\ell$-orthogonal basis $S'$ for $V$ 
	such that $S\subset S'$.  Then the span of $S'\setminus S$ will be an $\ell$-orthogonal complement to $W$.
\end{proof}

\begin{remark} Orthogonal complements in this setting are not unique: for example, if $V$ is spanned by two elements $v$ and $w$ with $\ell(v)>\ell(w)$, and $W=\mathrm{span}\{w\}$, then either of $\mathrm{span}\{v\}$ and $\mathrm{span}\{v+w\}$ serves as an $\ell$-orthogonal complement to $W$.
\end{remark}

\subsection{Ascending chain complexes}\label{fcc}
According to Definition \ref{ascdef}, an ascending chain complex over a field $\kappa$ consists of data $(C_{\uparrow*},\partial_{\uparrow},\ell_{\uparrow})$ where $C_{\uparrow *}=\oplus_{k\in \Z}C_{\uparrow k}$ is a graded vector space over $\kappa$, $(C_{\uparrow*},\ell_{\uparrow})$ is a filtered vector space in the sense of the previous section with $(C_{\uparrow*},\ell_{\uparrow})=\oplus_k(C_{\uparrow k},\ell|_{C_{\uparrow k}})$ as filtered vector spaces, and $\partial_{\uparrow}\co C_{\uparrow *}\to C_{\uparrow *}$ restricts to maps $\partial_{\uparrow k}\co C_{\uparrow k}\to C_{\uparrow\, k-1}$ satisfying $\partial_{\uparrow k}\partial_{\uparrow k+1}=0$ and $\ell_{\uparrow}(\partial_{\uparrow} c)\leq \ell_{\uparrow}(c)$ for all $c\in C_*$.  As usual we shall write $H_k(C_{\uparrow*})=\frac{\ker\partial_{\uparrow k}}{\Img\partial_{\uparrow k+1}}$, and $H_*(C_{\uparrow *})=\oplus_kH_k(C_{\uparrow *})$.  Ascending chain complexes form a category in which the morphisms are filtered chain maps, \emph{i.e.} those chain maps that are also morphisms of filtered vector spaces.  One obtains direct sums of filtered chain complexes in the obvious way, extending the notion of a direct sum of filtered vector spaces from the previous section.

In this setting, define the \textbf{spectral invariant function} $\rho_{\uparrow}\co H_*(C_{\uparrow *})\to \R\cup\{-\infty\}$ by \begin{equation}\label{specfn} \rho_{\uparrow}(h)=\inf\{\ell(c)|c\in \ker\partial_{\uparrow},\,[c]=h\}\end{equation} (where $[c]$ denotes the homology class of the cycle $c$).

\begin{prop}\label{specproj} Let $(C_{\uparrow *},\partial_{\uparrow},\ell_{\uparrow})$ be a filtered chain complex with notation as above, and assume that, for all $k$, $(C_{\uparrow k},\ell_{\uparrow}|_{C_{\uparrow k}})$ satisfies the best approximation property.  Then, for each $k$, there is a subspace $\mathcal{H}_k\leq \ker\partial_{\uparrow k}$ such that the canonical projection $\pi_k\co \ker\partial_k\to H_k(C_{\uparrow *})$ restricts to $\mathcal{H}_k$ as an isomorphism, and such that for each $c\in\mathcal{H}_k$ we have $\rho_{\uparrow}(\pi_k(c))=\ell_{\uparrow}(c)$.  In particular, $(H_k(C_{\uparrow *}),\rho_{\uparrow}|_{H_k(C_{\uparrow *})})$ is a filtered vector space satisfying the best approximation property.
\end{prop}

\begin{proof}
	The best approximation property is inherited by subspaces, so is satisfied by $\ker\partial_{\uparrow k}$.  Take $\mathcal{H}_k$ to be an $(\ell_{\uparrow}|_{\ker\partial_{\uparrow k}})$-orthogonal complement to the subspace $\Img\partial_{\uparrow k+1}$ of $\ker\partial_{\uparrow k}$ (as given by Corollary \ref{orthcomp}).  Since $(\Img\partial_{\uparrow k+1})\oplus \mathcal{H}_k=\ker\partial_{\uparrow k}$, the projection $\pi_k$ restricts to $\mathcal{H}_k$ as an isomorphism to $H_k(C_{\uparrow *})$.  Also, if $c\in \mathcal{H}_k$ and $c'\in\ker\partial_{\uparrow k}$ with $\pi_k(c)=\pi_k(c')$, then since $c'-c\in\Img\partial_{\uparrow k+1}$, the orthogonality of $\mathcal{H}_k$ and $\Img\partial_{\uparrow k+1}$ implies that $\ell_{\uparrow}(c')=\max\{\ell_{\uparrow}(c'-c),\ell_{\uparrow}(c)\}\geq \ell_{\uparrow}(c)$. Hence $\rho_{\uparrow}(\pi_k(c))=\ell_{\uparrow}(c)$.
	
	Thus $\pi_k$ gives an isomorphism of filtered vector spaces $(\mathcal{H}_k,\ell_{\uparrow}|_{\mathcal{H}_k})\to (H_k(C_{\uparrow *}),\rho_{\uparrow}|_{H_k(C_{\uparrow *})})$.  So since $\mathcal{H}_k$ inherits the best approximation property from $C_{\uparrow k}$, $(H_k(C_{\uparrow *}),\rho_{\uparrow}|_{H_k(C_{\uparrow *})})$ also satisfies this property.
\end{proof}

\begin{prop}\label{ahdecomp}
	Let $(C_{\uparrow *},\partial_{\uparrow},\ell_{\uparrow})$ be an ascending chain complex in which each $C_{\uparrow k}$ satisfies the best approximation property.  Then there are subcomplexes $A_*$ and $\mathcal{H}_*$ of $C_{\uparrow *}$ such that:
	\begin{itemize}
		\item $A_*$ and $\mathcal{H}_*$ are $\ell_{\uparrow}$-orthogonal complements to each other;
		\item $\partial_{\uparrow}|_{\mathcal{H}_*}=0$, and the restriction of the canonical projection $\ker\partial_{\uparrow k}\to H_k(C_{\uparrow *})$ induces an isomorphism of filtered vector spaces $(\mathcal{H}_k,\ell_{\uparrow}|_{\mathcal{H}_k})\to (H_k(C_{\uparrow *}),\rho_{\uparrow}|_{H_k(C_{\uparrow *})})$ for each $k$.
		\item  $H_k(A_*)=0$ for all $k$.
	\end{itemize}
\end{prop}

\begin{proof}
	For each $k$ let $\mathcal{H}_k$ be as in Proposition \ref{specproj}, and let $\mathcal{H}_*=\oplus_k\mathcal{H}_k$; evidently, $\mathcal{H}_*$ is a subcomplex of $C_{\uparrow *}$ with zero differential, which satisfies the second requirement of the proposition.  Also, let $F_k\leq C_{\uparrow k} $ be an $(\ell_{\uparrow}|_{C_{\uparrow k}})$-orthogonal complement to $\ker\partial_{\uparrow k}$.  Take $A_k=F_k\oplus \Img\partial_{\uparrow k+1}$.  Then $A_*=\oplus_k A_k$ gives a subcomplex of $C_{\uparrow *}$ since $A_k$ contains all of $\Img\partial_{\uparrow k+1}$.  Because $F_{k}$ is complementary to $\ker\partial_{\uparrow k}$, $\partial_{\uparrow k}$ restricts to $F_k$ as an isomorphism to $\Img(\partial_{\uparrow k-1})$. Hence $\partial_{\uparrow}|_{A_k}$ has kernel equal to $\Img(\partial_{\uparrow k+1})$, which coincides with $\Img(\partial_{\uparrow k+1}|_{A_{k+1}})$ since $A_{k+1}$ contains the complement $F_{k+1}$ to $\ker\partial_{\uparrow k+1}$.  Thus the subcomplex $A_*$ has trivial homology.
	It remains to check that $\mathcal{H}_k$ and $A_k$ are $\ell_{\uparrow}$-orthogonal for each $k$.  If $h\in \mathcal{H}_k$ and $a\in A_k$, we can write $a=f+b$ where $f\in F_k$ and $b\in \Img\partial_{\uparrow k+1}$.  If $\ell_{\uparrow}(b)>\ell_{\uparrow}(h)$ then $\ell_{\uparrow}(b+h)=\ell_{\uparrow}(b)$ by Remark \ref{strict}; if $\ell_{\uparrow}(b)\leq \ell(h)$ then $\ell_{\uparrow}(h)=\max\{\ell_{\uparrow}(b),\ell_{\uparrow}(h)\}\geq \ell_{\uparrow}(b+h)\geq \rho_{\uparrow}([h])=\ell_{\uparrow}(h)$, so equality holds throughout.  In any event, we will have $\ell_{\uparrow}(b+h)=\max\{\ell_{\uparrow}(b),\ell_{\uparrow}(h)\}$.  Now since $F_k$ is $\ell_{\uparrow}$-orthogonal to $\ker\partial_{\uparrow k}$ we have \begin{align*} \ell_{\uparrow}(a+h)&=\ell_{\uparrow}(f+b+h)=\max\{\ell_{\uparrow}(f),\ell_{\uparrow}(b+h)\}=\max\{\ell_{\uparrow}(f),\ell_{\uparrow}(b),\ell_{\uparrow}(h)\}\\ &=\max\{\ell_{\uparrow}(f+b),\ell_{\uparrow}(h)\}=\max\{\ell_{\uparrow}(a),\ell_{\uparrow}(h)\},\end{align*} as desired.
\end{proof}

In the setting of Proposition \ref{ahdecomp}, by choosing an $\ell_{\uparrow}|_{\mathcal{H}_k}$ orthogonal basis for each $\mathcal{H}_k$ (which exists by Proposition \ref{orthbasis}) we obtain a splitting of $\mathcal{H}_*$ as a direct sum of filtered, one-dimensional chain complexes, each with zero differential.  Similarly, the following proposition allows us to decompose the acyclic complex $A_*$ as a direct sum of simple filtered chain complexes (though, since $H_k(A_*)=0$, the summands need to have nontrivial differential, so they will be two-dimensional rather than one-dimensional).

\begin{prop}\label{acyclic}
	Let $(A_*,\partial_{\uparrow},\ell_{\uparrow})$ be an ascending chain complex such that $H_k(A_*)=0$ for all $k$, and assume that, for each $k$, $A_k$ satisfies the best approximation property.  Then, for suitable index sets $\mathcal{I}_k$, there are $(\ell_{\uparrow}|_{A_k})$-orthogonal bases $\{x_i^k|i\in \mathcal{I}_{k+1}\}\cup \{y_j^k|j\in \mathcal{I}_{k}\}$ for the various $A_k$ such that, for each $i\in\mathcal{I}_{k+1}$, we have $\partial_{\uparrow} y_{i}^{k+1}=x_{i}^{k}$.
\end{prop}

\begin{proof}
	Denote $B_k=\partial_{\uparrow}(A_{k+1})$, so that the hypothesis that $H_k(A_*)=0$ implies that also $B_k=\ker(\partial_{\uparrow}|_{A_k})$.  For each $k$ let $F_k$ be an $(\ell_{\uparrow}|_{A_k})$-orthogonal complement to $B_k$.  Then $\partial_{\uparrow}$ restricts to $F_{k+1}$ as a vector space isomorphism to $B_k$.  In order to prove the proposition it suffices to find, for arbitrary $k$, an $(\ell_{\uparrow}|_{F_{k+1}})$-orthogonal basis $\{y_{i}^{k+1}|i\in\mathcal{I}_{k+1}\}$ for $F_{k+1}$ whose image under $\partial_{\uparrow}$ is an $(\ell_{\uparrow}|_{B_k})$-orthogonal basis for $B_k$.  (Indeed, since $B_k$ and $F_k$ are $\ell_{\uparrow}$-orthogonal, the separate orthogonality of the two subsets $\{\partial_{\uparrow} y_{i}^{k+1}|i\in\mathcal{I}_{k+1}\}\subset B_k$ and $\{y_{j}^{k}|j\in \mathcal{I}_k\}\subset F_k$ would immediately imply the orthogonality of their union.)
	
	Proposition \ref{wellbest} implies that the set $\ell_{\uparrow}(A_{k+1}\setminus\{0\})$ is well-ordered.  We construct the desired basis using recursion indexed by the well-ordered set \[ J_{k+1}:=\ell_{\uparrow}(A_{k+1}\setminus\{0\})\cup\{\infty\} \] (where $\infty>\lambda$ for all $\lambda\in \ell_{\uparrow}(A_{k+1}\setminus\{0\})$).  For any $\lambda\in J_{k+1}$, denote \begin{align*} & F_{k+1}^{\leq \lambda}= \{y\in F_{k+1}|\ell_{\uparrow}(y)\leq \lambda\}\qquad  & F_{k+1}^{< \lambda}= \{y\in F_{k+1}|\ell_{\uparrow}(y)< \lambda\} \\ 
		& B_{k,\leq\lambda} =\partial_{\uparrow}(F_{k+1}^{\leq \lambda})  & B_{k,<\lambda}=\partial_{\uparrow}(F_{k+1}^{<\lambda})  \end{align*}
	Thus $\partial_{\uparrow}$ maps $F_{k+1}^{\leq\lambda}$ isomorphically to $B_{k,\leq \lambda}$ and $F_{k+1}^{<\lambda}$ isomorphically to $B_{k,<\lambda}$.
	As special cases, $F_{k+1}^{<\infty}=F_{k+1}$ and $B_{k,<\infty}=B_k$.
	
	Given $\lambda\in J_{k+1}$, as input to our recursive construction we will suppose that, for all $\mu\in J_{k+1}$ with $\mu<\lambda$, we have constructed an $(\ell|_{F_{k+1}^{< \mu}})$-orthogonal basis $S_{\mu}$ for $F_{k+1}^{< \mu}$ such that $\partial(S_{\mu})$ is also an $(\ell|_{B_{k,<\mu}})$-orthogonal basis for $B_{k,<\mu}$, and we suppose moreover that the various $S_{\mu}$ for $\mu<\lambda$ have the property that, if $\nu<\mu<\lambda$, then $S_{\nu}\subset S_{\mu}$.  To continue the recursion we need to extend this construction to the case $\mu=\lambda$.
	
	If $\lambda$ does not have an immediate predecessor in $J_{k+1}$ (\emph{i.e.}, if $\{\mu\in J_{k+1}|\mu<\lambda\}$ has no largest element), then we can simply take $S_{\lambda}=\cup_{\mu<\lambda}S_{\mu}$.  So now let us assume that $\lambda$ does have an immediate predecessor in $J_{k+1}$, denoted $\lambda^-$.  Then $F_{k+1}^{<\lambda}=F_{k+1}^{\leq \lambda^-}$ and $B_{k,<\lambda}=B_{k,\leq \lambda_-}$, so our task is to extend the basis $S_{\lambda^-}$ for $F_{k+1}^{<\lambda^-}$ to one for $F_{k+1}^{\leq \lambda^-}$, retaining the property that both the basis and its image under $\partial_{\uparrow}$ are $\ell_{\uparrow}$-orthogonal.
	
	Now $B_{k,\leq\lambda^-}$, being a subspace of $A_k$, satisfies the best approximation property, so by Proposition \ref{orthbasis} we may extend the $\ell_{\uparrow}$-orthogonal basis $\partial_{\uparrow}(S_{\lambda^-})$ for $B_{k,<\lambda^-}$ to an $\ell_{\uparrow}$-orthogonal basis $\partial_{\uparrow}(S_{\lambda^-})\cup T$ for $B_{k,\leq \lambda^-}$.  We claim that $S_{\lambda}:=S_{\lambda^-}\cup (\partial_{\uparrow}|_{F_{k+1}})^{-1}(T)$ satisfies the required properties.  That $\partial_{\uparrow}(S_{\lambda})$ is an $\ell_{\uparrow}$-orthogonal basis $B_{k,\leq\lambda^-}$ is immediate from the construction; we just need to check that $S_{\lambda}$ is an $\ell_{\uparrow}$-orthogonal basis for $F_{k+1}^{\leq\lambda^-}$.  
	
	The map $\partial_{\uparrow}|_{F_{k+1}}$ sends $F_{k+1}^{\leq\lambda^-}$ and $F_{k+1}^{<\lambda^-}$ isomorphically to, respectively, to $B_{k,\leq\lambda^-}$ and $B_{k,<\lambda^-}$; hence $\partial_{\uparrow}|_{F_{k+1}}$ descends to a vector space isomorphism $\frac{F_{k+1}^{\leq\lambda^-}}{F_{k+1}^{<\lambda^-}}\to \frac{B_{k,\leq\lambda^-}}{B_{k,<\lambda^-}}$.  Also, the subset $T\subset B_{k,\leq \lambda^-}$ projects to a basis for the quotient $\frac{B_{k,\leq\lambda^-}}{B_{k,<\lambda^-}}$.  Hence $(\partial_{\uparrow}|_{F_{k+1}})^{-1}(T)$ projects to a basis for $\frac{F_{k+1}^{\leq\lambda^-}}{F_{k+1}^{<\lambda^-}}$.  
	
	Quite generally, if $(V,\ell_V)$ is a filtered vector space, if $t\in \R$, and if $Z\subset V$ is a subset of $\ell_{V}^{-1}(\{t\})$, it is easy to check that $Z$ is $\ell_V$-orthogonal iff it descends to a linearly independent set in $\frac{V^{\leq t}}{V^{<t}}$.  Moreover, if this holds, then $\mathrm{span}(Z)$ and $V^{<t}$ are $\ell_V$-orthogonal subspaces (as follows from Remark \ref{strict}).  Applying this to our situation, since $S_{\lambda^-}$ is an $\ell_{\uparrow}$-orthogonal basis for $F_{k+1}^{<\lambda'}$ and $S_{\lambda^-}\cup  (\partial_{\uparrow}|_{F_{k+1}})^{-1}(T)$ spans $F_{k+1}^{\leq\lambda^-}$ with $ (\partial_{\uparrow}|_{F_{k+1}})^{-1}(T)$ projecting to a basis of $\frac{F_{k+1}^{\leq\lambda^-}}{F_{k+1}^{<\lambda^-}}$, it follows that $S_{\lambda^-}\cup  (\partial_{\uparrow}|_{F_{k+1}})^{-1}(T)$ is an $\ell_{\uparrow}$-orthogonal basis of $F_{k+1}^{\leq \lambda^-}$, \emph{i.e.} of $F_{k+1}^{<\lambda}$.  This completes the recursive step in the case that $\lambda$ has an immediate predecessor.
	
	Defining the bases $S_{\lambda}$ of $F_{k+1}^{<\lambda}$ for $\lambda\in J_{k+1}$ recursively in this fashion, we finally obtain a basis $S_{\infty}$ for $F_{k+1}^{<\infty}$, \emph{i.e.} for all of $F_{k+1}$, which satisfies the required properties.
\end{proof}

\begin{dfn}\label{elemdfn} Given a field $\kappa$, define the \textbf{elementary ascending chain complexes} $\mathcal{E}_k(a,b)_{\uparrow}$ and $\mathcal{E}_k(c,\infty)_{\uparrow}$ over $\kappa$ as follows: \begin{enumerate}\item
		If $k\in \Z$ and $a,b\in\R$ with $a\leq b$,  $\mathcal{E}_k(a,b)_{\uparrow}$ is the ascending chain complex $(C_{\uparrow *},\partial_{\uparrow},\ell_{\uparrow})$ having: \begin{itemize} \item $C_{\uparrow k}=C_{\uparrow k+1}=\kappa$, and $C_{\uparrow j}=0$ for $j\notin\{k,k+1\}$, \item $\partial_{\uparrow j}\co C_{\uparrow j}\to C_{\uparrow j-1}$ equal to zero for $j\neq k+1$, and equal to the identity map $\kappa\to\kappa$ for $j=k+1$;
			\item $\ell_{\uparrow}|_{C_{\uparrow k}\setminus\{0\}}\equiv a$; and $\ell_{\uparrow}|_{C_{\uparrow k+1}\setminus\{0\}}\equiv b$.\end{itemize}	
		\item If $k\in \Z$ and $c\in \R$, $\mathcal{E}_k(c,\infty)_{\uparrow}$ is the ascending  chain complex $(C_{\uparrow *},\partial,\ell)$ over $\R$ having:
		\begin{itemize} \item $C_{\uparrow k}=\kappa$, and $C_{\uparrow j}=0$ for $j\neq k$;
			\item $\partial_{\uparrow}=0$; and
			\item $\ell_{\uparrow}|_{C_{\uparrow k}\setminus\{0\}}\equiv c$.\end{itemize}\end{enumerate}
\end{dfn}

\begin{cor}\label{ascsplit}
	Let $(C_{\uparrow *},\partial_{\uparrow},\ell_{\uparrow})$ be a filtered chain complex with the property that, for all $k$, $(C_{\uparrow k},\ell_{\uparrow}|_{C_{\uparrow k}})$ satisfies the best approximation property (equivalently, has $\ell_{\uparrow}(C_{\uparrow k}\setminus\{0\})$ well-ordered).  Then $(C_{\uparrow *},\partial_{\uparrow},\ell_{\uparrow})$ is isomorphic in the category of ascending chain complexes
	to a direct sum of the form \begin{equation}\label{ascspliteqn} \left(\bigoplus_{i\in I} \mathcal{E}_{k_i}(a_i,b_i)_{\uparrow}\right)\oplus\left(\bigoplus_{j\in J} \mathcal{E}_{l_j}(c_j,\infty)_{\uparrow} \right) \end{equation} for suitable index sets $I$ and $J$, integers $k_i,l_j$, and real numbers $a_i,b_i,c_j$ with $a_i\leq b_i$.  Moreover, the subcomplex $\mathcal{H}_*=\bigoplus_{j\in J} \mathcal{E}_{l_j}(c_j,\infty)_{\uparrow}$ is isomorphic to the homology $(H_*(C_{\uparrow *}),0,\rho_{\uparrow})$, viewed as an ascending chain complex with zero differential and filtration function given by the spectral invariant function $\rho_{\uparrow}$ of (\ref{specfn}).
\end{cor}

\begin{proof}
	Proposition \ref{ahdecomp} gives a direct sum decomposition of filtered chain complexes $(C_{\uparrow *},\partial_{\uparrow},\ell_{\uparrow})=A_*\oplus \mathcal{H}_*$ where $A_*$ has trivial homology, and $\mathcal{H}_*$ has zero differential and is isomorphic to $(H_*(C_{\uparrow*}),0,\rho_{\uparrow})$.  Proposition \ref{acyclic} shows that $A_*$ decomposes up to isomorphism as a direct sum $\bigoplus_{i\in I} \mathcal{E}_{k_i}(a_i,b_i)_{\uparrow}$.  By taking  $\ell_{\uparrow}|_{\mathcal{H}_k}$ orthogonal bases for the filtered vector spaces $(\mathcal{H}_k,\ell_{\uparrow}|_{\mathcal{H}_k})$ (provided by Proposition \ref{orthbasis}), we see that $\mathcal{H}_*$ (with its zero differential) decomposes up to isomorphism as a direct sum of the form $\bigoplus_{j\in J} \mathcal{E}_{l_j}(c_j,\infty)_{\uparrow}$   
\end{proof}

\subsection{Filtered cospans}\label{cospan-decomp}

Recall that a filtered cospan consists of the data of:\begin{itemize}
	\item an ascending chain complex (in the sense of Section \ref{ascdef}) $\mathcal{C}_{\uparrow}=(C_{\uparrow *},\partial_{\uparrow},\ell_{\uparrow})$; 
	\item a descending chain complex, \emph{i.e.}, a triple $\mathcal{C}^{\downarrow}=(C^{\downarrow}_{*},\partial^{\downarrow},\ell^{\downarrow})$ such that $\overline{\mathcal{C}^{\downarrow}}:=(C^{\downarrow}_{*},\partial^{\downarrow},-\ell^{\downarrow})$ is an ascending chain complex;
	\item another (unfiltered) chain complex $(D_*,\partial_D)$, and chain maps $\psi_{\uparrow}\co C_{\uparrow *}\to D_*$ and $\psi^{\downarrow}\co C^{\downarrow}_{*}\to D_*$.\end{itemize}

In general, if $\mathcal{C}=(C_*,\partial,\ell)$ is an ascending (resp. descending) chain complex, we let $\overline{\mathcal{C}}=(C_*,\partial,-\ell)$ the descending (resp. ascending) chain complex obtained by negating the filtration function $\ell$.  As a special case of this, for $\infty>a\geq b\geq -\infty$, we write \[ \mathcal{E}_{k}(a,b)^{\downarrow}=\overline{\mathcal{E}_k(-a,-b)_{\uparrow}},\] so if $b>-\infty$ then $\mathcal{E}_{k}(a,b)^{\downarrow}$ has generators $x$ and $y$ in degrees $k$ and $k+1$, with respective filtration levels $a$ and  $b$, such that the boundary of $y$ is $x$. Descending chain complexes obviously form a category with morphisms $\mathcal{C}^{\downarrow}_{1}\to \mathcal{C}^{\downarrow}_{2}$ coinciding (set-theoretically as maps) with morphisms of ascending chain complexes $\overline{\mathcal{C}^{\downarrow}_{1}}\to\overline{\mathcal{C}^{\downarrow}_{2}}$.  The notion of a direct sum of ascending chain complexes then carries over via the conjugation operation $\mathcal{C}\leftrightarrow\bar{\mathcal{C}}$ to a notion of direct sum of descending chain complexes, and  Corollary \ref{ascsplit} implies that if $\mathcal{C}^{\downarrow}=(C^{\downarrow}_{*},\partial^{\downarrow},\ell^{\downarrow})$ has the property that, for each $k$, every subset of $\ell^{\downarrow}(C^{\downarrow}_{k}\setminus\{0\})$ has a greatest element, then $\mathcal{C}^{\downarrow}$ splits as a direct sum of the form \begin{equation}\label{descsplit} \mathcal{C}^{\downarrow}=\left(\bigoplus_{i\in I} \mathcal{E}_{k_i}(a_i,b_i)^{\downarrow}\right)\oplus\left(\bigoplus_{j\in J} \mathcal{E}_{l_j}(c_j,-\infty)^{\downarrow} \right).\end{equation}

The appropriate version of the spectral invariant function (\ref{specfn}) for a descending chain complex $\mathcal{C}^{\downarrow}$ is that given by negating the spectral invariant function $\rho_{\uparrow}$ associated to $\overline{\mathcal{C}^{\downarrow}}$, \emph{i.e.} it is the map $\rho^{\downarrow}\co H_*(C^{\downarrow}_{*})\to\R\cup\{\infty\}$ given by \begin{equation}\label{specdown} \rho^{\downarrow}(h)=\sup\{\ell(c)|c\in\ker\partial^{\downarrow},\,[c]=h\}.\end{equation} 

We introduce some terminology for certain types of filtered cospans:
\begin{dfn}\label{admdef}
	A filtered cospan $\mathcal{C}=\left((C_{\uparrow *},\partial_{\uparrow}^C,\ell_{\uparrow}^C),(C^{\downarrow}_{*},\partial^{\downarrow}_C,\ell^{\downarrow}_C),(D_*,\partial_D),\psi_{\uparrow},\psi^{\downarrow}\right)$ is said to be:\begin{itemize}\item[(i)] \textbf{admissible} if, for every $k\in \Z$, the images $\ell_{\uparrow}^{C}(C_{\uparrow k}\setminus \{0\})$ and $(-\ell^{\downarrow}_{C})(C^{\downarrow}_{k}\setminus\{0\})$ are well-ordered subsets of $\mathbb{R}$;
		\item[(ii)] \textbf{acyclic} if the homologies $H_*(C_{\uparrow *})$, $H_*(C^{\downarrow}_{*}),$ and $H_*(D_*)$ are all zero.
		\item[(iii)] \textbf{perfect} if the differentials $\partial_{\uparrow}^{C},\partial^{\downarrow}_{C},$ and $\partial_D$ are all zero. \end{itemize}
\end{dfn}

Thus for an admissible filtered cospan we have direct sum decompositions of both the ascending chain complex $(C_{\uparrow *},\partial_{\uparrow}^C,\ell_{\uparrow}^C)$ as in (\ref{ascspliteqn}) and the descending chain complex $(C^{\downarrow}_{*},\partial^{\downarrow}_C,\ell^{\downarrow}_C)$ as in (\ref{descsplit}).  If the filtered cospan is acyclic then only summands of form $\mathcal{E}_{k}(a,b)_{\uparrow}$ or $\mathcal{E}_k(a,b)^{\downarrow}$ will appear in the decompositions; if instead the filtered cospan is perfect then the only summands are of form $\mathcal{E}_k(c,\infty)_{\uparrow}$ or $\mathcal{E}_k(c,-\infty)^{\downarrow}$.

Direct sums in the category $\mathsf{CO}(\kappa)$ of filtered cospans are obtained in the obvious way, assembling the direct sums of their constituent parts. 

\begin{prop}\label{coah} 
	Any admissible filtered cospan $\mathcal{C}=\left((C_{\uparrow *},\partial_{\uparrow}^C,\ell_{\uparrow}^C),(C^{\downarrow}_{*},\partial^{\downarrow}_C,\ell^{\downarrow}_C),(D_*,\partial_D),\psi_{\uparrow},\psi^{\downarrow}\right)$ with coefficients in a field $\kappa$ is isomorphic in $\mathsf{HCO}(\kappa)$ to the direct sum of an acyclic filtered cospan $\mathcal{A}_{\mathcal{C}}$ and a perfect filtered cospan $\mathcal{H}_{\mathcal{C}}$.  Moreover:\begin{itemize}
		\item $\mathcal{A}_{\mathcal{C}}$ may be taken to be of the form $((A_{\uparrow *},\partial_{\uparrow}^{C}|_{A_{\uparrow *}},\ell_{\uparrow}^{C}|_{A_{\uparrow *}}),(A^{\downarrow}_{*},\partial^{\downarrow}_{C}|_{A^{\downarrow}_{*}},\ell^{\downarrow}_{C}|_{A^{\downarrow}_{*}}),0,0,0)$ for a suitable subcomplexes $A_{\uparrow *}$ of $C_{\uparrow *}$ and $A^{\downarrow}_{*}$ of $C^{\downarrow}_{*}$ each having trivial homology (with the remaining, unfiltered chain complex equal to the zero complex).
		\item $\mathcal{H}_{\mathcal{C}}$ can be taken to be of the form $\left((H_*(C_{\uparrow *}),0,\rho_{\uparrow}),(H_*(C^{\downarrow}_{*}),0,\rho^{\downarrow}),(H_*(D_*),0),\psi_{\uparrow *},\psi^{\downarrow}_{*}\right)$, where $\rho_{\uparrow}$ and $\rho^{\downarrow}$ are the spectral invariant functions of (\ref{specfn} and (\ref{specdown}), and $\psi_{\uparrow *}$ and $\psi^{\downarrow}_{*}$ are the maps induced by $\psi_{\uparrow}$ and $\psi^{\downarrow}$ on homology. \end{itemize}
\end{prop}

\begin{proof}
	Decompose the ascending chain complex $(C_{\uparrow *},\partial_{\uparrow}^{C},\ell_{\uparrow}^{C})$ as in Proposition \ref{ahdecomp}, so that $C_{\uparrow *}$ is a direct sum (as ascending chain complexes) of the acyclic subcomplex $A_{\uparrow *}$ and a subcomplex $\mathcal{H}_{\uparrow *}$ on which $\partial_{\uparrow}^{C}$ vanishes, with the filtered vector space $(\mathcal{H}_{\uparrow *},\ell_{\uparrow}^{C}|_{\mathcal{H}_{\uparrow *}})$ identified via the projection to homology with $(H_*(C_{\uparrow}),\rho_{\uparrow})$.  Denote by $\pi_{\uparrow}\co C_{\uparrow *}\to H_*(C_{\uparrow *})$ the composition of the projection $C_{\uparrow *}\to \mathcal{H}_{\uparrow *}$ associated to the direct sum decomposition $C_{\uparrow *}=A_{\uparrow *}\oplus \mathcal{H}_{\uparrow *}$ followed by the isomorphism $\mathcal{H}_{\uparrow *}\to H_*(C_{\uparrow *})$ given by passage to homology. Also define $j_{\uparrow}\co H_{*}(C_{\uparrow *})\to C_{\uparrow *}$ to be the composition of the inverse of the aforementioned isomorphism $\mathcal{H}_{\uparrow *}\to H_*(C_{\uparrow *})$  with the inclusion of $\mathcal{H}_{\uparrow *}$ into $C_{\uparrow *}$.
	
	The unfiltered chain complex $(D_*,\partial_D)$ can likewise be decomposed as a direct sum of an acyclic subcomplex $A_*^{D}$ and a subcomplex $H_*^D$ on which $\partial_D$ vanishes, with the projection to homology giving an isomorphism $H_*^D\cong H_*(D_*)$.\footnote{This is a standard consequence of the splitting of various short exact sequences of vector spaces over the field $\kappa$; it can also be regarded as a special case of Proposition \ref{ahdecomp} if we regard $D_*$ as an ascending chain complex with the trivial filtration function $\ell$ that evaluates as zero on all nonzero elements of $D_*$.}  As in the previous paragraph, we obtain a projection $\pi^D\co D_*\to H_*(D_*)$ and an inclusion $j^D\co H_*(D_*)\to D_*$.  
	
	The maps $\pi_{\uparrow},\pi^D,j_{\uparrow},j^D$ are all chain maps by virtue of the differentials on $H_*(C_{\uparrow *})$ and $H_*(D_*)$ being zero.  Evidently $\pi_{\uparrow}\circ j_{\uparrow}$ and $\pi^D\circ j^D$ are the identities on $H_*(C_{\uparrow *})$ and $H_*(D_*)$, respectively.  Meanwhile, $j_{\uparrow}\circ\pi_{\uparrow}$ and $j^D\circ \pi^D$ are the projections from the complexes $C_{\uparrow *}$ and $D_*$ to their direct summands $\mathcal{H}_{\uparrow *}$ and $H^D_*$. The maps $j_{\uparrow}\circ\pi_{\uparrow}$ and $j^D\circ \pi^D$ are chain homotopic to the respective identities (in the unfiltered sense).  Indeed, as in the proof of Proposition \ref{ahdecomp} one can let $F_{\uparrow k}$ be a vector space complement within $A_{\uparrow k}$ to $B_{\uparrow k}:=\Img(\partial^{C}_{\uparrow}|_{C_{\uparrow\, k+1}})$, so that $\partial_{\uparrow}^{C}$ maps $F_{\uparrow\,k+1}$ isomorphically to $B_{\uparrow k}$.  One obtains a map $K\co C_{\uparrow *}\to F_{\uparrow\,*+1}\subset C_{\uparrow *+1}$ given by first projecting $C_{\uparrow k}$ to $B_{\uparrow k}$ with respect to the direct sum decomposition $C_{\uparrow k}=F_{\uparrow k}\oplus B_{\uparrow k}\oplus \mathcal{H}_{\uparrow k}$ and then sending elements of $B_{\uparrow k}$ to their unique preimages in $F_{\uparrow\,k+1}$ under $\partial_{\uparrow}^{C}$.  This map $K$ gives a chain homotopy between $j_{\uparrow}\circ\pi_{\uparrow}$ and the identity on $C_{\uparrow *}$.  Similar reasoning applies to show that  $j^D\circ \pi^D$ is chain homotopic to the identity on $D_*$.
	
	We then have a commutative diagram \begin{equation}\label{split1}  \xymatrix{ D_* \ar[rr]^{\pi^D} & & H_*(D_*) \\ C_{\uparrow *} \ar[u]^{\psi_{\uparrow}\circ j_{\uparrow}\circ\pi_{\uparrow}} \ar[rr]_<<<<<<<<<{{\tiny \left[\begin{array}{c} 1_{C_{\uparrow *}}-j_{\uparrow}\circ\pi_{\uparrow} \\ \pi_{\uparrow}\end{array}\right] }} & & A_{\uparrow *}\oplus H_*(C_{\uparrow *}) \ar[u]_{{\tiny \left[\begin{array}{cc} 0 & \psi_{\uparrow *}\end{array}\right]} }    } \end{equation}  Indeed, either composition in (\ref{split1}) sends all elements of $A_{\uparrow *}$ to zero and sends an element $h\in \mathcal{H}_{\uparrow *}$ (which necessarily lies in $\ker(\partial_{\uparrow}^{C})$) to the homology class in $H_*(D_*)$ of $\psi_{\uparrow}h$.  Moreover, $\pi^D$ is a homotopy equivalence (using the zero differential on $H_*(D_*)$), and the bottom map in (\ref{split1}) is a filtered homotopy equivalence provided that $A_{\uparrow *}\oplus H_*(C_{\uparrow *})$ is endowed with the ascending chain complex structure given by taking the direct sum of the structure that $A_{\uparrow *}$ inherits as a subcomplex of $C_{\uparrow *}$ with the structure on $H_*(C_{\uparrow *})$ having zero differential and having $\rho_{\uparrow}$ as the filtration function.
	
	Since $j_{\uparrow}\circ\pi_{\uparrow}$ is chain homotopic to the identity, it follows that the diagram \begin{equation}\label{split2}  \xymatrix{
			D_* \ar[rr]^{\pi^D} & & H_*(D_*) \\ C_{\uparrow *} \ar[u]^{\psi_{\uparrow}} \ar[rr]_<<<<<<<<<{{\tiny \left[\begin{array}{c} 1_{C_{\uparrow *}}-j_{\uparrow}\circ\pi_{\uparrow} \\ \pi_{\uparrow}\end{array}\right] }} & & A_{\uparrow *}\oplus H_*(C_{\uparrow *}) \ar[u]_{{\tiny \left[\begin{array}{cc} 0 & \psi_{\uparrow *}\end{array}\right]} }}   \end{equation}
	commutes up to homotopy.
	
	By applying the same reasoning as above to the conjugate $\overline{(C^{\downarrow}_{*},\partial^{\downarrow}_{C},\ell^{\downarrow}_{C})}=(C^{\downarrow}_{*},\partial^{\downarrow}_{C},-\ell^{\downarrow}_{C})$ of the descending chain complex $(C^{\downarrow}_{*},\partial^{\downarrow}_{C},\ell^{\downarrow}_{C})$, we obtain a homotopy-commutative diagram \begin{equation}\label{split3}\xymatrix{  C^{\downarrow}_{*} \ar[d]_{\psi^{\downarrow}} \ar[rr]^<<<<<<<<<{{\tiny \left[\begin{array}{c} 1_{C^{\downarrow}_{ *}}-j^{\downarrow}\circ\pi^{\downarrow} \\ \pi^{\downarrow}\end{array}\right] }} & & A^{\downarrow}_{*}\oplus H_*(C^{\downarrow}_{ *}) \ar[d]^{{\tiny \left[\begin{array}{cc} 0 & \psi^{\downarrow}_{ *}\end{array}\right]} }  \\    D_* \ar[rr]_{\pi^D} & & H_*(D_*)   }  \end{equation}
	whose top row is a filtered homotopy equivalence of descending chain complexes; here $A^{\downarrow}_{*}$ is an acyclic subcomplex of $C^{\downarrow}_{*}$ obtained by applying Proposition \ref{ahdecomp} to $\overline{(C^{\downarrow}_{*},\partial^{\downarrow}_{C},\ell^{\downarrow}_{C})}$, and $H_*(C^{\downarrow}_{*})$ is regarded as a descending chain complex with zero differential and filtration function given by $\rho^{\downarrow}$.
	
	Putting (\ref{split2}) and (\ref{split3}) together, we see from Corollary \ref{hcoiso} that our original filtered cospan $\mathcal{C}$ is isomorphic in $\mathsf{HCO}(\kappa)$ to \[ \xymatrix{ A^{\downarrow}_{*}\oplus H_*(C^{\downarrow}_{*}) \ar[rd]^{{\tiny \left[\begin{array}{cc} 0 & \psi^{\downarrow}_{ *}\end{array}\right]}} & \\ & H_*(D_*) \\  A_{\uparrow *}\oplus H_*(C_{\uparrow *}) \ar[ur]_{{\tiny \left[\begin{array}{cc} 0 & \psi_{\uparrow *}\end{array}\right]} } &    } \] or equivalently to the direct sum in the category of filtered cospans of  \[ \xymatrix{ A^{\downarrow}_{*} \ar[rd] & \\ & 0 \\ A_{\uparrow *} \ar[ru] & } \quad \xymatrix{ \\  \mbox{ and  } \\ } \quad  \xymatrix{H_*(C^{\downarrow}_{ *})\ar[rd]^{\psi^{\downarrow}_{*}} & \\ & H_*(D_*). \\ H_*(C_{\uparrow *})\ar[ru]_{\psi_{\uparrow *}} &       }     \]
	
\end{proof}

The summands $\mathcal{A}_{\mathcal{C}}$ and $\mathcal{H}_{\mathcal{C}}$ of Proposition \ref{coah} admit further decompositions into indecomposables.  Since an acyclic chain complex over a field is always homotopy equivalent to the zero complex, any acyclic filtered cospan $\mathcal{A}$ is (using Corollary \ref{hcoiso}) isomorphic in $\mathsf{HCO}(\kappa)$ to one of the form \[ \xymatrix{A^{\downarrow}_{*}\ar[rd] & \\ & 0 \\ A_{\uparrow *}\ar[ru] & } \] Assuming $\mathcal{A}$ to be admissible, Corollary \ref{ascsplit} gives decompositions $A^{\downarrow}_{*}\cong \bigoplus_{i\in I^{\downarrow}} \mathcal{E}_{k_i^{\downarrow}}(a_i^{\downarrow},b_{i}^{\downarrow})^{\downarrow}$ and $A_{\uparrow *}\cong \bigoplus_{i\in I_{\uparrow}} \mathcal{E}_{k_{i \uparrow}}(a_{i\uparrow},b_{i\uparrow})_{\uparrow}$ in the categories, of, respectively, descending and ascending chain complexes.  Here $a_{i}^{\downarrow}\geq b_{i}^{\downarrow}>-\infty$ and $a_{i\uparrow}\leq b_{i\uparrow}<\infty$. (The $b_{i}^{\downarrow}$ and $b_{i\uparrow}$ are finite because $A^{\downarrow}_{*}$ and $A_{\uparrow *}$ have trivial homology.)  Note that any summands of form $\mathcal{E}_{k}(a,a)_{\uparrow}$ or $\mathcal{E}_{k}(a,a)^{\downarrow}$ are filtered homotopy equivalent to the zero complex.  Hence by Corollary \ref{hcoiso}, an admissible acyclic filtered cospan $\mathcal{A}$ is isomorphic to a direct sum of objects of form \begin{equation}  \label{basicacyc} \xymatrix{ \mathcal{E}_k(a,b)^{\downarrow} \ar[rd] & \\ & 0 \\  0 \ar[ru] & } \xymatrix{ \\ \mbox{ with $a>b>-\infty$, and  } \\ } \xymatrix{0\ar[rd] & \\ & 0 \\ \mathcal{E}_k(a,b)_{\uparrow}\ar[ru] & } \xymatrix{ \\ \mbox{ with $a<b<\infty$.  } \\ }   \end{equation}  We abbreviate the filtered cospans displayed in (\ref{basicacyc}) as $(\downarrow^a_b)_k$ and $(\uparrow_a^b)_k$, respectively.

Now let us describe decompositions of admissible, perfect filtered cospans. As the chain complexes involved in these have zero differential, we can immediately decompose them as direct sums of their graded pieces, with summands in degree $k$ having the general form \begin{equation}\label{genperfect} \xymatrix{ (H^{\downarrow}_k,\rho^{\downarrow}_{k}) \ar[rd]^{\phi^{\downarrow}_{k}} &  \\ & H_k \\ (H_{\uparrow k},\rho_{\uparrow}) \ar[ru]_{\phi_{\uparrow k}}  }  \end{equation} where $(H_{\uparrow k},\rho_{\uparrow})$ and $(H_k^{\downarrow},-\rho^{\downarrow})$ are filtered vector spaces that satisfy the best approximation property (and where the differentials, which are all zero, are all suppressed from the notation).  Consider the subspace $V_k=\phi_{\uparrow k}(H_{\uparrow k})\cap \phi^{\downarrow}_{k}(H^{\downarrow}_{k})\leq H_k$.  Choose: 
\begin{itemize} \item a $\rho_{\uparrow}$-orthogonal complement $I_{\uparrow k}$ to $\ker(\phi_{\uparrow k})$ within $H_{\uparrow k}$; 
	\item a $\rho_{\uparrow}$-orthogonal complement $J_{\uparrow k}$ to $(\phi_{\uparrow k})^{-1}(V_k)\cap I_{\uparrow k}$ within $I_{\uparrow k}$;
	\item a $(-\rho^{\downarrow})$-orthogonal complement $I^{\downarrow}_{ k}$ to $\ker(\phi^{\downarrow}_{ k})$ within $H^{\downarrow}_{ k}$;
	\item a $(-\rho_{\uparrow})$-orthogonal complement $J^{\downarrow}_k$ to $(\phi^{\downarrow}_k)^{-1}(V_k)\cap I^{\downarrow}_{ k}$ within $I^{\downarrow}_{ k}$; and 
	\item a complementary subspace $W_k$ to the subspace $\phi_{\uparrow k}(H_{\uparrow k})+\phi^{\downarrow}_{k}(H^{\downarrow}_{k})$ of $H_k$.
\end{itemize} These give rise to a direct sum decomposition of the filtered cospan (\ref{genperfect}) into the following six summands, where in each case the displayed maps (given by restricting $\phi_{\uparrow k}$ and $\phi^{\downarrow}_{k}$) are isomorphisms except when their domains or codomains are zero: 
\begin{align}\label{sixsummands} & \mbox{(i)}  \xymatrix{  \ker(\phi_{k}^{\downarrow})\ar[rd] & \\ & 0 \\ 0 \ar[ru] &    } \quad & \mbox{(ii)} \xymatrix{  0 \ar[rd] & \\ & 0 \\ \ker(\phi_{\uparrow k}) \ar[ru] &    } \quad & \mbox{(iii)}  \xymatrix{  (\phi^{\downarrow}_{k})^{-1}(V_k)\cap I^{\downarrow}_{k}   \ar[rd] & \\ & V_k \\ (\phi_{\uparrow k})^{-1}(V_k)\cap I_{\uparrow k} \ar[ru] &    } \\ \nonumber & \mbox{(iv)} \xymatrix{  J^{\downarrow}_{k}  \ar[rd] & \\ & \phi^{\downarrow}_{k}(J^{\downarrow}_{k}) \\ 0 \ar[ru] &    }  \quad & \mbox{(v)}  \xymatrix{  0  \ar[rd] & \\ & \phi_{\uparrow k}(J_{\uparrow k}) \\ J_{\uparrow k} \ar[ru] &    } \quad & \mbox{(vi)}
	\xymatrix{ 0 \ar[rd] & \\ & W_k \\ 0 \ar[ru] & } 
\end{align}

Let us write $\kappa_{k}$ for the graded $\kappa$-vector space---regarded as a chain complex with zero differential---whose $k$th graded piece is $\kappa$ and which is zero in all other degrees.  Then, the elementary ascending chain complex $\mathcal{E}_k(a,\infty)_{\uparrow}$ from Definition \ref{elemdfn} coincides as a chain complex with $\kappa_k$, and all of its nonzero elements have filtration level $a$; a similar statement holds for the elementary descending chain complex $\mathcal{E}_k(a,-\infty)^{\downarrow}=\overline{\mathcal{E}_k(-a,\infty)_{\uparrow}}$.

We then introduce the following notation for various indecomposable building blocks into which the summands in (\ref{sixsummands})(i)-(vi) can be further broken down:

\begin{align} \label{sixblocks}  (\downarrow^{a}_{-\infty})_k &=
	\left(\vcenter{\xymatrix{   \mathcal{E}_k(a,-\infty)^{\downarrow}\ar[rd] &   \\
			& 0 \\  0\ar[ru] & }}\right),	\quad (\uparrow_{a}^{\infty})_k=\left(\vcenter{
		\xymatrix{  0 \ar[rd] &   \\
			&  0, \\  \mathcal{E}_k(a,\infty)_{\uparrow} \ar[ru] & }}\right),\quad  (>_a^b)_k = \left(\vcenter{\xymatrix{ \mathcal{E}_k(b,-\infty)^{\downarrow}\ar[rd] & \\ & \kappa_k \\ \mathcal{E}_k(a,\infty)_{\uparrow} \ar[ru] &     }}\right), \\ \nonumber
	(\searrow^{a})_k & =
	\left(\vcenter{\xymatrix{   \mathcal{E}_k(a,-\infty)^{\downarrow}\ar[rd] &   \\
			& \kappa_k \\  0\ar[ru] & }}\right),	\quad (\nearrow_{a})_k=\left(\vcenter{
		\xymatrix{  0 \ar[rd] &   \\
			&  \kappa_k, \\  \mathcal{E}_k(a,\infty)_{\uparrow} \ar[ru] & }}\right),\quad \square_k=\left(\vcenter{\xymatrix{0 \ar[rd] & \\ & \kappa_k \\ 0 \ar[ru] &    }}\right).
\end{align}

Here each map that does not have zero as either its domain or its codomain is, set-theoretically, the identity map on $\kappa$.  By applying Proposition \ref{orthbasis} to $\overline{\ker(\phi^{\downarrow}_{k})},\ker(\phi_{\uparrow k}),\overline{J^{\downarrow}_{k}},$ and $J_{\uparrow k}$, it is easy to check that the filtered cospans in (\ref{sixsummands})(i),(ii),(iv), and (v) decompose up to isomorphism (in both $\mathsf{CO}(\kappa)$ and $\mathsf{HCO}(\kappa)$) into direct sums of terms of the form $(\downarrow^{a}_{-\infty})_k$, $(\uparrow_a^{\infty})_k$, $(\searrow^a)_k$, and $(\nearrow_a)_k$, respectively.  That (\ref{sixsummands})(vi) decomposes as a direct sum of copies of $\square_k$ follows by simply choosing a basis for $W_k$.    It is a little less trivial to see that (\ref{sixsummands})(iii) decomposes as a direct sum of terms form $(>_a^b)_k$, but this is a consequence of:

\begin{prop}\label{isomatch}
	Let $(X,\ell)$ and $(Y,\ell')$ be two filtered vector spaces such that $\ell(A\setminus\{0\})$ and $\ell'(B\setminus\{0\})$ are both well-ordered, and let $\phi\co A\to B$ be a vector space isomorphism.  Then there is an $\ell$-orthogonal basis $\{v_{\alpha}\}$ for $A$ such that $\{\phi(v_{\alpha})\}$ is an $\ell'$-orthogonal basis for $B$.
\end{prop}

\begin{proof}
	This follows formally from Proposition \ref{acyclic} applied to the ascending chain complex that is given in degree $1$ by $(X,\ell)$, in degree $0$ by $(Y,\ell')$, and in all other degrees by zero, and with differential given by $\phi$. 
\end{proof}

\begin{cor}\label{mixdecomp}
	The filtered cospan in (\ref{sixsummands})(iii) is isomorphic in $\mathsf{CO}(\kappa)$ and in $\mathsf{HCO}(\kappa)$ to a direct sum of the form $\oplus_{\alpha}(>_{a_{\alpha}}^{b_{\alpha}})_k$ for suitable choices of real numbers $a_{\alpha},b_{\alpha}$.
\end{cor}

\begin{proof}
	As noted before (\ref{sixsummands}), the maps appearing in (\ref{sixsummands})(iii), which are given by restricting $\phi^{\downarrow}_{k}$ and $\phi_{\uparrow k}$, are both vector space isomorphisms.  (In the case of (iii) this results from the facts that $I^{\downarrow}_{k}$ and $I_{\uparrow k}$ have trivial intersections with the kernels of $\phi^{\downarrow}_{k}$ and $\phi_{\uparrow k}$, respectively, and from the definition of $V_k$ as the intersection of the images of $\phi^{\downarrow}_{k}$ and $\phi_{\uparrow k}$.)  Hence we obtain a vector space  isomorphism $\phi \co (\phi_{\uparrow k})^{-1}(V_k)\cap I_k\to (\phi^{\downarrow}_{k})^{-1}(V_k)\cap I^{\downarrow}_{k}$ by composing the restriction of $\phi_{\uparrow k}$ with the inverse of the restriction of $\phi^{\downarrow}_{k}$.  Apply Proposition \ref{isomatch} to this isomorphism $\phi$, using $\ell=\rho_{\uparrow k}$ and $\ell'=-\rho^{\downarrow}_{k}$.  This yields a basis $\{v_{\alpha}\}$ of $(\phi_{\uparrow k})^{-1}(V_k)\cap I_k$.  Let $a_{\alpha}=\rho_{\uparrow k}(v_{\alpha})$ and $b_{\alpha}=\rho^{\downarrow}_{k}(\phi(v_{\alpha})$.  We then obtain an isomorphism from $\oplus_{\alpha}(>_{a_{\alpha}}^{b_{\alpha}})_k$ to (\ref{sixsummands})(iii) via the diagram \[ \xymatrix{\oplus_{\alpha}\mathcal{E}_k(b_{\alpha},-\infty)^{\downarrow}\ar[rr]\ar[rd] & & (\phi^{\downarrow}_{k})^{-1}(V_k)\cap I^{\downarrow}_{k} \ar[rd] & \\ &  \oplus_{\alpha}\kappa_k \ar[rr] & & V_k \\  \oplus_{\alpha}\mathcal{E}_k(a_{\alpha},\infty)_{\uparrow} \ar[ru] \ar[rr] & & (\phi_{\uparrow k})^{-1}(V_k)\cap I_k \ar[ru] }\] where the horizontal maps send the standard generator for $\mathcal{E}_k(a_{\alpha},\infty)_{\uparrow} $ to $v_{\alpha}$, the standard generator for the $\alpha$th summand of $\oplus_{\alpha}\kappa_k$ to $\phi_{\uparrow k}(v_{\alpha})$, and the standard generator for $\mathcal{E}_k(b_{\alpha},-\infty)^{\downarrow}$ to $\phi(v_{\alpha})$.
\end{proof}

\begin{dfn}\label{sesdef}
	A \emph{standard elementary summand} is any one of the building blocks $(\downarrow_{b}^{a})_k$ or $(\uparrow_{a}^{b})_k$ from (\ref{basicacyc}), or $(\downarrow^{a}_{-\infty})_k$, $(\uparrow_{a}^{\infty})_k$,  $(>_{a}^{b})_k$, $(\searrow^a)_k$, $(\nearrow_a)_k$, or $\square_k$ from (\ref{sixblocks}). 
\end{dfn}

\begin{cor}\label{decompexists}
	Any admissible filtered cospan $\mathcal{C}$ with coefficients in a field $\kappa$ is isomorphic in $\mathsf{HCO}(\kappa)$ to a direct sum of standard elementary summands.
\end{cor}

\begin{proof}
	Proposition \ref{coah} decomposes $\mathcal{C}$ into an acyclic filtered cospan $\mathcal{A}_{\mathcal{C}}$ and a perfect filtered cospan $\mathcal{H}_{\mathcal{C}}$.  As noted before (\ref{basicacyc}), Corollary \ref{ascsplit} implies that $\mathcal{A}_{\mathcal{C}}$ is isomorphic to a direct sum of a collection of objects of form $(\downarrow^{a}_{b})_k$ and $(\uparrow_{a}^{b})_k$.
	Meanwhile, $\mathcal{H}_{\mathcal{C}}$ decomposes as a direct sum of the summands in (\ref{sixsummands}), which are respectively isomorphic to direct sums of objects as in (\ref{sixblocks}) using Proposition \ref{orthbasis} and Corollary \ref{mixdecomp}.
\end{proof}

The summands in the decomposition of Corollary \ref{decompexists} are uniquely determined by $\mathcal{C}$; we defer the proof of this to Corollary \ref{decompunique}.

Of course, the conclusion of Corollary \ref{decompexists} also applies to any filtered cospan $\mathcal{C}$ that is isomorphic in $\mathsf{HCO}(\kappa)$ to an admissible filtered cospan, even if $\mathcal{C}$ is not itself admissible. For example, if $f\co\mathcal{X}\to [-\Lambda,\Lambda]$ is a simplexwise linear function on the geometric realization of a finite simplicial complex, the pinned singular filtered cospan $\mathcal{S}_{\partial}(\mathbb{X},f;\kappa)$ is (unless $f$ is locally constant) very far from being admissible, but by Proposition \ref{simpsing} $\mathcal{S}_{\partial}(\mathbb{X},f;\kappa)$ is isomorphic to the pinned simplicial filtered cospan $\mathcal{C}_{\partial}(X,f;\kappa)$ which is admissible.
 
\begin{remark}
	If $\mathcal{C}$ is $\Lambda$-bounded, then all values of the parameters $a$ and $b$ that arise in summands of form  $(\downarrow_{b}^{a})_k,(\uparrow_{a}^{b})_k,(\downarrow^{a}_{-\infty})_k, (\uparrow_{a}^{\infty})_k,  (>_{a}^{b})_k, (\searrow^a)_k,$ or $(\nearrow_a)_k$ in a decomposition of $\mathcal{C}$ will evidently lie in the open interval $(-\Lambda,\Lambda)$. 
\end{remark} 

\section{More isomorphisms in $\mathsf{HCO}(\kappa)$}\label{quasisect}

In this section we establish a sufficient criterion for filtered homotopy equivalence (Proposition \ref{quasihtopy}) that facilitates checking the condition in Corollary \ref{hcoiso} for two filtered cospans to be isomorphic in $\mathsf{HCO}(\kappa)$, and we use this to establish isomorphisms between the filtered cospans in Examples \ref{pinsing}, \ref{simpdef}, and \ref{morseex}.

\begin{dfn}
	If $(C_*,\partial_C,\ell_{\uparrow}^{C})$ and $(D_*,\partial_D,\ell_{\uparrow}^{D})$ are two ascending chain complexes and if $f\co C_*\to D_*$ is a filtered chain map, we say that $f$ is a \textbf{filtered quasi-isomorphism} if, for each $k\in\Z$ and $t\in J$, the induced map $f_*\co H_k(C_{*}^{\leq t})\to H_k(D_{*}^{\leq t})$ is an isomorphism.
\end{dfn}

If $f\co C_*\to D_*$ is a filtered chain map, the mapping cone $\mathrm{Cone}(f)$ has the structure of an ascending chain complex, with $\mathrm{Cone}(f)_{k}^{\leq t}=C_{k-1}^{\leq t}\oplus D_{k}^{\leq t}$.  From the standard exact sequence \[ \xymatrix{\cdots \ar[r] & H_k(C^{\leq t})\ar[r]^{f_*} & H_k(D^{\leq t})\ar[r]^<<<<{j_*} & H_k(Cone(f)^{\leq t})\ar[r]^<<<<{p_*} &  H_{k-1}(C^{\leq t})\ar[r]^{f_*} & H_{k-1}(D^{\leq t})  } \] (where $j\co D_*\to \mathrm{Cone}(f)$ and $p\co \mathrm{Cone}(f)\to C[-1]$ are the inclusion and projection) we see that $f$ is a filtered quasi-isomorphism if and only if $\mathrm{Cone}(f)$ is \textbf{filtered acyclic} in the sense that $H_k(\mathrm{Cone}(f)^{\leq t})=0$ for all $k$ and $t$.

\begin{prop} \label{quasihtopy}
	Let $(C_*,\partial_C,\ell_{\uparrow}^{C})$ and $(D_*,\partial_D,\ell_{\uparrow}^{D})$ be ascending chain complexes with the following properties:
	\begin{itemize} \item Each $C_k$ admits a $\ell_{\uparrow}^{C}$-orthogonal basis, and each $D_k$ admits a $\ell_{\uparrow}^{D}$-orthogonal basis;
		\item $C_k=D_k=0$ for all $k<0$.\end{itemize}
	Then every filtered quasi-isomorphism from $C_*$ to $D_*$ is a filtered homotopy equivalence.
\end{prop}

The proof depends on the following lemma, following a strategy similar to that used in \cite[Lemma 10.4.6]{Wei}.

\begin{lemma} \label{nullhtopy}
	Let $(D_*,\partial_D,\ell_{\uparrow}^{D})$ be an ascending chain complex satisfying the conditions in Proposition \ref{quasihtopy}, and let $(A_*,\partial_A,\ell_{\uparrow}^{A})$ be any ascending chain complex that is filtered acyclic.  Then every filtered chain map $j\co D_*\to A_*$ is filtered homotopic to $0$.
\end{lemma}

\begin{proof}[Proof of Lemma \ref{nullhtopy}]
	Write $j_n\co D_n\to A_n$, $\partial_{D}^{n}\co D_n\to D_{n-1}$, and $\partial_{A}^{n}\co A_n\to A_{n-1}$ for the restrictions of $j,\partial_D,\partial_A$ to the appropriate graded pieces.  We must construct, for each $n\in \Z$, filtered maps $K_n\co D_n\to A_{n+1}$ (thus $\ell_{\uparrow}^A\circ K_n\leq \ell_{\uparrow}^{D}|_{D_n}$) such that \begin{equation}\label{htopyeq} \partial^{n+1}_{A}\circ K_n+K_{n-1}\circ \partial^n_{D}=j_n.    \end{equation}  Since we assume that $D_n=0$ for all $n<0$, (\ref{htopyeq}) is satisfied for all $n<0$ using $K_n=0$ for all $n<0$.  So assume that $m\geq 0$ and that we have constructed $K_n$ for all $n<m$ in such a way that (\ref{htopyeq}) is satisfied for all $n<m$.  Since $j$ is a chain map, we then have \[ \partial^m_A\circ(j_m-K_{m-1}\circ\partial^m_D)=(j_{m-1}-\partial^m_A\circ K_{m-1})\circ \partial^{m}_D=K_{m-2}\circ\partial^{m-1}_{D}\circ\partial^m_D=0.  \]
	Moreover, $j_m-K_{m-1}\circ\partial_{D}^{m}$ is a filtered map by the inductive hypothesis.  So since $A_*$ is filtered acyclic, for each $d\in D_m$ there is $a_d\in A_{m+1}$ with $\ell_{\uparrow}^A(a_d)\leq \ell_{\uparrow}^D(d)$ and $\partial_{A}^{m+1}a_d=j_md-K_{m-1}\partial_D^md$.  
	
	If $\{d_i\}_{i\in\mathcal{I}}$ is a $\ell_{\uparrow}^{D}$-orthogonal basis for $D_{m}$, we may then set $K_{m}\left(\sum_i r_id_i\right)=\sum_i r_ia_{d_i}$.  Then $K_m$ is a filtered map by the orthogonality of $\{d_i\}$, and (\ref{htopyeq}) holds with $n=m$, completing the inductive argument.  
\end{proof}

\begin{proof}[Proof of Proposition \ref{quasihtopy}]
	Let $f\co C_*\to D_*$ be a filtered quasi-isomorphism and apply Lemma \ref{nullhtopy} with $A_*=\mathrm{Cone}(f)$ and $j$ equal to inclusion of $D_*$ into $\mathrm{Cone}(f)$.  The nullhomotopy $K$ produced by Lemma \ref{nullhtopy} is then a filtered map $D_*\to C_{*}\oplus D_*[1]$.  If we write the components of this map as $h\co D_*\to C_*$ and $L\co D_*\to D_*[1]$, the fact that $K$ is a nullhomotopy of $j$ is easily seen to be equivalent to the statements that $h$ is a chain map and that the filtered map $L$ is a homotopy between $f\circ h$ and the identity on $D_*$.
	
	Thus $h$ is a right filtered homotopy inverse to $f$.  Now $f$ is assumed to be a filtered quasi-isomorphism, so it follows that the maps induced by $h$ on filtered homology are the inverses to those induced by $f$.  In particular, $h$ is a filtered quasi-isomorphism.
	
	We may now apply the first part of the proof with the roles of $C_*$ and $D_*$ interchanged, and with $h$ in place of $f$.  This yields a filtered map $g\co C_*\to D_*$ such that $h\circ g$ is filtered homotopic to the identity on $C$.  But then consideration of $f\circ h\circ g$ readily implies that $f$ and $g$ are filtered homotopic to each other, and that either one of them is a two-sided filtered homotopy inverse to $h$.  In particular, $f$ is a filtered homotopy equivalence.
\end{proof}

\begin{remark}\label{qhrmk}
	The obvious analogue of Proposition \ref{quasihtopy} for descending chain complexes rather than ascending ones also holds, as can be seen by negating the filtration functions to convert from descending to ascending complexes.  Also, in the unfiltered case, if $(C_*,\partial_C)$ and $(D_*,\partial_D)$ are complexes of $\kappa$-vector spaces such that $C_k=D_k=0$ for $k<0$ then any quasi-isomorphism from $C_*$ to $D_*$ is a homotopy equivalence; this can be viewed as a special case of Proposition \ref{quasihtopy} if $C_*$ and $D_*$ are endowed with the trivial $\kappa$-ascending chain complex structures in which all nonzero elements have filtration level zero, or alternatively it can be inferred (more generally for bounded-below complexes of projective modules) from \cite[Corollary 10.4.7]{Wei}.
\end{remark}

We now establish the promised isomorphism between the pinned singular and pinned simplicial filtered cospans.

\begin{prop}\label{simpsing} With notation as in Examples \ref{pinsing} and \ref{simpdef}, for a simplicial complex $X$ with geometric realization $\mathbb{X}$ and for a simplexwise linear function $f\co\mathbb{X}\to [-\Lambda,\Lambda]$, we have an isomorphism in $\mathsf{HCO}^{\Lambda}(\kappa)$ \begin{equation}\label{simpsingdiag} \xymatrix{ \frac{C_*(X^+;\kappa)}{C_*(\partial^+X;\kappa)}\ar[rrrd]^{K^{\downarrow}} \ar[rr]^{\iota^{\downarrow}} \ar[rd] & & S_*(\mathbb{X}\setminus\partial^-\mathbb{X},\partial^+\mathbb{X};\kappa)\ar[rd]  & \\ &\frac{C_*(X;\kappa)}{C_*(\partial X;\kappa)} \ar[rr]^{\iota} &&  S_*(\mathbb{X},\partial \mathbb{X};\kappa) \\ \frac{C_*(X^-;\kappa)}{C_*(\partial^-X;\kappa)} \ar[ur] \ar[rr]_{\iota_{\uparrow}}\ar[rrru]^{K_{\uparrow}} && S_*(\mathbb{X}\setminus\partial^+\mathbb{X},\partial^-X;\kappa) \ar[ur] &             }\end{equation}
	between $\mathcal{C}_{\partial}(X,f;\kappa)$ and $\mathcal{S}_{\partial}(\mathbb{X},f;\kappa)$, for some choice of degree-one maps $K_{\uparrow},K^{\downarrow}$.	
\end{prop}

\begin{proof}
	The usual basis for $S_k(\mathbb{X}\setminus\partial^+\mathbb{X},\partial^-\mathbb{X};\kappa)$, consisting of maps $\sigma\co\Delta^k\to \mathbb{X}\setminus\partial^+\mathbb{X}$ with image not contained in $\partial^-\mathbb{X}$, is orthogonal for the filtration function $\ell_{\uparrow}$ on $S_k(\mathbb{X}\setminus\partial^+\mathbb{X},\partial^-\mathbb{X};\kappa)$.  Similar remarks apply to  $S_k(\mathbb{X}\setminus\partial^-\mathbb{X},\partial^+\mathbb{X};\kappa)$, to $
	\frac{C_*(X^-;\kappa)}{C_*(\partial^-X;\kappa)}$, and to $\frac{C_*(X^+;\kappa)}{C_*(\partial^+X;\kappa)}$.	The diagram (\ref{simpsingdiag}) commutes if we set $K_{\uparrow}=K^{\downarrow}=0$, so by Corollary \ref{hcoiso}, Proposition \ref{quasihtopy}, and Remark \ref{qhrmk}, it suffices to show that $\iota$ is a quasi-isomorphism and that $\iota_{\uparrow}$ and $\iota^{\downarrow}$ are filtered quasi-isomorphisms.
	
	That $\iota$ is a quasi-isomorphism is immediate from \cite[Theorem 2.27]{Ha}. The cases of $\iota_{\uparrow}$ and $\iota^{\downarrow}$ are identical (and can be converted to each other by reversing the sign of $f$), so we just consider $\iota_{\uparrow}$.  Given $t\in [-\Lambda,\Lambda)$, let $X^{\leq t}$ denote the subcomplex of the simplicial complex $X$ consisting of simplices whose vertices all have $f\leq t$, and let $\mathbb{X}^{\leq t}$ union of these simplices in the geometric realization $\mathbb{X}$ (thus $\mathbb{X}^{\leq t}$ is a geometric realization of $X^{\leq t}$).  We have $\mathbb{X}^{\leq t}\subset f^{-1}([-\Lambda,t])$; typically the inclusion is strict, due to the presence of simplices in which some vertices have $f\leq t$ and others have $f>t$.  The restriction of $\iota_{\uparrow}$ to the $t$-filtered part of 
	$\frac{C_*(X^-;\kappa)}{C_*(\partial^-X;\kappa)}$ is equal to the composition \[ \frac{C_*(X^{\leq t};\kappa)}{C_*(\partial^-X;\kappa)}\to S_*(\mathbb{X}^{\leq t},\partial^-\mathbb{X};\kappa)\to S_*(f^{-1}([-\Lambda ,t]),\partial^-\mathbb{X};\kappa) \] where the first map is the standard inclusion of simplicial into singular homology (which is a quasi-isomorphism by \cite[Theorem 2.27]{Ha}) and the second map is induced by the inclusion $(\mathbb{X}^{\leq t},\partial^-\mathbb{X})\to (f^{-1}([-\Lambda,t]),\partial^-\mathbb{X})$. But by a standard argument (see, \emph{e.g.}, the discussion around \cite[Figure VI.8]{EH}) the latter inclusion is a homotopy equivalence.  Thus, for any $t$, $\iota_{\uparrow}$ restricts as a quasi-isomorphism from the $t$-filtered part of   $\frac{C_*(X^-;\kappa)}{C_*(\partial^-X;\kappa)}$ to the $t$-filtered part of $S_*(\mathbb{X}\setminus\partial^+\mathbb{X},\partial^-\mathbb{X};\kappa)$, completing the proof that $\iota_{\uparrow}$ is a filtered quasi-isomorphism.
\end{proof}

The relation between the pinned singular and Morse filtered cospans  requires more work, involving in particular the subtleties of the construction of the homotopy equivalence $\mathcal{E}(f,v)$ from \cite{Paj} between the Morse and singular chain complexes.  As noted in Example \ref{morseex}, it is only the homotopy class of $\mathcal{E}(f,v)$ that is determined canonically by the description in \cite{Paj}; the map itself depends on various auxiliary choices made in the course of the construction.  We defer a detailed discussion of this to Appendix \ref{app}, in which we prove the following lemma.

\begin{lemma}\label{Efiltered}
	The map $\mathcal{E}(f,v)$ can be constructed consistently with the procedure in \cite{Paj} to have image contained in $S_*(\mathbb{X}\setminus\partial^+\mathbb{X},\partial^-\mathbb{X};\kappa)\subset S_*(\mathbb{X},\partial^-\mathbb{X};\kappa)$, and to be a filtered quasi-isomorphism $\mathcal{E}(f,v)\co CM_*(f,v;\kappa)\to S_*(\mathbb{X}\setminus\partial^+\mathbb{X},\partial^-\mathbb{X};\kappa)$.
\end{lemma} 

Given Lemma \ref{Efiltered}, the isomorphism between the pinned singular and Morse fitlered cospans follows readily:

\begin{prop}\label{morseiso}
	Let $\mathbb{X}$ be a compact smooth manifold with boundary and let $f\co\mathbb{X}\to [-\Lambda,\Lambda]$ be a Morse function having no critical points on $\partial\mathbb{X}$, with $\partial\mathbb{X}=\partial^-\mathbb{X}\cup\partial^+\mathbb{X}$ where $\partial^{\pm}\mathbb{X}=f^{-1}(\{\pm\Lambda\})$.  Then, for a gradient-like vector field $v$ for $f$, the Morse filtered cospan $\mathcal{M}(\mathbb{X},f,v;\kappa)$ from Example \ref{morseex} is isomorphic to the pinned singular filtered cospan $\mathcal{S}_{\partial}(\mathbb{X},f;\kappa)$, via an isomorphism of the form 
	\[
	\xymatrix{ CM_*(-f,-v;\kappa)\ar[rrrd]^{K^{\downarrow}} \ar[rr]^{\mathcal{E}(-f,-v)} \ar[rd]_{\psi^{\downarrow}} & & S_*(\mathbb{X}\setminus\partial^-\mathbb{X},\partial^+\mathbb{X};\kappa)\ar[rd]^{\pi^+}  & \\ & S_*(\mathbb{X},\partial \mathbb{X};\kappa) \ar@{=}[rr]  &&  S_*(\mathbb{X},\partial \mathbb{X};\kappa) \\ CM_*(f,v;\kappa)\ar[ur]^{\psi_{\uparrow}} \ar[rr]_{\mathcal{E}(f,v)}\ar[rrru]^{K_{\uparrow}} && S_*(\mathbb{X}\setminus\partial^+\mathbb{X},\partial^-X;\kappa) \ar[ur]_{\pi^-} &             }
	\] where the homotopy equivalences $\mathcal{E}(\pm f,\pm v)$ are defined using suitable choices that are consistent with the construction in \cite[Chapter 6]{Paj}. (In particular, with these choices, their images are contained in the subcomplexes $S_*(\mathbb{X}\setminus\partial^{\pm}\mathbb{X},\partial^{\mp}\mathbb{X};\kappa)$ of $S_*(\mathbb{X},\partial^{\mp}\mathbb{X};\kappa)$.)
\end{prop}

\begin{proof}
	Recall that $\psi^{\downarrow}=\pi^+\circ \mathcal{E}(-f,-v)$ and $\psi_{\uparrow}=\pi^-\circ\mathcal{E}(f,v)$, so the above diagram commutes if we set $K^{\downarrow}=K_{\uparrow}=0$.  (As noted in Example \ref{morseex}, different choices made in the construction of $\mathcal{E}(\pm f,\pm v)$ result in homotopic maps, so the diagram still commutes up to homotopy if the versions of $\mathcal{E}(\pm f,\pm v)$ that are used to construct $\psi^{\downarrow}$ and $\psi_{\uparrow}$ are different from the versions constructed in the diagram.)  
	
	By Lemma \ref{Efiltered}, $\mathcal{E}(f,v)$ can be constructed consistently with \cite{Paj} to be a filtered quasi-isomorphism from $CM_*(f,v;\kappa)$ to $S_*(\mathbb{X}\setminus\partial^+\mathbb{X},\partial^-\mathbb{X};\kappa)$.   By Proposition \ref{quasihtopy} this version of $\mathcal{E}(f,v)$ is then a filtered homotopy equivalence.  Similarly (after negating filtrations to convert descending complexes into ascending ones) a suitable version of $\mathcal{E}(-f,-v)$ is a filtered homotopy equivalence.  The conclusion now follows from Proposition \ref{hcoiso}.	
\end{proof}

\section{Calculation and examples}\label{calcsect}

We now discuss more practically how one finds the decomposition promised by Corollary \ref{decompexists}, and carry this out in some examples instantiating Examples \ref{simpdef} and \ref{morseex}.  We assume here that $\mathcal{C}=\left((C_{\uparrow *},\partial_{\uparrow}^C,\ell_{\uparrow}^C),(C^{\downarrow}_{*},\partial^{\downarrow}_C,\ell^{\downarrow}_C),(D_*,\partial_D),\psi_{\uparrow},\psi^{\downarrow}\right)$ is a filtered cospan over a field $\kappa$ such that $C_{\uparrow *}$ and $C^{\downarrow}_{*}$ are both finite-dimensional (which implies that $\mathcal{C}$ is admissible).  The discussion leading to Corollary \ref{decompexists} shows that the decomposition can be read off from bases $\mathcal{X}_{\uparrow k}=\mathcal{A}_{\uparrow k}\sqcup \mathcal{B}_{\uparrow k}\sqcup \mathcal{H}_{\uparrow k}$ and $\mathcal{X}^{\downarrow}_{k} =\mathcal{A}^{\downarrow}_{k}\sqcup \mathcal{B}^{\downarrow}_{k}\sqcup \mathcal{H}^{\downarrow}_{k}$ for the various $C_{\uparrow k}$ and $C^{\downarrow}_{k}$, respectively, with the following properties:
\begin{enumerate}
	\item \label{xbasis1} $\mathcal{X}_{\uparrow k}$ is $\ell_{\uparrow}^{C}$-orthogonal and $\mathcal{X}^{\downarrow}_{k}$ is $(-\ell^{\downarrow}_{C})$-orthogonal, \emph{i.e.} for $c_x\in \kappa$, \[ \ell_{\uparrow}^{C}\left(\sum_{x\in \mathcal{X}_{\uparrow k}}c_x x\right)=\max\{\ell_{\uparrow}^{C}(x)|c_x\neq 0\}\quad \mbox{and}\quad \ell^{\downarrow}_{C}\left(\sum_{x\in \mathcal{X}^{\downarrow}_{k}}c_xx\right)=\min\{\ell^{\downarrow}_{C}(x)|c_x\neq 0\}. \]
	\item \label{xbasis2} $\mathcal{B}_{\uparrow k}\sqcup \mathcal{H}_{\uparrow k}$ is a basis for $\ker(\partial_{\uparrow}^{C}|_{C_{\uparrow k}})$, and $\mathcal{B}^{\downarrow}_{k}\sqcup \mathcal{H}^{\downarrow}_{k}$ is a basis for $\ker(\partial^{\downarrow}_{C}|_{C^{\downarrow}_{k}})$.
	\item \label{xbasis3} $\partial_{\uparrow}^C$ maps $\mathcal{A}_{\uparrow\,k+1}$ bijectively to $\mathcal{B}_{\uparrow k}$, and $\partial^{\downarrow}_C$ maps $\mathcal{A}^{\downarrow}_{k+1}$ bijectively to $\mathcal{B}^{\downarrow}_{k}$.  Thus the quotient projections (modulo the images of $\partial^{C}_{\uparrow}$ and $\partial_{C}^{\downarrow}$) identify $\mathcal{H}_{\uparrow k}$ and $\mathcal{H}^{\downarrow}_{k}$ with bases for the respective homologies $H_k(C_{\uparrow *})$ and $H_k(C^{\downarrow}_{*})$.
	\item \label{xbasis4} For some nonnegative integers $m,p,p',q,q'$ (depending on  $k$), we can write \begin{align*} \mathcal{H}_{\uparrow k}&=\{x_1,\ldots,x_m,x_{m+1},\ldots,x_{m+p},x_{m+p+1},\ldots,x_{m+p+q}\},\\ \mathcal{H}^{\downarrow}_{k}&=\{y_1,\ldots,y_m,y_{m+1},\ldots,y_{m+p'},y_{m+p'+1},\ldots y_{m+p'+q'}\} \end{align*} where (denoting by $[d]$ the homology class of a cycle $d\in D_k$):
	\begin{itemize} \item[(i)] $[\psi_{\uparrow}x_i]=[\psi^{\downarrow}y_i]$ for $i=1,\ldots,m$;
		\item[(ii)] $\{[\psi_{\uparrow}x_1],\ldots,[\psi_{\uparrow}x_m],[\psi_{\uparrow}x_{m+1}],\ldots[\psi_{\uparrow}x_{m+p}],[\psi^{\downarrow}y_{m+1}],\ldots,[\psi^{\downarrow}y_{m+p'}]\}$ is a linearly independent set in $H_k(D_*)$; and 
		\item[(iii)] $[\psi_{\uparrow}x_{m+p+i}]=[\psi^{\downarrow}y_{m+p'+i'}]=0\in H_k(D_*)$ for all $i=1,\ldots,q$,\,$i'=1,\ldots,q'$. \end{itemize}
\end{enumerate}
In this case, $\mathcal{C}$ is isomorphic to a direct sum of the following, as $k$ varies:
\begin{itemize}
	\item For each $x\in \mathcal{A}_{\uparrow\,k+1}$ such that $\ell_{\uparrow}^{C}(\partial_{\uparrow}^{C}x)<\ell_{\uparrow}^{C}(x)$, a copy of $\left(\uparrow_{\ell_{\uparrow}^{C}(\partial_{\uparrow}^{C}x)}^{\ell_{\uparrow}^{C}(x)}\right)_k$;
\item For each $x\in \mathcal{A}^{\downarrow}_{k+1}$ such that $\ell^{\downarrow}_{C}(\partial^{\downarrow}_{C}x)>\ell^{\downarrow}_{C}(x)$, a copy of $\left(\downarrow^{\ell^{\downarrow}_{C}(\partial^{\downarrow}_{C}x)}_{\ell^{\downarrow}_{C}(x)}\right)_k$;
\item For $i=1,\ldots,m$, a copy of $\left(>_{\ell_{\uparrow}^{C}(x_i)}^{\ell^{\downarrow}_{C}(y_i)}\right)_k$;
\item For $i=m+1,\ldots,m+p$, a copy of $(\nearrow_{\ell_{\uparrow}^{C}(x_i)})_k$; 
\item For $i=m+1,\ldots,m+p'$, a copy of $(\searrow^{\ell^{\downarrow}_{C}(y_i)})_k$;
\item For $i=m+p+1,\ldots,m+p+q$, a copy of $\left(\uparrow_{\ell_{\uparrow}^{C}(x_i)}^{\infty}\right)_k$; 
\item For $i=m+p'+1,\ldots,m+p'+q'$, a copy of $\left(\downarrow^{\ell^{\downarrow}_{C}(y_i)}_{-\infty}\right)_k$; and 
\item A number of copies of $\square_k$ equal to $\dim H_k(D_*)-(m+p+p')$ (\emph{i.e.}, to the dimension of the cokernel of $\psi_{\uparrow *}+\psi^{\downarrow}_{*}\co H_k(C_{\uparrow *})\oplus H_k(C^{\downarrow}_{*})\to H_k(D_*)$).
\end{itemize}

While we do not pursue here the goal of finding an optimally efficient algorithm for finding suitable bases $\mathcal{X}_{\uparrow k}$ and $\mathcal{X}^{
\downarrow}_{k}$, we remark that the standard Gaussian-elimination-based algorithms for persistence barcodes that go back to \cite{Bar},\cite{ZC} are sufficient to achieve conditions (\ref{xbasis1}) through (\ref{xbasis3}), and that the bases $\mathcal{H}_{\uparrow k}$ and $\mathcal{H}^{\downarrow}_{k}$   can then be modified by related techniques in order to arrange for (\ref{xbasis4}) to hold as well.

In situations such as those as Examples \ref{simpdef} and \ref{morseex}, the standard bases (given by the $k$-simplices or the index-$k$ critical points) of $C_{\uparrow k}$ and $C^{\downarrow}_{k}$ are orthogonal, but do not typically satisfy any of (\ref{xbasis1}) through (\ref{xbasis4}), and thus must be modified to different orthogonal bases.  Proposition \ref{orthcrit}(ii) then gives a convenient way to check that these new bases are in fact orthogonal.

In all of the examples below, the filtered cospans will be $\Lambda$ bounded, and we set the parameter $\Lambda$ equal to 2.

\subsection{Simplicial complexes}

The following two examples will illustrate the calculation of the isomorphism type of the simplicial filtered cospan $\mathcal{C}_{\partial}(X,f;\kappa)$ for simplicial complex $X$ and a simplexwise linear function $f\co \mathbb{X}\to [-2,2]$ on the geometric realization of $X$, as defined in Example \ref{simpdef}.  (In the first example, $f$ will have image in $[-1,1]$, so that, in the notation of Example \ref{simpdef}, $\partial X=\varnothing$ and $X^-=X^+=X$).

\begin{ex}\label{hornex} Consider the simplicial complex $X$ and function $f$ depicted in Figure \ref{hornfig}.  Thus $X$ is the $0$-horn of the standard $3$-simplex (\emph{i.e.} the boundary of $\Delta^3$ with the face opposite the zeroth vertex deleted), and $f$ takes values $1,0,0,$ and $-1$ on the vertices $v_0,v_1,v_2,v_3,$ respectively.
	
As $\partial X$ is empty in this case, the complexes $(C_{\uparrow *},\partial_{\uparrow}^{C}),(C^{\downarrow}_{*},\partial^{\downarrow}_{C}),$ and $(D_*,\partial_D)$ are each equal to the simplicial chain complex of $X$; the functions $\ell_{\uparrow}^{C}$ and $\ell^{\downarrow}_{C}$ send a simplicial chain to, respectively, the maximal and minimal value of $f$ on the support of the chain.	

Denote by $\partial^k$ the degree-$k$ part of the simplicial boundary operator of $X$ (sending $k$-chains to $(k-1)$-chains).  We see that, with the various simplices labeled as in the figure, $\partial^2$ is injective,  $\ker\partial^1=\Img\partial^2=\mathrm{span}\{e_1-e_2+e_4,e_1-e_3+e_5,e_2-e_3+e_6\}$, and $\Img\partial^1=\left\{\left.\sum_{i=0}^{3}a_iv_i\right|a_0+a_1+a_2+a_3=0\right\}$. 	Thus the homology is, not surprisingly, trivial in degrees other than zero, and in degree zero it is one-dimensional, generated by the class of any one of the vertices $v_i$.

Let us find suitable bases $\mathcal{X}^{\downarrow}_{k}$ and $\mathcal{X}^{k}_{\uparrow}$ as described at the start of this section.    The main subtlety in this example arises for $k=1$: as the homology is trivial in degree $1$, we are to construct $(-\ell^{\downarrow}_{C})$- and $\ell_{\uparrow}^{C}$-orthogonal bases $\mathcal{X}^{\downarrow}_{1}=\mathcal{A}^{\downarrow}_{1}
\sqcup \mathcal{B}^{\downarrow}_{1}$ and $\mathcal{X}_{\uparrow}^{1}=\mathcal{A}_{\uparrow}^{1}\sqcup\mathcal{B}_{\uparrow}^{1}$ such that $\mathcal{B}^{\downarrow}_{1}$ and $\mathcal{B}_{\uparrow}^{1}$ span $\Img\partial_2$ and have orthogonal preimages under $\partial^2$, and such that the images of $\mathcal{A}_{\uparrow}^{1}$ and $\mathcal{A}^{\downarrow}_{1}$ under $\partial^1$ are also orthogonal.

\begin{center}
	\begin{figure}
		\includegraphics[width=3in]{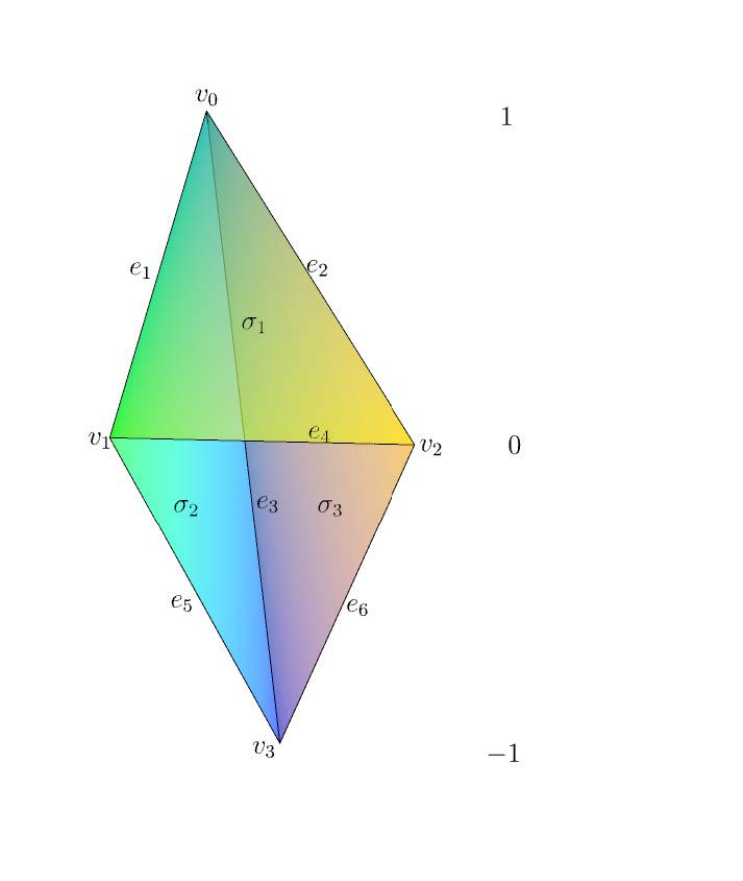}
		\caption{The function from Example \ref{hornex}}
		\label{hornfig}
	\end{figure}
\end{center}

To form $\mathcal{X}^{\downarrow}_{1}$ we may take $\mathcal{B}^{\downarrow}_{1}$ equal to $\{\partial \sigma_1,\partial \sigma_2,\partial \sigma_3\}=\{e_1-e_2+e_4,e_1-e_3+e_5,e_2-e_3+e_6\}$, and $\mathcal{A}_{1}^{\downarrow}=\{e_1,e_4,e_5\}$. (So by (\ref{xbasis3}) we must then also take $\mathcal{X}^{\downarrow}_{2}=\mathcal{A}^{\downarrow}_{2}=\{\sigma_1,\sigma_2,\sigma_3\}$ and $\mathcal{B}^{\downarrow}_{0}=\partial\left(\mathcal{A}^{\downarrow}_{1}\right)=\{v_1-v_0,v_2-v_1,v_3-v_1\}$.)  Indeed, it is not hard to check that $\mathcal{B}^{\downarrow}_{0}$ is then $(-\ell^{\downarrow}_{C})$-orthgonal, and $\mathcal{A}^{\downarrow}_{2}$ is   $(-\ell^{\downarrow}_{C})$-orthgonal directly by the definition of $\ell^{\downarrow}_{C}$.  To see that the basis $\mathcal{X}^{\downarrow}_{1}=\mathcal{A}^{\downarrow}_{1}\cup \mathcal{B}^{\downarrow}_{1}=\{e_1,e_4,e_5,e_1-e_2+e_4,e_1-e_3+e_5,e_2-e_3+e_6\}$ for the one-chains is $(-\ell^{\downarrow}_{C})$-orthogonal, one may appeal to  Proposition \ref{orthcrit}(ii), noting that $\ell^{\downarrow}_{C}$ maps both this basis and the ``standard'' basis $\{e_1,e_2,e_3,e_4,e_5,e_6\}$ to $\{0,0,0,-1,-1,-1\}$.  

To complete the discussion of the bases $\mathcal{X}^{\downarrow}_{k}$ we need to supplement the basis $\mathcal{B}^{\downarrow}_{0}=\{v_1-v_0,v_2-v_1,v_3-v_1\}$ for $\Img\partial_1$ to a $(-\ell^{\downarrow}_{C})$-orthogonal basis $\mathcal{B}^{\downarrow}_{0}\cup \mathcal{H}^{\downarrow}_{0}$ for $\mathrm{span}\{v_0,v_1,v_2,v_3\}$.  One readily sees (for instance by Proposition \ref{orthcrit}(ii) again)  that the unique way to do this (up to a meaningless scalar factor) is to take $\mathcal{H}^{\downarrow}_{0}=\{v_0\}$.

To start to translate this into the desired decomposition into standard elementary summands, observe that, for both $k=0$ and $k=1$, each
 $x\in\mathcal{A}^{\downarrow}_{k+1}$ has $\ell^{\downarrow}_{C}(\partial^{k+1}x)=\ell^{\downarrow}_{C}(x)$.  So the $\mathcal{A}^{\downarrow}_{k+1}$ and $\mathcal{B}^{\downarrow}_{k}$ do not contribute any terms $(\downarrow_{b}^{a})_k$ to the decomposition 
(the underlying reason being that what might be denoted as ``$(\downarrow_{a}^{a})_k$'' is isomorphic to zero in $\mathsf{HCO}(\kappa)$).  On the other hand, as the single element $v_0$ of $\mathcal{H}_{0}^{\downarrow}$ has $\ell^{\downarrow}_{C}(v_0)=1$, it will contribute a summand of the form $(>^{1}_{b})_0$; we will see below after constructing $\mathcal{H}_{\uparrow}^{0}$ that $b=-1$.\footnote{Since the maps that in the general
 theory are denoted by $\psi^{\downarrow}$ and $\psi_{\uparrow}$ in this case induce isomorphisms on homology (indeed, they are the identity), there will be no summands of any of the forms $(\uparrow_{a}^{\infty})_k,(\downarrow^{a}_{-\infty})_k,(\nearrow_a)_k,(\searrow^a)_k,$ or $\square_k$.}

We turn now to the $\mathcal{X}_{\uparrow}^{k}$.  It will not work to have $\mathcal{A}_{\uparrow}^{2}$ coincide with the standard basis $\mathcal{A}^{\downarrow}_{2}=\{\sigma_1,\sigma_2,\sigma_3\}$, because the image of this basis is not $\ell_{C}^{\downarrow}$-orthogonal: one has \[ \ell^{C}_{\uparrow}\left(\partial^2 (\sigma_1-\sigma_2+\sigma_3)\right)=\ell^{C}_{\uparrow}(e_4-e_5+e_6)=0 \quad \mbox{while}\quad \ell_{\uparrow}^{C}(\partial^2 \sigma_1)=\ell_{\uparrow}^{C}(\partial^2 \sigma_2)=\ell_{\uparrow}^{C}(\partial^2 \sigma_3)=1.\]  We remedy this by taking $\mathcal{A}_{\uparrow}^{2}=\{\sigma_1-\sigma_2+\sigma_3,\sigma_2,\sigma_3\}$ and hence $\mathcal{B}_{\uparrow}^{1}=\partial^2(\mathcal{A}_{\uparrow}^{1})=\{e_4-e_5+e_6,e_1-e_3+e_5,e_2-e_3-e_6\}$.  One can then take $\mathcal{A}_{\uparrow}^{1}=\{e_1,e_4,e_5\}$ (that $\mathcal{A}_{\uparrow}^{1}\cup \mathcal{B}_{\uparrow}^{1}$ is $\ell_{\uparrow}^{C}$-orthogonal follows from Proposition \ref{orthcrit}(ii) again), $\mathcal{B}_{\uparrow}^{0}=\partial(\mathcal{A}^{\uparrow}_{0})=\{v_1-v_0,v_2-v_1,v_3-v_1\}$, and $\mathcal{H}_{\uparrow}^{0}=\{v_3\}$.  (The choice of $v_3$ here is necessary to make $\mathcal{B}_{\uparrow}^{0}\cup \mathcal{H}_{\uparrow}^{0}$ $\ell_{\uparrow}^{C}$-orthogonal.)

Since $\ell_{\uparrow}^{C}(v_3)=-1$, this confirms the earlier assertion that the summand associated to the $0$th homology class represented simultaneously by the unique elements of $\mathcal{H}_{0}^{\downarrow}$ and $\mathcal{H}_{\uparrow}^{0}$ is $(>_{-1}^{1})_{0}$.  As was the case for the $\mathcal{A}^{\downarrow}_{k+1}$ most elements of the $\mathcal{A}_{\uparrow}^{k+1}$ have the same filtration levels as their images under $\partial^{k+1}$ and hence do not contribute to the decomposition, but there is one exception: $f_1-f_2+f_3\in \mathcal{A}_{\uparrow}^{2}$ has $\ell_{\uparrow}^{C}(f_1-f_2+f_3)=1$ and $\ell_{\uparrow}^{C}(\partial^2(f_1+f_2+f_3))=0$, yielding a summand $(\uparrow_{0}^{1})_1$.

We conclude that, in this instance, $\mathcal{C}_{\partial}(X,f;\kappa)$  is isomorphic to $(\uparrow_{0}^{1})_1\oplus (>_{-1}^{1})_{0}$.  Based on Propositions \ref{barclass} and \ref{simpsing}, this corresponds to the statement that the level set barcode of $f\co\mathbb{X}\to \R$ consists of an interval $[0,1)$ in degree $1$ and an interval $[-1,1]$ in degree $0$.  And indeed, the level sets of $f$ at levels in $[0,1)$ are homeomorphic to $S^1$, those at levels in $[-1,1]\setminus[0,1)$ are contractible, and the others are empty.	
\end{ex}

\begin{ex}\label{simpbdryex}
	For an example of $\mathcal{C}_{\partial}(X,f;\kappa)$ with $\partial X\neq \varnothing$, consider the function depicted in Figure \ref{simpfig2}.  Here $\partial^+X$ and $\partial^-X$ each consist of one point; these points are not named in the figure because they do not contribute generators to any of the relevant chain complexes.  According to the definitions in Example \ref{simpdef}, the complex $(D_*,\partial_D)$ is the quotient complex $\frac{C_*(X;\kappa)}{C_*(\partial X;\kappa)}$, so is freely generated in degree $0$ by $v_0,v_1,v_2$, in degree $1$ by $e_1,\ldots,e_6$, and in degree $2$ by $\sigma$.  To fix signs, we orient $e_1,e_5$, and $e_6$ to point away from the boundary, so $\partial_D e_1=v_0,\partial_D e_5=v_1,\partial_D e_6=v_2$.  Evidently $H_*(D_*,\partial_D)$ is isomorphic to the homology of $S^1$ relative to two points, so it has rank two in degree $1$ and is trivial in other degrees; as cycles generating the homology we may take $e_1+e_2-e_5$ and $e_4+e_5-e_6$. 

\begin{center}
	\begin{figure}
		\includegraphics[width=3in]{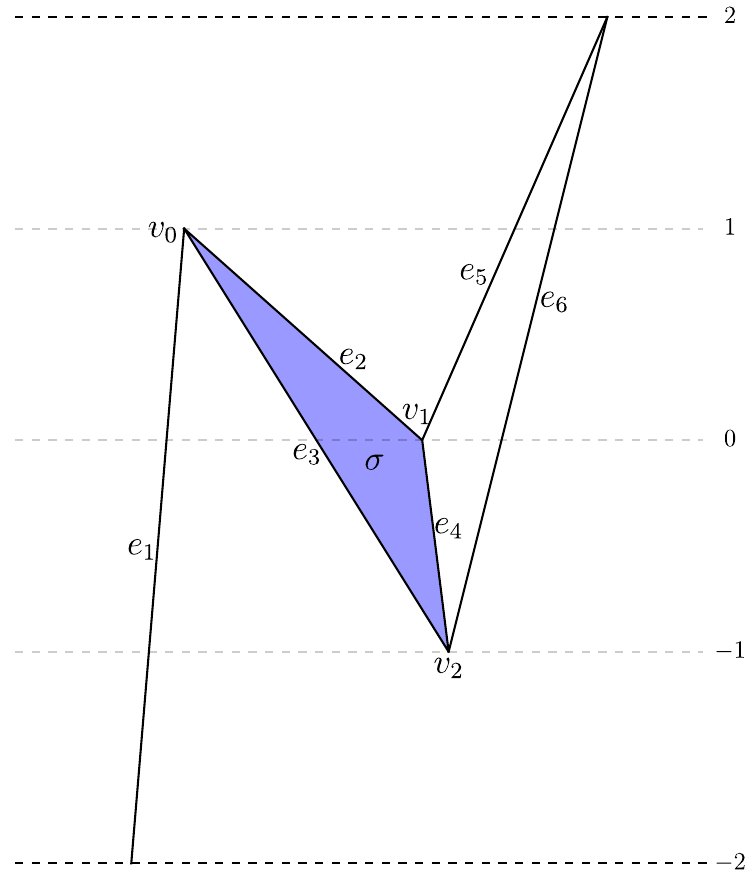}
		\caption{The function from Example \ref{simpbdryex}}
		\label{simpfig2}
	\end{figure}
\end{center}	
	
	The complexes $C_{\uparrow *}$ and $C^{\downarrow}_{*}$ are the subcomplexes of $D_*$ generated just by those simplices that do not intersect, respectively, $f^{-1}(\{2\})$ and $f^{-1}(\{-2\})$.  Thus they are identical to $D_*$ in degrees $0$ and $2$, while, $C_{\uparrow 1}=\mathrm{span}\{e_1,e_2,e_3,e_4\}$ and $C^{\downarrow}_{1}=\mathrm{span}\{e_2,e_3,e_4,e_5,e_6\}$.
	
	As both $C_{\uparrow 2}$ and $C^{\downarrow}_{2}$ are generated by $\sigma$, we take $\mathcal{X}_{\uparrow}^{2}=\mathcal{A}_{\uparrow}^{2}=\mathcal{X}^{\downarrow}_{2}=\mathcal{A}^{\downarrow}_{2}=\{\sigma\}$, and hence $\mathcal{B}^{1}_{\uparrow}=\mathcal{B}_{1}^{\downarrow}=\{e_2-e_3+e_4\}$.  
	
	The kernel of $\partial_D$ in degree one is spanned by $\{e_2-e_3+e_4,e_1+e_2-e_5, e_4+e_5-e_6\}$, which intersects $C_{\uparrow 1}$ only in $\mathrm{span}\{e_2-e_3+e_4\}=\mathrm{span}\mathcal{B}_{\uparrow}^{1}$.  So $\mathcal{H}_{\uparrow}^{1}$ is necessarily empty, and for $\mathcal{A}_{\uparrow}^{1}$ we need to arrange both that $\{e_2-e_3+e_4\}\cup \mathcal{A}_{\uparrow}^{1}$ is $\ell_{\uparrow}^{C}$-orthogonal and that $\partial_D(\mathcal{A}_{\uparrow}^{1})$ is  $\ell_{\uparrow}^{C}$-orthogonal.  The latter condition forces one element of $\mathcal{A}_{\uparrow}^{1}$ to have boundary in $\mathrm{span}\{v_2\}$, and another to have boundary in $\mathrm{span}\{v_1,v_2\}$.  A suitable choice is to take \[ \mathcal{A}_{\uparrow}^{1}=\{e_1+e_3,e_4,e_2\},\quad\mbox{and hence}\quad \mathcal{X}_{\uparrow}^{0}=\mathcal{B}_{\uparrow}^{0}=\{v_2,v_2-v_1,v_1-v_0\}.\]	As $C_{\uparrow *}$ has trivial homology, its only contributions to the decomposition of $\mathcal{C}_{\partial}(X,f;\kappa)$ come from elements $x$ of some $\mathcal{A}_{\uparrow}^{k+1}$ such that $\ell_{\uparrow}^{C}(\partial_Dx)<\ell_{\uparrow}^{C}(x)$.  There is one such $x$ above, namely $x=e_1+e_3\in \mathcal{A}_{\uparrow}^{1}$, giving rise to a summand $(\uparrow_{-1}^{1})_0$.
	
	Turning to $C^{\downarrow}_{*}$, its degree-$1$ cycles are spanned by the boundary $e_2-e_3+e_4$ and the homologically nontrivial element $e_4+e_5-e_6$. These are $(-\ell^{\downarrow}_C)$-orthogonal since any nontrivial linear combination of them has $\ell^{\downarrow}_{C}=-1$, so we take $\mathcal{H}^{\downarrow}_{1}=\{e_4+e_5-e_6\}$. For $\mathcal{A}^{\downarrow}_{1}$, we must arrange both that $\{e_2-e_3+e_4,e_4+e_5-e_6\}\cup\mathcal{A}^{\downarrow}_{1}$ is a $(-\ell^{\downarrow}_C)$-orthogonal basis for $\mathrm{span}\{e_2,e_3,e_4,e_5,e_6\}$ and that $\mathcal{B}^{\downarrow}_{0}:=\partial_D(\mathcal{A}^{\downarrow}_{1})$ is a
	 $(-\ell^{\downarrow}_C)$-orthogonal basis for $\mathrm{span}\{v_0,v_1,v_2\}$; this holds if we take \[ \mathcal{A}^{\downarrow}_{1}=\{-e_2+e_5,e_5,e_6\},\quad\mbox{and hence}\quad \mathcal{X}^{\downarrow}_{0}=\mathcal{B}^{\downarrow}_{0}=\{v_0,v_1,v_2\}.\]

The only element of any $\mathcal{A}^{\downarrow}_{k+1}$ having different $\ell_{C}^{\downarrow}$-filtration level than its boundary is $-e_2+e_5\in\mathcal{A}^{\downarrow}_{1}$; this yields a summand $(\downarrow^{1}_{0})_0$ in the decomposition of $\mathcal{C}_{\partial}(X,f;\kappa)$.  The homologically nontrivial cycle $e_4+e_5-e_6\in \mathcal{H}_{1}^{\downarrow}$, with $\ell^{\downarrow}_{C}=-1$, yields a summand $(\searrow^{-1})_1$ since its image in $H_1(D_*)$ is nontrivial.  Lastly, the one-dimensional complement to the image of $H_1(C_{\uparrow*})\oplus H_1(C^{\downarrow}_{*})\to H_1(D_*)$ gives a summand $\square_1$.

We conclude that $\mathcal{C}_{\partial}(X,f;\kappa)\cong 
(\uparrow_{-1}^{1})_0\oplus (\downarrow^{1}_{0})_0\oplus (\searrow^{-1})_1\oplus \square_1$, corresponding via Proposition \ref{barclass} to a level set barcode consisting of intervals $[-1,1),(0,1],(-1,2),$ and $(-2,2)$, all in degree zero, consistently with Figure \ref{simpfig2}.
\end{ex}

\subsection{Morse functions}

In this section we consider various instances of the Morse filtered cospan $\mathcal{M}(\mathbb{X},f,v;\kappa)$ from Example \ref{morseex}. In all cases $\kappa$ will be a field, and $f\co \mathbb{X}\to [-2,2]$ will be a Morse function on a compact manifold $\mathbb{X}$ with boundary, such that $-2$ and $2$ are regular values and $\partial \mathbb{X}=f^{-1}(\{-2,2\})$.
Then the ingredients in the Morse filtered cospan include the filtered Morse complexes $C_{\uparrow *}=CM_*(f,v;\kappa)$ and $C^{\downarrow}_{*}=CM_*(-f,-v;\kappa)$ associated to a generic gradient-like vector field $v$, as well as the relative singular chain complex 
$D_*=S_*(\mathbb{X},\partial\mathbb{X};\kappa)$.  

We will sometimes denote $I=[-2,2]$, and $\partial I=\{-2,2\}$.  As in Example \ref{morseex}, we also write $\partial^{\pm}\mathbb{X}=f^{-1}(\{\pm 2\})$.

\begin{ex}\label{trivmorse}
The simplest kind of Morse function is one which has no critical points.  In this case the flow of $v$ allows one to construct a diffeomorphism $\mathbb{X}\cong  M\times I$ where $M=\partial^-X$, identifying $f^{-1}(\{t\})$ with $M\times \{t\}$ for each $t\in I$.  The Morse complexes $C_{\uparrow *}$ and $C^{\downarrow}_{*}$ are both the zero complex, and the relative singular chain complex is homotopy equivalent to the relative homology $H_*(M\times I,M\times \partial I;\kappa)$, with zero differential.
 So $\mathcal{M}(X,f,v;\kappa)$ is isomorphic in $\mathsf{HCO}^{\Lambda}(\kappa)$ to \[ \xymatrix{ 0\ar[rd] & \\ & H_*(M\times I,M\times \partial I;\kappa), \\ 0 \ar[ru] &  } \] \emph{i.e.} to the direct sum $\oplus_{k}\left(\square_k\right)^{\oplus \dim  H_k(M\times I,M\times \partial I;\kappa)}$.  This corresponds by Theorem \ref{morseiso} and Proposition \ref{barclass} to a level set barcode having, for each $k$, a number of copies of $(-2,2)$ in degree $k-1$ equal to   $H_k(M\times I,M\times \partial I;\kappa)$, consistently with the suspension isomorphism $H_{k-1}(M;\kappa)\cong H_k(M\times I,M\times \partial I;\kappa)$. 
\end{ex}

\begin{ex}\label{cubic}
For another simple example, consider the function $f\co \mathbb{X}\to[-2,2]$  depicted in Figure \ref{cubicfig}.

\begin{center}
	\begin{figure}
		\includegraphics[width=2in]{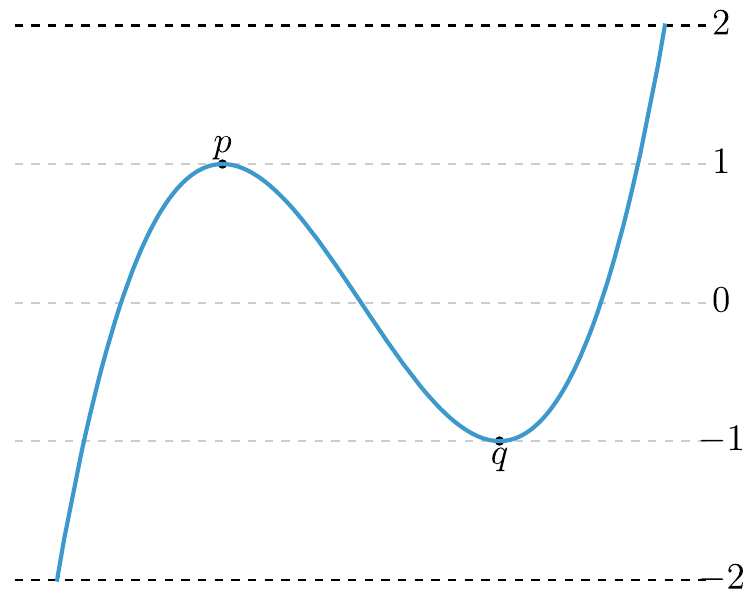}
		\caption{The function from Example \ref{cubic}}
		\label{cubicfig}
	\end{figure}
\end{center}	

Here $C_{\uparrow 1}=C^{\downarrow}_{0}=\mathrm{span}\{p\}$ while $C_{\uparrow}^{0}=C^{\downarrow}_{1}=\mathrm{span}\{q\}$, with the differentials given by $\partial_{\downarrow}^{C}p=\pm q$ and $\partial^{\downarrow}_{C}q=\pm p$ (for some signs $\pm$ dependent on auxiliary choices which do not affect the conclusion).  As $\ell_{\uparrow}(\pm p)=\ell^{\downarrow}(\pm p)=f(p)=1$ and $\ell_{\uparrow}(\pm q)=\ell^{\downarrow}(\pm q)=f(q)=-1$ it follows that the acyclic complexes $C_{\uparrow *}$ and $C^{\downarrow}_{*}$ are filtered homotopy equivalent to 
$\mathcal{E}_0(-1,1)_{\uparrow}$ and $\mathcal{E}_0(1,-1)^{\downarrow}$, respectively.  Meanwhile, as $\mathbb{X}$ is homeomorphic to an interval, $D_*=S_*(\mathbb{X},\partial\mathbb{X};\kappa)$ is homotopy equivalent to $\kappa_1$ (the homology of an interval rel its boundary).  We thus obtain in this case a decomposition \[ \mathcal{M}(\mathbb{X},f,v;\kappa)\cong (\uparrow_{-1}^{1})_0\oplus (\downarrow_{-1}^{1})_0\oplus\square_1.\] \end{ex}

\subsubsection{Elementary cobordisms}\label{elemsect}  We conclude our discussion of examples with a general consideration of the case that the Morse function $f\co \mathbb{X}\to [-2,2]$ has just a single critical point $p$, with critical value zero. We consistently denote \[ n=\dim \mathbb{X},\quad k=\mathrm{ind}_f(p).\]  So  \[ C_{\uparrow *}=\mathcal{E}_k(0,\infty)_{\uparrow},\quad C^{\downarrow}_{ *}=\mathcal{E}_{n-k}(0,-\infty)^{\downarrow}.\]  

For a suitable choice of gradient-like vector field $v$, a neighborhood of the critical point $p$ can be identified with the set 
\[ \mathbb{D}=\left\{(\vec{x},\vec{y})\in \R^k\times \R^{n-k}|-2\leq -\|\vec{x}\|^2+\|\vec{y}\|^2\leq 2,\, \|\vec{x}\|\vec{y}\|\leq 2\right\},\] see Figure \ref{dfig}, with the restriction of $f$ to this neighborhood identified with $(\vec{x},\vec{y})\mapsto  -\|\vec{x}\|^2+\|\vec{y}\|^2$ and with the time-$t$ flow of $v$ given by $(\vec{x},\vec{y})\mapsto (e^{-t}\vec{x},e^t\vec{y})$.
\begin{center}
	\begin{figure}
		\includegraphics[width=3in]{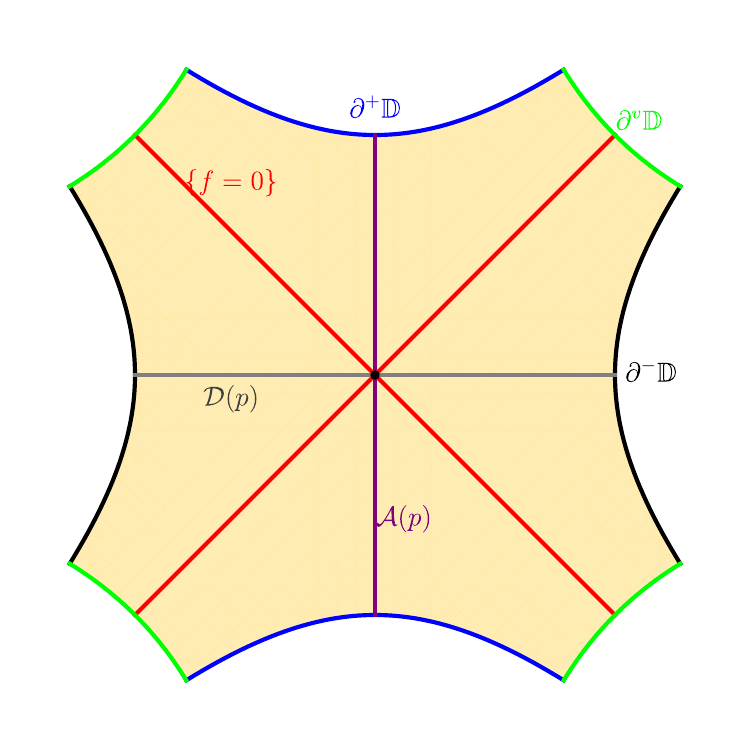}
		\caption{The standard neighborhood $\mathbb{D}$ in the case $k=1,n=2$.}
		\label{dfig}
	\end{figure}
\end{center}	

This set $\mathbb{D}$ is a smooth submanifold with corners of $\mathbb{R}^n$, with the closures of its codimension-one boundary strata comprising \[ \partial^+\mathbb{D}=\{(\vec{x},\vec{y})\in\mathbb{D}|\,-\|\vec{x}\|^2+\|\vec{y}\|^2=2\},\quad  \partial^-\mathbb{D}=\{(\vec{x},\vec{y})\in\mathbb{D}|\,-\|\vec{x}\|^2+\|\vec{y}\|^2=-2\},\] and \[ \partial^v\mathbb{D}=\left\{\left.(\vec{x},\vec{y})\in\mathbb{D}\right|\,\|\vec{x}\|\|\vec{y}\|=2\right\}.\]

One has $\mathbb{D}\cap\partial^{\pm}\mathbb{X}=\partial^{\pm}\mathbb{D}$ with $\partial^-\mathbb{D}$ and $\partial^+\mathbb{D}$ diffeomorphic to $S^{k-1}\times D^{n-k}$ and to $D^{k}\times S^{n-k-1}$, respectively, and there is a diffeomorphism \begin{align}\nonumber \psi\co \partial^v\mathbb{D} &\to S^{k-1}\times S^{n-k-1}\times I \\ (\vec{x},\vec{y}) &\mapsto \left(\frac{\vec{x}}{\|\vec{x}\|},\frac{\vec{y}}{\|\vec{y}\|},-\|\vec{x}\|^2+\|\vec{y}\|^2   \right). \label{dvdiffeo}\end{align}

As $f$ has no critical points outside of $\mathbb{D}$ and the flow of $v$ is tangent to $\partial^v\mathbb{D}$, one may use the flow to construct a diffeomorphism $\overline{\mathbb{X}\setminus \mathbb{D}}\cong Y\times I$, where $Y=\overline{\partial^-\mathbb{X}\setminus\partial^-\mathbb{D}}$ is a smooth $(n-1)$-manifold with boundary $\partial Y=\partial^v\mathbb{D}\cap \partial^-\mathbb{D}\cong S^{k-1}\times S^{n-k-1}$.  

In the other direction, we can recover $\mathbb{X}$, up to diffeomorphism respecting the Morse function $f$, by starting with a smooth $(n-1)$-manifold with boundary $Y$ and a diffeomorphism $\phi\co S^{k-1}\times S^{n-k-1}\to \partial Y$, and then forming $\mathbb{X}$ by gluing $\mathbb{D}$ to $Y\times I$ along $\partial^v\mathbb{D}$ and $(\partial Y)\times I$: \begin{equation}\label{xfromy} \mathbb{X}=\frac{\mathbb{D}\coprod (Y\times I)}{\psi^{-1}(\phi^{-1}(y),t)\sim (y,t)\mbox{ for }y\in\partial Y,\,t\in I} \end{equation}   	where $\psi$ is the map from (\ref{dvdiffeo}).  The function $f$ is simply given by projection to the $I$ factor on $(\partial Y)\times I$, and by $(\vec{x},\vec{y})\mapsto -\|\vec{x}\|^2+\|\vec{y}\|^2$ on $\mathbb{D}$.

Accordingly, for the rest of this discussion we will regard the input data for an $n$-dimensional elementary cobordism of index $k$ to consist of a compact $(n-1)$-manifold with boundary $Y$ together with a diffeomorphism $\phi\co S^{k-1}\times S^{n-k-1}\to \partial Y$.  We will assume that every connected component of $Y$ has nonempty boundary; this corresponds to assuming that the manifold $\mathbb{X}$ resulting from (\ref{xfromy}) is connected. (The general case can be obtained from this by just adding trivial cobordisms as in Example \ref{trivmorse}).

Choosing the gradient-like vector field appropriately, the descending manifold of the unique critical point $p=(\vec{0},\vec{0})\in\mathbb{D}$ is the $k$-disk $\mathcal{D}(p)=\mathbb{D}\cap(\R^k\times\{\vec{0}\})$
and the ascending manifold is the $(n-k)$-disk $\mathcal{A}(p)=\mathbb{D}\cap (\{\vec{0}\}\times\R^{n-k})$.  As in, \emph{e.g.}, \cite[Theorem 3.14]{Milhcob}, one has deformation retractions of $\mathbb{X}$ to either $\partial^-\mathbb{X}\cup \mathcal{D}(p)$ or to $\partial^+\mathbb{X}\cup \mathcal{A}(p)$, which restrict to  deformation retractions of $\mathbb{D}$ to $\partial^-\mathbb{D}\cup \mathcal{D}(p)$ or to $\partial^+\mathbb{D}\cup\mathcal{A}(p)$.

In this context, the map $\psi_{\uparrow}\co CM_*(f,v;\kappa)\to S_*(\mathbb{X},\partial\mathbb{X};\kappa)$ sends the unique generator $p$ of $CM_*(f,v;\kappa)$ to a relative cycle whose homology class in $H_*(\mathbb{X},\partial\mathbb{X};\kappa)$ is the same\footnote{modulo signs which depend on arbitrary choices and do not affect the conclusions} as that of the descending disk $\mathcal{D}(p)$; likewise $\psi^{\downarrow}\co CM_*(-f,-v;\kappa)\to S_*(\mathbb{X},\partial\mathbb{X};\kappa)$ sends the generator to a representative of the class of the ascending disk.  We denote these classes in $H_*(\mathbb{X},\partial\mathbb{X};\kappa)$ by $[\mathcal{D}(p)]$ and $[\mathcal{A}(p)]$, respectively.  The algebraic classification of $\mathcal{M}(\mathbb{X},f,v;\kappa)$ depends upon whether these classes are nontrivial and, if so, whether they are linearly dependent.  (The classes of $\mathcal{D}(p)$ and $\mathcal{A}(p)$ are always nontrivial in $H_k(\mathbb{X},\partial^-\mathbb{X};\kappa)$ and $H_{n-k}(\mathbb{X},\partial^+\mathbb{X};\kappa)$, respectively, but for homology relative to the full boundary $\partial\mathbb{X}$ the situation is more variable.)

\begin{ex}\label{k0n} A somewhat degenerate case occurs when either $k=0$ or $k=n$.  Since we are assuming both that $\partial Y\cong S^{k-1}\times S^{n-k-1}$ and that all components of $Y$ have nonempty boundary, and since $S^{-1}=\varnothing$, it follows that $Y$ must be the empty manifold.  Thus $\mathbb{X}=\mathbb{D}$, which in this case is the disk in $\mathbb{R}^n$ of radius $\sqrt{2}$, with $f$ being given by either $\vec{x}\mapsto \|\vec{x}\|^2$ or $\vec{x}\mapsto -\|\vec{x}\|^2$ depending on whether $k=0$ or $n$.
	
	If $k=0$, then $\partial^-\mathbb{X}=\varnothing$ and $\mathbb{X}$ coincides with the ascending manifold $\mathcal{A}(p)$.  Then $H_*(\mathbb{X},\partial\mathbb{X};\kappa)=H_*(\mathbb{X},\partial^+\mathbb{X};\kappa)$ is trivial in degrees other than $n$ and is generated by $[\mathcal{A}(p)]$ in degree $n$.  The descending manifold in this case is a single point, which\footnote{We are assuming $n\neq 0$ here; the reader can make the appropriate modifications in the very degenerate case that $n=0$.} vanishes in $H_*(\mathbb{X},\partial\mathbb{X};\kappa)$.  So since the generators for both $CM_{*}(f,v;\kappa)$ and $CM_*(-f,-v;\kappa)$ have filtration level $0$, it follows that $\mathcal{M}(\mathbb{X},f,v;\kappa)$ is isomorphic to $(\uparrow_{0}^{\infty})_0\oplus (\searrow^0)_n$ when $k=0$.
	
	If $k=n$ the situation is reversed: we have $\partial^+\mathbb{X}=\varnothing$ and $\mathbb{X}$ coincides with $\mathcal{D}(p)$, and the map $\psi_{\uparrow}\co CM_*(f,v;\kappa)\to S_*(\mathbb{X},\partial\mathbb{X};\kappa)$ induces an isomorphism on homology while $\psi^{\downarrow}$ induces the zero map.  So in this case $\mathcal{M}(\mathbb{X},f,v;\kappa)\cong (\nearrow_0)_n\oplus (\downarrow^{0}_{-\infty})_0$.
\end{ex}

Having dispensed with Example \ref{k0n}, assume from now on that $0<k<n$.  Given the diffeomorphism $\phi\co S^{k-1}\times S^{n-k-1}\to Y$, choose basepoints $\ast$ in $S^{k-1}$ and $S^{n-k-1}$, and introduce the spheres \begin{equation} S_-=\phi(S^{k-1}\times\{\ast\}),\qquad S_+=\phi(\{\ast\}\times S^{n-k-1}) \label{spmdef}\end{equation} in $\partial Y$.  
We can think of $S_-$ as a pushoff to $\partial Y\cong \partial^v\mathbb{D}\cap \partial^-\mathbb{D}$ of the boundary $\partial\mathcal{D}(p)\subset \partial^-\mathbb{D}$ of the descending disk $\mathcal{D}(p)$, and likewise of $S_+$ as a pushoff of $\partial\mathcal{A}(p)$ to $\partial^v\mathbb{D}\cap \partial^+\mathbb{X}$.  

Recalling the diffeomorphism $\psi\co \partial^v\mathbb{D}\to S^{k-1}\times S^{n-k-1}\times I$, we have cylinders $C_-=\psi^{-1}(S^{k-1}\times\{\ast\}\times I)$ and $C_+=\psi^{-1}(\{\ast\}\times S^{n-k-1}\times I)$ in $\partial^v\mathbb{D}$ having boundary on $\partial^v\mathbb{D}\cap(\partial^+\mathbb{D}\cup\partial^-\mathbb{D})$.  By capping off one boundary component of $C_-$ by a disk in $\partial^+\mathbb{D}$ one sees that, up to sign, $C_-$ represents the same class in $H_{k}(\mathbb{D},\partial^-\mathbb{D}\cup \partial^+\mathbb{D};\kappa)$ as does the descending manifold $\mathcal{D}(p)$.  Likewise $C_+$ represents the same class, modulo sign, as $\mathcal{A}(p)$ in $H_{n-k}(\mathbb{D},\partial^-\mathbb{D}\cup\partial^+\mathbb{D};\kappa)$.  Note that under the gluing (\ref{xfromy}) that forms $\mathbb{X}$, the cylinders $C_{\pm}$ correspond to $S_{\pm}\times I\subset (\partial Y)\times I$.

\begin{prop}\label{dacrit} Assume either that $\mathrm{char}\,\kappa=2$ or that $\mathbb{X}$ is orientable,\footnote{To give a criterion in terms of the input data $Y$ and $\phi$, $\mathbb{X}$ is orientable iff $Y$ admits an orientation whose induced orientation on the boundary makes the diffeomorphism $\phi\co S^{k-1}\times S^{n-k-1}\to \partial Y$ orientation-reversing, where $S^{k-1}\times S^{n-k-1}$ carries the standard product orientation induced by the orientations of $S^{k-1}$ and $S^{n-k-1}$ as the boundaries of $D^k$ and $D^{n-k}$.  In case $k=1$ or $k=n-1$, note that the orientation of $S^0$ as the boundary of $D^1$ gives opposite signs to the two points of $S^0$. If $1<k<n-1$, so that $\partial Y$ is connected, it is equivalent to just require that $Y$ be orientable.} and let $\iota\co \partial Y\to Y$ be the inclusion. Then: \begin{itemize} \item[(i)] $[\mathcal{D}(p)]=0\in H_k(\mathbb{X},\partial\mathbb{X};\kappa)$ iff $\iota_*[S_-]=0\in H_{k-1}(Y;\kappa)$.
		\item[(ii)] $[\mathcal{A}(p)]=0\in H_{n-k}(\mathbb{X},\partial\mathbb{X};\kappa)$ iff $\iota_*[S_+]=0\in H_{n-k-1}(Y;\kappa)$.
		\item[(iii)] $\mathrm{span}\{[\mathcal{D}(p)],[\mathcal{A}(p)]\}$ is always a one-dimensional subspace of $H_*(\mathbb{X},\partial\mathbb{X};\kappa)$. \end{itemize}
\end{prop}

\begin{remark}
	If $n\neq 2k$ then Proposition \ref{dacrit}(iii) clearly implies that exactly one of $[\mathcal{D}(p)]$ and $[\mathcal{A}(p)]$ is nonzero, since these classes are in different gradings.  If $n=2k$, however, it can happen that $[\mathcal{D}(p)]$ and $[\mathcal{A}(p)]$ are both nonzero, and are linearly dependent.
\end{remark}

\begin{proof}[Proof of Proposition \ref{dacrit}]
	Consider the diagram \begin{equation}\label{mvdy} 
	\xymatrix{ H_*((\partial Y)\times I,(\partial Y)\times \partial I;\kappa)\ar[r]^<<<<<{(j^{D}_*,j^{Y}_*)} & H_*(\mathbb{D},\partial^-\mathbb{D}\cup\partial^+\mathbb{D};\kappa)\oplus H_*(Y\times I,Y\times\partial I;\kappa)\ar[r]& H_*(\mathbb{X},\partial\mathbb{X};\kappa) \\ H_{*-1}(\partial Y;\kappa)\ar[u]^{\sigma}_{\cong}
	} 	\end{equation} where the top row is part of the Mayer-Vietoris sequence (with $j^D,j^Y$ the inclusions of $\partial Y\times I\cong \partial^v\mathbb{D}$ into $\mathbb{D}$ and $Y\times I$, respectively) and $\sigma$ is the suspension isomorphism.  As in the proposition, we write $[\mathcal{D}(p)]$ and $[\mathcal{A}(p)]$ for the relative homology classes of the descending and ascending manifolds in $H_*(\mathbb{X},\partial\mathbb{X};\kappa)$; let us introduce the notation $[\mathcal{D}(p)]_{\mathbb{D}}$ and $[\mathcal{A}(p)]_{\mathbb{D}}$ for the class of the descending and ascending manifolds within $H_*(\mathbb{D},\partial^-\mathbb{D}\cup\partial^+\mathbb{D};\kappa)$.   The discussion before the proposition shows that \begin{equation}\label{jsig} j^{D}_{*}\sigma[S_-]=\ep_1[\mathcal{D}(p)]_{\mathbb{D}}\quad\mbox{and}\quad j^{D}_{*}\sigma[S_+]=\ep_2[\mathcal{A}(p)]_{\mathbb{D}}\end{equation} for appropriate choices of $\ep_1,\ep_2\in \{-1,1\}$.

	It will be helpful to establish the following:
	\begin{claim}\label{jdjy}
		$\ker j^{D}_{*}\subset \ker j^{Y}_{*}$.
	\end{claim}
	\begin{proof}[Proof of Claim \ref{jdjy}]
		We have a Mayer-Vietoris sequence \[  
		H_{*}(\mathbb{D},\partial^-\mathbb{D};\kappa)\oplus H_{*}(\mathbb{D},\partial^-\mathbb{D};\kappa)\to H_*(\mathbb{D},\partial^-\mathbb{D}\cup\partial^+\mathbb{D};\kappa)\to H_{*-1}(\mathbb{D};\kappa)\to H_{*-1}(\mathbb{D},\partial^-\mathbb{D};\kappa)\oplus H_{*-1}(\mathbb{D},\partial^-\mathbb{D};\kappa), \] where
		$H_{*}(\mathbb{D},\partial^-\mathbb{D};\kappa)$ and $H_{*}(\mathbb{D},\partial^-\mathbb{D};\kappa)$ are one-dimensional and generated by the classes of the descending and ascending manifolds, respectively.  It follows that $H_*(\mathbb{D},\partial^-\mathbb{D}\cup\partial^+\mathbb{D};\kappa)$ (summed over all gradings) is three-dimensional, with a basis consisting of $[\mathcal{D}(p)]_{\mathcal{D}}$ in degree $k$, $[\mathcal{A}(p)]_{\mathcal{D}}$ in degree $n-k$, and the class of any path in $\mathbb{D}$ connecting $\partial^-\mathbb{D}$ to $\partial^+\mathbb{D}$ in degree $1$.   These generators are, up to sign, the images under $j^{D}_{*}\circ \sigma$ of the classes $[S_-],[S_+],$ and $[pt]$ in $H_*(\partial Y;\kappa)$.  In particular, $j^{D}_{*}\circ\sigma\co H_{*-1}(\partial Y;\kappa)\to H_*(\mathbb{D},\partial^-\mathbb{D}\cup\partial^+\mathbb{D};\kappa)$ restricts injectively to the span of these three classes.  Now, due to our assumption that either $\mathrm{char}\,\kappa=2$ or $\mathbb{X}$ is orientable, we have a fundamental class $[\mathbb{X},\partial \mathbb{X}]\in H_{n}(\mathbb{X},\partial\mathbb{X};\kappa)$.  As $H_n(\mathbb{D},\partial^-\mathbb{D}\cup\partial^+\mathbb{D};\kappa)=H_n(Y\times I,Y\times\partial I;\kappa)=\{0\}$, the connecting homomorphism $\delta\co H_{n}(\mathbb{X},\partial\mathbb{X};\kappa)\to  H_{n-1}((\partial Y)\times I,(\partial Y)\times\partial I;\kappa)$ maps $[\mathbb{X},\partial \mathbb{X}]$ to a nonzero element of   $H_{n-1}((\partial Y)\times I,(\partial Y)\times\partial I;\kappa)$, which is then annihilated by $(j^{*}_{D},j^{*}_{Y})$.  (In fact, $\delta[\mathbb{X},\partial \mathbb{X}]=\pm\sigma[\partial Y]$.)  
		
		So $H_*((\partial Y)\times I,(\partial Y)\times \partial I;\kappa)$, which is four-dimensional since it is isomorphic to $H_{*-1}(S^{k-1}\times S^{n-k-1};\kappa)$,  has  a three-dimensional subspace on which $j_{*}^{D}$ is injective and a one-dimensional subspace on which both $j_{*}^{D}$ and $j_{*}^{Y}$ vanish.  It follows that $\ker j_{*}^{D}\subset \ker j_{*}^{Y}$.     
	\end{proof}

Returning now to the main body of the proof, it follows from the exactness of the top row of (\ref{mvdy}) that a general linear combination $x[\mathcal{D}(p)]_{\mathbb{D}}+y[\mathcal{A}(p)]_{\mathbb{D}}\in H_*(\mathbb{D},\partial^-\mathbb{D}\cup \partial^+\mathbb{D};\kappa)$ vanishes under the inclusion-induced map to $H_*(\mathbb{X},\partial\mathbb{X};\kappa)$ if and only if $\left(x[\mathcal{D}(p)]_{\mathbb{D}}+y[\mathcal{A}(p)]_{\mathbb{D}},0\right)$ lies in $\mathrm{Im}(j^D_*,j^Y_*)$. 
Now Claim \ref{jdjy} implies than an element of $\Img(j^D_*,j^Y_*)$ is uniquely determined by its component in $H_*(\mathbb{D},\partial^-\mathbb{D}\cup \partial^+\mathbb{D};\kappa)$.   So in view of (\ref{jsig}), the unique element of $\Img(j^D_*,j^Y_*)$ having $H_*(\mathbb{D},\partial^-\mathbb{D}\cup \partial^+\mathbb{D};\kappa)$-component equal to $x[\mathcal{D}(p)]_{\mathbb{D}}+y[\mathcal{A}(p)]_{\mathbb{D}}$ is $\left(x[\mathcal{D}(p)]_{\mathbb{D}}+y[\mathcal{A}(p)]_{\mathbb{D}},j^{Y}_{*}\sigma(\ep_1x[S_-]+\ep_2y[S_+])\right)$ for certain universal signs $\ep_1,\ep_2$.   So $x[\mathcal{D}(p)]_{\mathbb{D}}+y[\mathcal{A}(p)]_{\mathbb{D}}$ vanishes on inclusion to $(\mathbb{X},\partial\mathbb{X})$ iff $j^{Y}_{*}\sigma(\ep_1x[S_-]+\ep_2y[S_+])=0$.  But $j^{Y}_{*}\sigma(\ep_1x[S_-]+\ep_2y[S_+])$ coincides with the image under the suspension isomorphism $H_{*-1}(Y;\kappa)\cong  H_*(Y\times I,Y\times\partial I;\kappa)$ of $\iota_*(\ep_1x[S_-]+\ep_2y[S_+])\in H_{*-1}(Y;\kappa)$.

Item (i) of the proposition follows directly from specializing the above discussion to the case $x=1$ and $y=0$, and Item (ii) follows from the case that $x=0$ and $y=1$.  The above discussion also shows that \[\dim\left(\left\{(x,y)\in \kappa^2|x[\mathcal{D}(p)]+y[\mathcal{A}(p)]=0\right\}\right)=\dim\left(\ker \iota_*|_{\mathrm{span}\{[S_-,S_+]\}}\right).\]  So Item (iii) of the proposition is equivalent to the statement that $\iota_*[S_-]$ and $\iota_*[S_+]$ span a one-dimensional subspace of $H_*(Y;\kappa)$.  We conclude the proof by establishing this statement.

If $k=1$ and $n=2$ this claim holds by a brief case analysis which we defer to Example \ref{k1n2}.  If $k=1$ and $n>2$, so that $\partial Y\cong S^0\times S^{n-2}$ has two components, then $S_-$ consists of one point on each of these two components and $S_+$ is equal to one component of $\partial Y$.  Then $\iota_*[S_-]=0$ if and only if $Y$ is connected.  In that case, the kernel of $\iota_{*}\co H_{n-2}(\partial Y;\kappa)\to H_{n-2}(Y;\kappa)$ is spanned by the fundamental class of $\partial Y$ (as this is the image under the boundary map of the fundamental class of $(Y,\partial Y)$, which generates $H_{n-1}(Y;\kappa)$), any representative of which is supported in both components of $\partial Y$.  So if $\iota_*[S_-]=0$ then $\iota_*[S_+]\neq 0$.  On the other hand, if $\iota_*[S_-]\neq 0$, so that $Y$ is disconnected, then $S_+$ is the boundary of a connected component of $Y$ and so $\iota_*[S_+]=0$  So if $k=1$ and $n>2$ then exactly one of $\iota_*[S_-]$ and $\iota_*[S_+]$ is nonzero.  The same conclusion holds if $k=n-1$ and $n>2$, which corresponds to reversing the roles of $S_-$ and $S_+$ in the above argument.

So assume for the rest of the proof that $1<k<n-1$, which has the effect of ensuring that both factors of $S^{k-1}\times S^{n-k-1}\cong \partial Y$ are postive-dimensional.  Then, from the connection provided by Poincar\'e-Lefschetz duality between the homology and cohomology sequences of the pair $(Y,\partial Y)$ (see, \emph{e.g}, \cite[Theorem VI.9.2]{Bre}), it follows that the rank of $\iota^*\co H^{k-1}(Y;\kappa)\to H^{k-1}(\partial Y;\kappa)$ (equivalently, the rank of $\iota_*\co H_{k-1}(\partial Y;\kappa)\to H_{k-1}(Y;\kappa)$) coincides with the dimension of the kernel of $\iota_*\co H_{n-k-1}(\partial Y;\kappa)\to H_{n-k-1}(Y;\kappa)$.  If $n\neq 2k$, so that $H_{k-1}(\partial Y;\kappa)$ and $H_{n-k-1}(\partial Y;\kappa)$ are one-dimensional, generated by $[S_-]$ and $[S_+]$, respectively, this implies that $\iota_*[S_-]=0$ if and only if $\iota_*[S_+]\neq 0$, as in the above $k=1$ case.  If instead $n=2k$, so that $H_{k-1}(\partial Y;\kappa)=H_{n-k-1}(\partial Y;\kappa)=\mathrm{span}\{[S_-],[S_+]\}$, it follows that the image and kernel of $i_*\co H_{k-1}(\partial Y;\kappa)\to H_{k-1}(Y;\kappa)$ both have dimension one.  So in any event the conclusion of (iii) in the statement of the proposition holds.

\end{proof}

\begin{cor}\label{morsesummary}
Let $f\co\mathbb{X}\to [-2,2]$ be a Morse function on a smooth connected compact $n$-dimensional manifold $\mathbb{X}$, having just one critical point $p$, such that $f(p)=0$ and $p$ has index $k\in\{1,\ldots,n-1\}$. If $\mathrm{char}\,\kappa\neq 2$, assume that $\mathbb{X}$ is orientable.  Let $S_-,S_+$ be the spheres from (\ref{spmdef}). Then, for a generic gradient-like vector field $v$, the Morse filtered cospan $\mathcal{M}(\mathbb{X},f,v;\kappa)$ is isomorphic to the direct sum of an appropriate\footnote{Specifically, the number of copies of $\square_j$ is $\dim H_j(\mathbb{X},\partial\mathbb{X};\kappa)-1$ for $j=n-k$ in case (i) or $j=k$ in cases (ii) or (iii), and $\dim H_j(\mathbb{X},\partial\mathbb{X};\kappa)$ for all other $j$.  To relate this to the $(n-1)$-manifold with boundary $Y$, the analysis of the sequence (\ref{mvdy}) in the proof of Proposition \ref{dacrit}, along with the suspension isomorphism $H_j(Y\times I,Y\times \partial I;\kappa)\cong H_{j-1}(Y;\kappa)$, imply that $H_j(\mathbb{X},\partial\mathbb{X};\kappa)\cong H_{j-1}(Y;\kappa)$ if $j<n$, while $H_j(\mathbb{X},\partial\mathbb{X};\kappa)\cong\kappa$.} number of copies of the various $\square_j$ for $1\leq j\leq n$ together with either:
\begin{itemize}
	\item[(i)] $(\uparrow_{0}^{\infty})_k\oplus (\searrow^0)_{n-k}$, in case $\iota_*[S_-]=0$ and $\iota_*[S_+]\neq 0$;
	\item[(ii)] $(\nearrow_0)_k\oplus (\downarrow_{-\infty}^{0})_{n-k}$, in case $\iota_*[S_-]\neq 0$ and $\iota_*[S_+]=0$; or 
	\item[(iii)] $(>_0^0)_k$ in case $\iota_*[S_-]$ and $\iota_*[S_+]$ are both nonzero, which can occur only if $n=2k$.
\end{itemize}
\end{cor}

\begin{proof}
By composing the maps $\psi_{\uparrow},\psi^{\downarrow}$ from Example \ref{morseex} with a chain homotopy equivalence between $S_*(\mathbb{X},\partial\mathbb{X};\kappa)$ and its homology (with the zero differential), we obtain an isomorphism of $\mathcal{M}(\mathbb{X},f,v;\kappa)$ with the filtered cospan \[ \xymatrix{ \mathcal{E}_{n-k}(0,-\infty)^{\downarrow} \ar[rd]^{\bar{\psi}^{\downarrow}} & \\ & H_*(\mathbb{X},\partial\mathbb{X};\kappa) \\ \mathcal{E}_k(0,\infty)_{\uparrow} \ar[ur]_{\bar{\psi}_{\uparrow}} &  }     \] where all complexes involved have zero differential and where the maps $\bar{\psi}^{\downarrow}$ and $\bar{\psi}_{\uparrow}$ have images spanned by the classes $[\mathcal{A}(p)]$ and $[\mathcal{D}(p)]$, respectively. By Proposition \ref{dacrit}, $V:=\mathrm{span}\{[\mathcal{A}(p)],[\mathcal{D}(p)]\}$ is one-dimensional, so, by taking a graded vector space complement $W$ to $V$ in $H_k(\mathbb{X},\partial\mathbb{X};\kappa)$ we obtain a further decomposition of $\mathcal{M}(\mathbb{X},f,v;\kappa)$ as a direct sum of \[ \xymatrix{ \mathcal{E}_{n-k}(0,-\infty)^{\downarrow} \ar[rd]^{\bar{\psi}^{\downarrow}} & & & 0 \ar[rd] & \\ & V &  \mbox{and} & & W.  \\ \mathcal{E}_k(0,\infty)_{\uparrow} \ar[ur]_{\bar{\psi}_{\uparrow}} &  &  & 0\ar[ru]      & }\]  The second of these is isomorphic to a direct sum of various $\square_j$, and the classification of the first of these follows directly from Proposition \ref{dacrit}.
\end{proof}

To rephrase the topological situation somewhat, we may begin with a smooth $(n-1)$-manifold $Y$ equipped with a parametrization of its boundary by $S^{k-1}\times S^{n-k-1}$, singling out spheres $S_-\cong S^{k-1}\{\ast\}$ and $S_+\cong \{\ast\}\times S^{n-k-1}$ on $\partial Y$.  The resulting cobordism $\mathbb{X}$ has boundary components \[ \partial^-\mathbb{X}\cong (S^{k-1}\times D^{n-k})\cup_{\partial}Y,\qquad \partial^+\mathbb{X}\cong (D^k\times S^{n-k-1})_{\partial}Y.\]  In the language of surgery theory, one has an attaching sphere $S^{k-1}\times\{0\}\subset S^{k-1}\times D^{n-k}\subset \partial^-\mathbb{X}$ and a belt sphere $\{0\}\times S^{n-k-1}\subset D^k\times S^{ n-k-1}\subset \partial^+\mathbb{X}$, each with trivialized normal bundles; our cobordism is the one associated to framed surgery on the attaching sphere.  The spheres $S_{\pm}\subset\partial Y$ are pushoffs of the attaching and belt spheres to boundaries of their tubular neighborhoods, in directions constant with respect to the trivializations  of their normal bundles.  In terms of the Morse function $f$, the attaching and belt spheres are the boundaries of the descending and ascending disks $\mathcal{D}(p)$ and $\mathcal{A}(p)$, respectively, and their framings are obtained up to isotopy by restricting trivializations of the normal bundles to   $\mathcal{D}(p)$ and $\mathcal{A}(p)$ to their boundaries.

\begin{ex}\label{k1n2}
Consider the case that $k=1$ and $n=2$.  Then $Y$ is to be a compact $1$-manifold whose boundary is parametrized by the four-point space $S^0\times S^0$, with the property that each component of $Y$ has nonempty boundary.  Up to diffeomorphism, the only possible choice of $Y$ is then a disjoint union of two closed intervals $I_1\sqcup I_2$.  We have $S^0\times S^0=\{(1,1),(1,-1),(-1,1),(-1,1)\}$; we may relabel such that the parametrization $\phi\co S^0\times S^0\to \partial Y$ satisfies $\phi(1,1)\in \partial I_1$, and then we have three possible cases depending upon which of the other three points $(1,-1),(-1,1),(-1,-1)$ is sent to the other boundary point of $I_1$.

Now the standard orientation of $S^0\times S^0$ assigns a positive sign to the points $(1,1)$ and $(-1,-1)$ and a negative sign to $(1,-1)$ and $(-1,1)$, and we will obtain an orientable $2$-manifold $\mathbb{X}$ after gluing iff $\phi^{-1}(\partial I_1)$ (and hence also $\phi^{-1}(\partial I_2)$) consists of two points of opposite sign.  Thus, for compatibility with the assumptions of Proposition \ref{dacrit}, if $\phi(-1,-1)\in\partial I_1$ then we should assume that $\mathrm{char}\,\kappa=2$, while in the remaining two cases  this restriction is not necessary.

In $H_0(\partial Y;\kappa)$ we have $S_-=[\phi(1,1)]-[\phi(-1,1)]$ and $S_+=[\phi(1,1)]-[\phi(1,-1)]$.  So if $\phi(-1,1)\in \partial I_1$ then $[S_-]$ vanishes under the inclusion-induced map to $H_0(Y;\kappa)$ and $[S_+]$ does not, while the reverse is true if $\phi(1,-1)\in \partial I_1$.  In the remaining case that $\phi(-1,-1)\in\partial I_1$, both $\iota_*[S_-]$ and $\iota_*[S_+]$ are nonzero in $H_0(Y;\kappa)$, while $\iota_*[S_+]-\iota_*[S_-]=0$ since $\phi(1,-1)$ and $\phi(-1,1)$ are then the two boundary points of $I_2$.  This confirms the statement in the proof of Proposition \ref{dacrit} that $\mathrm{span}\{\iota_*[S_-],\iota_*[S_+]\}$ is always one-dimensional (which was proved therein for all other choices of $k$ and $n$ but was deferred to this example when $n=2$ and $k=1$).  So the conclusions of Proposition \ref{dacrit} and Corollary \ref{morsesummary} do apply to the present situation.

It then follows from Corollary \ref{morsesummary} that  
$\mathcal{M}(\mathbb{X},f,v;\kappa)$ decomposes as:
\begin{itemize}\item[(i)] $\square_1\oplus \square_2\oplus (\uparrow_{0}^{\infty})_{1}\oplus (\searrow^0)_1$ if $\phi^{-1}(\partial I_1)=\{(1,1),(-1,1)\}$; \item[(ii)]  $\square_1\oplus \square_2\oplus (\nearrow_{0})_{1}\oplus (\downarrow^0_{-\infty})_1$ if $\phi^{-1}(\partial I_1)=\{(1,1),(1,-1)\}$; and \item[(iii)]  $\square_1\oplus \square_2\oplus (>_0^0)_1$ if $\phi^{-1}(\partial I_1)=\{(1,1),(-1,-1)\}$ and $\mathrm{char}\,\kappa=2$. \end{itemize}

The three cases correspond to the three ways of gluing two bands $I_1\times I$ and $I_2\times I$ to the standard neighborhood $\mathbb{D}$ in Figure \ref{dfig}, matching $\cup_j(\partial I_j)\times I$ to $\partial^v\mathbb{D}$, and $\cup_j(\partial I_j)\times \{\pm 2\}$ to $\partial^v\mathbb{D}\cap \partial^{\pm}\mathbb{D}$. In case (i), this yields a  pair-of-pants cobordism from $\partial^-\mathbb{X}\cong S^1$ to  $\partial^+\mathbb{X}\cong S^1\sqcup S^1$, and in case (ii) one gets a pair-of-pants cobordism from  $S^1\sqcup S^1$ to $S^1$.  In case (iii), both $\partial^{+}\mathbb{X}$ and $\partial^-\mathbb{X}$ are circles, and $\mathbb{X}$ is diffeomorphic to the complement of two disks in $\mathbb{R}P^2$.  See Figure \ref{k1n2fig}.  
\begin{center}
	\begin{figure}
		\includegraphics[width=5.5in]{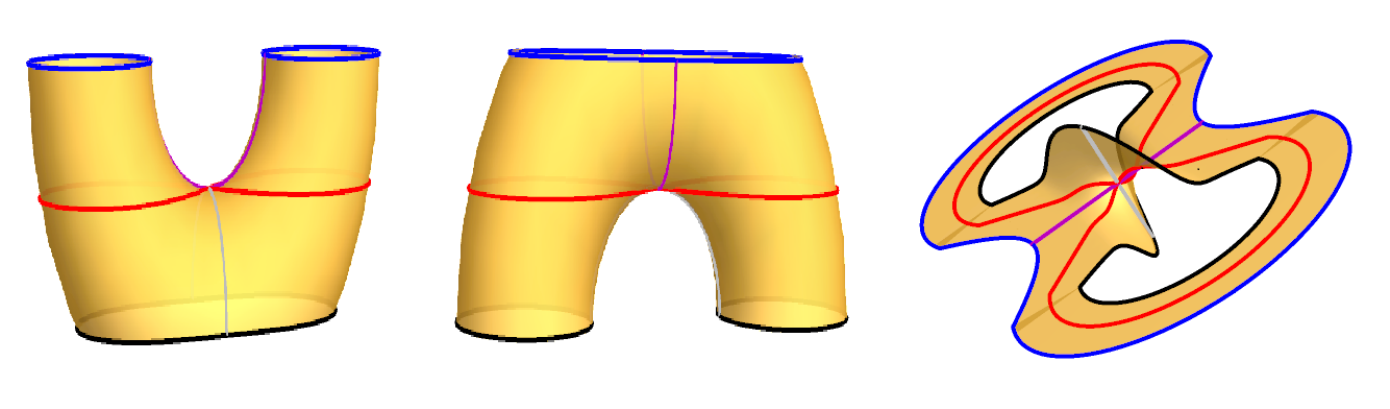}
		\caption{The three types of two-dimensional, index-one elementary cobordisms. As in Figure \ref{dfig}, the level sets at $-2,0,$ and $2$ are shown in black, red, and blue, respectively; the ascending manifold is purple and the descending manifold is light gray.}
		\label{k1n2fig}
	\end{figure}
	\end{center}	

Given Theorem \ref{morseiso} and Proposition \ref{barclass}, our decompositions correspond to the statement that the level set barcode of $f$ consists of the following intervals:\begin{itemize} \item in case (i),  $(0,2)$ and $(-2,2)$ in degree zero, and $[0,2)$ and $(-2,2)$ in degree one; \item in case (ii), $(-2,0)$ and $(-2,2)$ in degree zero, and $(-2,0]$ and $(-2,2)$ in degree one; \item in case (iii), assuming $\mathrm{char}\,\kappa=2$, only $(-2,2)$ in degree zero and $(-2,2)$ and $[0,0]$ in degree $1$.
\end{itemize}  This is reflected in Figure \ref{k1n2fig}.   If one considers case (iii) over a field of characteristic other than 2, then Proposition \ref{dacrit} does not apply due to the use of the fundamental class of $\mathbb{X}$ in the proof of Claim \ref{jdjy}.  In this case $H_*(\mathbb{X},\partial \mathbb{X};\kappa)$ is one-dimensional, supported in degree one, and both $[\mathcal{D}(p)]$ and $[\mathcal{A}(p)]$ are zero in $H_1(\mathbb{X},\partial\mathbb{X};\kappa)$ because neither $\mathcal{D}(p)$ nor $\mathcal{A}(p)$ intersects both boundary components of $\mathbb{X}$.  Consequently  $\mathcal{M}(\mathbb{X},f,v;\kappa)$ decomposes as $\square_1\oplus (\uparrow_{0}^{\infty})_1\oplus (\downarrow^{0}_{-\infty})_1$, corresponding to a level set barcode consisting of $(-2,2)$ in degree zero and $(-2,0]$ and $[0,-2)$ in degree one.
\end{ex}

\begin{ex}\label{k2n4}
	Finally, let us consider a family of cases where $k=2$ and $n=4$.  Thus $Y$ should be a $3$-manifold with torus boundary; specifically, we shall take $Y=S^1\times D^2$.  There is of course a standard parametrization of $\partial Y$ by $S^1\times S^1$, but we get a range of different cobordisms by considering different parametrizations $\phi\co S^1\times S^1\to \partial Y$ represented (with respect to the standard identification of $S^1$ with $\mathbb{R}/\mathbb{Z}$)  by matrices $\left(\begin{array}{cc} p & r \\ q & s\end{array}\right)\in SL(2;\Z)$.  The resulting manifold $\mathbb{X}$ will always be orientable, so we can apply Proposition \ref{dacrit} with an arbitrary coefficient field $\kappa$. 
	
	In view of the discussion before Example \ref{k1n2}, the bottom part $\partial^-\mathbb{X}$ of the boundary of $\mathbb{X}$ is formed by gluing another copy of $S^1\times D^2$ to $Y$ in such a way that the meridian of the other copy is sent to a curve in the class $r[S^1\times\{pt\}]+s[\{pt\}\times \partial D^2]\in H_1(\partial Y;\Z)$.  Thus $\partial^-\mathbb{X}$ is diffeomorphic to the lens space $L(r,s)$; see, \emph{e.g.}, \cite[Chapter 9]{Rol}.  Similarly, as $\partial^+\mathbb{X}$ is formed by gluing $D^2\times S^1$ to $Y$ via $\phi$, $\partial^+\mathbb{X}$ is diffeomorphic to $L(p,q)$.
	
	Now, in the notation of Corollary \ref{morsesummary}, we have $\iota_*[S_-]=0$ if and only $\phi$ maps $S^1\times\{pt\}$ to a multiple (in $H_1(\partial Y;\kappa)$) of the meridian of $S^1\times \partial D^2$, and $\iota_*[S_+]=0$ if and only if $\phi$ maps $\{pt\}\times S^1$ to a multiple of the meridian of $S^1\times \partial D^2$.  These conditions amount to the statements that $p=0\in\kappa$ and $r=0\in\kappa$, respectively.  So we conclude that our elementary cobordism $\mathbb{X}$ from $L(r,s)$ to $L(p,q)$ has $\mathcal{M}(\mathbb{X},f,v;\kappa)$ isomorphic to: \begin{itemize} \item $\square_1\oplus\square_4\oplus (\uparrow_0^{\infty})_2\oplus (\searrow^0)_2$ if $p=0\in\kappa$; \item $\square_1\oplus \square_4\oplus (\nearrow_0)_2\oplus (\downarrow^{0}_{-\infty})_2$ if $r=0\in\kappa$; \item $\square_1\oplus \square_4\oplus(>_0^0)_2$ if $pr\neq 0\in\kappa$.
\end{itemize}
These correspond to level set barcodes $\{(-2,2)_0,(0,2)_1,[0,2)_2,(-2,2)_3\}$, $\{(-2,2)_0,(-2,0)_1,(-2,0]_2,(-2,2)_3\}$, and $\{(-2,2)_0,[0,0]_2,(-2,2)_3\}$, respectively, where ``$I_j$'' refers to an interval $I$ in grading $j$.
\end{ex}

\section{From filtered cospans to $\mathbb{M}$-indexed persistence modules}\label{funsect}

This section is concerned with the construction and properties of a functor $\mathcal{F}\co \mathsf{HCO}^{\Lambda}(\kappa)\to \mathsf{K}(\kappa)^{\mathbb{M}}$. Here $\mathsf{HCO}^{\Lambda}(\kappa)$ is the category of  $\Lambda$-bounded filtered cospans from Section \ref{catsect}, $\mathsf{K}(\kappa)$ is the homotopy category of chain complexes of $\kappa$-vector spaces, and, following \cite{BBF20} (who, in our language, generally take $\Lambda=\frac{\pi}{2}$), \[ \mathbb{M}=\{(x,y)\in \R^2|-2\Lambda\leq x+y\leq 2\Lambda\},\] regarded as a category with a unique morphism $(x,y)\to (x',y')$ whenever $x\geq x'$ and $y\leq y'$, and no morphisms $(x,y)\to (x',y')$ otherwise. The notation $\mathsf{K}(\kappa)^{\mathbb{M}}$ then refers to the category of functors from $\mathbb{M}$ to $\mathsf{K}(\kappa)$, with morphisms given by natural transformations.  (In other words, $\mathsf{K}(\kappa)^{\mathbb{M}}$ is the category of persistence modules in $\mathsf{K}(\kappa)$ indexed by the sub-poset $\mathbb{M}$ of $\mathbb{R}^{\circ}\times \R$, where $\R^{\circ}$ denotes the set of real numbers regarded as a poset whose order is reversed from the usual order.)  The functor $\mathcal{F}$, whose construction requires several steps, will allow us to relate the properties of $\mathsf{HCO}(\kappa)$ discussed up to this point in the paper with interlevel persistence.

\subsection{The strip $\mathbb{M}$}\label{stripsect}

Before defining our functor $\mathsf{HCO}^{\Lambda}(\kappa)\to \mathsf{K}(\kappa)^{\mathbb{M}}$ we introduce conventions regarding regions in \cite{BBF20}'s strip $\mathbb{M}=\{(x,y)\in \R^2|-2\Lambda\leq x+y\leq2\Lambda\}$  Recall that this strip is regarded as a partially ordered set with respect to the order $\preceq$  given by $(x,y)\preceq (x',y')$ iff $x\geq x'$ and $y\leq y'$.  Following \cite{BBF20}, there is an important strictly increasing, order-preserving map \begin{align*} T & \co \mathbb{M}\to \mathbb{M} \\ (x,y) &\mapsto (-2\Lambda-y,2\Lambda-x).\end{align*}  This is a glide-reflection, given by first reflecting across the line $\{x+y=0\}$ and then translating by $(-2\Lambda,2\Lambda)$. This map $T$ generates a $\Z$-action on $\mathbb{M}$, with quotient a M\"obius strip.  As a fundamental domain for this action we shall use the set \[ \mathfrak{D}=S\sqcup L\sqcup A\] where: \begin{align*} S&=\left\{(x,y)\in\mathbb{R}^2\left|-\Lambda<x\leq \Lambda,\,-\Lambda\leq y<\Lambda\right.\right\};  \\ L&=\left\{(x,y)\in \R^2\left|x+y\geq -2\Lambda,\,x\leq -\Lambda,\,y<\Lambda\right.\right\}; \\ A&=\left\{(x,y)\in \R^2 \left| x+y \leq 2\Lambda,\,x>-\Lambda,\,y\geq \Lambda \right.\right\}.    \end{align*} (Here ``$S$'' stands for ``square,'' ``$L$'' for ``left,'' and ``$A$'' for ``above.'')  We leave to the reader the verification that $\mathbb{M}$ is the disjoint union of the sets $T^k(\mathfrak{D})$ as $k$ varies through $\Z$. See Figure \ref{stripfig}. Note that the fundamental domain used in most of \cite{BBF20},\cite{BBF21} is different from $\mathfrak{D}$.   

\begin{center}
	\begin{figure}
		\includegraphics[width=3in]{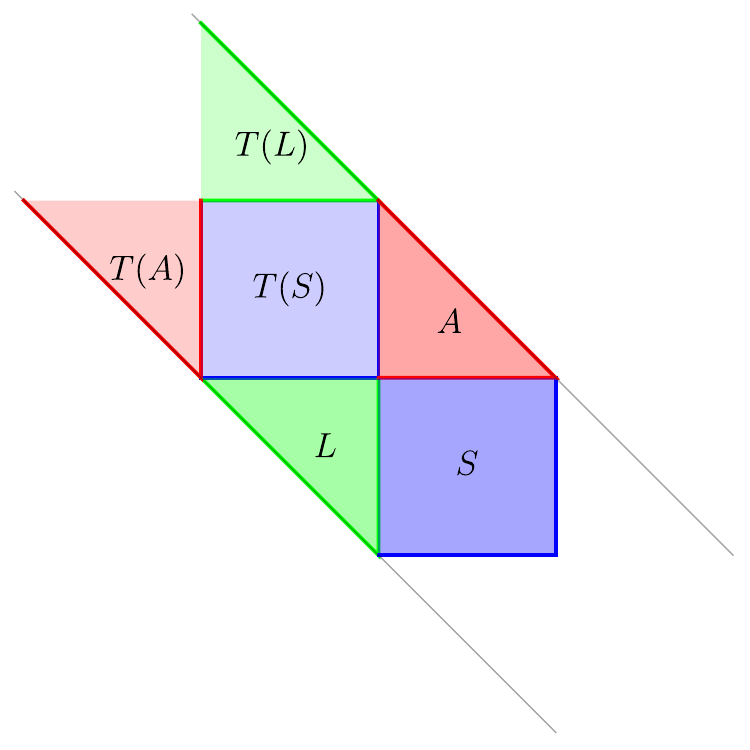}
	\caption{The strip $\mathbb{M}$, showing the regions $S,L,A$ that comprise the fundamental domain $\mathfrak{D}$ for the action generated by $T$, as well as the images of these regions under $T$.}
	\label{stripfig}
	\end{figure}
\end{center}

\subsection{The functors $\mathcal{F}_{(s,t)}^{0}$}\label{f0sec}

In this subsection, for each $(s,t)\in\mathbb{M}$ we shall define a functor \[ \mathcal{F}^{0}_{(s,t)}\co \mathsf{CO}^{\Lambda}(\kappa)\to \mathsf{Ch}(\kappa). \]  
First, we set up a bit more notation.

For an object of $\mathsf{CO}^{\Lambda}(\kappa)$, given by the data $\left((C_{\uparrow *},\partial_{\uparrow}^C,\ell_{\uparrow}^C),(C^{\downarrow}_{*},\partial^{\downarrow}_C,\ell^{\downarrow}_C),(D_*,\partial_D),\psi_{\uparrow},\psi^{\downarrow}\right)$ as in Definition \ref{codfn} and for $t\in [-\Lambda,\Lambda]$, recall the notations for the subcomplexes \begin{equation}\label{subcxs} C_{\uparrow}^{\leq t}=(\ell_{\uparrow}^{C})^{-1}((-\infty,t]),\quad C^{\downarrow}_{\geq t}=(\ell^{\downarrow}_{C})^{-1}([t,\infty]).  \end{equation}
Because objects of $\mathsf{CO}^{\Lambda}(\kappa)$ are by assumption $\Lambda$-bounded, we have \[ C_{\uparrow}^{\leq -\Lambda}=C^{\downarrow}_{\geq \Lambda}=\{0\}.\]

For an object $C=(C_*,\partial_C)$ of $\mathsf{Ch}(\kappa)$ and for $k\in \Z$, $C[k]$ will denote the chain complex given by setting $C[k]_n=C_{n+k}$ and $\partial_{C[k]}=(-1)^k\partial_C$.  The assignment $C\mapsto C[k]$ can be enriched to a dg-functor $\mathsf{Ch}(\kappa)\to\mathsf{Ch}(\kappa)$ by taking its action on $\mathrm{Hom}_{\mathsf{Ch}(\kappa)}^{m}(C,D)$ to be multiplication by $(-1)^{mk}$.

 The formula for  $\mathcal{F}^{0}_{(s,t)}$ depends on where $(s,t)$ lies with respect to the subsets introduced in Section \ref{stripsect}, noting that \begin{equation}\label{mdecomp} \mathbb{M}=\sqcup_{k\in \Z}T^k(S\sqcup L\sqcup A);\end{equation} indeed we shall set, for $k\in \Z$ and $(x,y)\in\mathfrak{D}=S\sqcup L\sqcup A$, \begin{equation}\label{f0dfn} \mathcal{F}_{T^k(x,y)}^{0}=\left\{\begin{array}{ll} \mathcal{S}^{k}_{x,y} & \mbox{if }(x,y)\in S \\ \mathcal{L}^{k}_{x,y} & \mbox{if }(x,y)\in L \\ \mathcal{A}^{k}_{x,y} & \mbox{ if }(x,y)\in A  \end{array}\right. \end{equation} where $\mathcal{S}^{k}_{x,y},\mathcal{L}^{k}_{x,y},$ $\mathcal{A}^{k}_{x,y}$, are defined as follows.  Given an object $\mathcal{C}=\left((C_{\uparrow *},\partial_{\uparrow}^C,\ell_{\uparrow}^C),(C^{\downarrow}_{*},\partial^{\downarrow}_C,\ell^{\downarrow}_C),(D_*,\partial_D),\psi_{\uparrow},\psi^{\downarrow}\right)$ as in Definition \ref{codfn}, with notation as in (\ref{subcxs}) we set: \begin{align*}
 	\mathcal{S}^{k}_{x,y}(\mathcal{C}) &= \mathrm{Cone}(-\psi^{\downarrow}+\psi_{\uparrow}\co C^{\downarrow}_{\geq x}\oplus C_{\uparrow}^{\leq y}\to D)[-k+1] \\
 	\mathcal{L}^{k}_{x,y}(\mathcal{C}) &= \mathrm{Cone}(C_{\uparrow}^{\leq -2\Lambda-x}\hookrightarrow C_{\uparrow}^{\leq y})[-k] \\ \mathcal{A}^{k}_{x,y}(\mathcal{C})&= \mathrm{Cone}(C^{\downarrow}_{\geq 2\Lambda-y}\hookrightarrow C^{\downarrow}_{\geq x})[-k] .
 	\end{align*}  

(In the last two lines above, $\hookrightarrow$ denotes inclusion; these inclusions are well-defined since $x\leq 2\Lambda-y$ on $A$, and $y\geq -2\Lambda-x$ on $L$.)

Now given two objects $\mathcal{C}=\left((C_{\uparrow *},\partial_{\uparrow}^C,\ell_{\uparrow}^C),(C^{\downarrow}_{*},\partial^{\downarrow}_C,\ell^{\downarrow}_C),(D_*,\partial_D),\psi_{\uparrow},\psi^{\downarrow}\right)$    and\\ $\mathcal{X}=\left((X_{\uparrow *},\partial_{\uparrow}^X,\ell_{\uparrow}^{X}),(X^{\downarrow}_{*},\partial^{\downarrow}_{X},\ell^{\downarrow}_{X}),(Y_*,\partial_Y),\phi^{\downarrow},\phi_{\uparrow}\right)$ of $\mathsf{CO}^{\Lambda}(\kappa)$, and given a degree-$m$ morphism $\mathfrak{a}=(\alpha^{\downarrow},\alpha_{\uparrow},\alpha,K^{\downarrow},K_{\uparrow})\in \mathrm{Hom}_{\mathsf{CO}^{\Lambda}(\kappa)}^m(\mathcal{C},\mathcal{X})$, we take the actions of our functors on $\mathfrak{a}$ to be given in block form by: \begin{equation}\label{sfun} \mathcal{S}^{k}_{x,y}(\mathfrak{a})=(-1)^{m(k-1)}\left(\begin{array}{ccc} \alpha^{\downarrow} & 0 & 0 \\ 0 & \alpha_{\uparrow} & 0 \\ -K^{\downarrow} & K_{\uparrow} & \alpha  \end{array} \right)\end{equation}
(as a graded map $C^{\downarrow}_{\geq x}[-k]\oplus C_{\uparrow}^{\leq y}[-k]\oplus D[-k+1]\to X^{\downarrow}_{\geq x}[-m-k]\oplus X_{\uparrow}^{\leq y}[-m-k]\oplus Y[-m-k+1]$); 
\begin{equation}\label{lfun} \mathcal{L}^{k}_{x,y}(\mathfrak{a})= (-1)^{mk}\left(\begin{array}{cc} \alpha_{\uparrow} & 0 \\ 0 & (-1)^{m}\alpha_{\uparrow} \end{array} \right)  \end{equation} (as a graded map $C_{\uparrow}^{\leq -2\Lambda-x}[-k-1]\oplus C_{\uparrow}^{\leq y}[-k]\to X_{\uparrow}^{\leq -2\Lambda-x}[-m-k-1]\oplus X_{\uparrow}^{\leq y}[-m-k]$);
and \begin{equation} \label{afun} \mathcal{A}^{k}_{x,y}(\mathfrak{a})= (-1)^{mk}\left(\begin{array}{cc} \alpha^{\downarrow} & 0 \\ 0 & (-1)^{m}\alpha^{\downarrow} \end{array} \right)  \end{equation} (as a graded map $C^{\downarrow}_{\geq 2\Lambda-y}[-k-1]\oplus C^{\downarrow}_{\geq x}[-k]\to X^{\downarrow}_{\geq 2\Lambda-y}[-m-k-1]\oplus X^{\downarrow}_{\geq x}[-m-k]$).

It is a routine matter to check that each of $\mathcal{S}^{k}_{x,y},\mathcal{L}^{k}_{x,y},\mathcal{A}^{k}_{x,y}$ respect compositions and intertwine the differentials $\delta$, so they each define dg-functors, and thus induce functors $\mathsf{HCO}^{\Lambda}(\kappa)\to \mathsf{K}(\kappa)$ between homotopy categories. For $(s,t) \in\mathbb{M}$, we denote by $\underline{\mathcal{F}}^{0}_{(s,t)}$ the functor on homotopy categories induced by, respectively, $\mathcal{S}^{k}_{x,y},\mathcal{L}^{k}_{x,y},$ or $\mathcal{A}^{k}_{x,y}$ if $(s,t)=T^k(x,y)$ where $(x,y)$ lies in, respectively $S,L,$ or $A$. 

\subsection{Transformations between the $\mathcal{F}^{0}_{(s,t)}$}\label{transf}

For an object $\mathcal{C}=\left((C_{\uparrow *},\partial_{\uparrow}^C,\ell_{\uparrow}^C),(C^{\downarrow}_{*},\partial^{\downarrow}_C,\ell^{\downarrow}_C),(D_*,\partial_D),\psi_{\uparrow},\psi^{\downarrow}\right)$ of $\mathsf{CO}^{\Lambda}(\kappa)$, let us write \begin{align*} \mathcal{S}^k(\mathcal{C})&=\mathrm{Cone}(-\psi^{\downarrow}+\psi_{\uparrow}\co C^{\downarrow}\oplus C_{\uparrow}\to D)[-k+1] \\ \mathcal{L}^{k}(\mathcal{C})&=\mathrm{Cone}(C_{\uparrow}\to C_{\uparrow})[-k] \\ \mathcal{A}^k(\mathcal{C})&=\mathrm{Cone}(C^{\downarrow}\to C^{\downarrow})[-k] \end{align*}
where the maps in the cones defining $\mathcal{L}^k(\mathcal{C}),\mathcal{A}^k(\mathcal{C})$ are the respective identities.  Thus the various $\mathcal{S}^{k}_{x,y}(\mathcal{C}),\mathcal{L}^{k}_{x,y}(\mathcal{C}),$ and $\mathcal{A}^{k}_{x,y}(\mathcal{C})$ are certain filtered subcomplexes of, respectively, $\mathcal{S}^{k}(\mathcal{C}),\mathcal{L}^{k}(\mathcal{C}),$ and $\mathcal{A}^k(\mathcal{C})$.

We define the following maps, written in block matrix forms in the obvious way:

\begin{align*}  \phi_{T^kS}^{T^kL}(\mathcal{C})&=\left(\begin{array}{ccc} 0 & 0 & 0 \\ 0 & 1_{C_{\uparrow}[-k]} & 0    \end{array}\right)\co \mathcal{S}^k(\mathcal{C})\to \mathcal{L}^k(\mathcal{C})   \\ \phi_{T^kS}^{T^kA}(\mathcal{C})&=\left(\begin{array}{ccc} 0 & 0 & 0 \\  1_{C^{\downarrow}[-k]} & 0 & 0    \end{array}\right)\co \mathcal{S}^k(\mathcal{C})\to \mathcal{A}^k(\mathcal{C}) \\ \phi_{T^kL}^{T^{k+1}S}(\mathcal{C}) &= \left(\begin{array}{cc} 0 & 0 \\ 1_{C_{\uparrow}[-k-1]} & 0 \\ 0 & \psi_{\uparrow}\end{array}\right) \co \mathcal{L}^{k}(\mathcal{C})\to \mathcal{S}^{k+1}(\mathcal{C})  
\\ \phi_{T^kA}^{T^{k+1}S}(\mathcal{C}) &= \left(\begin{array}{cc} -1_{C^{\downarrow}[-k-1]} & 0 \\  0 & 0 \\ 0 & \psi^{\downarrow}\end{array}\right) \co \mathcal{A}^{k}(\mathcal{C})\to \mathcal{S}^{k+1}(\mathcal{C})   	
	\end{align*} 

Routine computation shows that each of these maps are chain maps.  They are intended to provide the structure maps corresponding to passing between adjacent regions in Figure \ref{stripfig} for the $\mathbb{M}$-indexed persistence module associated to $\mathcal{C}$.   

Let us record some computations concerning compositions of the above maps. Recall that the differentials $\delta_{C,D}$ on the complexes $\mathrm{Hom}_{\mathsf{Ch}(\kappa)}(C,D)$ are given by, for $\psi\in\mathrm{Hom}^m_{\mathsf{Ch}(\kappa)}(C,D)$,  $\delta_{C,D}\psi=\partial_D\psi-(-1)^m\psi\partial_C$.  We have: \begin{equation}\label{lazero}
	\phi_{T^{k+1}S}^{T^{k+1}A}(\mathcal{C})\phi_{T^kL}^{T^{k+1}S}(\mathcal{C})=0; \qquad \phi_{T^{k+1}S}^{T^{k+1}L}(\mathcal{C})\phi_{T^kA}^{T^{k+1}S}(\mathcal{C})=0; 	
\end{equation}
\begin{equation}\label{llhtopy}
	\phi_{T^{k+1}S}^{T^{k+1}L}(\mathcal{C})\phi_{T^kL}^{T^{k+1}S}(\mathcal{C})=\left(\begin{array}{cc} 0 & 0 \\ 1 & 0 \end{array}\right)=\delta_{\mathcal{L}^k(\mathcal{C}),\mathcal{L}^{k+1}(\mathcal{C})}\left(\begin{array}{cc} 0 & 0 \\ 0 & (-1)^k\end{array}\right);
\end{equation}
\begin{equation}\label{aahtopy}
		\phi_{T^{k+1}S}^{T^{k+1}A}(\mathcal{C})\phi_{T^kA}^{T^{k+1}S}(\mathcal{C})=\left(\begin{array}{cc} 0 & 0 \\ -1 & 0 \end{array}\right)=\delta_{\mathcal{A}^k(\mathcal{C}),\mathcal{A}^{k+1}(\mathcal{C})}\left(\begin{array}{cc} 0 & 0 \\ 0 & (-1)^{k+1}\end{array}\right);	
\end{equation}
\begin{equation}\label{sls}
		\phi^{T^{k+1}S}_{T^{k}L}(\mathcal{C})\phi_{T^kS}^{T^{k}L}(\mathcal{C})=\delta_{\mathcal{S}^{k}(\mathcal{C}),\mathcal{S}^{k+1}(\mathcal{C})}\left(\begin{array}{ccc} 0 & 0 & 0 \\ 0 & (-1)^k & 0 \\ 0 & 0 & 0 \end{array}\right);
	\end{equation}
\begin{equation}\label{sas}
	\phi^{T^{k+1}S}_{T^{k}A}(\mathcal{C})\phi_{T^kS}^{T^{k}A}(\mathcal{C})=\delta_{\mathcal{S}^{k}(\mathcal{C}),\mathcal{S}^{k+1}(\mathcal{C})}\left(\begin{array}{ccc} (-1)^{k+1} & 0 & 0 \\ 0 & 0 & 0 \\ 0 & 0 & 0 \end{array}\right); \mbox{ and }
\end{equation}
\begin{equation}\label{ssdiff}
	\phi^{T^{k+1}S}_{T^{k}A}(\mathcal{C})\phi_{T^kS}^{T^{k}A}(\mathcal{C})-	\phi^{T^{k+1}S}_{T^{k}L}(\mathcal{C})\phi_{T^kS}^{T^{k}L}(\mathcal{C})=\delta_{\mathcal{S}^{k}(\mathcal{C}),\mathcal{S}^{k+1}(\mathcal{C})}\left(\begin{array}{ccc} 0 & 0 & 0 \\ 0 &  0 & 0 \\ 0 & 0 & (-1)^k \end{array}\right).
	\end{equation}

Relating to the naturality of the maps $\phi_{T^kS}^{T^{k}L},\phi_{T^kS}^{T^kA},\phi_{T^kL}^{T^{k+1}S},\phi_{T^kA}^{T^{k+1}S}$, suppose that $\mathcal{X}=\left((X_{\uparrow *},\partial_{\uparrow}^X,\ell_{\uparrow}^{X}),(X^{\downarrow}_{*},\partial^{\downarrow}_{X},\ell^{\downarrow}_{X}),(Y_*,\partial_Y),\phi^{\downarrow},\phi_{\uparrow}\right)$ is another object of $\mathsf{CO}^{\Lambda}(\kappa)$ and that $\mathfrak{a}=(\alpha^{\downarrow},\alpha_{\uparrow},\alpha,K^{\downarrow},K_{\uparrow})\in \mathrm{Hom}_{\mathsf{CO}^{\Lambda}(\kappa)}^{m}(\mathcal{C},\mathcal{X})$.  Write $\mathcal{S}^k(\mathfrak{a})\in \mathrm{Hom}_{\mathsf{Ch}(\kappa)}^{m}(\mathcal{S}^k(\mathcal{C},\mathcal{X}))$ for the morphism given by the right hand side of (\ref{sfun}) (but now considered as a map between the full unfiltered complexes $\mathcal{S}^k(\mathcal{C})$ and $\mathcal{S}^k(\mathcal{X})$), and similarly let  $\mathcal{L}^k(\mathfrak{a})$ and $\mathcal{A}^k(\mathfrak{a})$ be given  by the right hand sides of (\ref{lfun}) and (\ref{sfun}).  We find: 
\begin{equation}\label{natkk}
	\phi_{T^kS}^{T^kL}(\mathcal{X})\mathcal{S}^k(\mathfrak{a})=\mathcal{L}^k(\mathfrak{a})\phi_{T^kS}^{T^kL}(\mathcal{C});\qquad 	\phi_{T^kS}^{T^kA}(\mathcal{X})\mathcal{S}^k(\mathfrak{a})=\mathcal{A}^k(\mathfrak{a})\phi_{T^kA}^{T^kL}(\mathcal{C});
\end{equation}
\begin{equation}\label{lsnat}
	\phi_{T^kL}^{T^{k+1}S}(\mathcal{X})\mathcal{L}^k(\mathfrak{a})-\mathcal{S}^{k+1}(\mathfrak{a})\phi_{T^kL}^{T^{k+1}S}(\mathcal{C})=\delta_{\mathcal{L}^{k}(\mathcal{C}),\mathcal{S}^{k+1}(\mathcal{X})}(\mathcal{K}_{\uparrow}\mathfrak{a})+\mathcal{K}_{\uparrow}(\delta_{\mathcal{C},\mathcal{X}}\mathfrak{a});\mbox{ and}
\end{equation}
\begin{equation}\label{asnat}
	\phi_{T^kA}^{T^{k+1}S}(\mathcal{X})\mathcal{A}^k(\mathfrak{a})-\mathcal{S}^{k+1}(\mathfrak{a})\phi_{T^kA}^{T^{k+1}S}(\mathcal{C})=\delta_{\mathcal{A}^{k}(\mathcal{C}),\mathcal{S}^{k+1}(\mathcal{X})}(\mathcal{K}^{\downarrow}\mathfrak{a})+\mathcal{K}^{\downarrow}(\delta_{\mathcal{C},\mathcal{X}}\mathfrak{a}),
\end{equation}
where we define $\mathcal{K}_{\uparrow}\co \mathrm{Hom}^m_{\mathsf{CO}^{\Lambda}(\kappa)}(\mathcal{C},\mathcal{X})\to \mathrm{Hom}^{m-1}_{\mathsf{Ch}(\kappa)}(\mathcal{L}^k(\mathcal{C}),\mathcal{S}^{k+1}(\mathcal{X}))$ and  $\mathcal{K}^{\downarrow}\co \mathrm{Hom}^m_{\mathsf{CO}^{\Lambda}(\kappa)}(\mathcal{C},\mathcal{X})\to \mathrm{Hom}^{m-1}_{\mathsf{Ch}(\kappa)}(\mathcal{A}^k(\mathcal{C}),\mathcal{S}^{k+1}(\mathcal{X}))$
by \[ \mathcal{K}_{\uparrow}(\alpha^{\downarrow},\alpha_{\uparrow},\alpha,K^{\downarrow},K_{\uparrow})=(-1)^{(m-1)(k-1)}\left(\begin{array}{ccc} 0 & 0 \\ 0 & 0 \\ 0 & K_{\uparrow}\end{array}\right)\quad\mbox{and}\]\[ \mathcal{K}^{\downarrow}(\alpha^{\downarrow},\alpha_{\uparrow},\alpha,K^{\downarrow},K_{\uparrow})=(-1)^{(m-1)(k-1)}\left(\begin{array}{ccc} 0 & 0 \\ 0 & 0 \\ 0 & K^{\downarrow}\end{array}\right)\]

Equations (\ref{natkk}) show that $\phi_{T^kS}^{T^kL}$ defines a natural transformation between the functors $\mathcal{S}^k,\mathcal{L}^k\co \mathsf{CO}^{\Lambda}(\kappa)\to \mathsf{Ch}(\kappa)$, and likewise that $\phi_{T^kS}^{T^kA}$ defines a natural transformation between $\mathcal{S}^k$ and $\mathcal{A}^k$.  Equations (\ref{lsnat}) and (\ref{asnat}), on the other hand, show that the transformations $\phi_{T^kL}^{T^{k+1}S}$ and $\phi_{T^kA}^{T^{k+1}S}$  do not give natural transformations between functors from $\mathsf{CO}^{\Lambda}(\kappa)$ to $\mathsf{Ch}(\kappa)$,  but that they induce natural transformations $\uph_{T^kL}^{T^{k+1}S}$ and $\uph_{T^kA}^{T^{k+1}S}$ between the functors $\mathsf{HCO}^{\Lambda}(\kappa)\to \mathsf{K}(\kappa)$ that are induced by $\mathcal{S}^k,\mathcal{L}^k,\mathcal{A}^k$ at the level of homotopy categories.

Our interest will be not so much on the actions of these transformations as maps between the full complexes $\mathcal{S}^k(\mathcal{C}),\mathcal{L}^k(\mathcal{C}),$ and $\mathcal{A}^{k}(\mathcal{C})$, but rather as maps between the filtered subcomplexes $\mathcal{S}_{x,y}^{k}(\mathcal{C}),\mathcal{L}_{x,y}^{k}(\mathcal{C}),$ and $\mathcal{A}_{x,y}^{k}(\mathcal{C})$ that are used to define the $\mathcal{F}_{(s,t)}^{0}(\mathcal{C})$ via (\ref{f0dfn})\footnote{We will continue to write general elements of the strip $\mathbb{M}$ as $(s,t)$, reserving $(x,y)$ for elements that are known or chosen to belong to the fundamental domain $\mathfrak{D}=S\cup L\cup A$.}, considered up to homotopy.  

One may verify that, if $(s,t),(s',t')\in\mathbb{M}$ with $(s,t)\preceq (s',t')$,  and either $(s,t)\in T^kS$ and $(s',t')\in T^kL$, or $(s,t)\in T^kS$ and $(s',t')\in T^kA$, or $(s,t)\in T^kL$ and $(s',t')\in T^{k+1}S$, or $(s,t)\in T^kA$ and $(s',t')\in T^{k+1}S$, then the appropriate map from among $\phi_{T^kS}^{T^kL}(\mathcal{C}), \phi_{T^kS}^{T^kA}(\mathcal{C}), \phi_{T^kL}^{T^{k+1}S}(\mathcal{C}),$ and  $\phi_{T^kA}^{T^{k+1}S}(\mathcal{C})$ restricts as a map from $\mathcal{F}_{(s,t)}^{0}(\mathcal{C})$ to $\mathcal{F}_{(s',t')}^{0}(\mathcal{C})$.  Let us confirm this in the third case, where  $(s,t)\in T^kL$ and $(s',t')\in T^{k+1}S$ (still assuming $(s,t)\preceq (s',t')$); the other cases are similar or simpler.  Write $(s,t)=T^k(x,y)$ and $(s',t')=T^{k+1}(x',y')=T^k(-2\Lambda-y',2\Lambda-x')$ where $(x,y)\in L$ and $(x',y')\in S$.  Then (as graded $\kappa$-vector spaces) we have \[ \mathcal{F}_{(s,t)}^{0}(\mathcal{C})=C_{\uparrow}^{\leq -2\Lambda-x}[-k-1]\oplus C_{\uparrow}^{\leq y}[-k],\quad \mathcal{F}_{(s',t')}^{0}(\mathcal{C})=C^{\downarrow}_{\geq x'}[-k-1]\oplus C_{\uparrow}^{\leq y'}[-k-1]\oplus D[-k].\]  
By definition, $\phi_{T^kL}^{T^{k+1}S}(\mathcal{C})$ maps $(a,b)\in C_{\uparrow}[-k-1]\oplus C_{\uparrow}[-k]$ to $(0,a,\psi_{\uparrow}b)\in C^{\downarrow}[-k-1]\oplus C_{\uparrow}[-k-1]\oplus D[-k]$.  That this restricts to a map $\mathcal{F}_{(s,t)}^{0}(\mathcal{C})\to \mathcal{F}_{(s',t')}^{0}(\mathcal{C})$ is equivalent to the condition that $C_{\uparrow}^{\leq -2\Lambda-x}\subset C_{\uparrow}^{\leq y'}$, \emph{i.e.} that $-2\Lambda-x\leq y'$. This  does indeed hold under the assumption that $(s,t)\preceq (s',t')$, since the map $T$ is strictly increasing with respect to $\preceq$ and so $T^{-k}(s,t)\preceq T^{-k}(s',t')$, \emph{i.e.}, $(x,y)\preceq (-2\Lambda-y',2\Lambda-x')$, whence $x\geq -2\Lambda-y'$.

On a similar note, if $(s,t)\in T^kA$ and $(s',t')\in T^{k+1}A$, say $(s,t)=T^k(x,y)$ and $(s',t')=T^{k+1}(x',y')=T^k(-2\Lambda-y',2\Lambda-x')$, and if $(s,t)\preceq (s',t')$, then the map \[ \left(\begin{array}{cc} 0 & 0 \\ -1 & 0\end{array}\right)\co C^{\downarrow}[-k-1]\oplus C^{\downarrow}[-k] \to C^{\downarrow}[-k-2]\oplus C^{\downarrow}[-k-1],   \] which appears in (\ref{aahtopy}), restricts as a map $\mathcal{F}_{(s,t)}^{0}(\mathcal{C})\to \mathcal{F}_{(s',t')}^{0}(\mathcal{C})$.  Indeed, this is equivalent to the statement that $C^{\downarrow}_{\geq 2\Lambda-y}\subset C^{\downarrow}_{\geq x'}$, \emph{i.e.} that $2\Lambda-y\geq x'$, and this inequality is implied by the assumption that $T^k(x,y)\preceq T^{k}(-2\Lambda-y',2\Lambda-x')$.  Similarly, the map $\left(\begin{array}{cc} 0 & 0 \\ 1 & 0\end{array}\right)\co C_{\uparrow}[-k-1]\oplus C_{\uparrow}[-k]\to C_{\uparrow}[-k-2]\oplus C_{\uparrow}[-k-1]$ appearing in  (\ref{llhtopy}) restricts as a map $\mathcal{F}_{(s,t)}^{0}(\mathcal{C})\to \mathcal{F}_{(s',t')}^{0}(\mathcal{C})$ in the case that $(s,t)\in T^kL$ and $(s',t')\in T^{k+1}L$ with $(s,t)\preceq (s',t')$.  

Now Equations (\ref{llhtopy}) and (\ref{aahtopy}) show that these maps $\left(\begin{array}{cc} 0 & 0 \\ -1 & 0\end{array}\right)$ and $\left(\begin{array}{cc} 0 & 0 \\ 1 & 0\end{array}\right)$ are nullhomotopic when considered as maps $\mathcal{L}^k(\mathcal{C})\to \mathcal{L}^{k+1}(\mathcal{C})$ and $\mathcal{A}^k(\mathcal{C})\to \mathcal{A}^{k+1}(\mathcal{C})$ on the full complexes (as well they must be, since $\mathcal{A}^k(\mathcal{C})$ and $\mathcal{L}^k(\mathcal{C})$ are zero objects in the homotopy category).  However, when these maps are restricted to maps $\mathcal{F}_{(s,t)}^{0}(\mathcal{C})\to \mathcal{F}_{(s',t')}^{0}(\mathcal{C})$ on the filtered subcomplexes, the only case in which (\ref{llhtopy}) or (\ref{aahtopy}) implies that they are nullhomotopic is when the map $\pm\left(\begin{array}{cc} 0 & 0 \\ 0 & 1\end{array}\right)$ appearing on the right of (\ref{llhtopy}) or (\ref{aahtopy}) restricts as a map between these subcomplexes.  

The compositions $\phi_{T^kL}^{T^{k+1}S}(\mathcal{C})\phi_{T^kS}^{T^{k+1}L}(\mathcal{S})(\mathcal{C})$ and $\phi_{T^kA}^{T^{k+1}L}(\mathcal{C})\phi_{T^kL}^{T^{k+1}A}(\mathcal{S})(\mathcal{C})$, mapping $\mathcal{S}^k(\mathcal{C})$ to $\mathcal{S}^{k+1}(\mathcal{C})$, each have image in the summand $D[-k]$ of $\mathcal{S}^{k+1}(\mathcal{C})=C^{\downarrow}[-k-1]\oplus C_{\uparrow}[-k-1]\oplus D[-k]$.  As this summand is contained in every filtered subcomplex $\mathcal{S}^{k+1}_{x,y}(\mathcal{C})$, it follows that these maps restrict as maps $\mathcal{S}^k_{x,y}(\mathcal{C})\to\mathcal{S}^{k+1}_{x',y'}(\mathcal{C})$ for any choice of $(x,y),(x',y')\in S$.  Moreover, again because the summand $D[-k]$ is common to each filtered subcomplex of $\mathcal{S}^{k+1}(\mathcal{C})$, (\ref{ssdiff}) shows that $\phi_{T^kL}^{T^{k+1}S}(\mathcal{C})\phi_{T^kS}^{T^{k+1}L}(\mathcal{S})(\mathcal{C})$ and $\phi_{T^kA}^{T^{k+1}L}(\mathcal{C})\phi_{T^kL}^{T^{k+1}A}(\mathcal{S})(\mathcal{C})$ are homotopic as maps $\mathcal{S}^k_{x,y}(\mathcal{C})\to\mathcal{S}^{k+1}_{x',y'}(\mathcal{C})$ for any choice of $(x,y),(x',y')$.  (Meanwhile, (\ref{sas}) and (\ref{sls}) show that these maps are nullhomotopic just for certain choices of $(x,y),(x',y')$.)

\subsection{The functor $\mathcal{F}\co \mathsf{HCO}^{\Lambda}(\kappa)\to \mathsf{K}(\kappa)^{\mathbb{M}}$} \label{fsect}

For an object $\mathcal{C}=\left((C_{\uparrow *},\partial_{\uparrow}^C,\ell_{\uparrow}^C),(C^{\downarrow}_{*},\partial^{\downarrow}_C,\ell^{\downarrow}_C),(D_*,\partial_D),\psi_{\uparrow},\psi^{\downarrow}\right)$ of $\mathsf{CO}^{\Lambda}(\kappa)$, and for $(s,t),(s',t')\in\mathbb{M}$ with $(s,t)\preceq (s',t')$, we have objects $\mathcal{F}^{0}_{(s,t)}(\mathcal{C}),\mathcal{F}^{0}_{(s',t')}$ of the dg-category $\mathsf{Ch}(\kappa)$; these are thus also objects of the homotopy category $\mathsf{K}(\kappa)$.  We then define the morphism $\Phi_{(s,t)}^{(s',t')}(\mathcal{C})\in \mathsf{Hom}_{\mathsf{K}(\kappa)}( \mathcal{F}^{0}_{(s,t)}(\mathcal{C}),\mathcal{F}^{0}_{(s',t')}(\mathcal{C}))$ as follows:

\begin{enumerate}
	\item[\namedlabel{(a)}{(a)}] If $(s,t)$ and $(s',t')$ either both lie in $T^kS$, or both lie in $T^kL$, or both lie in $T^{k}A$, so that $\mathcal{F}^{0}_{(s,t)}(\mathcal{C})$ is a subcomplex of $\mathcal{F}^{0}_{(s',t')}(\mathcal{C})$, then  $\Phi_{(s,t)}^{(s',t')}(\mathcal{C})$  is the homotopy class of the inclusion.
	\item[\namedlabel{(b)}{(b)}] If $(s,t)\in T^kS$ and $(s',t')$ in $T^kL$, then  $\Phi_{(s,t)}^{(s',t')}(\mathcal{C})$  is the homotopy class of the restriction to $\mathcal{F}^{0}_{(s,t)}(\mathcal{C})$ of  $\phi_{T^kS}^{T^kL}(\mathcal{C})$.
	\item[\namedlabel{(c)}{(c)}] If $(s,t)\in T^kS$ and $(s',t')$ in $T^kA$, then  $\Phi_{(s,t)}^{(s',t')}(\mathcal{C})$  is the homotopy class of the restriction to $\mathcal{F}^{0}_{(s,t)}(\mathcal{C})$ of  $\phi_{T^kS}^{T^kA}(\mathcal{C})$.
	\item[\namedlabel{(d)}{(d)}] If $(s,t)\in T^kL$ and $(s',t')$ in $T^{k+1}S$, then  $\Phi_{(s,t)}^{(s',t')}(\mathcal{C})$  is the homotopy class of the restriction to $\mathcal{F}^{0}_{(s,t)}(\mathcal{C})$ of  $\phi_{T^kL}^{T^{k+1}S}(\mathcal{C})$.
	\item[\namedlabel{(e)}{(e)}] If $(s,t)\in T^kA$ and $(s',t')$ in $T^{k+1}S$, then  $\Phi_{(s,t)}^{(s',t')}(\mathcal{C})$  is the homotopy class of the restriction to $\mathcal{F}^{0}_{(s,t)}(\mathcal{C})$ of  $\phi_{T^kA}^{T^{k+1}S}(\mathcal{C})$.
	\item[\namedlabel{(f)}{(f)}] If $(s,t)\in T^kL$ and $(s',t')\in T^{k+1}L$, then   $\Phi_{(s,t)}^{(s',t')}(\mathcal{C})$  is the homotopy class of the restriction to $\mathcal{F}^{0}_{(s,t)}(\mathcal{C})$ of $\phi_{T^{k+1}S}^{T^{k+1}L}(\mathcal{C})\circ \phi_{T^kL}^{T^{k+1}S}(\mathcal{C})=\left(\begin{array}{cc}0 & 0 \\ 1 & 0\end{array}\right)$.
	\item[\namedlabel{(g)}{(g)}] If $(s,t)\in T^kA$ and $(s',t')\in T^{k+1}A$, then   $\Phi_{(s,t)}^{(s',t')}(\mathcal{C})$  is the homotopy class of the restriction to $\mathcal{F}^{0}_{(s,t)}(\mathcal{C})$ of $\phi_{T^{k+1}S}^{T^{k+1}A}(\mathcal{C})\circ \phi_{T^kA}^{T^{k+1}S}(\mathcal{C})=\left(\begin{array}{cc}0 & 0 \\ -1 & 0\end{array}\right)$.
	\item[\namedlabel{(h)}{(h)}] If $(s,t)\in T^kS$ and $(s',t')\in T^{k+1}S$, then $\Phi_{(s,t)}^{(s',t')}(\mathcal{C})$  is the homotopy class of the restriction to $\mathcal{F}^{0}_{(s,t)}(\mathcal{C})$ of $\phi_{T^kL}^{T^{k+1}S}(\mathcal{C})\circ \phi_{T^kS}^{T^{k}L}(\mathcal{C})=\left(\begin{array}{ccc}0 & 0 & 0 \\ 0 & 0 & 0 \\ 0 & \psi_{\uparrow} & 0\end{array}\right)$, or equivalently of  $\phi_{T^kA}^{T^{k+1}S}(\mathcal{C})\circ \phi_{T^kS}^{T^{k}A}(\mathcal{C})=\left(\begin{array}{ccc}0 & 0 & 0 \\ 0 & 0 & 0 \\ \psi^{\downarrow} & 0  & 0\end{array}\right)$.
	\item[\namedlabel{(i)}{(i)}] In all other cases, $\Phi_{(s,t)}^{(s',t')}(\mathcal{C})=0$.
\end{enumerate}

\begin{prop}\label{mainnat}
	The morphisms  $\Phi_{(s,t)}^{(s',t')}(\mathcal{C})$ satisfy the following properties:
	\begin{itemize}
		\item[(I)] If $(s,t)\preceq (s',t')$ with either $s'\leq -2\Lambda-t$ or $t'\geq 2\Lambda-s$, then $\Phi_{(s,t)}^{(s',t')}(\mathcal{C})=0$.
		\item[(II)] If $(s,t)\preceq (s',t')\preceq (s'',t'')$ then $\Phi_{(s',t')}^{(s'',t'')}(\mathcal{C})\circ \Phi_{(s,t)}^{(s',t')}(\mathcal{C})=\Phi_{(s,t)}^{(s'',t'')}(\mathcal{C})$.
		\item[(III)] If $\mathcal{X}$ is another object of $\mathsf{HCO}^{\Lambda}(\kappa)$, if $\mathfrak{a}\in \mathrm{Hom}_{\mathsf{CO}^{\Lambda}(\kappa)}(\mathcal{C},\mathcal{X})$ has $\delta_{\mathcal{C},\mathcal{X}}\mathfrak{a}=0$, inducing $\underline{\mathfrak{a}}\in\mathrm{Hom}_{\mathsf{HCO}^{\Lambda}(\kappa)}(\mathcal{C},\mathcal{X})$, then $\Phi_{(s,t)}^{(s',t')}(\mathcal{X})\circ \underline{\mathcal{F}}_{(s,t)}^{0}(\underline{\mathfrak{a}})= \underline{\mathcal{F}}_{(s',t')}^{0}(\underline{\mathfrak{a}})\circ \Phi_{(s,t)}^{(s',t')}(\mathcal{C})$.
	\end{itemize}
\end{prop}

As at the end of Section \ref{f0sec}, $\underline{\mathcal{F}}_{(s,t)}^{0}$ denotes the functor on homotopy categories induced by $\mathcal{F}_{(s,t)}^{0}$.

This proposition allows us to make the following definition:

\begin{dfn}\label{fdef}
	For an object $\mathcal{C}$ of $\mathsf{HCO}^{\Lambda}(\kappa)$, we let $\mathcal{F}(\mathcal{C})$ denote the functor $\mathbb{M}\to \mathsf{K}(\kappa)$ which sends an object $(s,t)\in\mathbb{M}$ to $\mathcal{F}^{0}_{(s,t)}(\mathcal{C})$, and, for $(s,t)\preceq (s',t')$, the unique morphism $(s,t)\to (s',t')$ to $\Phi_{(s,t)}^{(s',t')}(\mathcal{C})$.  When the object $\mathcal{C}$ is understood we may just write $\Phi_{(s,t)}^{(s',t')}$.
	
	For a morphism $\underline{\mathfrak{a}}\in \mathrm{Hom}_{\mathsf{HCO}^{\Lambda}(\kappa)}(\mathcal{C},\mathcal{X})$, we let $\mathcal{F}(\underline{\mathfrak{a}})$ be the natural transformation of functors $\mathcal{F}(\mathcal{C})\to\mathcal{F}(\mathcal{X})$ given by, for each $(s,t)\in\mathbb{M}$, $\mathcal{F}(\underline{\mathfrak{a}})(s,t)=\underline{\mathcal{F}}_{(s,t)}^{0}(\underline{\mathfrak{a}})\in \mathrm{Hom}_{\mathsf{K}(\kappa)}(\mathcal{F}^{0}_{(s,t)}(\mathcal{C}),\mathcal{F}^{0}_{(s,t)}(\mathcal{C}))$.
\end{dfn}

Indeed, Proposition \ref{mainnat}(II) implies that $\mathcal{F}(\mathcal{C})$ is  a functor from $\mathbb{M}$ to $\mathsf{K}(\kappa)$, and Proposition \ref{mainnat}(III) implies that $\mathcal{F}(\underline{\mathfrak{a}})$ is a natural transformation of functors.

\begin{proof}[Proof of Proposition \ref{mainnat}]
	Property (III) follows directly from the definitions together with (\ref{natkk}),(\ref{lsnat}), and (\ref{asnat}), noting that the maps $\left(\begin{array}{cc} 0 & 0\\ 0 & 0 \\ 0 & K_{\uparrow}\end{array}\right)$ and 
$\left(\begin{array}{cc} 0 & 0\\ 0 & 0 \\ 0 & K^{\downarrow}\end{array}\right)$ that appear in the definitions of $\mathcal{K}_{\uparrow}$ and $\mathcal{K}^{\downarrow}$ each have image contained in the summand $Y[-k]$ that is common to all of the $\mathcal{S}^{k+1}_{x,y}(\mathcal{X})$ for $(x,y)\in S$.

To prove (I), let $(s,t)=T^k(x,y)$ where $(x,y)\in S\cup L\cup A$, and let $(s',t')=T^k(w,z)$.\footnote{We write $(w,z)$ instead of $(x',y')$ because this point might not be in $S\cup L\cup A$.}  The condition that $(s,t)\preceq (s',t')$ with either $s'\leq -2\Lambda-t$ or $t'\geq 2\Lambda-s$ is equivalent to the condition that $(x,y)\preceq (w,z)$ with either $w\leq -2\Lambda-y$ or $z\geq 2\Lambda-x$.  Let us assume that $z\geq 2\Lambda-x$; the case that $w\leq -2\Lambda-y$ is similar. (One considers $(-2\Lambda-y,y)$ instead of $(x,2\Lambda-x)$.)

If $(x,y)\in S\cup A$, then $(x,2\Lambda-x)\in A$, and $\mathcal{F}^{0}_{T^k(x,2\Lambda-x)}(\mathcal{C})=\mathrm{Cone}(C^{\downarrow}_{x}\to C^{\downarrow}_{x})[-k]$ is (as the  cone of an identity map) isomorphic to zero in $\mathsf{K}(\kappa)$.  Since $ (x,y)\preceq (w,z)$ with $z\geq 2\Lambda-x$ we have $(x,2\Lambda-x)\preceq (w,z)$, and hence $(w,z)\in A\cup\bigcup_{j\geq 1}T^j(S\cup L\cup A)$.  If either $(x,y)\in S$ and $(w,z)\notin A\cup T(S)$, or if $(x,y)\in S$ and $(w,z)\notin A\cup T(S)\cup T(A)$, then the map $\Phi_{(s,t)}^{(s',t')}(\mathcal{C})=\Phi_{T^k(x,y)}^{T^k(w,z)}(\mathcal{C})$ falls into the final case \ref{(i)} of our prescription and so is zero by definition.  In each of the remaining cases (where $(x,y)\in S$ and $(w,z)\in A\cup T(S)$, or $(x,y)\in A$ and $(w,z)\in A\cup T(S)\cup T(A)$), since $(x,y)\preceq (x,2\Lambda-x)\preceq (w,z)$ with $(x,2\Lambda-x)\in A$, one sees from some combination of cases \ref{(a)}, \ref{(c)}, \ref{(e)}, \ref{(g)}, and \ref{(h)} that we have a factorization $\Phi_{(s,t)}^{(s',t')}(\mathcal{C})=\Phi_{T^k(x,2\Lambda-x)}^{(s',t')}(\mathcal{C})\circ \Phi_{(s,t)}^{T^k(x,2\Lambda-x)}(\mathcal{C})$.  This composition vanishes since it factors through the zero object $\mathcal{F}^{0}_{T^k(x,2\Lambda-x)}(\mathcal{C})$.  This completes the proof of (I) in case $(x,y)\in S\cup A$ and $z\geq 2\Lambda-x$.
	
If instead $(x,y)\in L$ with $z\geq 2\Lambda-x$, then we have $(x,y)\preceq (x,2\Lambda-x)\preceq (w,z)$ with $(x,2\Lambda-x)\in T(L)$.  Hence $(w,z)\in T(L)\cup T(S)\cup\bigcup_{j\geq 2}T^j(S\cup L\cup A)$. If $(w,z)\notin T(L)$ then $\Phi_{(s,t)}^{(s',t')}(\mathcal{C})$ falls into case \ref{(i)} and so vanishes by definition.  If $(w,z)\in T(L)$, so that $(s,t)\in T^k(L)$, $T^k(x,2\Lambda-x)\in T^{k+1}(L)$, and $T(s',t')\in T^{k+1}(L)$, then by cases \ref{(a)} and \ref{(f)} of our definition the map $\Phi_{(s,t)}^{(s',t')}(\mathcal{C})$ factors through the zero object $\mathcal{F}^{0}_{T^k(x,2\Lambda-x)}(\mathcal{C})$ and so is zero.  We have now proven statement (I) of the proposition in all cases where $z\geq 2\Lambda-x$. As noted earlier, the alternative case in which $w\leq -2\Lambda-y$ follows by a very similar analysis which we leave to the reader.

Having dispensed with (I), we turn to (II).  Let $(s,t),(s',t'),(s'',t'')\in \mathbb{M}$ with $(s,t)\preceq (s',t')\preceq (s'',t'')$.  As we have already used in the proof of (I), there are many cases in which the relation $\Phi_{(s',t')}^{(s'',t'')}(\mathcal{C}) \circ
  \Phi_{(s,t)}^{(s',t')}(\mathcal{C})=\Phi_{(s,t)}^{(s'',t'')}(\mathcal{C})$ is immediate from the construction.  Let us say that a pair $\left((s_0,t_0),(s_1,t_1)\right)\in\mathbb{M}\times\mathbb{M}$ with $(s_0,t_0)\preceq (s_1,t_1)$ is \textbf{proximal} if the location of $(s_0,t_0)$ and $(s_1,t_1)$ with respect to the decomposition $\mathbb{M}=\cup_k(T^kS\cup T^kL\cup T^kA)$ are such that the construction of $\Phi_{(s_0,t_0)}^{(s_1,t_1)}(\mathcal{C})$ is dictated by one of the cases (a)-(h). If $\left((s,t),(s'',t'')\right)$ is proximal in this sense, the so too are both pairs $\left((s,t),(s',t')\right)$ and $\left((s',t'),(s'',t'')\right)$ and $(s'',t'')$, and inspection of the appropriate formulas from cases \ref{(a)}-\ref{(h)} shows that the identity $\Phi_{(s',t')}^{(s'',t'')}(\mathcal{C}) \circ
  \Phi_{(s,t)}^{(s',t')}(\mathcal{C})=\Phi_{(s,t)}^{(s'',t'')}(\mathcal{C})$holds in this case.  
  
  Let us say that the pair $\left((s_0,t_0),(s_1,t_1)\right)\in \mathbb{M}\times\mathbb{M}$ is \textbf{strongly proximal} if $(s_1,t_1)\in (-2\Lambda-t_0,s_0]\times [t_0,2\Lambda-s_0)$. It is easy to see that a strongly proximal pair is also proximal, and that if $\left((s,t),(s'',t'')\right)$ is strongly proximal and if $(s,t)\preceq (s',t')\preceq (s'',t'')$ then both  $\left((s,t),(s',t')\right)$ and $\left((s',t'),(s'',t'')\right)$ are strongly proximal.  Now part (I) of the proposition asserts that if $\left((s_0,t_0),(s_1,t_1)\right)$ is \emph{not} strongly proximal then $\Phi_{(s_0,t_0)}^{(s_1,t_1)}(\mathcal{C})=0$.  Hence if either $\left((s,t),(s',t')\right)$ or  $\left((s',t'),(s'',t'')\right)$ is not strongly proximal then the identity $\Phi_{(s',t')}^{(s'',t'')}(\mathcal{C}) \circ
  \Phi_{(s,t)}^{(s',t')}(\mathcal{C})=\Phi_{(s,t)}^{(s'',t'')}(\mathcal{C})$
holds because both sides are zero.

The only remaining cases are those in which both $\left((s,t),(s',t')\right)$ and $\left((s',t),(s'',t'')\right)$ are strongly proximal but $\left((s,t),(s'',t'')\right)$ is not proximal (and hence also not strongly proximal).  In this case $\Phi_{(s,t)}^{(s'',t'')}(\mathcal{C})=0$ by defnition, and so we need to check that $\Phi_{(s',t')}^{(s'',t'')}(\mathcal{C}) \circ
\Phi_{(s,t)}^{(s',t')}(\mathcal{C})=0$. Now since $((s,t),(s'',t''))$ is not strongly proximal, we have either $s''\leq -2\Lambda-t$ or $t''\geq 2\Lambda-s$.  Let us suppose that $s''\leq -2\Lambda-t$.  Since $((s,t),(s',t'))$ is strongly proximal, $s'>-2\Lambda-t$ and so we get a chain $(s,t)\preceq(s',t')\preceq (-2\Lambda-t,t')\preceq (s'',t'')$.  Now the pair $((s,t),(-2\Lambda-t,t'))$ is proximal\footnote{Indeed, pairs $((s,t),T(s,t))$ are always proximal, and $(s,t)\preceq (-2\Lambda-t,t')\preceq T(s,t)$.} but not strongly proximal; by what we have already shown, the fact that it is proximal implies that $\Phi_{(s',t')}^{(-2\Lambda-t,t')}(\mathcal{C})\circ \Phi_{(s,t)}^{(s',t')}(\mathcal{C})=\Phi_{(s,t)}^{(-2\Lambda-t,t')}(\mathcal{C})$, and the fact that it is not strongly proximal implies that $\Phi_{(s,t)}^{(-2\Lambda-t,t')}(\mathcal{C})=0$.  Meanwhile, we also have $\Phi_{(s',t')}^{(s'',t'')}(\mathcal{C})=\Phi_{(-2\Lambda-t,t')}^{(s'',t'')}(\mathcal{C})\circ \Phi_{(s',t')}^{(-2\Lambda-t,t')}(\mathcal{C})$.  Hence \begin{align*}
	\Phi_{(s',t')}^{(s'',t'')}(\mathcal{C}) &\circ
	\Phi_{(s,t)}^{(s',t')}(\mathcal{C})= \left(\Phi_{(-2\Lambda-t,t')}^{(s'',t'')}(\mathcal{C})\circ \Phi_{(s',t')}^{(-2\Lambda-t,t')}(\mathcal{C})\right)\circ \Phi_{(s,t)}^{(s',t')}(\mathcal{C}) \\ & = \Phi_{(-2\Lambda-t,t')}^{(s'',t'')}(\mathcal{C})\circ \left(\Phi_{(s',t')}^{(-2\Lambda-t,t')}(\mathcal{C})\circ \Phi_{(s,t)}^{(s',t')}(\mathcal{C}) \right)=0.
\end{align*}
In the final case that $t''\geq 2\Lambda-s$, we similarly deduce that $\Phi_{(s',t')}^{(s'',t'')}(\mathcal{C}) \circ
\Phi_{(s,t)}^{(s',t')}(\mathcal{C})=0$ by using the chain $(s,t)\preceq (s',t')\preceq (s',2\Lambda-s)\preceq (s'',t'')$.  Thus in all cases where both $\left((s,t),(s',t')\right)$ and $\left((s',t),(s'',t'')\right)$ are strongly proximal but $\left((s,t),(s'',t'')\right)$ is not proximal we have $\Phi_{(s',t')}^{(s'',t'')}(\mathcal{C}) \circ
\Phi_{(s,t)}^{(s',t')}(\mathcal{C})=\Phi_{(s,t)}^{(s'',t'')}(\mathcal{C})=0$, completing the proof of (II).

\end{proof}

\subsection{Distinguished triangles}\label{dist}

For chain complexes $(C,\partial_C),(D,\partial_D)$ and a chain map $\alpha\co C\to D$, we have chain maps $j\co D\to \mathrm{Cone}(\alpha)$ and $p\co \mathrm{Cone}(\alpha)\to C[-1]$ given respectively by inclusion and projection; these can be placed into a sequence \begin{equation}\label{coneseq} \xymatrix{ C\ar[r]^{\alpha} & D \ar[r]^<<<<{j} & \mathrm{Cone}(\alpha)\ar[r]^{p} & C[-1].    } \end{equation}  

Recall that a sequence of morphisms in the homotopy category $\mathsf{K}(\kappa)$ of the form \[ \xymatrix{ X\ar[r]^{u} & Y \ar[r]^v & Z \ar[r]^{w} & X[-1]   } \] is said to be a \emph{distinguished triangle} if and only if it is isomorphic, as a diagram in $\mathsf{K}(\kappa)$, to one of the form (\ref{coneseq}).  As is well-known (see, \emph{e.g.} \cite[Proposition 10.2.4]{Wei}), $\mathsf{K}(\kappa)$ satisfies the axioms of a triangulated category using this notion of distinguished triangles.

\begin{theorem}\label{distmain}
	Let $\mathcal{C}$ be an object of $\mathsf{HCO}^{\Lambda}(\kappa)$, let $(s,t),(u,v)\in\mathbb{M}$ with $(s,t)\preceq (u,v)\preceq T(s,t)$ and let $k$ be the integer such that $(s,t)\in T^k(S\cup L\cup A)$.  Then the following sequence is a distinguished triangle in $\mathsf{K}(\kappa)$:
	\begin{equation}\label{distmaineq}
	\xymatrix{ \mathcal{F}^0_{(s,t)}(\mathcal{C})\ar[rr]^-{(-1)^k\left(\Phi_{(s,t)}^{(u,t)}\oplus \Phi_{(s,t)}^{(s,v)}\right)} & & \mathcal{F}^0_{(u,t)}(\mathcal{C})\oplus \mathcal{F}^0_{(s,v)}(\mathcal{C})   \ar[rr]^-{-\Phi_{(u,t)}^{(u,v)}+\Phi_{(s,v)}^{(u,v)}} & & \mathcal{F}^0_{(u,v)}(\mathcal{C})\ar[r]^{\Phi_{(u,v)}^{T(s,t)}} & \mathcal{F}^0_{T(s,t)}(\mathcal{C})
	}
	\end{equation}
\end{theorem}
\begin{remark}\label{signdist}
	As explanation for the sign $(-1)^k$ on the first map, note that, in general, the map $\Phi_{T^k(s_0,t_0)}^{T^k(s_1,t_1)}\co \mathcal{F}^{0}_{T^k(s_0,t_0)}(\mathcal{C})\to \mathcal{F}^{0}_{T^k(s_1,t_1)}(\mathcal{C})$ is formed by applying the functor $[-k]$ to the map $\Phi_{(s_0,t_0)}^{(s_1,t_1)}\co \mathcal{F}^{0}_{(s_0,t_0)}(\mathcal{C})\to \mathcal{F}^{0}_{(s_1,t_1)}(\mathcal{C}) $.  In general, applying the $k$-fold iterate of the shift functor to a distinguished triangle in a triangulated category results in a distinguished triangle only after multiplying (any) one of the maps by $(-1)^k$.  (On the level of chain complexes this can be understood in terms of the fact that $\mathrm{Cone}(f)[-k]$ is the cone of (the shift of) the map $(-1)^kf$.)  This remark also shows that it suffices to prove Theorem \ref{distmain} in the case $k=0$, \emph{i.e.} the case that $(s,t)\in S\cup L\cup A$.
\end{remark}

\begin{proof}
	As noted in Remark \ref{signdist}, it suffices to prove the result when $k=0$. We construct  chain maps $f,g$ in the diagram \begin{equation}\label{diagiso}
\xymatrix{	\mathcal{F}^0_{(s,t)}(\mathcal{C})\ar[r] \ar@{=}[d] & \mathcal{F}^0_{(u,t)}(\mathcal{C})\oplus \mathcal{F}^0_{(s,v)}(\mathcal{C})   \ar[r] \ar@{=}[d] & \mathrm{Cone}(\Phi_{(s,t)}^{(u,t)}(\mathcal{C})\oplus \Phi_{(s,t)}^{(s,v)}(\mathcal{C})) \ar[r] \ar@<-1ex>[d]_{g}  & \mathcal{F}^0_{T(s,t)}(\mathcal{C}) \ar@{=}[d]
 \\ \mathcal{F}^0_{(s,t)}(\mathcal{C})\ar[r]  & \mathcal{F}^0_{(u,t)}(\mathcal{C})\oplus \mathcal{F}^0_{(s,v)}(\mathcal{C})   \ar[r] & \mathcal{F}^0_{(u,v)}(\mathcal{C})\ar[r] \ar@<-1ex>[u]_{f}  & \mathcal{F}^0_{T(s,t)}(\mathcal{C})  
 }\end{equation}
which are homotopy inverses to each other and cause the diagram to commute up to homotopy.  Here the horizontal maps in the first row are those of (\ref{coneseq}) (with $\alpha=\Phi_{(s,t)}^{(u,t)}(\mathcal{C})\oplus \Phi_{(s,t)}^{(s,v)}(\mathcal{C})$), noting that $\mathcal{F}^0_{T(s,t)}(\mathcal{C})=\mathcal{F}^0_{(s,t)}(\mathcal{C})[-1]$, and those in the second row are as in (\ref{distmaineq}).  There are in all ten cases to check (see Figure \ref{rects}), and in Table \ref{bigtable} we specify the maps $f$ and $g$ in each case, along with a degree-$1$ map $h\co \mathrm{Cone}(\Phi_{(s,t)}^{(u,t)}(\mathcal{C})\oplus \Phi_{(s,t)}^{(s,v)}(\mathcal{C}))\to \mathrm{Cone}(\Phi_{(s,t)}^{(u,t)}(\mathcal{C})\oplus \Phi_{(s,t)}^{(s,v)}(\mathcal{C})) $ for which $\partial h+h\partial=1-fg$.  The other composition $gf$ is in each case equal to the identity (without even passing to homotopy), and it is routine to check that in each case the diagram commutes in $\mathsf{K}(\kappa)$.  

\begin{center}
	\begin{figure}
		\includegraphics[width=6in]{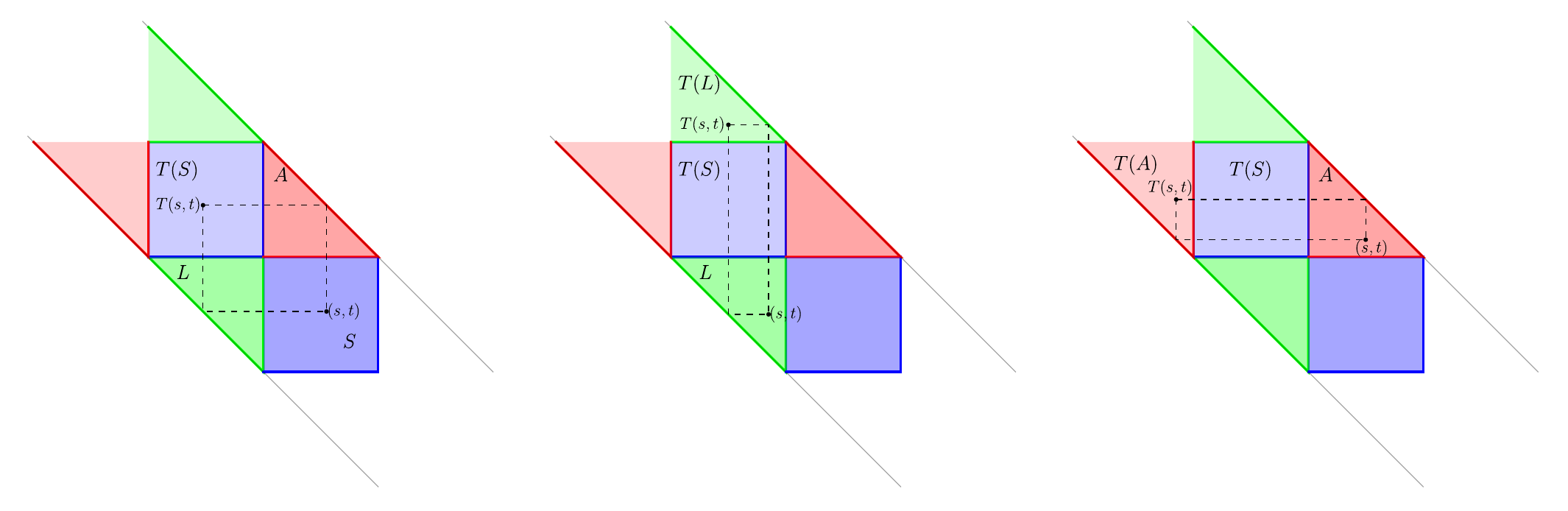}
		\caption{\label{rects} If $(s,t)\in S\cup L\cup A$, the condition that $(s,t)\preceq (u,v)\preceq T(s,t)$, is equivalent to $(u,v)$ lying in the appropriate dashed rectangle shown above.  This leads to the ten cases listed in Table \ref{bigtable}.}
	\end{figure}
\end{center}

To illustrate what calculations are involved in these routine checks, and perhaps also to clarify the meanings of the block matrices in Table \ref{bigtable}, let us give more extensive details for a representative choice of one of the ten cases, namely the sixth one in Table \ref{bigtable}, where $(s,t)\in L$ and $(u,v)\in T(S)$, still with $(s,t)\preceq (u,v)\preceq T(s,t)$.  Then $(u,t)\in L$ and $(s,v)\in T(S)$, 
so we have \begin{align*}
	\mathcal{F}_{(s,t)}^{0}(\mathcal{C}) &=\mathcal{L}_{s,t}^{0}(\mathcal{C})=\mathrm{Cone}(C_{\uparrow}^{\leq -2\Lambda-s}\hookrightarrow C_{\uparrow}^{\leq t}), \\ \mathcal{F}_{(u,t)}^{0}(\mathcal{C})&=\mathrm{Cone}(C_{\uparrow}^{\leq -2\Lambda-u}\hookrightarrow C_{\uparrow}^{\leq t}), \\ \mathcal{F}_{(s,v)}^{0}(\mathcal{C})&=\mathcal{S}^{1}_{2\Lambda-v,-2\Lambda-s}(\mathcal{C})=\mathrm{Cone}(-\psi^{\downarrow}+\psi_{\uparrow}\co C^{\downarrow}_{\geq 2\Lambda-v}\oplus C_{\uparrow}^{\leq -2\Lambda-s}\to D), \\ 
	\mathcal{F}_{(u,v)}^{0}(\mathcal{C})&=\mathrm{Cone}(-\psi^{\downarrow}+\psi_{\uparrow}\co C^{\downarrow}_{\geq 2\Lambda-v}\oplus C_{\uparrow}^{\leq -2\Lambda-u}\to D).	
\end{align*}

The map $\Phi_{(s,t)}^{(u,t)}\oplus \Phi_{(s,t)}^{(s,v)}$ that appears as the first map of both rows in (\ref{diagiso}) is then the map \[ C_{\uparrow}^{\leq -2\Lambda-s}[-1]\oplus C_{\uparrow}^{\leq t}\to C_{\uparrow}^{\leq -2\Lambda-u}[-1]\oplus C_{\uparrow}^{\leq t}\oplus C^{\downarrow}_{\geq 2\Lambda-v}[-1]\oplus C_{\uparrow}^{\leq -2\Lambda-s}[-1]\oplus D \] given in block form by the matrix (with each ``$1$'' denoting an identity or an inclusion map) \[ \left( \begin{array}{cc} 1 & 0 \\ 0 & 1 \\ 0 & 0 \\ 1 & 0 \\ 0 & \psi_{\uparrow} \end{array}\right). \]  The differential on $\mathrm{Cone}(\Phi_{(s,t)}^{(u,t)}(\mathcal{C})\oplus \Phi_{(s,t)}^{(s,v)}(\mathcal{C}))$ is given with respect to the splitting $
C_{\uparrow}^{\leq -2\Lambda-s}[-2]\oplus C_{\uparrow}^{\leq t}[-1]\oplus C_{\uparrow}^{\leq -2\Lambda-u}[-1]\oplus C_{\uparrow}^{\leq t}\oplus C^{\downarrow}_{\geq 2\Lambda-v}[-1]\oplus C_{\uparrow}^{\leq -2\Lambda-s}[-1]\oplus D$ by \[ \left(\begin{array}{ccccccc}\partial^{C}_{\uparrow} & 0 & 0 & 0 & 0 & 0 & 0 \\ -1 & -\partial_{\uparrow}^{C} & 0 & 0 & 0 & 0 & 0 \\  1 & 0 & -\partial_{\uparrow}^{C} & 0 & 0 & 0 & 0 \\ 0 & 1 & 1 & \partial_{\uparrow}^{C} & 0 & 0 & 0 \\ 0  & 0 & 0 & 0 & -\partial^{\downarrow}_{C} & 0 & 0 \\ 1 & 0 & 0 & 0 & 0 & -\partial_{\uparrow}^{C} & 0 \\ 0 & \psi_{\uparrow} & 0 & 0 & -\psi^{\downarrow} & \psi_{\uparrow} & \partial_D   \end{array}\right),  \] while that on $\mathcal{F}_{(u,v)}^{0}(\mathcal{C})$ is given with respect to the splitting $C^{\downarrow}[-1]\oplus C_{\uparrow}[-1]\oplus D$ by $\left(\begin{array}{ccc} -\partial^{\downarrow}_{C} & 0 & 0 \\ 0 & \partial_{\uparrow}^{C} & 0 \\ -\psi^{\downarrow} & -\psi_{\uparrow} & \partial_D \end{array}\right)$.  Multiplying matrices, one checks that the maps $f$ and $g$ from the sixth row of Table \ref{bigtable} intertwine these two differentials, that $gf$ is the identity on $\mathcal{F}_{(u,v)}(\mathcal{C})$, and that the map $h$ from the sixth row of Table \ref{bigtable} gives a homotopy between $fg$ and the identity. Note also that this map $h$ is well-defined as a degree-one map on $\mathrm{Cone}(\Phi_{(s,t)}^{(u,t)}(\mathcal{C})\oplus \Phi_{(s,t)}^{(s,v)}(\mathcal{C}))$, since its only nonzero components are the identity maps $C_{\uparrow}^{\leq -2\Lambda-s}[-1]\to C_{\uparrow}^{\leq -2\Lambda-s}[-2]$ and $C_{\uparrow}^{\leq t}\to C_{\uparrow}^{\leq t}[-1]$ (viewed as maps of degree one).

It remains to check that that (\ref{diagiso}) commutes in $\mathsf{K}(\kappa)$.  That the left square in (\ref{diagiso}) commutes is a tautology, as the first maps in the two rows are the same.  There are two middle squares and two right squares in (\ref{diagiso}), depending upon whether one uses the map $f$ or $g$.  However, since we have already seen that $f$ and $g$ are inverses to each other in $\mathsf{K}(\kappa)$, if one of the two middle squares commutes in $\mathsf{K}(\kappa)$ then the other will as well, and likewise for the two right squares.   

The map $\mathcal{F}^{0}_{(u,t)}(\mathcal{C})\oplus \mathcal{F}^{0}_{(s,v)}(\mathcal{C})\to \mathcal{F}^{0}_{(u,v)}(\mathcal{C})$ in the bottom row of (\ref{diagiso}) is given as a map $C_{\uparrow}^{\leq -2\Lambda-u}[-1]\oplus C_{\uparrow}^{\leq t}\oplus C^{\downarrow}_{\geq 2\Lambda-v}[-1]\oplus C_{\uparrow}^{\leq -2\Lambda-s}[-1]\oplus D\to C^{\downarrow}_{\geq 2\Lambda-v}[-1]\oplus C_{\uparrow}^{\leq -2\Lambda-u}[-1]\oplus D$ by the block matrix $\left(\begin{array}{ccccc} 0 & 0 & 1 & 0 & 0 \\ -1 & 0 & 0 & 0 & 0 \\ 0 & -\psi_{\uparrow} & 0 & 0 & 1\end{array}\right)$.  Note that this coincides with the last five columns of the matrix representing $g$.  Since the second map in the top row of (\ref{diagiso}) is given by the inclusion associated to the mapping cone, it follows that the version of the middle square of (\ref{diagiso}) involving $g$ commutes (both in $\mathsf{Ch}(\kappa)$ and in $\mathsf{K}(\kappa)$).  Hence the version of the middle square of (\ref{diagiso}) involving $f$ also commutes in $\mathsf{K}(\kappa)$.  

The map $\mathcal{F}^{0}_{(u,v)}(\mathcal{C})\to \mathcal{F}^{0}_{T(s,t)}(\mathcal{C})$ in the bottom row of (\ref{diagiso}) is, according to Case (b) at the start of Section \ref{fsect}, given by restricting the map $\phi_{TS}^{TL}(\mathcal{C})$ to $\mathcal{S}^1_{2\Lambda-v,-2\Lambda-u}(\mathcal{C})$; thus it is given in block form by $\left(\begin{array}{ccc} 0 & 0 & 0 \\ 0 & 1 & 0\end{array}\right)$ from $C^{\downarrow}_{\geq 2\Lambda-v}[-1]\oplus C_{\uparrow}^{\leq -2\Lambda-u}[-1]\oplus D\to C_{\uparrow}^{\leq -2\Lambda-s}[-2]\oplus C_{\uparrow}^{\leq t}[-1]$.  Since the last map in the top row of (\ref{diagiso}) is the projection associated to the mapping cone, and since the above $2\times 3$ matrix coincides with the first two rows of the matrix representing $f$ in Table \ref{bigtable}, it follows that the version of the right square in (\ref{diagiso}) that involves $f$ commutes; hence the version involving $g$ commutes in $\mathsf{K}(\kappa)$ since $f$ and $g$ are inverses in $\mathsf{K}(\kappa)$.  

This concludes the verification that (\ref{diagiso}) is an isomorphism of diagrams in the case that $(s,t)\in L$ and $(u,v)\in T(S)$.  The other nine cases can be verified by arguments with the same outline as the one just provided; in each of these cases the main subtlety is to find the appropriate maps $f,g,h$, and these are supplied in Table \ref{bigtable}.

\begin{table}
	\caption{ \label{bigtable} The ten cases for the maps 
		$f$ and $g$ and the homotopy $h$ between $fg$ and the identity in (\ref{diagiso})}
\begin{tabular}{c|c|c|c|c} 
	$(s,t)$ & $(u,v)$ & $f$ & $g$ & $h$  \\ \hline 
	$L$ & $L$ & $\left(\begin{array}{cc} 0 & 0 \\ 1 & 0 \\ -1 & 0 \\ 0 & 0 
	\\ 0 & 0 \\ 0 & 1\end{array}\right)$ & $\left(\begin{array}{cccccc} 0 & 0 & -1 & 0 & 1 & 0 \\ 0 & 0 & 0 & -1 & 0 & 1\end{array}\right)$ & $\left(\begin{array}{cccccc} 0 & 0 & 0 & 0 & 1 & 0 \\ 0 & 0 & 0 & 1 & 0 & 0 \\ 0 & 0 & 0 & 0 & 0 & 0 \\ 0 & 0 & 0 & 0 & 0 & 0\\ 0 & 0 & 0 & 0 & 0 & 0\\ 0 & 0 & 0 & 0 & 0 & 0\end{array}\right)$ 
	\\ \hline $A$ & $A$ & $\left(\begin{array}{cc} 0 & 0 \\ -1 & 0 \\ 0 & 0 \\ 0 & -1 \\ 1 & 0 \\ 0 & 0 \end{array}\right)$  & $\left(\begin{array}{cccccc} 0 & 0 & -1 & 0 & 1 & 0 \\ 0 & 0 & 0 & -1 & 0 & 1  \end{array}\right)$ & $\left(\begin{array}{cccccc} 0 & 0 & 1 & 0 & 0 & 0 \\ 0 & 0 & 0 & 0 & 0 & 1 \\ 0 & 0 & 0 & 0 & 0 & 0  \\ 0 & 0 & 0 & 0 & 0 & 0 \\ 0 & 0 & 0 & 0 & 0 & 0 \\ 0 & 0 & 0 & 0 & 0 & 0   \end{array}\right)$ \\ \hline 
	$S$ & $S$ & $\left(\begin{array}{ccc} 0 & 0 & 0 \\ 0 & 0 & 0 \\ \psi^{\downarrow} & 0 & 0 \\ -1 & 0 & 0 \\ 0 & 0 & 0 \\ 0 & 0 & 0 \\ 0 & 0 & 0 \\ 0 & 1 & 0 \\ 0 & 0 & 1 \end{array}\right)$ & $\left(\begin{array}{ccccccccc} 0 & 0 & 0 & -1 & 0 & 0 & 1 & 0 & 0 \\ 0 & 0 & 0 & 0 &-1 & 0 & 0 & 1 & 0 \\ 0 & 0 & 0 & 0 & 0 & -1 & 0 & 0 & 1 \end{array}\right)$  & $\left(\begin{array}{ccccccccc} 0 & 0 & 0 & 0 & 0 & 0 & 1 & 0 & 0 \\ 0 & 0 & 0 & 0 & 1 & 0 & 0 & 0 & 0 \\ 0 & 0 & 0 & 0 & 0 & 1 & 0 & 0 & 0  \\ 0 & 0 & 0 & 0 & 0 & 0 & 0 & 0 & 0  \\ 0 & 0 & 0 & 0 & 0 & 0 & 0 & 0 & 0 \\ 0 & 0 & 0 & 0 & 0 & 0 & 0 & 0 & 0 \\ 0 & 0 & 0 & 0 & 0 & 0 & 0 & 0 & 0 \\ 0 & 0 & 0 & 0 & 0 & 0 & 0 & 0 & 0 \\ 0 & 0 & 0 & 0 & 0 & 0 & 0 & 0 & 0 \end{array}\right)$  \\ \hline
	$S$ & $L$ & $\left(\begin{array}{cc} 0 & 0 \\ 1 & 0 \\ 0 & \psi_{\uparrow} \\ -1 & 0 \\ 0 & 0 \\ 0 & 0 \\ 0 & 1 \\ 0 & 0\end{array}\right)$ & $\left(\begin{array}{cccccccc} 0 & 0 & 0 & -1 & 0 & 0 & 0 & 0 \\ 0 & 0 & 0 & 0 & -1 & 0 & 1 & 0\end{array}\right)$ & $\left( \begin{array}{cccccccc} 0 & 0 & 0 & 0 & 0 & 1 & 0 & 0 \\ 0 & 0 & 0 & 0 & 1 & 0 & 0 & 0 \\ 0 & 0 & 0 & 0 & 0 & 0 & 0 & 1 \\ 0 & 0 & 0 & 0 & 0 & 0 & 0 & 0 \\ 0 & 0 & 0 & 0 & 0 & 0 & 0 & 0 \\ 0 & 0 & 0 & 0 & 0 & 0 & 0 & 0 \\ 0 & 0 & 0 & 0 & 0 & 0 & 0 & 0 \\ 0 & 0 & 0 & 0 & 0 & 0 & 0 & 0 \end{array}\right)$  \\ \hline
	$S$ & $A$  &  $\left(\begin{array}{cc} -1 & 0 \\ 0 & 0 \\ 0 & \psi^{\downarrow} \\ 0 & -1 \\ 0 & 0 \\ 0 & 0 \\ 1 & 0 \\ 0 & 0\end{array}\right)$ &  $\left(\begin{array}{cccccccc} 0 & 0 & 0 & 0 & 0 & 0 & 1 & 0 \\ 0 & 0 & 0 & -1 & 0 & 0 & 0 & 1\end{array}\right)$ & $\left( \begin{array}{cccccccc} 0 & 0 & 0 & 0 & 0 & 0 & 0 & 1 \\ 0 & 0 & 0 & 0 & 1 & 0 & 0 & 0 \\ 0 & 0 & 0 & 0 & 0 & 1 & 0 & 0 \\ 0 & 0 & 0 & 0 & 0 & 0 & 0 & 0 \\ 0 & 0 & 0 & 0 & 0 & 0 & 0 & 0 \\ 0 & 0 & 0 & 0 & 0 & 0 & 0 & 0 \\ 0 & 0 & 0 & 0 & 0 & 0 & 0 & 0 \\ 0 & 0 & 0 & 0 & 0 & 0 & 0 & 0 \end{array}\right)$ \\ \hline
	$L$ & $T(S)$ & $\left(\begin{array}{ccc} 0 & 0 & 0 \\ 0 & 1 & 0 \\ 0 & -1 & 0 \\ 0 & 0 & 0 \\ 1 & 0 & 0 \\ 0 & 0 & 0 \\ 0 & 0 & 1\end{array}\right)$  & $\left(\begin{array}{ccccccc} 0 & 0 & 0 & 0 & 1 & 0 & 0 \\ 0 & 0 & -1 & 0 & 0 & 1 & 0 \\ 0 & 0 & 0 & -\psi_{\uparrow} & 0 & 0 & 1\end{array}\right)$ & $\left(\begin{array}{ccccccc} 0 & 0 & 0 & 0 & 0 & 1 & 0 \\ 0 & 0 & 0 & 1 & 0 & 0 & 0 \\ 0 & 0 & 0 & 0 & 0 & 0 & 0  \\ 0 & 0 & 0 & 0 & 0 & 0 & 0 \\ 0 & 0 & 0 & 0 & 0 & 0 & 0 \\ 0 & 0 & 0 & 0 & 0 & 0 & 0 \\ 0 & 0 & 0 & 0 & 0 & 0 & 0 \end{array}\right)$
	\\ 
	\end{tabular}\end{table}\begin{table}\begin{tabular}{c|c|c|c|c} 
	$(s,t)$ & $(u,v)$ & $f$ & $g$ & $h$\\  \hline $A$ & $T(S)$ & $\left(\begin{array}{ccc} 0 & 0 & 0 \\ 1 & 0 & 0 \\ 0 & 0 & 0 \\ 0 & -1 & 0 \\ 0 & 0 & -1 \\ -1 & 0 & 0 \\ 0 & 0 & 0\end{array}\right)$ & $\left(\begin{array}{ccccccc} 0 & 0 & -1 & 0 & 0 & -1 & 0 \\ 0 & 0 & 0 & -1 & 0 & 0 & 0 \\ 0 & 0 & 0 & 0 & -1 & 0 & \psi^{\downarrow}\end{array}\right)$ & $\left(\begin{array}{ccccccc} 0 & 0 & -1 & 0 & 0 & 0 & 0 \\ 0 & 0 & 0 & 0 & 0 & 0 & 1 \\ 0 & 0 & 0 & 0 & 0 & 0 & 0 \\ 0 & 0 & 0 & 0 & 0 & 0 & 0 \\ 0 & 0 & 0 & 0 & 0 & 0 & 0 \\ 0 & 0 & 0 & 0 & 0 & 0 & 0 \\ 0 & 0 & 0 & 0 & 0 & 0 & 0\end{array}\right)$ \\ \hline 
	$S$ & $T(S)$ & $\left(\begin{array}{ccc} 1 & 0 & 0 \\ 0 & 1 & 0 \\ 0 & 0 & 1 \\ 0 & -1 & 0 \\ 0 & 0 & 0 \\ -1 & 0 & 0 \\ 0 & 0 & 0 \end{array}\right)$ & $\left(\begin{array}{ccccccc} 0 & 0 & 0 & 0 & 0 & -1 & 0 \\ 0 & 0 & 0 & -1 & 0 & 0 & 0 \\ 0 & 0 & 1 & 0 & -\psi^{\downarrow} &  0 & \psi_{\uparrow}\end{array}\right)$ & $\left(\begin{array}{ccccccc} 0 & 0 & 0 & 0 & 0 & 0 & 1 \\ 0 & 0 & 0 & 0 & 1 & 0 & 0 \\ 0 & 0 & 0 & 0 & 0 & 0 & 0 \\ 0 & 0 & 0 & 0 & 0 & 0 & 0 \\ 0 & 0 & 0 & 0 & 0 & 0 & 0 \\ 0 & 0 & 0 & 0 & 0 & 0 & 0 \\ 0 & 0 & 0 & 0 & 0 & 0 & 0 \end{array}\right)$	\\ \hline $L$ & $T(L)$ & $\left(\begin{array}{cc} 1 & 0 \\ 0 & 1 \\ 0 & -1 \\ 0 & 0 \\ 1 & 0 \\ 0 & 0 \end{array}\right)$ & $\left(\begin{array}{cccccc} 0 & 0 & 0 & 0 & 1 & 0 \\ 0 & 0 & -1 & 0 & 0 & 1\end{array}\right)$ & $\left(\begin{array}{cccccc} 0 & 0 & 0 & 0 & 0 & 1 \\ 0 & 0 & 0 & 1 & 0 & 0 \\ 0 & 0 & 0 & 0 & 0 & 0 \\ 0 & 0 & 0 & 0 & 0 & 0 \\ 0 & 0 & 0 & 0 & 0 & 0 \\ 0 & 0 & 0 & 0 & 0 & 0\end{array}\right)$ \\ \hline $A$ & $T(A)$ & $\left(\begin{array}{cc} 1 & 0 \\ 0 & 1 \\ -1 & 0 \\ 0 & 0 \\ 0 & -1 \\ 0 & 0 \end{array}\right)$ & $\left(\begin{array}{cccccc} 0 & 0 & -1 & 0 & 0 & 0 \\ 0 & 0 & 0 & -1 & -1 & 0 \end{array}\right)$ & $\left(\begin{array}{cccccc} 0 & 0 & 0 & -1 & 0 & 0 \\ 0 & 0 & 0 & 0 & 0 & 1 \\ 0 & 0 & 0 & 0 & 0 & 0 \\ 0 & 0 & 0 & 0 & 0 & 0 \\ 0 & 0 & 0 & 0 & 0 & 0 \\ 0 & 0 & 0 & 0 & 0 & 0\end{array}\right)$	
\end{tabular}
\end{table}
\end{proof}

Given a functor $\mathcal{G}\co \mathbb{M} \to \mathsf{K}(\kappa)$, one obtains functors $H_0(\mathcal{G})\co \mathbb{M}\to \mathsf{Vect}_{\kappa}$ and $H^0(\mathcal{G}): \mathbb{M}^{op} \to \mathsf{Vect}_{\kappa}$,

 respectively, by composing with the zeroth homology and cohomology functors $\mathsf{K}(\kappa)\to\mathsf{Vect}_{\kappa}$.

\begin{cor}\label{cohfunct} For any object $\mathcal{C}$ of $\mathsf{HCO}^{\Lambda}(\kappa)$, the objects $H_0(\mathcal{F}(\mathcal{C}))$ of $\mathsf{Vect}_{\kappa}^{\mathbb{M}}$ and $H^0(\mathcal{F}(\mathcal{C}))$ of $\mathsf{Vect}_{\kappa}^{\mathbb{M}^{op}}$ are, respectively, homological and cohomological functors in the sense of \cite[Section 3]{BBF21}; that is, for any $(s,t),(u,v)\in \mathbb{M}$ with $(s,t)\preceq (u,v)\preceq T(s,t)$ we have exact sequences  
	\[ \xymatrix{ & \qquad\qquad\qquad\qquad\qquad \cdots\ar[r] &   H_0(\mathcal{F}(\mathcal{C}))(T^{-1}(u,v))\ar[dll]  \\ H_0(\mathcal{F}(\mathcal{C}))(s,t)\ar[r]  & H_0(\mathcal{F}(\mathcal{C}))(u,t)\oplus H_0(\mathcal{F}(\mathcal{C}))(s,v) \ar[r] &  H_0(\mathcal{F}(\mathcal{C}))(u,v) \ar[dll]\\ H_0(\mathcal{F}(\mathcal{C}))(T(s,t))\ar[r] & \cdots\qquad\qquad\qquad\qquad\qquad &  }\]
	and 	\[ \xymatrix{ & \qquad\qquad\qquad\qquad\qquad \cdots\ar[r] &   H^0(\mathcal{F}(\mathcal{C}))(T(s,t))\ar[dll]  \\ H^0(\mathcal{F}(\mathcal{C}))(u,v)\ar[r]  & H^0(\mathcal{F}(\mathcal{C}))(u,t)\oplus H^0(\mathcal{F}(\mathcal{C}))(s,v) \ar[r] &  H^0(\mathcal{F}(\mathcal{C}))(s,t) \ar[dll]\\ H^0(\mathcal{F}(\mathcal{C}))(T(u,v))\ar[r] & \cdots,\qquad\qquad\qquad\qquad\qquad &  }\] the maps being induced on (co)homology by (\ref{distmaineq}).  Moreover, both functors $H_0(\mathcal{F}(\mathcal{C}))$ and $H^0(\mathcal{F}(\mathcal{C}))$ restrict to the boundary $\partial\mathbb{M}=\{(s,t)\in \R^2||s+t|=2\Lambda\}$ as zero.
\end{cor}

\begin{proof}
	If $(s,t)\in \partial\mathbb{M}$, then if $(x,y)\in\mathfrak{D}$ is chosen such that $T^k(x,y)=(s,t)$ for some $k\in \Z$, we will have either $(x,y)\in L$ with $y=-2\Lambda-x$, or $(x,y)\in A$ with $x=2\Lambda-y$.  In either case, the object $\mathcal{F}(C)(s,t)$ of $\mathsf{K}(\kappa)$ is (up to grading shift) the cone of an identity map, so it is isomorphic to zero in $\mathsf{K}(\kappa)$ and in particular its zeroth homology and cohomology are both zero.
	
	The fact that the sequences in the corollary are exact follows immediately from Theorem \ref{signdist} and the general fact (see, \emph{e.g.}, \cite[Proposition IV.1.3]{GM}, applied with $U$ equal to a complex given by $\kappa$ in degree zero and $0$ in all other degrees) that applying $H_0$ or $H^0$ to a distinguished triangle in $\mathsf{K}(\kappa)$ yields an exact sequence.	
\end{proof}

\subsection{The functor and the standard elementary summands} \label{elemfun}

In Section \ref{cospan-decomp}, we have shown that every filtered cospan over a field $\kappa$ that satisfies an admissibility hypothesis is isomorphic in $\mathsf{HCO}^{\Lambda}(\kappa)$ to a direct sum of standard elementary summands, enumerated in (\ref{basicacyc}) and (\ref{sixblocks}).  If $\mathcal{C}$ is $\Lambda$-bounded, then the parameters $a$ and $b$ arising in (\ref{basicacyc}) and (\ref{sixblocks}) will all lie in the open interval $(-\Lambda,\Lambda)$.  Correspondingly, \cite[Theorem 3.21]{BBF21} asserts that any homological functor $\mathbb{M}\to \mathsf{Vect}_{\kappa}$ that vanishes on $\partial\mathbb{M}$ and satisfies a ``sequential continuity'' hypothesis decomposes as a direct sum of ``blocks'' $B^v$.  Here, if $v=(v_1,v_2)$, the functor $B^v\co \mathbb{M}\to\mathsf{Vect}_{\kappa}$ is given by, for $(w_1,w_2)\in\mathbb{M}$, \begin{equation}\label{bvdef}
B^v(w_1,w_2)=\left\{\begin{array}{ll}\kappa & \mbox{if }-2\Lambda-v_2<w_1\leq v_1\mbox{ and }v_2\leq w_2<2\Lambda-v_1 \\ 0 & \mbox{otherwise}\end{array}  \right. 
\end{equation} with structure maps $B^v(w_1,w_2)\to B^v(w'_1,w'_2)$ given by the identity when both domain and codomain are equal to $\kappa$ and by $0$ otherwise.

Accordingly, we make the following definition:

\begin{dfn}\label{persrish}
	Suppose that $F\in \mathsf{Vect}_{\kappa}^{\mathbb{M}}$ is a homological functor that vanishes on $\partial\mathbb{M}$ and is sequentially continuous.  The \textbf{persistence diagram} of $F$, denoted $\mathcal{D}(F)$, is the multiset of elements of $\mathbb{M}$ with the property that $F\cong \oplus_{v\in\mathcal{D}(F)}B^v$.
\end{dfn}

It is not hard to check that if $\mathcal{C}$ is a tame filtered cospan then $H_0(\mathcal{F}(\mathcal{C}))$ is sequentially continuous in the sense of \cite{BBF21}.  Indeed, if $(s_n,t_n)$ is a decreasing sequence in $\mathbb{M}$ that converges to $(s,t)$, then the definition of $\mathcal{F}$ and the admissibility of $\mathcal{C}$ imply that the complexes $\mathcal{F}(\mathcal{C})(s_n,t_n)$ and $\mathcal{F}(\mathcal{C})(s,t)$ are eventually equal to each other.  Moreover, by Corollary \ref{cohfunct}, $H_0(\mathcal{F}(\mathcal{C}))$ is homological and vanishes on $\partial\mathbb{M}$.  So on general grounds, $H_0(\mathcal{F}(\mathcal{C}))$ decomposes as a direct sum of blocks $B^v$. 

We may therefore likewise define the persistence diagram of an admissible filtered cospan:

\begin{dfn}
	Let $\mathcal{C}$ be a $\Lambda$-bounded admissible filtered cospan.  The \textbf{persistence diagram} of $\mathcal{C}$, denoted $\mathcal{D}(\mathcal{C})$, is the persistence diagram (as defined in Definition \ref{persrish}) of $H_0(\mathcal{F}(\mathcal{C}))$.	
\end{dfn}

To see how to determine the persistence diagram of an admissible $\Lambda$-bounded admissible filtered cospan from a direct sum decomposition as in Section \ref{cospan-decomp}, we should work out the effect of $H_0\circ\mathcal{F}$ on the standard elementary summands  from (\ref{basicacyc}) and (\ref{sixblocks}).  (Note that the functor $H_0\circ \mathcal{F}$ is clearly additive, so a decomposition in $\mathsf{HCO}^{\Lambda}(\kappa)$ carries over to one in $\mathsf{Vect}_{\kappa}^{\mathbb{M}}$.)  

\begin{prop}\label{blockclass}
If the $\Lambda$-bounded filtered cospan $\mathcal{C}$ is one of the indecomposables from (\ref{basicacyc}) or (\ref{sixblocks}), then $H_0(\mathcal{F}(\mathcal{C}))$ is isomorphic to the block $B^{v(\mathcal{C})}$, where $v(\mathcal{C})\in \mathbb{M}$ is given by the following table:
\begin{center}
\begin{tabular}{c|c} $\mathcal{C}$ & $v(\mathcal{C})$  \\ 	\hline $\left(\uparrow_{a}^{b}\right)_k$ (for $a<b<\Lambda$) & $T^{-k-1}(-2\Lambda-a,b)$  \\ \hline $\left(\uparrow_{a}^{\infty}\right)_k$ & $T^{-k}(\Lambda,a)$ \\ 
\hline $\left(\nearrow_{a}\right)_k$ & $T^{-k}(-\Lambda,a)$ \\ \hline $\left(\downarrow^{a}_{b}\right)_{k}$ (for $a>b>-\Lambda$) & $T^{-k-1}(b,2\Lambda-a)$ \\ \hline 
 $\left(\downarrow^{a}_{-\infty}\right)_{k}$ & $T^{-k}(a,-\Lambda)$ \\ \hline $\left(\searrow^{a}\right)_k$ & $T^{-k}(a,\Lambda)$ \\ \hline $\left(>_{b}^{a}\right)_k$ & $T^{-k}(a,b)$ \\ \hline $\square_k$ & $T^{-k+1}(\Lambda,-\Lambda)$ 
\end{tabular}
\end{center}
\end{prop}

\begin{proof}
	These follow by routine computation; we will discuss the cases of $\left(\uparrow_{a}^{b}\right)_k$, $(\searrow^a)_k$, $\left(>_{b}^{a}\right)_k$, and $\square_k$ as representative examples and leave the others to the reader.
	
\paragraph{\textbf{Case} $\mathcal{C}=\square_k$}	We address this case first since it is the simplest.  Clearly $\mathcal{F}(\square_k)(T^{j}(x,y))=0$ whenever $(x,y)\in L\cup A$ and $j\in \Z$, while if $(x,y)\in S$ and $j\in \Z$ we have \[ \mathcal{F}(\square_k)(T^{j}(x,y))=\mathrm{Cone}(0\to\kappa_{k})[-j+1]=(\kappa_k)[-j+1],\] so for $(x,y)\in S$, \[ H_0(\mathcal{F}(\square_k)(T^{j}(x,y)))=H_0((\kappa_k)[-j+1])=H_{-j+1}(\kappa_k)=\left\{\begin{array}{ll}\kappa &  \mbox{if }j=-k+1 \\ 0 & \mbox{otherwise} \end{array}\right.\]  Now $B^{T^{-k+1}(\Lambda,-\Lambda)}$ is likewise given by $\kappa$ on $T^{-k+1}(S)$ and $0$ elsewhere.  So since the structure maps for $\mathcal{F}(\square_k)$ within $T^{-k+1}(S)$ are  given by inclusion of subcomplexes (which in this case is just the identity on $\mathrm{Cone}(0\to \kappa_{k})[k]$), we indeed have $H_0(\mathcal{F}(\mathcal{C}))=B^{T^{-k+1}(\Lambda,-\Lambda)}$ as functors $\mathbb{M}\to \mathsf{Vect}_{\kappa}$.

\paragraph{\textbf{Case} $\mathcal{C}=\left(\uparrow_{a}^{b}\right)_k$ with $a<b<\Lambda$} In this case, the complexes $C^{\downarrow}_{*}$ and $D_*$ are both zero while $C_{\uparrow*}=\mathcal{E}_k(a,b)_{\uparrow}$.  As before, we will consider the values of $\mathcal{F}\left(\left(\uparrow_{a}^{b}\right)_k\right)$ on $T^{j}(x,y)$ for $j\in \Z$ and $(x,y)\in\mathfrak{D}=S\cup L\cup A$.  Clearly, if $(x,y)\in A$, then $\mathcal{F}\left(\left(\uparrow_{a}^{b}\right)_k\right)(T^j(x,y))=0$.  If $(x,y)\in S$, we have \[ \mathcal{F}\left(\left(\uparrow_{a}^{b}\right)_k\right)(T^j(x,y))=\mathrm{Cone}\left(\mathcal{E}_k(a,b)_{\uparrow}^{\leq y}\to 0\right)[-j+1]=\mathcal{E}_k(a,b)_{\uparrow}^{\leq y}[-j]. \] So for $(x,y)\in S$, \begin{equation}\label{upabS} H_0\left(\mathcal{F}\left(\left(\uparrow_{a}^{b}\right)_k\right)(T^j(x,y)) \right)=H_{-j}\left(\mathcal{E}_k(a,b)_{\uparrow}^{\leq y}\right)=\left\{\begin{array}{ll}\kappa & \mbox{if }j=-k\mbox{ and } a\leq y<b \\ 0 & \mbox{otherwise}\end{array}\right. \end{equation}
Lastly, if $(x,y)\in L$ (so in particular $-2\Lambda-x\leq y<\Lambda$) then \begin{equation}\label{upabcone} H_0\left(\mathcal{F}\left(\left(\uparrow_{a}^{b}\right)_k\right)(T^j(x,y))\right) = H_{-j}\left(\mathrm{Cone}\left(\mathcal{E}_k(a,b)_{\uparrow}^{\leq -2\Lambda-x}\hookrightarrow \mathcal{E}_k(a,b)_{\uparrow}^{\leq y} \right)\right). \end{equation}
If $a> -2\Lambda-x$ then $\mathcal{E}_{k}(a,b)_{\uparrow}^{\leq -2\Lambda-x}=0$, so (\ref{upabcone}) is $H_{-j}\left(\mathcal{E}_k(a,b)_{\uparrow}^{\leq y}\right)$, which is $\kappa$ if $j=-k$ and $a\leq y<b$, and $0$ otherwise.  

If $b\leq y$, then  $\mathcal{E}_{k}(a,b)_{\uparrow}^{\leq y}$ is acyclic, so (\ref{upabcone}) yields $H_{-j}(\mathcal{E}_{k}(a,b)_{\uparrow}^{\leq -2\Lambda-x}[-1])$, which is $\kappa$ if $j=-k-1$ and $a\leq -2\Lambda-x<b$, and $0$ otherwise.  

In the remaining case that $a\leq -2\Lambda-x$ and $b>y$, the inclusion $\mathcal{E}_k(a,b)_{\uparrow}^{\leq-2\Lambda-x}\hookrightarrow \mathcal{E}_{k}(a,b)_{\uparrow}^{\leq y}$ is an isomorphism, so (\ref{upabcone}) is zero.

Summing up, if $(x,y)\in L$ and $j\in \Z$, \begin{equation}\label{upabL} H_0\left(\mathcal{F}\left(\left(\uparrow_{a}^{b}\right)_k\right)(T^j(x,y)) \right)=\left\{\begin{array}{ll}\kappa & \mbox{if }j=-k,\, a\leq y<b,\mbox{ and }x>-2\Lambda-a \\ \kappa & \mbox{if } j=-k-1 ,\,-2\Lambda-b<x\leq -2\Lambda-a,\mbox{ and }y\geq b\\ 0 & \mbox{otherwise}\end{array}\right.   \end{equation}  Reparametrizing slightly, we obtain from (\ref{upabS}) and (\ref{upabL}) that, for $(s,t)\in\mathbb{M}$, \[
H_0\left(\mathcal{F}\left(\left(\uparrow_{a}^{b}\right)_k\right)(T^{-k-1}(s,t))\right)=\left\{\begin{array}{ll}\kappa & \mbox{if }(s,t)\in L,\,-2\Lambda-b<s\leq -2\Lambda-a,\mbox{ and }t\geq b  \\ \kappa & \mbox{if }(s,t)\in T(S)\mbox{ and }-2\Lambda-b<s\leq -2\Lambda-a \\ \kappa & \mbox{if }(s,t)\in T(L),\,-2\Lambda-b<s\leq -2\Lambda-a, \mbox{ and }t<4\Lambda+a \\ 0 & \mbox{otherwise}\end{array} \right. \]

It is easy to see that the above agrees with the values of the block $B^{T^{-k-1}(-2\Lambda-a,b)}$.  Moreover, inspection of the above computation together with cases (a),(b),(d), and (f) of the definition of the morphisms $\Phi_{(s,t)}^{(s',t')}(\mathcal{C})$ in Section \ref{fsect} shows that the induced maps
 $H_0\left(\mathcal{F}\left(\left(\uparrow_{a}^{b}\right)_k\right)(T^{-k-1}(s,t))\right)\to H_0\left(\mathcal{F}\left(\left(\uparrow_{a}^{b}\right)_k\right)(T^{-k-1}(s',t'))\right)$ 
 are isomorphisms whenever both the domain and the codomain are nonzero.  So
  $H_0\left(\mathcal{F}\left(\left(\uparrow_{a}^{b}\right)_k\right)\right)$ is indeed isomorphic to $B^{T^{-k-1}(-2\Lambda-a,b)}$.

\paragraph{\textbf{Case} $\mathcal{C}=(\searrow^a)_k$} In this case $C_{\uparrow *}=0$, so $\mathcal{F}((\searrow^a)_k)(T^{j}(x,y))=0$ if $(x,y)\in L$.  If $(x,y)\in A$, then \begin{align*} H_0\left(\mathcal{F}((\searrow^a)_k)(T^{j}(x,y)) \right)&= H_{-j}\left(\mathrm{Cone}(\mathcal{E}_k(a,-\infty)^{\downarrow}_{\geq 2\Lambda-y}\hookrightarrow \mathcal{E}_k(a,-\infty)^{\downarrow}_{\geq x})\right)\\&=\left\{\begin{array}{ll}\kappa & \mbox{if }j=-k\mbox{ and }2\Lambda-y>a\geq x \\ 0 & \mbox{otherwise} \end{array}\right. 
\end{align*}
If instead $(x,y)\in S$, then $ H_0\left(\mathcal{F}((\searrow^a)_k)(T^{j}(x,y)) \right)$ is the homology in degree $-j+1$ of the cone of a map $\psi_x\co \mathcal{E}_k(a,-\infty)^{\downarrow}_{\geq x}\to\kappa_k$.  The domain of $\psi_x$ is zero if $x>a$ (so in this case $H_0\left(\mathcal{F}((\searrow^a)_k)(T^{j}(x,y)) \right)$ agrees with $H_{-j+1}(\kappa_k)$), while if $x\leq a$ then 
$\mathcal{E}_k(a,-\infty)^{\downarrow}_{\geq x}=\kappa_k$ and $\psi_x$ is multiplication by $-1$, which has acyclic mapping cone. Using the reparametrization
 $(s,t)=T(x,y)=(-2\Lambda-y,2\Lambda-x)\in T(S)$, we see that if $(s,t)\in T(S)$ then 
 \[ H_0\left(\mathcal{F}((\searrow^a)_k)(T^{j}(s,t))\right)=\left\{ \begin{array}{ll} \kappa &\mbox{if }j=-k\mbox{ and }t<2\Lambda-a \\ 0 & \mbox{otherwise} \end{array}\right.  \]  Now the block $B^{T^{-k}(a,\Lambda)}$ is nontrivial precisely at those points of form $T^{-k}(s,t)$ where either $(s,t)\in A$ with $s\leq a$ and $t<2\Lambda-a$ or $(s,t)\in T(S)$ with $t<2\Lambda-a$.  So $B^{T^{-k}(a,\Lambda)}$ coincides pointwise with $H_0\left(\mathcal{F}((\searrow^a)_k)\right)$, and the definition of $\phi_{T^{-k}A}^{T^{-k+1}(S)}$ from Section \ref{transf} implies that the structure maps coincide.

\paragraph{\textbf{Case} $\mathcal{C}=(>_{b}^{a})_k$} Here $C_{\uparrow*},C^{\downarrow}_{*}$ and $D_*$ are zero in all degrees other than $k$, and are one-dimensional in degree $k$, with the generators for $C_{\uparrow*}$ and $C^{\downarrow}_{*}$ having filtration levels $b$ and $a$ respectively.  So if $(x,y)\in L$, the inclusion $C_{\uparrow}^{\leq -2\Lambda-x}\hookrightarrow C_{\uparrow}^{\leq y}$ is an isomorphism unless $-2\Lambda-x<b\leq y$, in which case it is the map from $0$ to $\kappa_k$; hence
 \[ \mbox{for }(x,y)\in L,\,H_0\left(\mathcal{F}((>_{b}^{a})_k)(T^{j}(x,y))\right) = \left\{\begin{array}{ll}\kappa & j=-k,\,x>-2\Lambda-b,\mbox{ and }y\geq b \\ 0 & \mbox{otherwise}\end{array}\right..\]  
 Similar consideration of the inclusions $C^{\downarrow}_{\geq 2\Lambda-y}\hookrightarrow C^{\downarrow}_{\geq x}$ shows that \[ \mbox{for }(x,y)\in A,\,H_0\left(\mathcal{F}((>_{b}^{a})_k)(T^{j}(x,y))\right) = \left\{\begin{array}{ll}\kappa & j=-k,\,x\leq a,\mbox{ and }y<2\Lambda-a \\ 0 & \mbox{otherwise}\end{array}\right..\] 

Now if $(x,y)\in S$, $\mathcal{F}((>_{b}^{a})_k)(T^{j}(x,y))$ is a shift of the cone of a map $\psi_{x,y}\co C^{\downarrow}_{\geq x}\oplus C_{\uparrow}^{\leq y}\to D_*$ which is an epimorphism with one-dimensional kernel in degree $k$ if $a\geq x$ and $b\leq y$, a monomorphism with one-dimensional cokernel in degree $k$ if $a<x$ and $b>y$, and an isomorphism otherwise.  So, for $(x,y)\in S$, \[ H_0\left(\mathcal{F}((>_{b}^{a})_k)(T^{j}(x,y))\right)=H_{-j+1}(\mathrm{Cone}(\psi_{x,y}))\cong \left\{\begin{array}{ll}\kappa & \mbox{if }j=-k,\,x\leq a,\mbox{ and }y\geq b \\ \kappa & \mbox{if }j=-k+1,\, x>a,\mbox{ and }y<b \\ 0 & \mbox{otherwise}  \end{array}\right.. \]  
Now since $a$ and $b$ are between $-\Lambda$ and $\Lambda$, a point $(s,t)\in\mathbb{M}$ lies in the rectangle $(-2\Lambda-b,a]\times [b,2\Lambda-a)$ iff either $(s,t)\in S$ with $s\leq a$ and $t\geq b$, or $(s,t)\in L$ with $s>-2\Lambda-b$ and $y\geq b$, or $(s,t)\in A$ with $s\leq a$ and $t<2\Lambda-a$, or $(s,t)=T(x,y)$ where $(x,y)\in S$ with $x>a$ and $y<b$.  So the above calculations imply that  $H_0\left(\mathcal{F}((>_{b}^{a})_k)\right)$ is nonzero at exactly the same points as the block $B^{T^{-k}(a,b)}$.  Moreover the structure maps defined in Section \ref{fsect} induce isomorphisms from $H_0\left(\mathcal{F}((>_{b}^{a})_k)(s,t)\right)$ to $H_0\left(\mathcal{F}((>_{b}^{a})_k)(s',t')\right)$ whenever both are nontrivial, completing the proof in this case.  

The remaining cases $(\uparrow_{a}^{\infty})_k$, $(\nearrow_a)_k$, $(\downarrow_{b}^{a})_k$, and  $(\downarrow^{a}_{-\infty})_k$, of Proposition \ref{blockclass} follow by the same methods as those above, so we omit the details.

\end{proof}

We see from Proposition \ref{blockclass} that $H_0\circ \mathcal{F}$ establishes a \emph{bijection} between our standard elementary summands from (\ref{basicacyc}) and (\ref{sixblocks}) and the blocks $B^v$ of \cite{BBF20} as $v$ varies through $\mathbb{M}\setminus\partial\mathbb{M}$; the inverse of this bijection is depicted in Figure \ref{blockfig}.

\begin{center}
	\begin{figure}
		\includegraphics[width=5in]{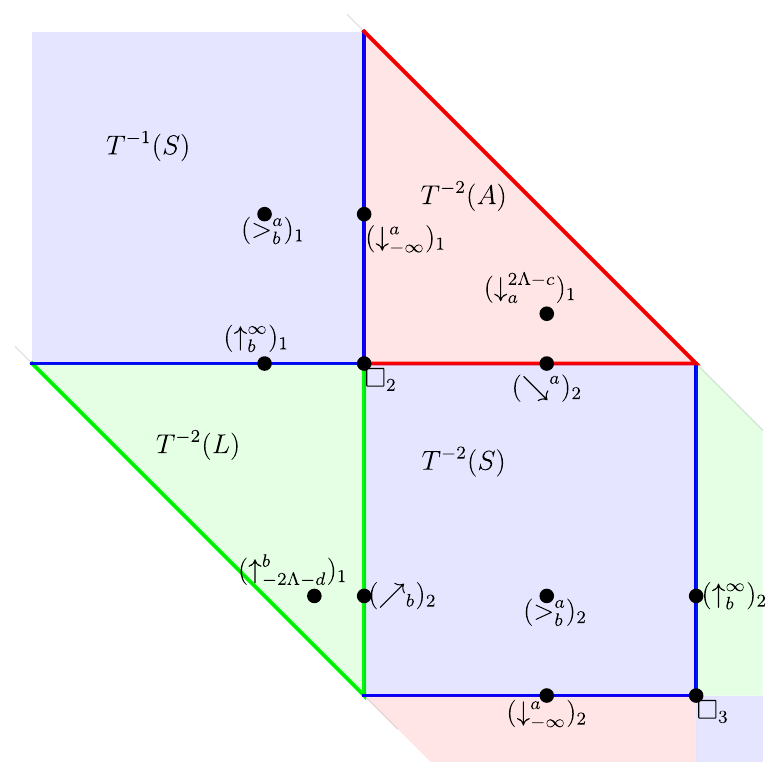}
		\caption{Various points $v$ of $\mathbb{M}$, labeled by the standard elementary summand $\mathcal{C}$ in $\mathsf{HCO}^{\Lambda}(\kappa)$ for which $H_0(\mathcal{F}(\mathcal{C}))$ is isomorphic to the block $B^v$ according to Proposition \ref{blockclass}.  The points labeled by $(>_{b}^{a})_2$, $(\uparrow_{-2\Lambda-d}^{b})_1$, and $(\downarrow_{a}^{2\Lambda-c})_1$ are equal respectively to $T^{-2}(a,b)$, $T^{-2}(d,b)$, and $T^{-2}(a,c)$, where $(a,b)\in S$, $(d,b)\in L$, and $(a,c)\in A$.}\label{blockfig}
	\end{figure}
\end{center}

This allows for a quick proof of the following uniqueness result for the decompositions of Section \ref{cospan-decomp}.

\begin{cor}\label{decompunique}
	Let $\mathcal{C}$ be a $\Lambda$-bounded filtered cospan that decomposes up to isomorphism in $\mathsf{HCO}^{\Lambda}(\kappa)$ as $\bigoplus (\mathcal{C}_{\alpha})^{\oplus n(\mathcal{C}_{\alpha})}$ where the $\mathcal{C}_{\alpha}$ are distinct blocks from among those in (\ref{basicacyc}) and (\ref{sixblocks}), and $n(\mathcal{C}_{\alpha})\in \N$.  Then the numbers $n(\mathcal{C}_{\alpha})$ are uniquely determined by $\mathcal{C}$.
\end{cor}

\begin{proof}
	By Proposition \ref{blockclass}, $H_0(\mathcal{F}(\mathcal{C}))\cong \bigoplus (B^{v(\mathcal{C}_{\alpha})})^{\oplus n(\mathcal{C}_{\alpha})}$ for certain $ v(\mathcal{C}_{\alpha})\in \mathbb{M}\setminus\partial\mathbb{M}$, 
	with the assignment $\mathcal{C}_{\alpha}\mapsto v(\mathcal{C}_{\alpha})$ being one-to-one, so that the various $v(\mathcal{C}_{\alpha})$ appearing in the sum are all distinct.  
	But, as observed in \cite[Section 2.2]{BBF20}, the multiplicities $n(\mathcal{C}_{\alpha})$ are then uniquely determined by  $H_0(\mathcal{F}(\mathcal{C}))$ (and hence by $\mathcal{C}$), as the dimension of the quotient of
	 $H_0(\mathcal{F}(\mathcal{C}))(v(\mathcal{C}_{\alpha}))$ by the sum of the images in $H_0(\mathcal{F}(\mathcal{C}))(v(\mathcal{C}_{\alpha}))$  of the various   $H_0(\mathcal{F}(\mathcal{C}))(v)$ for $v\preceq v(\mathcal{C}_{\alpha})$ under the appropriate structure maps.
\end{proof}
	
\section{The isomorphism with relative interlevel set homology}\label{rishsect}

Let us call a continuous function $f\co \mathbb{X}\to [-\Lambda,\Lambda]$ \emph{tame} if, for every $y\in (-\Lambda,\Lambda)$, there is a number $\ep=\ep(y)$ with $0<\ep<\Lambda-y$ and a strong deformation retraction from $f^{-1}([-\Lambda,y+\ep))$ to $f^{-1}([-\Lambda,y])$.  (For example, this holds if $f$ is a Morse function on a smooth manifold with boundary such that $\partial\mathbb{X}$ is the union of the regular level sets $f^{-1}(\{-\Lambda\})$ and $f^{-1}(\{\Lambda\})$ as in Example \ref{morseex}, or if $f$ is a simplexwise linear function on  a simplicial complex.)  Following \cite{BBF20}, to any tame function $f\co \mathbb{X}\to [-\Lambda,\Lambda]$ one may associate an object $h(f)$ of $\mathsf{Vect}_{\kappa}^{\mathbb{M}}$, called the \emph{relative interlevel set homology of $f$}.\footnote{As both the tameness of $f$ and the object $h(f)$ only depend on the restriction of $f$ to $f^{-1}((-\Lambda,\Lambda))$, one could instead notate this as $h(f|_{\mathbb{X}\setminus\partial\mathbb{X}})$ where $\partial\mathbb{X}:=f^{-1}(\{-\Lambda,\Lambda\})$. (The case $\partial\mathbb{X}=\varnothing$ is of course allowed, and is indeed the case that is considered in \cite{BBF20}.) To compare more precisely to \cite{BBF20}, their input is a function $g\co \mathbb{X}\to \R$, which is then composed it with the arctangent function to obtain a function with values in $(-\Lambda,\Lambda)$ with $\Lambda=\frac{\pi}{2}$; thus our $f$ should be identified with their $\arctan\circ g$, and what we denote by $h(f)$ would be $h(\tan\circ g)$ in their notation.}  At the same time,  we have a pinned singular filtered cospan $\mathcal{S}_{\partial}(\mathbb{X},f;\kappa)$ as in Example \ref{pinsing}. Applying the functor $\mathcal{F}$ to this object of $\mathsf{HCO}^{\Lambda}(\kappa)$ yields an object of $\mathsf{K}(\kappa)^{\mathbb{M}}$; composing with the zeroth homology functor $\mathsf{K}(\kappa)\to \mathsf{Vect}_{\kappa}$ then yields the object $H_0\circ \mathcal{F}(\mathcal{S}_{\partial}(\mathbb{X},f;\kappa))$.  After recalling the definition of $h(f)$, this section will be devoted to the proof of:

\begin{theorem} \label{H0iso}
	For any tame function $f\co \mathbb{X}\to [-\Lambda,\Lambda]$, there is an isomorphism $h\left(f\right)\cong H_0\circ \mathcal{F}(\mathcal{S}_{\partial}(\mathbb{X},f;\kappa))$.  
\end{theorem}	

This isomorphism is constructed fairly concretely below and depends naturally on the pair $(\mathbb{X},f)$ in a suitable sense which we leave to the reader to formulate precisely.
  
Before starting the proof of Theorem \ref{H0iso} we make precise the definition of $h(f)$ for a tame function $f\co \mathbb{X}\to [-\Lambda,\Lambda]$.  First, recalling that $S=(-\Lambda,\Lambda]\times[-\Lambda,\Lambda)$, let \[ S^+=\{(x,y)\in S|y\geq x\}\quad \mbox{and}\quad S^-=\{(x,y)\in (x,y)\in S|y<x\}. \]  Then, for $(x,y)\in \mathfrak{D}=S^+\sqcup S^-\sqcup L\sqcup A$ and $k\in \Z$, set: 
\[ 
h(f)_{T^{-k}(x,y)}=\left\{\begin{array}{ll} H_k(f^{-1}([x,y];\kappa)) & \mbox{if }(x,y)\in S^+ \\ H_{k+1}(\mathbb{X},\mathbb{X}\setminus f^{-1}((y,x));\kappa) & \mbox{if }(x,y)\in S^- \\ H_k(f^{-1}([-\Lambda,y]),f^{-1}([-\Lambda,-2\Lambda-x]);\kappa) & \mbox{if }(x,y)\in L \\ H_k(f^{-1}([x,\Lambda]),f^{-1}([2\Lambda-y,\Lambda]);\kappa) & \mbox{if }(x,y)\in A \end{array}\right.
\]

To describe the structure maps $h(f)_{(s,t)}\to h(f)_{(s',t')}$ of the $\mathbb{M}$-indexed persistence module $h(f)$ (for $(s,t)\preceq (s',t')$), it suffices to restrict to the following cases, as arbitrary structure maps are necessarily given by compositions of ones corresponding to one of these cases:
\begin{enumerate} \item \label{same} Both $(s,t)$ and $(s',t')$ belong to the same one of $T^{-k}(S^+)$, or $T^{-k}(S^-)$, or $T^{-k}(L)$, or $T^{-k}(A)$;
	\item \label{S+L} $(s,t)\in T^{-k}(S^+)$ and $(s',t')\in T^{-k}(L)$; \item \label{S+A} $(s,t)\in T^{-k}(S^+)$ and $(s',t')\in T^{-k}(A)$;
	\item \label{LS-} $(s,t)\in T^{-k}(L)$ and $(s',t')\in T^{-k+1}(S^-)$;
	\item \label{AS-} $(s,t)\in T^{-k}(A)$ and $(s',t')\in T^{-k+1}(S^-)$; 
	\item \label{S-S+} $(s,t)\in T^{-k}(S^-)$ and $(s',t')\in T^{-k}(S^+)$
\end{enumerate} 
 
 In each of cases (\ref{same})-(\ref{AS-}), the assumption that $(s,t)\preceq (s',t')$ implies that $h(f)_{(s,t)}$ and $h(f)_{(s',t')}$ can be expressed, respectively, as relative homologies $H_k(A,B;\kappa)$ and $H_k(A',B';\kappa)$ for certain subsets $A,B,A',B'$ of $\mathbb{X}$ with $A\subset A'$ and $B\subset B'$, so the map $h(f)_{(s,t)}\to h(f)_{(s',t')}$ is just taken to be the map on homology associated to the inclusion of pairs $(A,B)\subset (A',B')$.  In case (vi), writing $(s,t)=T^{-k}(x,y)$ and $(s',t')=T^{-k}(x',y')$ with $x\geq x'$, $y\leq y'$, $y<x$, and $y'\geq x'$, we have \[ h(f)_{(s,t)}=H_{k+1}(\mathbb{X},\mathbb{X}\setminus f^{-1}((y,x));\kappa),\qquad h(f)_{(s',t')}=H_k(f^{-1}([x',y']);\kappa).\]  The map $h(f)_{(s,t)}\to h(f)_{(s',t')}$ is taken in this case to be the connecting homomorphism in the Mayer-Vietoris sequence associated as in \cite[Section 10.7]{tD} to the pair of excisive triads $\left(\mathbb{X}\setminus f^{-1}((y,x));f^{-1}([-\Lambda,y]),f^{-1}([x,\Lambda])\right)$ and $\left(\mathbb{X};f^{-1}([-\Lambda,y']),f^{-1}([x',\Lambda])\right)$. (The tameness of $f$ ensures that the second triad is always excisive for singular homology, including in the case that $x'=y'$.)
 
This connecting homomorphism $\delta$ can be formulated as follows.  Continuing to assume that $x'\leq y'$, let $S_{*}^{sm,x'
,y'}(\mathbb{X})$ denote the subcomplex of $S_*(\mathbb{X})$ generated by singular simplices having image in either $f^{-1}([-\Lambda,y'+\ep(y')))$ or $f^{-1}([x',\Lambda])$. (Here the small positive parameter $\ep(y')$ is as in the definition of tameness.)  By the theorem of small chains (\cite[Proposition 2.21]{Ha}), the inclusion $S_{*}^{sm,x'
,y'}(\mathbb{X};\kappa)\to S_*(\mathbb{X};\kappa)$ is a chain homotopy equivalence.  We then have a commutative diagram \begin{equation} \xymatrix{  S_*(f^{-1}([x',y']);\kappa) \ar[r]^<<<<<<<<<{{\tiny \left[\begin{array}{c} 1 \\ 1\end{array}\right]}} \ar[d] & \frac{S_*(f^{-1}([x',\Lambda]);\kappa)}{S_*(f^{-1}([x,\Lambda]);\kappa)}\oplus \frac{S_*(f^{-1}([-\Lambda,y']);\kappa)}{S_*(f^{-1}([-\Lambda,y]);\kappa)} \ar[d] \ar[r]^<<<<<<{{\tiny [-1\,\,\,1]}} & \frac{S_*(\mathbb{X};\kappa)}{
S_*(\mathbb{X}\setminus f^{-1}((y,x));\kappa)} \\
 S_*(f^{-1}([x',y'+\ep(y'));\kappa) \ar[r]^<<<{{\tiny \left[\begin{array}{c} 1\\ 1\end{array}\right]} } & \frac{S_*(f^{-1}([x',\Lambda]);\kappa)}{S_*(f^{-1}([x,\Lambda]);\kappa)}\oplus \frac{S_*(f^{-1}([-\Lambda,y'+\ep(y')));\kappa)}{S_*(f^{-1}([-\Lambda,y]);\kappa)}  \ar[r]^<<<<<{{\tiny [-1 \,\,\, 1] }} & \frac{S_{*}^{sm,x',y'}(\mathbb{X};\kappa)}{S_*(\mathbb{X}\setminus f^{-1}((y,x));\kappa)}  \ar[u]   } \label{mvdiag}
\end{equation} where we use ``$1$'' to denote a map induced by inclusion. Here the vertical arrows are inclusions which are chain homotopy equivalences, and the bottom row is short exact.  The connecting homomorphism of the long exact sequence on homology associated to the bottom row then induces, via the above homotopy equivalences, the connecting homomorphism $\delta\co H_{k+1}(\mathbb{X},\mathbb{X}\setminus f^{-1}((y,x));\kappa)\to H_k(f^{-1}([x',y']);\kappa)$ which we use as a structure map for $h(f)$.
 
 The following algebraic lemma will help us relate $h(f)$ to the various complexes arising in $\mathcal{F}(\mathcal{S}_{\partial}(\mathbb{X},f;\kappa))$.

\begin{lemma}\label{coneids}
\begin{itemize} \item[(i)] Let $\phi\co C_*\to D_*$ be an injective chain map between two chain complexes.  Then the map $\alpha_{\phi}\co \mathrm{Cone}(\phi)\to \mathrm{coker}(\phi)$ defined by $\alpha_{\phi}(c,d)=\left(d \mod \Img(\phi)\right)$ is a quasi-isomorphism.
	\item[(ii)] Let $\phi'\co C'_{*}\to D_*$ be a surjective map between two chain complexes.  Then the map $\beta_{\phi'}\co \ker(\phi')\to \mathrm{Cone}(\phi')[1]$ defined by $\beta_{\phi'}(c')=(c',0)$ is a quasi-isomorphism.
\item[(iii)]  Suppose that $C_*$ is a subcomplex of $C'_*$, and that the surjective chain map $\phi'\co C'_*\to D_*$ restricts to an \emph{injective} chain map $\phi=\phi'|_{C_*}\co C_*\to D_*$, thus inducing a short exact sequence of chain complexes \begin{equation}\label{quotses}
\xymatrix{	0\ar[r] & \ker(\phi')\ar[r] & \frac{C'_*}{C_*}\ar[r]^<<<<{\phi'} & \mathrm{coker}(\phi)\ar[r]& 0. }\end{equation}
Then for each $j\in \Z$ there is a commutative diagram \[ \xymatrix{ H_{j}(\mathrm{Cone}(\phi)) \ar[r]\ar[d]_{\alpha_{\phi}} & H_{j}(\mathrm{Cone}(\phi'))=H_{j-1}(\mathrm{Cone}(\phi')[1]) \\ H_{j}(\mathrm{coker}(\phi)) \ar[r]^{\Delta} & H_{j-1}(\ker(\phi')) \ar[u]_{\beta_{\phi'}}   }\]  where $\Delta$ is the connecting homomorphism of the long exact sequence associated to (\ref{quotses}) and where the top arrow is induced by the inclusion.
\end{itemize}
\end{lemma}

\begin{proof}
	(i) and (ii) are well-known and can be proven by starting from the commutative diagrams \[ \xymatrix{ & 0 \ar[r] & D_* \ar[r] \ar@{=}[d] & \mathrm{Cone}(\phi) \ar[r] \ar[d]_{\alpha_{\phi}} & C_*[-1] \ar[r] & 0 \\ 0 \ar[r] & C_* \ar[r]^{\phi} & D_* \ar[r] & \mathrm{coker}(\phi) \ar[r] & 0 &  }  \]
	and \[  \xymatrix{ & 0 \ar[r] & \ker(\phi') \ar[r] \ar[d]_{\beta_{\phi'}} & C'_* \ar[r] \ar@{=}[d]&  D_* \ar[r] & 0 \\ 0 \ar[r] & D_*[1] \ar[r] & \mathrm{Cone}(\phi')[1] \ar[r] & C'_* \ar[r] & 0 &  }\] and applying the five-lemma to the resulting diagrams of long exact sequences (which commmute up to sign). 
	
	 For (iii), if $(c,d)$ is a cycle in $\mathrm{Cone}(\phi)_j$ (so necessarily $\partial c=0$ and $\partial d=-\phi(c)$), then the homology class of $(c,d)$ is mapped by $\beta_{\phi'}\circ \Delta\circ \alpha_{\phi}$ to the homology class of the element $(c+\partial c',0)\in Cone(\phi')_{j}$, for any choice of $c'\in C'_{j}$ such that $\phi'(c')=d$.  Since the boundary of $(c',0)$ in $Cone(\phi')$ is then $(-\partial c',d)$, the elements $(c,d)$ and $(c+\partial c',0)$ of $Cone(\phi')_{j}$ are homologous, so $\beta_{\phi'}\circ \Delta\circ \alpha_{\phi}\co H_j(\mathrm{Cone}(\phi))\to H_j(\mathrm{Cone}(\phi'))$ indeed coincides with the  map induced by  inclusion.
\end{proof}

It is essentially immediate from the constructions in Section \ref{f0sec} and \ref{fsect} that, with notation as in Example \ref{singular}, the appropriate quasi-isomorphisms $\alpha_{\phi}$ as in Lemma \ref{coneids}(i) induce isomorphisms $H_0(\mathcal{F}(\mathcal{S}_{\partial}(\mathbb{X},f;\kappa)(s,t)))\cong h(f)_{(s,t)}$ in the cases that $(s,t)=T^{-k}(x,y)$ for $(x,y)\in S^-\cup L\cup A$.  For example, if $(x,y)\in S^-$, then $ \mathcal{F}(\mathcal{S}_{\partial}(\mathbb{X},f;\kappa)(s,t))$ is \[ \mathrm{Cone}\left(\xymatrix{\frac{S_*(f^{-1}([x,\Lambda]);\kappa)}{S_*(\partial^+\mathbb{X};\kappa)}\oplus \frac{S_*(f^{-1}([-\Lambda,y]);\kappa)}{S_*(\partial^-\mathbb{X};\kappa)}\ar[r]^<<<<<{{\tiny [-1 \,\,\, 1] }} & \frac{S_{*}(\mathbb{X};\kappa)}{S_*(\partial\mathbb{X};\kappa)},   }\right)[k+1]\] which is the cone of an injection since $y<x$ and since $\partial\mathbb{X}$ is the disjoint union of $\partial^-\mathbb{X}$ and $\partial^+\mathbb{X}$. So Lemma \ref{coneids}(i) identifies $H_0\left(\mathcal{F}(\mathcal{S}(\mathbb{X},f;\kappa)(s,t))\right)$ with \begin{align*} H_0\left(\mathrm{coker}\left({\tiny [-1 \,\,\, 1] } \right)[k+1]\right)&=H_{k+1}\left(\frac{S_*(\mathbb{X};\kappa)/S_*(\partial\mathbb{X};\kappa)}{S_*(f^{-1}([-\Lambda,y]\cup[x,\Lambda]);\kappa)/S_*(\partial\mathbb{X};\kappa)}   \right)  \\ &\cong H_{k+1}\left(\frac{S_*(\mathbb{X};\kappa)}{S_*(f^{-1}([-\Lambda,y]\cup[x,\Lambda]);\kappa)}\right)=h(f)_{(s,t)}. \end{align*}

The case that $(x,y)\in S^+$ requires perhaps a bit more explanation, particularly when $y=x$ in which case the tameness hypothesis on $f$ is needed.  In the next proposition, note that if $(x,y)\in S^+$ then both $x$ and $y$ lie in the open interval $(-\Lambda,\Lambda)$ (since $x>-\Lambda$, $y<\Lambda$, and $y\geq x$), so $S_*(f^{-1}([x,y];\kappa))$ can be regarded as a subcomplex both of $\frac{S_*(f^{-1}([x,\Lambda]);\kappa)}{S_*(\partial^+\mathbb{X};\kappa)}$ and of   $\frac{S_*(f^{-1}([-\Lambda,y]);\kappa)}{S_*(\partial^-\mathbb{X};\kappa)}$.

\begin{prop}\label{htopyker}
	For $(x,y)\in S^+$, we have a quasi-isomorphism \[ 
S_*(f^{-1}([x,y]);\kappa) \xrightarrow{\beta^{x,y}}  \mathrm{Cone}\left(\xymatrix{\frac{S_*(f^{-1}([x,\Lambda]);\kappa)}{S_*(\partial^+\mathbb{X};\kappa)}\oplus \frac{S_*(f^{-1}([-\Lambda,y]);\kappa)}{S_*(\partial^-\mathbb{X};\kappa)}\ar[r]^<<<<<{{\tiny [-1 \,\,\, 1] }} & \frac{S_{*}(\mathbb{X};\kappa)}{S_*(\partial\mathbb{X};\kappa)}   }\right)[1]	\] defined by $c\mapsto(i_1c,i_2c,0)$ where $i_1$ and $i_2$ are the inclusions of $ S_*(f^{-1}([x,y]);\kappa)$ into $\frac{S_*(f^{-1}([x,\Lambda]);\kappa)}{S_*(\partial^+\mathbb{X};\kappa)}$ and    $\frac{S_*(f^{-1}([-\Lambda,y]);\kappa)}{S_*(\partial^-\mathbb{X};\kappa)}$, respectively.
\end{prop}

\begin{proof}
	By the tameness assumption on $f$, there is $\ep=\ep(y)\in(0,\Lambda-y)$ and a homotopy $H\co [0,1]\times f^{-1}([-\Lambda,y+\ep))\to f^{-1}([-\Lambda,y+\ep))$  from the identity to a map $r$ having image in $f^{-1}([-\Lambda,y])$, such that $H(t,\cdot)$ restricts to  $f^{-1}(([-\Lambda,y])$ as the identity for every $t$.  Letting $\iota\co S_*(f^{-1}([-\Lambda,y]);\kappa)\to S_*(f^{-1}([-\Lambda,y+\ep));\kappa)$ denote the inclusion, $r_*\circ\iota$ is thus the identity, while there is a map $K\co S_*(f^{-1}([-\Lambda,y+\ep));\kappa)\to S_{*+1}(f^{-1}([-\Lambda,y+\ep));\kappa)$, constructed by the procedure described in, \emph{e.g.}, \cite[p. 112]{Ha}, such that $\partial K+K\partial=1_{S_*(f^{-1}([-\Lambda,y+\ep));\kappa)}-\iota\circ r_*$. Given that our homotopy $H$ is stationary on  $F^{-1}([-\Lambda,y])$, the chain homotopy $K$ has the following property that will be relevant later: for any singular simplex $\sigma\co \Delta^k\to f^{-1}([-\Lambda,y+\ep))$ with image contained in $f^{-1}([-\Lambda,y])$, the singular $(k+1)$-chain $K\sigma$ is a linear combination of \emph{degenerate} singular simplices, each having the same image as $\sigma$.  (Recall that a singular $k+1$-simplex $\tau\co \Delta^{k+1}\to Z$ is degenerate if it factors as a composition $\Delta^{k+1}\to\Delta^k\to Z$ where the first map is one of the standard degeneracy maps given in barycentric coordinates, for some $i$, by $(t_0,\ldots,t_{k+1})\mapsto(t_0,\ldots,t_{i-1},t_i+t_{i+1},t_{i+2},\ldots,t_{k+1})$.) 
	
	This chain homotopy $K$ induces a map on quotients $\frac{S_*(f^{-1}([-\Lambda,y+\ep));\kappa)}{S_*(\partial^-\mathbb{X};\kappa)}\to \frac{S_{*+1}(f^{-1}([-\Lambda,y+\ep));\kappa)}{S_{*+1}(\partial^-\mathbb{X};\kappa)}$ which we continue to denote by $K$.
	Also, 
	throughout the rest of the proof, we frequently either denote an inclusion or a quotient projection by ``$1$,'' or suppress the notation for the composition of an inclusion or quotient projection with some other map.
	
	As in (\ref{mvdiag}), let $S_{*}^{sm,x,y}(\mathbb{X};\kappa)$ denote the subcomplex of $S_*(\mathbb{X};\kappa)$ generated by simplices each of whose image lies either in $f^{-1}([x,\Lambda])$ or in $f^{-1}([-\Lambda,y+\ep))$; thus the inclusion $S_{*}^{sm,x,y}(\mathbb{X};\kappa)\to S_*(\mathbb{X};\kappa)$ is a chain homotopy equivalence, which descends to a chain homotopy equivalence between the quotients $\frac{S_{*}^{sm,x,y}(\mathbb{X};\kappa)}{S_*(\partial\mathbb{X};\kappa)}$ and $\frac{S_{*}(\mathbb{X};\kappa)}{S_*(\partial\mathbb{X};\kappa)}$.  We then have a diagram \begin{equation}\label{kerdiag} \xymatrix{  \frac{S_*(f^{-1}([x,\Lambda]);\kappa)}{S_*(\partial^+\mathbb{X};\kappa)}\oplus \frac{S_*(f^{-1}([-\Lambda,y+\ep));\kappa)}{S_*(\partial^-\mathbb{X};\kappa)} \ar[d]_{{\tiny\left[   \begin{array}{cc} 1 & 0 \\ 0 & r_* \end{array}\right]}} \ar[r]^<<<<<{{\tiny [-1 \,\,\, 1] }} & \frac{S_{*}^{sm,x,y}(\mathbb{X};\kappa)}{S_*(\partial\mathbb{X};\kappa)} \ar[d]^{1} \\ \frac{S_*(f^{-1}([x,\Lambda]);\kappa)}{S_*(\partial^+\mathbb{X};\kappa)}\oplus \frac{S_*(f^{-1}([-\Lambda,y]);\kappa)}{S_*(\partial^-\mathbb{X};\kappa)} \ar[r]^<<<<<<{{\tiny [-1 \,\,\, 1] }} & \frac{S_{*}(\mathbb{X};\kappa)}{S_*(\partial\mathbb{X};\kappa)}  } \end{equation} which commutes up to a homotopy given by (the composition of the inclusion with) the map that acts on the second factor by $K$.  Let us denote the top horizontal arrow in (\ref{kerdiag}) by $p_{\ep}$, and the bottom horizontal arrow by $p$.  As the vertical arrows of (\ref{kerdiag}) are both quasi-isomorphisms, the homotopy-commutativity of the diagram implies that we have a quasi-isomorphism \[ \left[\begin{array}{ccc} 1 & 0 & 0 \\ 0 & r_* & 0 \\ 0 & K & 1 \end{array}\right]\co \mathrm{Cone}(p_{\ep})\xrightarrow{\simeq}\mathrm{Cone}(p). \]  Now the diagonal inclusion ${\tiny \left[\begin{array}{c}1 \\ 1\end{array}\right]}\co S_*(f^{-1}([x,y+\ep));\kappa)\to\frac{S_*(f^{-1}([x,\Lambda]);\kappa)}{S_*(\partial^+\mathbb{X};\kappa)}\oplus \frac{S_*(f^{-1}([-\Lambda,y+\ep));\kappa)}{S_*(\partial^-\mathbb{X};\kappa)}$ is an injection whose image is $\mathrm{ker}(p_{\ep})$, so Lemma \ref{coneids}
gives a quasi-isomorphism $S_*(f^{-1}([x,y+\ep));\kappa)\to \mathrm{Cone}(p_{\ep})[1]$, expressible in block form as ${\tiny \left[\begin{array}{c}1 \\ 1 \\ 0\end{array}\right]}$.  Moreover, our strong deformation retraction $H$ restricts to a deformation retraction of $f^{-1}([x,y+\ep))$ to $f^{-1}([x,y])$, so the inclusion of the latter into the former induces an isomorphism on homology.  Combining the above, we have a sequence of quasi-isomorphisms 

\[ \xymatrix{ S_*(f^{-1}([x,y]);\kappa)\ar[r]^1 & S_*(f^{-1}([x,y+\ep));\kappa) \ar[r]^<<<<{{\tiny \left[\begin{array}{c}1 \\ 1 \\ 0\end{array}\right]}} & \mathrm{Cone}(p_{\ep})[1]\ar[rr]^{{\tiny \left[\begin{array}{ccc} 1 & 0 & 0 \\ 0 & r_* & 0 \\ 0 & K & 1 \end{array}\right]} } & & Cone(p)[1].    }.\] Since the map $r$ restricts to $f^{-1}([x,y])$ as the identity, this composition can be expressed as $\beta^{x,y}+j\circ K$ where $\beta^{x,y}$ is the map from the statement of the proposition, $j\co \frac{S_*(\mathbb{X};\kappa)}{S_*(\partial\mathbb{X};\kappa)}[1]\to Cone(p)[1]$ is the inclusion (which is of course a chain map), and $K\co S_*(f^{-1}([x,y]);\kappa)\to \frac{S_*(\mathbb{X};\kappa)}{S_*(\partial\mathbb{X};\kappa)}[1]$ is the composition of the previously-discussed chain homotopy with the appropriate inclusions and projections.  Since we are restricting $K$ here to chains in $f^{-1}([x,y])$, on which the retraction $r$ acts as the identity, this restricted $K$ anti-commutes with the singular boundary operator, and so is a chain map when regarded, as here, as a map $S_*(f^{-1}([x,y]);\kappa)\to \frac{S_*(\mathbb{X};\kappa)}{S_*(\partial\mathbb{X};\kappa)}[1]$ (since the sign of the boundary operator on the codomain is reversed by the degree shift).     In fact, this chain map acts as zero on homology, by the observation at the start of the proof that its image consists of degenerate singular simplices in $\mathbb{X}$, and the complex of degenerate simplices in $\mathbb{X}$ is acyclic (see \cite[(10.6)]{EZ} or \cite[Proof of Theorem 8.3.8]{Wei}).  

We have thus shown both that $\beta^{x,y}+j\circ K$ is a quasi-isomorphism and that $K$ (and hence also $j\circ K$) acts as zero on homology.  So $\beta^{x,y}$ is a quasi-isomorphism, as desired. 
\end{proof}

  The foregoing discussion shows that, for any $(s,t)=T^{-k}(x,y)\in\mathbb{M}$ (with $(x,y)\in \mathfrak{D}$), we have an isomorphism \[ h(F)_{(s,t)}\cong H_0\left(\mathcal{F}(\mathcal{S}_{\partial}(\mathbb{X},f;\kappa))(s,t) \right), \] induced on homology by a map $\alpha$ as in Lemma \ref{coneids}(ii) if $(x,y)\in S^-\cup L\cup A$ and by the map $\beta^{x,y}$ of Proposition \ref{htopyker} if $(x,y)\in S^+$.  To complete the proof of Theorem \ref{H0iso}, we need to check that these isomorphisms are compatible with the structure maps $h(f)_{(s,t)}\to h(f)_{(s',t')}$ recalled at the start of this section and the structure maps $\mathcal{F}(\mathcal{S}_{\partial}(\mathbb{X},f;\kappa))(s,t)\to \mathcal{F}(\mathcal{S}_{\partial}(\mathbb{X},f;\kappa))(s',t')$  defined in Section \ref{fsect}.  As these structure maps are functorial with respect to the pairs $(s,t)$, it suffices to check this in each of the Cases (\ref{same}) through (\ref{S-S+}).  Case (\ref{same}) is clear since the maps $\alpha$ and $\beta^{x,y}$ are natural with respect to inclusions.  For Case (\ref{S+L}), where $(x,y)\in S^+$ and $(x',y')\in L$, note that the composition  $\phi_{T^{-k}S}^{T^{-k}L}\circ \beta^{x,y}$ sends $S_*(f^{-1}([x,y]);\kappa)[k]$ to the second factor of $\mathrm{Cone}\left(\frac{S_*(f^{-1}([-\Lambda,-2\Lambda-x']);\kappa)}{S_*(\partial^-\mathbb{X};\kappa)}\to \frac{S_*(f^{-1}([-\Lambda,y']);\kappa)}{S_*(\partial^-\mathbb{X};\kappa)} \right)[k]$ by inclusion, and so the composition of this map with $\alpha$ induces on $H_0$ the inclusion-induced map $H_k(F^{-1}([x,y]);\kappa)\to H_k(F^{-1}((-\infty,y']),F^{-1}((-\infty,-2\Lambda-x']);\kappa)$, as desired.  Case (\ref{S+A}) is entirely similar to Case (\ref{S+L}).  Cases (\ref{LS-}) and (\ref{AS-}) are likewise essentially the same as each other; the latter follows from the commutativity of the diagram \[ \xymatrix{
  	\mathrm{Cone}\left(\frac{S_*(f^{-1}([2\Lambda-y,\Lambda]);\kappa)}{S_*(\partial^+\mathbb{X};\kappa)}\to \frac{S_*(f^{-1}([x,\Lambda]);\kappa)}{S_*(\partial^+\mathbb{X};\kappa)}\right) \ar[r]^<<<<<<<{{\tiny[ 0 \,\, 1 ]}} \ar[d]_{\phi_{T^{-k}A}^{T^{-k+1}S}={\tiny \left[\begin{array}{cc} -1 & 0 \\ 0 & 0 \\ 0 & 1 \end{array}\right] }} & \frac{S_*(f^{-1}([x,
  	\Lambda];\kappa)/S_*(\partial^+\mathbb{X};\kappa)  }{S_*(f^{-1}([2\Lambda-y,\Lambda]);\kappa)/S_*(\partial^+\mathbb{X};\kappa)  } \ar[d]^{1} \\  \mathrm{Cone}\left( \frac{S_*(F^{-1}([x',\Lambda];\kappa)}{S_*(\partial^+\mathbb{X};\kappa)}\oplus \frac{S_*(f^{-1}([-\Lambda,y']);\kappa)}{S_*(\partial^-\mathbb{X};\kappa)}\to \frac{S_*(\mathbb{X};\kappa)}{S_*(\partial\mathbb{X};\kappa)}\right) \ar[r]^<<<<<{{\tiny [0\,\,0\,\,1]}} & \frac{S_*(\mathbb{X};\kappa)/S_*(\partial\mathbb{X};\kappa)}{S_{*}(\mathbb{X}\setminus f^{-1}((y',x'));\kappa)/S_*(\partial\mathbb{X};\kappa)}
  }
  \] where we continue denoting various combinations of inclusions and projections just by ``1.''  
  
  We finally consider the remaining case (\ref{S-S+}), in which the structure map for $\mathcal{F}(\mathcal{S}_{\partial}(\mathbb{X},f;\kappa))$ is the inclusion \begin{align*} \mathrm{Cone}&\left(\frac{S_{*}(f^{-1}([x,\Lambda]);\kappa)}{S_*(\partial^+\mathbb{X};\kappa)}\oplus \frac{S_*(f^{-1}([-\Lambda,y]);\kappa)}{S_*(\partial^-\mathbb{X};\kappa)}\to \frac{S_*(\mathbb{X};\kappa)}{S_*(\partial\mathbb{X};\kappa)}\right)[k+1] \\ & \hookrightarrow \mathrm{Cone}\left(\frac{S_{*}(f^{-1}([x',\Lambda]);\kappa)}{S_*(\partial^+\mathbb{X};\kappa)}\oplus \frac{S_*(f^{-1}([-\Lambda,y']);\kappa)}{S_*(\partial^-\mathbb{X};\kappa)}\to \frac{S_*(\mathbb{X};\kappa)}{S_*(\partial\mathbb{X};\kappa)}\right)[k+1]  
\end{align*} while that for $h(f)$ is the connecting homomorphism in a Mayer-Vietoris sequence.  The relation between these is provided by Lemma \ref{coneids}(iii).  
Indeed, applying that lemma with $\phi'$ equal to the map $\frac{S_*(f^{-1}([x',\Lambda]);\kappa)}{S_*(\partial^+\mathbb{X};\kappa)}\oplus \frac{S_*(f^{-1}([-\Lambda,y'+\ep(y
	)));\kappa)}{S_*(\partial^-\mathbb{X};\kappa)}\to \frac{S^{sm,x',y'}(\mathbb{X};\kappa)}{S_*(\partial\mathbb{X};\kappa)}$ that induces the bottom row of (\ref{mvdiag}),
	and with $\phi$ equal to the restriction of $\phi'$ to $\frac{S_*(f^{-1}([x,\Lambda]);\kappa)}{S_*(\partial^+\mathbb{X};\kappa)}\oplus \frac{S_*(f^{-1}([-\Lambda,y]));\kappa)}{S_*(\partial^-\mathbb{X};\kappa)}$, 
we get a commutative diagram \[ \xymatrix{ H_0(\mathcal{F}(\mathcal{S}_{\partial}(X,f;\kappa))(s,t)) \ar[r] \ar[d]_{\alpha_{\phi}} & H_{k+1}(\mathrm{Cone}(\phi')) & H_0(\mathcal{F}(\mathcal{S}_{\partial}(X,f;\kappa)(s',t')) \ar[l] \\ h(f)_{s,t} \ar[r]^<<<<<<<<{\Delta} & H_k(f^{-1}([x',y'+\ep(y')));\kappa)\ar[u]_{{\tiny \left[\begin{array}{c}1\\1\\0\end{array}\right]_*}} & h(f)_{s',t'}\ar[u]_{\beta^{x',y'}}\ar[l]   } \] where the unmarked horizontal arrows are all induced by inclusions and the two left-pointing arrows are both isomorphisms by the tameness of $f$.  The structure maps $H_0(\mathcal{F}(\mathcal{S}_{\partial}(X,f;\kappa))(s,t))\to H_0(\mathcal{F}(\mathcal{S}_{\partial}(X,f;\kappa)(s',t'))$ and $h(f)_{s,t}\to h(f)_{s',t'}$ are respectively given by inverting the two left-pointing arrows in the above diagram and then taking the compositions of the top and bottom rows; hence these structure maps indeed coincide under our isomorphisms $\alpha_{\phi}$ and $\beta^{x',y'}$. This completes the proof of Theorem \ref{H0iso}.

\subsection{Relation to level set barcodes}
As is discussed (with cohomology in place of homology) in \cite[Section 3.2.1]{BBF21}, if the relative interlevel set homology $h(f)$ for a tame function $f\co \mathbb{X}\to [-\Lambda,\Lambda]$ decomposes as a direct sum of blocks $B^v$ as defined in (\ref{bvdef}), then the level set barcode of \cite{CDM} of $f|_{\mathbb{X}\setminus\partial\mathbb{X}}$ can be read off as follows: for any $v$ in the interior of
 $\mathbb{M}$, there is a unique integer $j(v)$ with the property that the support of $B^v$ intersects the segment $T^{-j}(\{(t,t)|-\Lambda<t<\Lambda\})$ in a nonempty interval $I(v)$, and then the contribution of $B^v$ to the level set barcode is precisely the interval $I(v)$ in the degree $j(v)$. The nature of the interval $I(v)$ (\emph{e.g.}, whether it is open, half-open, or closed, and whether its endpoints include $-\Lambda$ and/or $\Lambda$) depend upon where $v$ lies with respect to the regions $T^{k}(S^+),T^k(S^-),T^k(L)$, and $T^k(A)$; by going through the various cases in Proposition \ref{blockclass}, one deduces the following from that Proposition and from Theorem \ref{H0iso}:

\begin{prop}\label{barclass}
	Let $f\co \mathbb{X}\to [-\Lambda,\Lambda]$ be a tame function, let $\kappa$ be a field, and suppose that the pinned singular filtered cospan $\mathcal{S}_{\partial}(\mathbb{X},f;\kappa)$ is isomorphic to a direct sum of standard elementary summands.  The level set barcode of $f|_{\mathbb{X}\setminus\partial{X}}$ (with coefficients in $\kappa$) then consists of one interval for each of these standard elementary summands, associated via the following table (where we denote an interval $I$ in grading $j$ by $I_j$):
	\begin{center}
		\begin{tabular}{c|c} $\mathcal{C}$ & $I(v(\mathcal{C}))_{j(v(\mathcal{C}))}$  \\ 	\hline $\left(\uparrow_{a}^{b}\right)_k$ for $a<b<\Lambda$ & $[a,b)_k$  \\ \hline $\left(\uparrow_{a}^{\infty}\right)_k$ &  $[a,\Lambda)_k$\\ 
			\hline $\left(\nearrow_{a}\right)_k$ & $(-\Lambda,a)_{k-1}$ \\ \hline $\left(\downarrow^{a}_{b}\right)_{k}$ for $a>b>-\Lambda$ & $(b,a]_k$ \\ \hline 
			$\left(\downarrow^{a}_{-\infty}\right)_{k}$ &  $(-\Lambda,a]_k$ \\ \hline $\left(\searrow^{a}\right)_k$ &  $(a,\Lambda)_{k-1}$ \\ \hline $\left(>_{b}^{a}\right)_k$ for $-\Lambda<b\leq a<\Lambda$ & $[b,a]_k$ \\ 
			\hline $\left(>_{b}^{a}\right)_k$ for $-\Lambda<a< b<\Lambda$ & $(a,b)_{k-1}$ \\ \hline $\square_k$ & $(-\Lambda,\Lambda)_{k-1}$
		\end{tabular}
	\end{center}
	\end{prop}

\section{Interleavings and stability}\label{stabsect}

As the input in our motivating examples is a function $f\co \mathbb{X}\to [-\Lambda,\Lambda]$ for a fixed value $\Lambda$, and elements in the corresponding barcodes are then subintervals of $(-\Lambda,\Lambda)$,  a bit of care is required in setting up the usual considerations regarding interleaving and stability, because the action of translations does not preserve the interval $[-\Lambda,\Lambda]$.  However, on the principle that our function $f$ can  be regarded as a composition $\varphi\circ g$ where $g\co \mathbb{X}\to [-\infty,\infty]$ and $\varphi\co [-\infty,\infty]\to [-\Lambda,\Lambda]$ is an increasing homeomorphism, we can expect some kind of stability for $f$ arising from the standard types of stability results for $g=\varphi^{-1}\circ f$.

Without explicitly invoking any such results, we can adapt the language of interleavings of categories with a flow from \cite{dSMS} to work directly with $[-\Lambda,\Lambda]$-valued functions, their associated $\Lambda$-bounded filtered cospans, and their corresponding diagrams (which take values in the strip $\mathbb{M}$, whose construction and interpretation benefit from taking a finite, fixed value of $\Lambda$). 

Let us choose and fix in what follows an increasing homeomorphism $\varphi\co [-\infty,\infty]\to [-\Lambda,\Lambda]$. This determines, for any $s\in \R$, what will function as a version of  ``translation by s'' on the interval $[-\Lambda,\Lambda]$, namely the map $\rho^{\varphi}_s\co [-\Lambda,\Lambda]\to[-\Lambda,\Lambda]$ defined by \[ \rho^{\varphi}_{s}(t)=\varphi\left(s+\varphi^{-1}(t)\right).\] (Here, of course, $s+\infty=\infty$ and $s+(-\infty)=-\infty$, so $\rho^{\varphi}_{s}(\pm\Lambda)=\pm\Lambda$.)  Evidently \begin{equation}\label{rhofunct} \rho^{\varphi}_{s+t}=\rho^{\varphi}_{s}\circ\rho^{\varphi}_{t},\qquad \rho^{\varphi}_{0}=\mathbf{1}_{[-\Lambda,\Lambda]},\end{equation} and \begin{equation}\label{monotone} \mbox{if }s\leq s',\mbox{ then }\rho^{\varphi}_{s}(t)\leq \rho_{s'}^{\varphi}(t)\mbox{ for all }t.\end{equation}

\subsection{Flows}
Following \cite{dSMS}, a \textbf{strict flow} on a category $\mathsf{C}$ is defined to be a monoidal functor $\mathcal{T}\co\ep\mapsto \mathcal{T}_{\ep}$ from the poset category $[0,\infty)$ (with monoidal structure given by addition) to the category $\mathbf{End}(\mathsf{C})$ of endofunctors of $\mathsf{C}$.  Thus $\mathcal{T}_0$ is the identity endofunctor of $\mathsf{C}$, we have $\mathcal{T}_{s+t}=\mathcal{T}_s\circ\mathcal{T}_t$, and for $s\leq t$ we have a natural transformation $\mathcal{T}_{(s\leq t)}\co \mathcal{T}_s\Rightarrow \mathcal{T}_t$.  If $X$ and $Y$ are two objects of $\mathsf{C}$ and $\ep\geq 0$, an $\ep$-\textbf{interleaving} $(\alpha,\beta)$ of $X$ and $Y$ is then defined to consist of morphisms $\alpha\co X\to \mathcal{T}_{\ep}Y$ and $\beta\co Y\to \mathcal{T}_{\ep}X$ such that \begin{equation}\label{intlvcrit} (\mathcal{T}_{\ep}\beta)\circ \alpha=\mathcal{T}_{(0\leq 2\ep),X},\qquad (\mathcal{T}_{\ep}\beta)\circ \alpha=\mathcal{T}_{(0\leq 2\ep),Y}.\end{equation}
This yields an interleaving distance on $\mathsf{C}$: for objects $X,Y$ of $\mathsf{C}$ we set \begin{equation}\label{intergen} d_{\mathrm{int}}^{\mathcal{T}}(X,Y)=\inf\left\{\ep >0|\mbox{there is an $\ep$-interleaving between $X$ and $Y$}\right\}\end{equation} where the infimum of the empty set is $\infty$.

We also borrow the following terminology from \cite{U22}

\begin{dfn}\label{implantdfn}
	Given a strict flow $\mathcal{T}$ on a category $\mathsf{C}$, two objects $X$ and $Y$ of $\mathsf{C}$, and $\ep>0$, an $\ep$-\textbf{implantation} of $X$ into $Y$ consists of morphisms $\alpha\co X\to\mathcal{T}_{\ep}Y$ and $\beta\co Y\to \mathcal{T}_\ep X$ such that $(\mathcal{T}_{\ep}\beta)\circ\alpha=\mathcal{T}_{(0\leq 2\ep),X}$.
\end{dfn}

Thus we omit the second condition in (\ref{intlvcrit}).  While an $\ep$-interleaving is a sort of approximate isomorphism between $X$ and $Y$, an $\ep$-implantation expresses $X$ as approximately a retract of $Y$.

We now describe how the system of translations $\rho_{s}^{\varphi}$ of $[-\Lambda,\Lambda]$, associated to the increasing homeomorphism $\varphi\co [-\infty,\infty]\to [-\Lambda,\Lambda]$, induces strict flows on various categories in our story.  

\subsubsection{The flow on filtered cospans}\label{cospanflow}

First, on the  homotopy category of filtered cospans $\mathsf{HCO}^{\Lambda}(\kappa)$, we get a strict flow $\mathcal{T}^{\varphi}$ by taking, on objects \[ \mathcal{T}^{\varphi}_{\ep}\left((C_{\uparrow *},\partial_{\uparrow}^{C},\ell_{\uparrow}^{C}),(C^{\downarrow}_{*},\partial^{\downarrow}_{C},\ell^{\downarrow}_{C}),(D_*,\partial_D),\psi_{\uparrow},\psi^{\downarrow} \right)=\left((C_{\uparrow *},\partial_{\uparrow}^{C},\rho_{-\ep}^{\varphi}\circ\ell_{\uparrow}^{C}),(C^{\downarrow}_{*},\partial^{\downarrow}_{C},\rho_{\ep}^{\varphi}\circ\ell^{\downarrow}_{C}),(D_*,\partial_D),\psi_{\uparrow},\psi^{\downarrow} \right) \] (so the only effect of $\mathcal{T}_{\ep}^{\varphi}$ is to ``translate'' the ascending filtration function by $-\ep$ and the descending filtration function by $\ep$), and by taking $\mathcal{T}_{\ep}^{\varphi}$ to act as the identity on morphisms,  noting that the conditions (\ref{quint}) and (\ref{filtpres}) for a tuple $(\alpha^{\downarrow},\alpha_{\uparrow},\alpha,K^{\downarrow},K_{\uparrow})$ to define a morphism $\mathcal{C}\to\mathcal{X}$ are unchanged if $\mathcal{C}$ and $\mathcal{X}$ are simultaneously replaced by $\mathcal{T}^{\varphi}_{\ep}\mathcal{C}$ and $\mathcal{T}^{\varphi}_{\ep}\mathcal{X}$.  The natural transformations $\mathcal{T}_{(s\leq t)}^{\varphi}$, applied to a general object $\mathcal{C}$ as above, are given by taking the identity maps on $C_{\uparrow *},C^{\downarrow}_*,D_*$ and taking the maps $K_{\uparrow}$ and $K^{\downarrow}$ as in (\ref{quint}) to be zero.  ((\ref{filtpres}) is satisfied due to (\ref{monotone}).)

\begin{ex} \label{function-interleave}
 We apply this to the pinned singular filtered cospans $\mathcal{S}_{\partial}(\mathbb{X},f;\kappa)$ of Example \ref{pinsing}.  So we fix a topological space $\mathbb{X}$ with a subset $\partial\mathbb{X}=\partial^-\mathbb{X}\sqcup \partial^+\mathbb{X}$ and consider functions $f\co \mathbb{X}\to [-\Lambda,\Lambda]$ that are required to have $f^{-1}(\{\pm\Lambda\})=\partial^{\pm}\mathbb{X}$.
 
 If $f,g\co\mathbb{X}\to [-\Lambda,\Lambda]$ are two such functions, the complexes $C_{\uparrow *},C^{\downarrow}_{*},D_*$ prescribed by Example \ref{pinsing} are the same for $f$ as for $g$, while the filtration functions $\ell_{\uparrow}$ and $\ell^{\downarrow}$ differ.  One may easily check that, if  $\sup_{\mathbb{X}\setminus\partial\mathbb{X}}|\varphi^{-1}\circ f-\varphi^{-1}\circ g|\leq \ep$, then the behavior of the identity maps on $C_{\uparrow *},C^{\downarrow}_{*},D_*$ with respect to these filtrations is such that they induce morphisms in $\mathsf{HCO}^{\Lambda}(\kappa)$ from $\mathcal{S}_{\partial}(\mathbb{X},f;\kappa)$ to 
 $\mathcal{T}_{\ep}^{\varphi}\mathcal{S}_{\partial}(\mathbb{X},g;\kappa)$ and from $\mathcal{S}_{\partial}(\mathbb{X},g;\kappa)$ to 
 $\mathcal{T}_{\ep}^{\varphi}\mathcal{S}_{\partial}(\mathbb{X},f;\kappa)$  (with the chain homotopy maps $K_{\uparrow}$ and $K^{\downarrow}$ taken to be zero in both cases).  
 Since the natural transformations $\mathcal{T}^{\varphi}_{(0\leq 2\ep)}$ in $\mathcal{C}$ are likewise given by the identities on the various chain complexes, we conclude that, if   $\sup_{\mathbb{X}\setminus\partial\mathbb{X}}|\varphi^{-1}\circ f-\varphi^{-1}\circ g|\leq \ep$, then there is an $\ep$-interleaving between $\mathcal{S}_{\partial}(\mathbb{X},f;\kappa)$ and $\mathcal{S}_{\partial}(\mathbb{X},g;\kappa)$. (Compare \cite[Theorem 4.4]{BBF21}.)
\end{ex}

\subsubsection{The flow on $\mathbb{M}$}\label{mflowsect}

Our homeomorphism $\varphi\co [-\infty,\infty]\to[-\Lambda,\Lambda]$ also induces a strict flow $\mathcal{T}_{\mathbb{M}}^{\varphi}$ on the strip $\mathbb{M}$ (regarded as a poset category) by the following prescription with respect to the decomposition of $\mathbb{M}$ from (\ref{mdecomp}): for $(x,y)\in\mathfrak{D}=S\sqcup L\sqcup A$, for $k\in \Z$, and for $\ep\geq 0$ we set: \begin{equation}\label{mflow} (\mathcal{T}_{\mathbb{M}}^{\varphi})_{\ep}\left(T^k(x,y)\right)=\left\{\begin{array}{ll} T^k\left(\rho^{\varphi}_{-\ep}(x),\rho^{\varphi}_{\ep}(y)\right) & \mbox{if }(x,y)\in S \\ T^k\left(-2\Lambda-\rho^{\varphi}_{\ep}(-2\Lambda-x), \rho^{\varphi}_{\ep}(y)\right) & \mbox{if }(x,y)\in L\\ T^k\left(\rho^{\varphi}_{-\ep}(x),2\Lambda-\rho^{\varphi}_{-\ep}(2\Lambda-y)\right) & \mbox{if }(x,y)\in A    \end{array}\right.\end{equation}

(Since there is at most one morphism between any two objects of $\mathbb{M}$, the action of $(\mathcal{T}_{\mathbb{M}}^{\varphi})_{\ep}$ on morphisms is determined by its action on objects.)

Note that the maps $(\mathcal{T}_{\mathbb{M}}^{\varphi})_{\ep}\co\mathbb{M}\to\mathbb{M}$ commute with the shift $T$.  Also, each of the subdomains $T^kS, T^kL, T^kA$ are preserved by the $(\mathcal{T}_{\mathbb{M}}^{\varphi})_{\ep}$, as are each of the edges $T^k(\{\pm \Lambda\}\times (-\Lambda,\Lambda))$ and $T^k((-\Lambda,\Lambda)\times\{\pm\Lambda\})$ as well as  the corners $T^k(\{(\pm\Lambda,\pm\Lambda)\})$ of the squares $T^kS$.

\subsubsection{The flow on $\mathsf{C}^{\mathbb{M}}$}

If $\mathsf{C}$ is any category then the above strict flow $\mathcal{T}_{\mathbb{M}}^{\varphi}$ on $\mathbb{M}$ induces, as in \cite[Section 3]{dSMS},  a strict flow $\mathcal{T}_{\mathsf{C}^{\mathbb{M}}}^{\varphi}$ on the functor category $\mathsf{C}^{\mathbb{M}}$ by setting, for any functor $f\co\mathbb{M}\to\mathsf{C}$ and $\ep\geq 0$, \[ (\mathcal{T}_{\mathsf{C}^{\mathbb{M}}}^{\varphi})_{\ep}f=f\circ(\mathcal{T}_{\mathbb{M}}^{\varphi})_{\ep}  \] and using the obvious corresponding action on natural transformations.

Upon unraveling the definitions, one  finds that the functors $\mathcal{F}^{0}_{(s,t)}\co \mathsf{CO}^{\Lambda}(\mathbb{R})\to \mathsf{Ch}(\kappa)$, defined in (\ref{f0dfn}) for each $(s,t)\in\mathbb{M}$, satisfy $\mathcal{F}^{0}_{(s,t)}(\mathcal{T}^{\varphi}_{\ep}(\mathcal{C}))=\mathcal{F}^{0}_{(\mathcal{T}^{\varphi}_{\mathbb{M}})_{\ep}(s,t)}(\mathcal{C})$ for each $\ep\in[0,\infty)$ and each object $\mathcal{C}$ of $\mathsf{CO}^{\Lambda}(\kappa)$.   Hence:

\begin{prop}\label{flowfunctor}
	The functor $\mathcal{F}\co \mathsf{HCO}^{\Lambda}(\kappa)\to \mathsf{K}(\kappa)^{\mathbb{M}}$ of Definition \ref{fdef} satisfies $\mathcal{F}\circ \mathcal{T}^{\varphi}_{\ep}=(\mathcal{T}^{\varphi}_{\mathsf{K}(\kappa)^{\mathbb{M}}})_{\ep}\circ \mathcal{F}$.  
\end{prop}

That is, $\mathcal{F}$ is a strict $[0,\infty)$-equivariant functor in the sense of \cite[Section 4]{dSMS}.  It follows immediately that $H_0\circ\mathcal{F}\co \mathsf{HCO}^{\Lambda}(\kappa)\to \mathsf{Vect}_{\kappa}^{\mathbb{M}}$ is likewise a strict $[0,\infty)$-equivariant functor with respect to the flows $\mathcal{T}^{\varphi}$ on $\mathsf{HCO}^{\Lambda}(\kappa)$ and $\mathcal{T}^{\varphi}_{\mathsf{Vect}_{\kappa}^{\mathbb{M}}}$ on $\mathsf{Vect}_{\kappa}^{\mathbb{M}}$.

\subsection{Metrics}

\subsubsection{Embeddings and matchings}

Let $(X,d_X)$ be an extended pseudometric space\footnote{An extended pseudometric on $X$ is a function $d_X\co X\times X\to [0,\infty]$ that is symmetric and satisfies the triangle inequality and the property that $d(x,x)=0$ for all $x$.  The word ``extended'' refers to the fact that $d$ may take the value $\infty$.} and let $\Delta\subset X$.  Given two multisets\footnote{Formally, a multiset of elements of $X$ can be defined as a set $S$ together with a function $\pi\co S\to X$; our notation will conflate an element of $S$ with its image under this map.  For $x\in X$, the multiplicity of $x$ in the multiset $S$ is the cardinality of $\pi^{-1}\{x\}$.} $S$ and $T$ of elements of $X$, and given $\ep\geq 0$, an $\ep$-\textbf{matching between $S$ and $T$ relative to} $\Delta$ consists of the data of: \begin{itemize} 
	\item submultisets $S'\subset S$ and $T'\subset T$ such that for all $x\in (S\setminus S')\cup (T\setminus T')$ there exists $y\in \Delta$ such that $d_X(x,y)\leq \ep$; and
	\item a bijection $f\co S'\to T'$ such that $d_X(x,f(x))\leq\ep$ for all $x\in S'$.
	\end{itemize}

We define an $\ep$-\textbf{embedding of $S$ into $T$ relative to }$\Delta$ to consist of the data of:
\begin{itemize}\item a submultiset $S'\subset S$ such that for all $x\in S\setminus S'$ there exists $y\in \Delta$ such that $d_X(x,y)\leq \ep$; and 
\item an injection $f\co S'\to T$ such that $d_X(x,f(x))\leq\ep$ for all $x\in S'$.
\end{itemize}

The familiar \textbf{bottleneck distance} between multisets $S$ and $T$ of elements of $X$ (relative to $\Delta$) is defined to be \[ d_{\infty}^{\Delta,d_X}(S,T)=\inf\{\ep\geq 0|\mbox{There exists an $\ep$-matching between $S$ and $T$ relative to $\Delta$}\}. \] (See \cite{CEH07} for the original formulation in the context of persistence diagrams and \cite{CGGM} for a general treatment.)  Let us define the bottleneck \textbf{hemidistance from $S$ to} $T$ relative to $\Delta$ to be \begin{equation}\label{hemidfn} \delta^{\Delta,d_X}_{\infty}(S,T)=\inf\{ \ep\geq 0|\mbox{There exists an $\ep$-embedding from $S$ into $T$ relative to $\Delta$}  \}.\end{equation}  Clearly $\delta^{\Delta,d_X}_{\infty}$ is not symmetric, hence the term hemidistance.  

\begin{prop}\label{symm} For any multisets $S$ and $T$ of elements $X$ we have \[ d_{\infty}^{\Delta,d_X}(S,T)=\max\{\delta^{\Delta,d_X}_{\infty}(S,T),\delta^{\Delta,d_X}_{\infty}(T,S)\}.\]
\end{prop}

\begin{proof}
	As an $\ep$-matching between $S$ and $T$ clearly gives rise to $\ep$-embeddings both from $S$ into $T$ and from $T$ into $S$, the inequality ``$\geq$'' is immediate.  For the reverse inequality we shall show that, given $\ep$-embeddings $f\co S'\to T$ and $g\co T'\to S$, there may be constructed an $\ep$-matching between $S$ and $T$.
	
	This may be achieved by adapting a standard proof of the Schr\"oder-Bernstein theorem (the statement that if each of two sets can each be injected into the other then there is a bijection between them). Set $B_0=T\setminus \Img(f)$ and inductively define, for $n\geq 1$, $A_n=g(T'\cap B_{n-1})\subset S$ and $B_n=f(S'\cap A_{n})\subset T$.  Let $S''=S'\cup\left(\cup_{n=1}^{\infty}A_n\right)$ and define $h\co S''\to T$ by \[ h(a)=\left\{\begin{array}{ll} g^{-1}(a) & \mbox{if }a\in A_n\mbox{ for some }n\geq 1 \\ f(a) & \mbox{if $x\in S'\setminus\cup_{n=1}^{\infty}A_n$}  \end{array}\right..  \]  (This is well-defined since $g$ is injective and since each $A_n$ is contained in $\Img g$.)   If $x_1\in A_n$ and $x_2\in S'\setminus(\cup_n A_n)$, the element $h(x_1)=g^{-1}(x_1)\in B_{n-1}$ cannot equal $f(x_2)$: indeed, if $n=1$ this follows because $B_0\cap \Img(f)=0$, while if $n>1$ it follows because $B_{n-1}\subset f(A_{n-1})$ while $x_2\notin A_{n-1}$ and $f$ is injective.  From this, together with the evident fact that the restrictions of $h$ to $\cup_n A_n$ and $S''\setminus(\cup_nA_n)$ are  each  injective, it follows that $h$ is injective.   
	
	Let us show that the image of $h$ contains $T'$.	
	Given $y\in T'$, if we have $y\in B_n$ for some $n\geq 0$, then  $g(y)\in A_{n+1}$ and $y=h(g(y))$.  If instead $y\in T'\setminus\cup_{n\geq 0}B_n$, then the fact that $y\notin B_0$ implies that $y\in \Img(f)$, and the fact that $y\notin \cup_{n\geq 1}B_n$ implies that $y\notin f(S'\cap A_n)$ for any $n$. So the element $x$ of $S'$ such that $f(x)=y$ does not belong to any $A_n$, whence $h(x)=y$.  So indeed $T'\subset \Img(h)$.
	
	Setting $T''=\Img(h)$, we thus have a bijection $h\co S''\to T''$ where $S'\subset S''\subset S$ and $T'\subset T''\subset T$, with the additional property that if $y=h(x)$ then either $y=f(x)$ or $x=g(y)$.  The fact that $f\co S'\to T$ and $g\co T'\to S$ are $\ep$-embeddings relative to $\Delta$ then immediately implies that $h$ is an $\ep$-matching relative to $\Delta$.
	\end{proof}

\subsubsection{Metrics from flows}\label{flowmetrics}

Now let $(P,\preceq)$ be any partially ordered set, and let $\Delta\subset P$.  For the purposes of this section, one should imagine that one has some notion of persistence diagrams which are multisets of elements of $P\setminus \Delta$. (The case most relevant to this paper is that $P=\mathbb{M}$ and $\Delta=\partial\mathbb{M}$, corresponding to Proposition \ref{blockclass}, but we will mention some simpler examples for context.) Regarding $P$ as a category in the usual way, suppose that we have a strict flow $\mathcal{T}\co [0,\infty)\to\mathrm{End}(P)$. This induces an interleaving distance $d_{\mathrm{int}}^{\mathcal{T}}$ on $P$ by the general prescription in (\ref{intergen}), given in the present context by \[ d_{\mathrm{int}}^{\mathcal{T}}(x,y)=\inf\{\ep>0|x\preceq \mathcal{T}_{\ep}y\mbox{ and }y\leq \mathcal{T}_{\ep}x\}.  \]  For example, if $P=\R^n$ with its standard partial order and if $\mathcal{T}_{\ep}\co P\to P$ is given by $(x_1,\ldots,x_n)\mapsto (x_1+\ep,\ldots,x_n+\ep)$ then, as pointed out in \cite[Theorem 3.9]{dSMS}, $d_{\mathrm{int}}^{\mathcal{T}}$ coincides with the usual $\ell^{\infty}$ distance on $\R^n$.

Regarding $P$ now as an extended pseudometric space with respect to the distance $d_{\mathrm{int}}^{\mathcal{T}}$, we then obtain a bottleneck distance \[ d_{B,\mathcal{T}}^{P,\Delta}:=d^{\Delta,d_{\mathrm{int}}^{\mathcal{T}}}_{\infty} \] on multisets of elements of $P$.  For example, one may take \[ P=\{(a,b)\in (-\infty,\infty)\times (-\infty,\infty]|a\leq b\}\mbox{ and }\Delta=\{(a,b)\in P|a=b\},\] with the flow $\mathcal{T}_{\ep}(a,b)=(a+\ep,b+\ep)$, and then one obtains the usual bottleneck distance on persistence diagrams from \cite{CEH07}.  Likewise one obtains a bottleneck hemidistance (see (\ref{hemidfn})) $\delta_{B,\mathcal{T}}^{P,\Delta}:=\delta^{\Delta,d_{\mathrm{int}}^{\mathcal{T}}}_{\infty}$.

The persistence diagrams in this paper are multisets of elements of $\mathbb{M}\setminus\partial\mathbb{M}$, so the bottleneck distance of interest for us will be $d_{B,\mathcal{T}^{\varphi}_{\mathbb{M}}}^{\mathbb{M},\partial\mathbb{M}}$, derived from the extended pseudometric $d_{\mathrm{int}}^{\mathcal{T}_{\mathbb{M}}^{\varphi}}$. on $\mathbb{M}$.  Let us describe some features of   $d_{\mathrm{int}}^{\mathcal{T}_{\mathbb{M}}^{\varphi}}$.

Write $S^{\circ},L^{\circ}$, and $A^{\circ}$ for the interiors of the subsets $S,L,$ and $A$ respectively (with respect to the usual topology on $\R^2$, so for example $S^{\circ}=(-\Lambda,\Lambda)\times(-\Lambda,\Lambda)$).  Suppose that $R$ is any of the following subsets of the fundamental domain $\mathfrak{D}$ (whose union is all of $\mathfrak{D}$): \begin{align} \label{regionlist}\nonumber & S^{\circ},\,L^{\circ},\, A^{\circ},\,(-\Lambda,\Lambda)\times\{-\Lambda\},\,(-\Lambda,\Lambda)\times\{\Lambda\},\\ & \{-\Lambda\}\times (-\Lambda,\Lambda),\,\{\Lambda\}\times (-\Lambda,\Lambda),\,\{(\Lambda,-\Lambda)\},\mbox{ and } \mathfrak{D}\cap \partial\mathbb{M}.\end{align}  Then one may verify from (\ref{mflow}) that, for each $k\in\Z$ and each $\ep\geq 0$, $(\mathcal{T}_{\mathbb{M}}^{\varphi})_{\ep}(T^kR)=T^kR$. Moreover, each of the regions $R$ from (\ref{regionlist}), \emph{except} for $L^{\circ},A^{\circ}$, and $\mathfrak{D}\cap \partial\mathbb{M}$, has the property that if $v,w\in T^kR$ with $v\preceq w$, then the entire interval $[v,w]=\{u\in\mathbb{M}:v\preceq u\preceq w\}$ is contained in $T^kR$.   We deduce that:

\begin{prop}\label{sepregions} If $v\in \mathbb{M}\setminus\partial\mathbb{M}$ and $w\in\mathbb{M}$ with
	 $d_{\mathrm{int}}^{\mathcal{T}_{\mathbb{M}}^{\varphi}}(v,w)<\infty$, 
	 and if $v\in T^kR$ where $R$ is one of the regions from (\ref{regionlist}), then either $w\in T^kR$, or $R$ is equal to $L^{\circ}$ or $A^{\circ}$ and $w\in\partial\mathbb{M}$.
\end{prop}

\begin{proof}
	That $d^{\mathcal{T}_{\mathbb{M}}^{\varphi}}_{\mathrm{int}}(v,w)<\infty$ means that there is $\ep>0$ such that both $v\preceq (\mathcal{T}_{\mathbb{M}}^{\varphi})_{\ep}w$ and $w\preceq (\mathcal{T}_{\mathbb{M}}^{\varphi})_{\ep}v$, and hence $ (\mathcal{T}_{\mathbb{M}}^{\varphi})_{\ep}w\in [v,(\mathcal{T}_{\mathbb{M}}^{\varphi})_{2\ep}v ]$.  Since we have assumed that $v\notin\partial\mathbb{M}$, if $R$ is neither $L^{\circ}$ nor $A^{\circ}$ then the remarks before the proposition then imply that  $ (\mathcal{T}_{\mathbb{M}}^{\varphi})_{\ep}w\in T^k R$, and hence also that $w\in T^k R$.  
	
	If instead $v\in T^kL^{\circ}$, then the interval $[v,(\mathcal{T}_{\mathbb{M}}^{\varphi})_{2\ep}v ]$ (which is the intersection of $\mathbb{M}$ with an axis-aligned rectangle having opposite corners $v$ and $(\mathcal{T}_{\mathbb{M}}^{\varphi})_{2\ep}v$, both of which lie in $T^kL^{\circ}$) is contained in $T^kL^{\circ}\cup \partial\mathbb{M}$ and so $(\mathcal{T}_{\mathbb{M}}^{\varphi})_{\ep}w$, and hence also $w$, lies in $T^kL^{\circ}\cup\partial\mathbb{M}$. Similarly, if $v\in T^k A^{\circ}$, then $[v,(\mathcal{T}_{\mathbb{M}}^{\varphi})_{2\ep}v ]\subset T^kA^{\circ}\cup\partial\mathbb{M}$, from which it follows that $w\in T^kA^{\circ}\cup\partial\mathbb{M}$. 
\end{proof}

 For the bottleneck distance $d_{B,\mathcal{T}^{\varphi}_{\mathbb{M}}}^{\mathbb{M},\partial\mathbb{M}}$ it is particularly important to understand how $d_{\mathrm{int}}^{\mathcal{T}_{\mathbb{M}}^{\varphi}}$ measures distances from a point of $\mathbb{M}\setminus \partial\mathbb{M}$ to the boundary $\partial\mathbb{M}$.  The following criterion is helpful.
 \begin{prop}\label{bdrydist}
 	Let $v\in\mathbb{M}\setminus\partial\mathbb{M}$ and $0<\ep<\infty$.  Then there exists $w\in\partial \mathbb{M}$ such that $d_{\mathrm{int}}^{\mathcal{T}_{\mathbb{M}}^{\varphi}}(v,w)< \ep$ if and only if we have $(\mathcal{T}_{\mathbb{M}}^{\varphi})_{2\ep}v \npreceq Tv$.
 \end{prop}
 
 \begin{proof}
If $v$ lies in $T^kR$ for any of the regions $R$ in (\ref{regionlist}) other than $L^{\circ},A^{\circ},$ and $\mathfrak{D}\cap\partial\mathbb{M}$, then Proposition \ref{sepregions} shows that
 $d_{\mathrm{int}}^{\mathcal{T}_{\mathbb{M}}^{\varphi}}(v,w)=\infty$ 
	for all $w\in\partial\mathbb{M}$, so in this case the current proposition is equivalent to the statement that 
	$(\mathcal{T}_{\mathbb{M}}^{\varphi})_{2\ep}v \preceq Tv$ for all (finite) $\ep$.  
	But for each of these regions $R$ one may readily check that every element of $T^kR$ is less than (according to $\preceq$) every element of $T^{k+1}R$, so since $(\mathcal{T}_{\mathbb{M}}^{\varphi})_{2\ep}v\in T^kR$ and $Tv\in T^{k+1}R$ the conclusion follows.
	
	Since we assume $v\notin \partial\mathbb{M}$, there remain only the cases that $v\in T^kL^{\circ}$ and that $v\in T^{k}A^{\circ}$ (for some $k\in\Z$).
	
	Suppose that $v=T^k(x,y)$ where $(x,y)\in L^{\circ}$, so that $x<-\Lambda, y<\Lambda$, and $x+y>-2\Lambda$.  Thus $-\Lambda<-2\Lambda-x<y$.  We have $Tv=T^k(-2\Lambda-y,2\Lambda-x)$ while $(\mathcal{T}_{\mathbb{M}}^{\varphi})_{2\ep}v=T^k\left(-2\Lambda-\rho_{2\ep}^{\varphi}(-2\Lambda-x),\rho_{2\ep}^{\varphi}(y)\right)$.  
	Since $T^k$ respects $\preceq$ (in the sense that $v_1\preceq v_2\Leftrightarrow T^kv_1\preceq T^kv_2$) and since $\rho_{2\ep}^{\varphi}(y)<\Lambda<3\Lambda<2\Lambda-x$, the condition that
		 $(\mathcal{T}_{\mathbb{M}}^{\varphi})_{2\ep}(v)\npreceq Tv$ is equivalent to the condition that $-2\Lambda-\rho_{2\ep}^{\varphi}(-2\Lambda-x)<-2\Lambda-y$, \emph{i.e.} that $y<\rho_{2\ep}^{\varphi}(-2\Lambda-x)$.  
If there is $w\in \partial \mathbb{M}$ such that $d_{\mathrm{int}}^{\mathcal{T}_{\mathbb{M}}^{\varphi}}(v,w)< \ep$, then $w\in [v,(\mathcal{T}_{\mathbb{M}}^{\varphi})_{2\ep}v]$, which is contained in $T^kL$, so $w\in T^k(L\cap \partial\mathbb{M})$ and we can write $w=T^k(x',-2\Lambda-x')$ where $-3\Lambda< x'\leq-\Lambda$.  Moreover we will have $v\preceq (\mathcal{T}_{\mathbb{M}}^{\varphi})_{\ep'}w$ and $w\preceq (\mathcal{T}_{\mathbb{M}}^{\varphi})_{\ep'}v$ for some $   \ep'<\ep$; the first of these inequalities implies that  \[ y<\rho_{\ep}^{\varphi}(-2\Lambda-x')\] while the second implies that $x'>-2\Lambda-\rho_{\ep}^{\varphi}(-2\Lambda-x)$, \emph{i.e.} that $-2\Lambda-x'<\rho_{\ep}(-2\Lambda-x)$, and hence that \[ \rho_{\ep}^{\varphi}(-2\Lambda-x')<\rho_{2\ep}(-2\Lambda-x).\]  Thus, if $v=T^k(x,y)$ where $(x,y)\in L$ and if there is $w\in\partial\mathbb{M}$ with $d_{\mathrm{int}}^{\mathcal{T}_{\mathbb{M}}^{\varphi}}(v,w)< \ep$, we deduce that $y<\rho_{2\ep}^{\varphi}(-2\Lambda-x)$, which as noted earlier is equivalent (when $v\in T^kL^{\circ}$) to the statement that $(\mathcal{T}_{\mathbb{M}}^{\varphi})_{2\ep}(v)\npreceq Tv$. 

Conversely, suppose that $v=T^k(x,y)$ with $(x,y)\in L^{\circ}$ has $(\mathcal{T}_{\mathbb{M}}^{\varphi})_{2\ep}(v)\npreceq Tv$. Then $-2\Lambda-x<y<\rho_{2\ep}^{\varphi}(-2\Lambda-x)$, so for some $\delta$ with $0<\delta<\ep$ we have $y=\rho_{2\delta}^{\varphi}(-2\Lambda-x)$.  Let $y'=\rho_{\delta}^{\varphi}(-2\Lambda-x)=\rho_{-\delta}^{\varphi}(y)$, and $x'=-2\Lambda-y'$, so that $-2\Lambda-x'=\rho_{\delta}^{\varphi}(-2\Lambda-x)$.  

One then finds \[ (\mathcal{T}_{\mathbb{M}}^{\varphi})_{\delta}(T^k(x',y'))=T^k\left(-2\Lambda-\rho_{2\delta}^{\varphi}(-2\Lambda-x),y\right)\succeq T^k(x,y)=v\] and \[ (\mathcal{T}_{\mathbb{M}}^{\varphi})_{\delta}(v)= (\mathcal{T}_{\mathbb{M}}^{\varphi})_{\delta}(T^k(x,y))=T^k(x',\rho_{2\delta}^{\varphi}(y'))\succeq T^k(x',y').\]  Thus, letting $w=T^k(x',y')$, we have $w\in\partial\mathbb{M}$ with  $d_{\mathrm{int}}^{\mathcal{T}_{\mathbb{M}}^{\varphi}}(v,w)\leq \delta< \ep$.  

This completes the proof of the desired equivalence in the case that $v\in T^kL^{\circ}$.  The sole remaining case is that $v\in T^kA^{\circ}$; this is a mirror image of the one just analyzed and is left to the reader.
 \end{proof}

\subsection{Converse stability}

Stability theorems in persistent homology generally express the continuity of persistence diagrams with respect to their geometric or algebraic input. In algebraic contexts one can often also prove an inequality in the opposite direction, see, \emph{e.g.}, \cite[Section 8]{BL}.   In our case we have:

\begin{prop}\label{convstab}
	Let $\mathcal{C}$ and $\mathcal{X}$ be admissible $\Lambda$-bounded filtered cospans, with persistence diagrams $\mathcal{D}(\mathcal{C})$ and $\mathcal{D}(\mathcal{X})$.  Then \[ d_{\mathrm{int}}^{\mathcal{T}^{\varphi}}(\mathcal{X},\mathcal{C})\leq d_{B,\mathcal{T}_{\mathbb{M}}^{\varphi}}^{\mathbb{M},\partial\mathbb{M}}(\mathcal{D}(\mathcal{X}),\mathcal{D}(\mathcal{C})).  \]
\end{prop}

\begin{proof}
	Suppose that there is an $\ep$-matching between $\mathcal{D}(\mathcal{X})$ and $\mathcal{D}(\mathcal{C})$, where $0<\ep<\infty$. Up to isomorphism in $\mathsf{HCO}^{\Lambda}(\kappa)$, we then have decompositions of $\mathcal{C}$ and $\mathcal{X}$ into standard elementary summands: \[ \mathcal{C}=\left(\bigoplus_{i\in I}\mathcal{C}_i\right)\oplus \left(\bigoplus_{j\in J_1}\mathcal{C}'_j\right),\qquad \mathcal{X}=\left(\bigoplus_{i\in I}\mathcal{X}_i\right)\oplus \left(\bigoplus_{j\in J_2}\mathcal{X}'_j\right) \] where: \begin{itemize} \item For each $i\in I$, $H_0\circ\mathcal{F}(\mathcal{C}_i)\cong B^{c_i}$ and $H_0\circ\mathcal{F}(\mathcal{X}_i)\cong B^{x_i}$ for some $c_i,x_i\in\mathbb{M}\setminus\partial\mathbb{M}$ such that $d_{\mathrm{int}}^{\mathcal{T}_{\mathbb{M}}^{\varphi}}(c_i,x_i)\leq \ep$.  (Here the block modules $B^{c_i},B^{x_i}$ are defined as in (\ref{bvdef}).)
		\item Each $H_0\circ\mathcal{F}(\mathcal{C}'_j)$ and each $H_0\circ \mathcal{F}(\mathcal{X}'_j)$ is isomorphic to some block module $B^{y_j}$ with the property that, for some $z_j\in\partial\mathbb{M}$,   $d_{\mathrm{int}}^{\mathcal{T}_{\mathbb{M}}^{\varphi}}(y_j,z_j)\leq\ep$. \end{itemize}
 The standard elementary summands $\mathcal{C}_i,\mathcal{X}_i,\mathcal{C}'_j$, $\mathcal{X}'_j$ can be read off from the corresponding $c_i,x_i,y_j,$ and $z_j$ via the table in Proposition \ref{blockclass}. 
 
 We wish to show that there is an $\ep$-interleaving between $\mathcal{X}$ and $\mathcal{C}$.  Since the functor $\mathcal{T}_{\ep}\co \mathsf{HCO}^{\Lambda}(\kappa)\to\mathsf{HCO}^{\Lambda}(\kappa)$ is additive, a direct sum of $\ep$-interleavings is an $\ep$-interleaving, so it suffices to show that the $\mathcal{C}'_j$ and $\mathcal{X}'_j$ are each $\ep$-interleaved with the zero object of $\mathsf{HCO}^{\Lambda}(\kappa)$ (having $C_{\uparrow *}=C^{\downarrow}_{*}=D_*=\{0\}$), and that each $\mathcal{X}_i$ is $\ep$-interleaved with $\mathcal{C}_i$.  
 
 First let $\mathcal{Y}$ be any of the $\mathcal{C}'_j$ or $\mathcal{X}'_j$, with $H_0\circ\mathcal{F}(\mathcal{Y})\cong B^{y_j}$ where $d_{\mathrm{int}}^{\mathcal{T}_{\ep}^{\varphi}}(y_j,z_j)\leq \ep$ for some $z_j\in\partial\mathbb{M}$.  To show that $\mathcal{Y}$ is $\ep$-interleaved with the zero object, it suffices to show that the effect of the natural transformation $\mathcal{T}_{(0\leq 2\ep)}$ on $\mathcal{Y}$ is the zero map $\mathcal{Y}\to \mathcal{T}^{\varphi}_{2\ep}\mathcal{Y}$, since then the unique maps to and from the zero object will give an $\ep$-interleaving.
 
   Proposition \ref{sepregions} implies that, for some $k$, either $y_j\in T^kL^{\circ}$ or $y_j\in T^kA^{\circ}$.  The two cases are similar; let us assume that $y_j\in T^kA^{\circ}$, so $y_j=T^k(b,2\Lambda-a)$ where $\Lambda>a>b>-\Lambda$.  By Proposition \ref{blockclass},  $\mathcal{Y}$ must then be $(\downarrow^{a}_{b})_{-k-1}$.  We have $T(b,2\Lambda-a)=(-4\Lambda+a,2\Lambda-b)$ and, for any $\delta>0$, $(\mathcal{T}_{\mathcal{M}}^{\varphi})_{\delta}(b,2\Lambda-a)= (\rho_{-\delta}^{\varphi}(b),2\Lambda-\rho_{-\delta}^{\varphi}(a))$, so the existence of a point of $\partial\mathbb{M}$ within distance $\ep$ from $y_j$ implies via Proposition \ref{bdrydist} that $2\Lambda-\rho_{-\delta}^{\varphi}(a)>2\Lambda-b$ for all $\delta>2\ep$, and hence that \[ a\leq \rho_{2\ep}^{\varphi}(b).\] 
   Now $\mathcal{T}^{\varphi}_{(0\leq 2\ep),(\downarrow_{b}^{a})_{-k-1}}$ is the cohomology class in $\mathsf{HCO}^{\Lambda}_{\kappa}$ of the element \\$I\in \mathrm{Hom}^0_{\mathsf{CO}^{\Lambda}(\kappa)}(( \downarrow_{b}^{a})_{-k-1},\mathcal{T}^{\varphi}_{2\ep}(\downarrow_{b}^{a})_{-k-1})$ that is given by the identity on the underlying chain complexes of $(\downarrow_{b}^{a})_{-k-1}$, the only nontrivial component of which is the identity map on the underlying chain complex of the elementary descending chain complex $(\mathcal{E}_{-k-1}(a,b)^{\downarrow})$.  This identity map is homotopic to $0$ by the map $K^{\downarrow}$ that sends the generator in degree $-k-1$ (with filtration level $a$) to that in degree $-k$ (with filtration level $b$),  and the fact that $a\leq \rho_{2\ep}^{\varphi}(b)$ implies that $(-K^{\downarrow},0,0,0,0)$ is well-defined as an element of $\mathrm{Hom}_{\mathsf{CO}^{\Lambda}(\kappa)}^{-1}\left(( \downarrow_{b}^{a})_{-k-1},\mathcal{T}^{\varphi}_{2\ep}(\downarrow_{b}^{a})_{-k-1}  \right)$; evidently this element has differential equal to $I$ according to the definition in (\ref{diffl}).  So the cohomology class of $I$ vanishes.
 
Thus $\mathcal{T}^{\varphi}_{(0\leq 2\ep),\mathcal{Y}}=0$ in the case that $\mathcal{Y}$ is $\mathcal{C}'_j$ or $\mathcal{X}'_j$ and $y_j\in T^k A^{\circ}$.  Similarly, in the case that $y_j\in T^kL^{\circ}$, one finds that $\mathcal{Y}=(\uparrow_{a}^{b})_{-k-1}$ with $\rho_{-2\ep}^{\varphi}(b)\leq a$, which again implies that $\mathcal{T}^{\varphi}_{(0\leq 2\ep),\mathcal{Y}}=0$.  So the summands $\mathcal{C}'_j$ and $\mathcal{X}'_j$ are each $\ep$-interleaved with the zero object of $\mathsf{HCO}^{\Lambda}(\kappa)$.

We now turn to the summands $\mathcal{C}_i$ and $\mathcal{X}_i$, with $H_0\circ\mathcal{F}(\mathcal{C}_i)\cong B^{c_i}$ and $H_0\circ\mathcal{F}(\mathcal{X}_i)\cong B^{x_i}$ where $c_i,x_i\in\mathbb{M}\setminus\partial\mathbb{M}$ and $d_{\mathrm{int}}^{\mathcal{T}_{\mathbb{M}}^{\varphi}}(c_i,x_i)\leq \ep$.  Fix the index $i\in I$.  By Proposition \ref{sepregions}, the elements $c_i$ and $x_i$ both belong to the same $T^kR$ where $R$ is one of the eight regions  
 $S^{\circ},L^{\circ}, A^{\circ},(-\Lambda,\Lambda)\times\{-\Lambda\},(-\Lambda,\Lambda)\times\{\Lambda\},\{-\Lambda\}\times (-\Lambda,\Lambda),\{\Lambda\}\times (-\Lambda,\Lambda),$ or $\{(\Lambda,-\Lambda)\}$ from (\ref{regionlist}), corresponding to the eight lines in the table in Proposition \ref{blockclass}.  In particular, the underlying chain complexes of $\mathcal{C}_i$ and $\mathcal{X}_i$ are equal to each other, though the filtrations differ.  Using Proposition \ref{blockclass} and (\ref{mflow}), the fact that $d_{\mathrm{int}}^{\mathcal{T}_{\mathbb{M}}^{\varphi}}(c_i,x_i)\leq \ep$, \emph{i.e.}, that $c_i\preceq (\mathcal{T}^{\varphi}_{\mathbb{M}})_{\ep}x_i$ and $x_i\preceq (\mathcal{T}^{\varphi}_{\mathbb{M}})_{\ep}c_i$, is easily seen in each of the eight cases to imply that the identity maps on the underlying chain complexes respect the filtrations in such a way as to give well-defined morphisms in $\mathsf{HCO}^{\Lambda}(\kappa)$ both from $\mathcal{C}_i$ to $\mathcal{T}^{\varphi}_{\ep}\mathcal{X}_i$ and from $\mathcal{X}_i$ to $\mathcal{T}^{\varphi}_{\ep}\mathcal{C}_i$.  Since the natural transformation $\mathcal{T}_{(0\leq 2\ep)}$ is also given by the identity maps on the underlying chain complexes, this yields and $\ep$-interleaving between $\mathcal{C}_i$ and $\mathcal{X}_i$.
 
Combining these $\ep$-interleavings between standard elementary summands, we conclude that the $\ep$-matching between $\mathcal{D}(\mathcal{X})$ and $\mathcal{D}(\mathcal{C})$ implies the existence of an $\ep$-interleaving between $\mathcal{X}$ and $\mathcal{C}$, from which the proposition immediately follows.  
\end{proof}

\subsection{Local stability for relative interlevel set homology}

Recall that the output of relative interlevel set homology as in \cite{BBF20},\cite{BBF21} (with coefficients in a field $\kappa$) is a functor $\mathbb{M}\to \mathsf{Vect}_{\kappa}$ vanishing on $\partial\mathbb{M}$ which is \emph{homological} in the sense of \cite[Definition C.3]{BBF21}.  When such a functor $F$ is sequentially continuous (\cite[Definition 2.4]{BBF21}), \cite[Theorem 3.21]{BBF21} asserts that $F$ decomposes as a direct sum of block modules $B^v$ (see (\ref{bvdef})) where each $v\in\mathbb{M}\setminus\partial\mathbb{M}$, and we defined the persistence diagram $\mathcal{D}(F)$ to be the multiset of elements of $\mathbb{M}\setminus\partial\mathbb{M}$ arising in this decomposition. One may then compare the interleaving and bottleneck distances between such functors and their associated persistence diagrams.  While we do not prove a full stability result in this context, we do prove the following local statement.

\begin{prop}\label{rishstab}
	Let $F\in\mathrm{Vect}_{\kappa}^{\mathbb{M}}$ be a sequentially continuous homological functor vanishing on $\partial\mathbb{M}$ whose persistence diagram $\mathcal{D}(F)$ assigns nonzero multiplicity to only finitely many points in $\mathbb{M}\setminus\partial\mathbb{M}$.  Then  there is $\delta_F>0$ such that, for any sequentially continuous homological functor $G: \mathbb{M}\to\mathsf{Vect}_{\kappa}$ that vanishes on $\partial\mathbb{M}$ and satisfies $d_{\mathrm{int}}^{\mathcal{T}_{\mathsf{Vect}_{\kappa}^{\mathbb{M}}}^{\varphi}}(G,F)< \delta_F$, 
		we have \[ d_{B,\mathcal{T}_{\mathsf{Vect}_{\kappa}^{\mathbb{M}}}^{\varphi}}^{\mathbb{M},\partial\mathbb{M}}(\mathcal{D}(G),\mathcal{D}(F))\leq d_{\mathrm{int}}^{\mathcal{T}_{\mathsf{Vect}_{\kappa}^{\mathbb{M}}}^{\varphi}} (G,F).\]		
\end{prop}

\begin{proof}
We remark first that the formula defining the maps $(\mathcal{T}_{\mathbb{M}}^{\varphi})_{\ep}\co \mathbb{M}\to\mathbb{M}$ from (\ref{mflow}) makes sense for all $\ep\in \mathbb{R}$, not just $\ep\in [0,\infty)$; below we will use these maps $(\mathcal{T}_{\mathbb{M}}^{\varphi})_{\ep}$ for all $\ep\in\R$.

Recall that the partial order on $\mathbb{M}$ is defined by saying that $(s,t)\preceq (s',t')$ iff $s\geq s'$ and $t\leq t'$.  Let us introduce the notation $(s,t)\ll (s',t')$ to signify the stronger condition that both $s>s'$ and $t<t'$.

With this said, for $v\in \mathbb{M}\setminus\partial\mathbb{M}$, the block module $B^v$ of (\ref{bvdef})  has the property that, for $w\preceq w'$, we have \begin{equation} \mathrm{rank}(B^v(w)\to B^v(w'))=\left\{\begin{array}{ll} 1 & \mbox{if }v\preceq w\preceq w'\ll T(v)\\ 0 & \mbox{otherwise}\end{array}\right..\label{blockrank1}\end{equation}

For any functor $G\co \mathbb{M}\to \mathsf{Vect}_{\kappa}$, and any $w\in \mathbb{M}$, the structure maps of $G$ make $\{G(v)|v\ll w\}$ into a directed system.  So we define \[ \underline{G}(w)=\varinjlim_{v\ll w}G(v).\]  Thus we have a canonical map $G(v)\to \underline{G}(w)$ whenever $v\ll w$.  In the special case $G=B^v$, (\ref{blockrank1}) yields, for $w\ll w'$,  
\begin{equation} \mathrm{rank}(B^v(w)\to \underline{B}^v(w'))=\left\{\begin{array}{ll} 1 & \mbox{if }v\preceq w\ll w'\preceq T(v)\\ 0 & \mbox{otherwise}\end{array}\right..\label{blockrank2}\end{equation}
Given  $w\in \mathbb{M}\setminus\partial\mathbb{M}$ and a positive number $\delta$ such that $(\mathcal{T}_{\mathbb{M}}^{\varphi})_{2\delta}w\ll Tw$, and given any functor $G\co \mathbb{M}\to\mathsf{Vect}_{\kappa}$ that vanishes on $\partial\mathbb{M}$, let \[ n_{\delta}(w,G)=\mathrm{rank}\left(G\left((\mathcal{T}^{\varphi}_{\mathbb{M}})_{\delta}w\right)\to \underline{G}\left((\mathcal{T}^{\varphi}_{\mathbb{M}})_{-\delta}Tw \right)\right).\]  (The condition on $\delta$, which by Proposition \ref{bdrydist} will hold provided that $\inf\{d_{\mathrm{int}}^{\mathcal{T}_{\mathbb{M}^{\varphi}}}(w,w')|w'\in\partial\mathbb{M}\}>\delta$, ensures that the map $G\left((\mathcal{T}^{\varphi}_{\mathbb{M}})_{\delta}w\right)\to \underline{G}\left((\mathcal{T}^{\varphi}_{\mathbb{M}})_{-\delta}Tw\right)$ is well-defined.) 
	
	Assuming $G$ to decompose as a direct sum of block modules, $G=\oplus_{i\in I}B^{w_i}$, we see from (\ref{blockrank2}) that (continuing to assume that $(\mathcal{T}_{\mathbb{M}}^{\varphi})_{2\delta}w\ll Tw$) \begin{align}\nonumber n_{\delta}(w,G)&=\#\{i\in I|w_i\preceq (\mathcal{T}^{\varphi}_{\mathbb{M}})_{\delta}w\ll (\mathcal{T}^{\varphi}_{\mathbb{M}})_{-\delta}Tw\preceq Tw_i \}\\  \nonumber &= \#\{i\in I| w_i\preceq (\mathcal{T}^{\varphi}_{\mathbb{M}})_{\delta}w\mbox{ and }w\preceq (\mathcal{T}^{\varphi}_{\mathbb{M}})_{\delta}w_i\} \\&= \#\left\{i\in I\left|d_{\mathrm{int}}^{\mathcal{T}_{\mathbb{M}}^{\varphi}}(w,w_i) \leq \delta  \right.\right\}\label{blockmult} \end{align}
	
	Let $F$ be as in the hypotheses of the proposition, and take $\delta_F>0$ sufficiently small that every point occurring with nonzero multiplicity in $\mathcal{D}(F)$ lies a distance (as measured by $d_{\mathrm{int}}^{\mathcal{T}_{\mathbb{M}}^{\varphi}}$) greater than $2\delta_F$ from $\partial\mathbb{M}$, and also that any two distinct points occurring with nonzero multiplicity in $\mathcal{D}(F)$ lie a distance greater than $2\delta_F$ from each other.  For any $v\in\mathbb{M}\setminus\partial\mathbb{M}$ let $m(v,F)$ denote the multiplicity of $v$ in the persistence diagram of $F$.  Then by Proposition \ref{bdrydist}, each $v$ occurring in the persistence diagram of $F$ has $(\mathcal{T}_{\mathbb{M}}^{\varphi})_{4\delta_F}v\preceq Tv$, so that $n_{\delta}(v,G)$ is defined for any $\delta<2\delta_F$ and any functor $G\co\mathbb{M}\to \mathrm{Vect}_{\kappa}$, and  (\ref{blockmult}) implies that \[ m(v,F)=n_{\delta}(v,F)\mbox{ whenever } m(v,F)\neq 0\mbox{ and }0<\delta<2\delta_F.  \]
	
	Now suppose that $G$, as in the hypothesis of the proposition, has $d_{\mathrm{int}}^{\mathcal{T}_{\mathsf{Vect}_{\kappa}^{\mathbb{M}}}^{\varphi}}(G,F)< \delta_F$, so that there is an $\ep$-interleaving between $F$ and $G$ for some $\ep<\delta_F$.  For any $\ep'$ with $\ep<\ep'<\delta_F$, and any $v$ occurring with nonzero multiplicity in $\mathcal{D}(F)$, this $\ep$-interleaving gives rise to factorizations 
	\[ F((\mathcal{T}_{\mathbb{M}}^{\varphi})_{\ep'-\ep}v)\to G((\mathcal{T}_{\mathbb{M}}^{\varphi})_{\ep'}v)\to \underline{G}((\mathcal{T}_{\mathbb{M}}^{\varphi})_{-\ep'}Tv)\to \underline{F}((\mathcal{T}_{\mathbb{M}}^{\varphi})_{\ep-\ep'}Tv) \] and \[  G((\mathcal{T}_{\mathbb{M}}^{\varphi})_{\ep'}v)\to F((\mathcal{T}_{\mathbb{M}}^{\varphi})_{\ep'+\ep}v)\to \underline{F}((\mathcal{T}_{\mathbb{M}}^{\varphi})_{-\ep'-\ep}Tv)\to  \underline{G}((\mathcal{T}_{\mathbb{M}}^{\varphi})_{-\ep'}Tv), \] which together imply that \[ n_{\ep'-\ep}(v,F)\leq n_{\ep'}(v,G)\leq n_{\ep'+\ep}(v,F)\]  whenever $m(v,F)\neq 0$, $\ep<\ep'<\delta$, and $F$ and $G$ are $\ep$-interleaved.  But in this situation $n_{\ep'-\ep}(v,F)=n_{\ep'+\ep}(v,F)=m(v,F)$.  Thus, decomposing $G$ as $G=\oplus_{i\in I}B^{w_i}$, we see from (\ref{blockmult}) that, if $G$ and $F$ are $\ep$-interleaved and $\ep<\ep'<\delta_F$, \[ \#\left\{i\in I\left|d_{\mathrm{int}}^{\mathcal{T}_{\mathbb{M}}^{\varphi}}(v,w_i)\leq \ep'   \right. \right\}=m(v,F)\mbox{ whenever }m(v,F)\neq 0.  \]  Meanwhile, if a point $w\in\mathbb{M}$ lies a distance greater than $\ep'$ both from all points of $\partial\mathbb{M}$ and from  all points having nonzero multiplicity in $\mathcal{D}(F)$, then the $\ep$-interleaving induces a factorization \[  G((\mathcal{T}_{\mathbb{M}}^{\varphi})_{\ep'-\ep}w)\to F((\mathcal{T}_{\mathbb{M}}^{\varphi})_{\ep'}w)\to \underline{F}((\mathcal{T}_{\mathbb{M}}^{\varphi})_{-\ep'}Tw)\to \underline{G}((\mathcal{T}_{\mathbb{M}}^{\varphi})_{\ep-\ep'}Tv)  \] where the middle map is zero by (\ref{blockmult}); this shows that $n_{\ep'-\ep}(w,G)=0$ and hence that $w$ does not appear in the persistence diagram of $G$.
	
	We conclude that when $F$ and $G$ are $\ep$-interleaved and $\ep<\ep'<\delta_F$, there are exactly $m(v,F)$ points in $\mathcal{D}(G)$, counted with multiplicity, that lie within $\ep'$ of any point $v$ with $m(v,F)\neq 0$, and that any points of $\mathcal{D}(G)$ that are not within distance $\ep'$ of any points of $\mathcal{D}(F)$ must lie within distance $\ep'$ of some point of $\partial\mathbb{M}$.  Our assumption that the distinct points of $\mathcal{D}(F)$ are further than $2\delta_F$ away from each other implies that a point of $\mathcal{D}(G)$ cannot be within $\ep'$ of two different points of $\mathcal{D}(F)$, and the assumption that the points of $\mathcal{D}(F)$ are further than $2\delta_F$ away from $\partial\mathbb{M}$ implies that   a point of $\mathcal{D}(G)$ cannot be within $\ep'$ of both a point of $\mathcal{D}(F)$ and a point of $\partial\mathbb{M}$.
	
	So we obtain an $\ep'$-matching between $\mathcal{D}(\mathcal{G})$ and $\mathcal{D}(F)$ relative to $\partial\mathbb{M}$ by assigning each point  $w$ of $\mathcal{D}(G)$ that is not within $\ep'$ of some point of $\partial\mathbb{M}$ to (one of the $m(v,F)$ copies of) the unique point $v$ of $\mathcal{D}(F)$ that is within $\ep'$ of $w$.
	
	We have thus shown that if $0<\ep<\delta_F$ and if there is an $\ep$-interleaving between $F$ and $G$ then, for all $\ep'>\ep$, there is an $\ep'$-matching between $\mathcal{D}(G)$ and $\mathcal{D}(F)$ relative to $\partial\mathbb{M}$.  The proposition now follows by taking $\ep'$ arbitarily close to $\ep$. 
\end{proof}

This quickly yields the following local stability result for filtered cospans, which we will later leverage to prove a full stability result.

\begin{cor}\label{localhco}
	Let $\mathcal{C}$ be an admissible filtered cospan such that the persistence diagram $\mathcal{D}(\mathcal{C})$ assigns nonzero multiplicity to only finitely many points of $\mathbb{M}\setminus\partial\mathbb{M}$.  Then there is $\delta_{\mathcal{C}}>0$ such that, for any other admissible filtered cospan $\mathcal{X}$ with
	 $d_{\mathrm{int}}^{\mathcal{T}^{\varphi}}(\mathcal{X},\mathcal{C})<\delta_{\mathcal{C}}$, we have
	  \[ d_{B,\mathcal{T}^{\varphi}_{\mathbb{M}}}^{\mathbb{M},\partial\mathbb{M}}(\mathcal{D}(\mathcal{X}),\mathcal{D}(\mathcal{C}))\leq d_{\mathrm{int}}^{\mathcal{T}^{\varphi}}(\mathcal{X},\mathcal{C}). \] 
\end{cor}

\begin{proof}
	Recall that, by definition, $\mathcal{D}(\mathcal{C})=\mathcal{D}(H_0\circ \mathcal{F}(\mathcal{C}))$ and likewise $\mathcal{D}(\mathcal{X})=\mathcal{D}(H_0\circ \mathcal{F}(\mathcal{X}))$.  In view of Proposition \ref{flowfunctor} and the evident analogous statement for the functor 
	$H_0\co \mathsf{K}(\kappa)^{\mathbb{M}}\to \mathsf{Vect}_{\kappa}^{\mathbb{M}}$, an $\ep$-interleaving between $\mathcal{C}$ and $\mathcal{X}$ is carried by $H_0\circ\mathcal{F}$ to an $\ep$-interleaving between $H_0\circ \mathcal{F}(\mathcal{C})$ and $H_0\circ \mathcal{F}(\mathcal{X})$.  Thus  
	 $d_{\mathrm{int}}^{\mathcal{T}^{\varphi}_{\mathsf{Vect}_{\mathbb{M}}^{\kappa}}}(H_0\circ\mathcal{F}(\mathcal{X}),H_0\circ \mathcal{F}(\mathcal{C}))\leq d_{\mathrm{int}}^{\mathcal{T}^{\varphi}}(\mathcal{X},\mathcal{C})$ (cf. \cite[Theorem 4.3]{dSMS}).  The corollary is now immediate from Proposition \ref{rishstab}, taking $\delta_{\mathcal{C}}=\delta_{H_0\circ\mathcal{F}(\mathcal{C})}$.
\end{proof}

\subsection{The stability theorem}

Let us say that a $\Lambda$-bounded filtered cospan \\ $\mathcal{C}=\left((C_{\uparrow *},\partial_{\uparrow}^C,\ell_{\uparrow}^C),(C^{\downarrow}_{*},\partial^{\downarrow}_C,\ell^{\downarrow}_C),(D_*,\partial_D),\psi_{\uparrow},\psi^{\downarrow}\right)$ is \emph{finitely-generated} if each of the chain complexes $C_{\uparrow *},C^{\downarrow}_{*},D_*$ is finitely-generated as a $\kappa$-vector space, and that $\mathcal{C}$ is of finite type if $\mathcal{C}$ is isomorphic in $\mathsf{HCO}^{\Lambda}(\kappa)$ to a finitely-generated filtered cospan.  This implies that we have a well-defined persistence diagram $\mathcal{D}(\mathcal{C})$, which is a finite multiset of elements of $\mathbb{M}\setminus\partial\mathbb{M}$.

Here is our main stability theorem:

\begin{theorem}\label{mainstab}
	Let $\mathcal{C}$ and $\mathcal{X}$ be two $\Lambda$-bounded filtered cospans that are of finite type.  Then $d_{B,\mathcal{T}_{\mathbb{M}}^{\varphi}}^{\mathbb{M},\partial\mathbb{M}}(\mathcal{D}(\mathcal{X}),\mathcal{D}(\mathcal{C}))\leq d_{\mathrm{int}}^{\mathcal{T}^{\varphi}}(\mathcal{X},\mathcal{C})$.
\end{theorem}

\begin{cor}\label{fnstable} In the context of Example \ref{pinsing},
	let $f,g\co\mathbb{X}\to  [-\Lambda,\Lambda]$ with $f^{-1}(\{\pm\Lambda\})=g^{-1}(\{\pm\Lambda\})=\partial^{\pm}\mathbb{X}$ be such that the pinned singular filtered cospans $\mathcal{S}_{\partial}(\mathbb{X},f;\kappa)$ and $\mathcal{S}_{\partial}(\mathbb{X},g;\kappa)$ are of finite type.  Then \[ d_{B,\mathcal{T}_{\mathbb{M}}^{\varphi}}^{\mathbb{M},\partial\mathbb{M}}\left(\mathcal{D}(\mathcal{S}_{\partial}(\mathbb{X},f;\kappa)),\mathcal{D}(\mathcal{S}_{\partial}(\mathbb{X},g;\kappa))\right)\leq \sup_{\mathbb{X}\setminus\partial\mathbb{X}}\left|\varphi^{-1}\circ f-\varphi^{-1}\circ g\right|. \]
\end{cor}
For example, the finite type condition would hold if $f,g$ are Morse functions on a compact manifold, or simplexwise linear functions on a finite simplicial complex.
\begin{proof}[Proof of Corollary \ref{fnstable}] This follows immediately from Theorem \ref{mainstab} and Example \ref{function-interleave}.\end{proof}

The main ingredients in the proof of Theorem \ref{mainstab} are the local result Corollary \ref{localhco} and the upcoming \textbf{interpolation lemma}.  This follows a well-worn strategy for proving stability results in persistent homology by combining a local stability theorem with the existence of suitable interpolations, see, \emph{e.g.}, \cite[Section 3.3]{CEH07}, \cite[Section 5.6]{CDGO},\cite[Section 9]{UZ}.  Our specific approach adapts an argument in \cite{UZ}, with additional assistance from Proposition \ref{symm}.\footnote{In fact, by using Proposition \ref{symm} in the same way as we do below, one can strengthen the  stability result \cite[Theorem 1.4(i)]{UZ} to an isometry theorem, removing the factor of $2$ that appears therein.}  Recall Definition \ref{implantdfn} of an $\ep$-implantation as a weaker, asymmetric version of an $\ep$-interleaving.

\begin{lemma}[Interpolation Lemma] \label{interp}
	Let $\ep>0$, let $\mathcal{C}$ and $\mathcal{X}$ be  $\Lambda$-bounded filtered cospans, and suppose that there is an $\ep$-implantation in $\mathsf{HCO}^{\Lambda}(\kappa)$ of $\mathcal{C}$ into $\mathcal{X}$.  Then there is a one-parameter family $\{\mathcal{Y}_t\}_{t\in [0,1]}$ of finite-type $\Lambda$-bounded filtered cospans such that:
	\begin{itemize}
		\item[(i)] For any $s,t\in [0,1]$ with $s\leq t$, $\mathcal{Y}_s$ and $\mathcal{Y}_t$ are $\ep(t-s)$-interleaved;
		\item[(ii)] $\mathcal{Y}_0$ is isomorphic in $\mathsf{HCO}^{\Lambda}(\kappa)$ to $\mathcal{X}$; and
		\item[(iii)] For some finite-type $\Lambda$-bounded filtered cospan $\mathcal{B}$, $\mathcal{Y}_1$ is isomorphic in $\mathsf{HCO}^{\Lambda}(\kappa)$ to $\mathcal{B}\oplus\mathcal{C}$.
	\end{itemize}
	\end{lemma}
\begin{proof}[Proof of Theorem \ref{mainstab}, assuming Lemma \ref{interp}] Given $\mathcal{C},\mathcal{X}$ as in the statement of the theorem, for $\ep$ arbitrarily close to $d_{\mathrm{int}}^{\mathcal{T}^{\varphi}}(\mathcal{X},\mathcal{C})$ there is an $\ep$-interleaving between $\mathcal{C}$ and $\mathcal{X}$.  This $\ep$-interleaving may be interpreted both as an $\ep$-implantation of $\mathcal{C}$ into $\mathcal{X}$ and as an $\ep$-implantation of $\mathcal{X}$ into $\mathcal{C}$.  Regarding it first as an $\ep$-implantation of $\mathcal{C}$ into $\mathcal{X}$, let $\mathcal{Y}_t$ be as given by Lemma \ref{interp}. For $t\in [0,1]$ write $\delta_t$ for the number denoted in Corollary \ref{localhco} by $\delta_{\mathcal{Y}_t}$.  By taking a finite subcover of the open cover $\{(t-\delta_t,t+\delta_t)\}_{t\in [[0,1]}$ of $[0,1]$, we find numbers $t_0,\ldots,t_m$ with $0=t_0<\cdots<t_i<t_{i+1}<\cdots<t_m=1$ such that $\ep(t_{i+1}-t_i)<\delta_{t_i}$.  So since $\mathcal{Y}_{t_i}$ and $\mathcal{Y}_{t_{i+1}}$ are $\ep(t_{i+1}-t_i)$-interleaved, Corollary \ref{localhco} shows that $d_{B,\mathcal{T}_{\mathbb{M}}^{\varphi}}^{\mathbb{M},\partial\mathbb{M}}(\mathcal{D}(\mathcal{Y}_{i+1}),\mathcal{D}(\mathcal{Y}_i))\leq \ep(t_{i+1}-t_i)$ for $i=0,\ldots,m-1$.  So by the triangle inequality for the bottleneck distance, $d_{B,\mathcal{T}_{\mathbb{M}}^{\varphi}}^{\mathbb{M},\partial\mathbb{M}}(\mathcal{D}(\mathcal{Y}_{1}),\mathcal{D}(\mathcal{Y}_0))\leq \ep$, \emph{i.e.},
 \[  d_{B,\mathcal{T}_{\mathbb{M}}^{\varphi}}^{\mathbb{M},\partial\mathbb{M}}\left(\mathcal{D}(\mathcal{B}\oplus\mathcal{C}),\mathcal{D}(\mathcal{X})\right)\leq \ep,\] where $\mathcal{B}$ is as in Lemma \ref{interp}(iii).  Now the persistence diagram of $\mathcal{B}\oplus\mathcal{C}$ is the disjoint union of those of $\mathcal{C}$ and of $\mathcal{B}$, so a $\ep$-matching between  $\mathcal{D}(\mathcal{B}\oplus\mathcal{C})$ and $\mathcal{D}(\mathcal{X})$ relative to $\partial\mathbb{M}$ restricts to an $\ep$-embedding of $\mathcal{D}(\mathcal{C})$ into $\mathcal{D}(\mathcal{X})$ relative to $\partial \mathbb{M}$.  Hence, with notation as in Section \ref{flowmetrics}, $\delta_{B,\mathcal{T}_{\mathbb{M}}}^{M,\partial\mathbb{M}}(\mathcal{D}(\mathcal{C}),\mathcal{D}(\mathcal{X}))\leq \ep$. 

But since our $\ep$-interleaving also serves as an $\ep$-implantation of $\mathcal{X}$ into $\mathcal{C}$, we may reverse the roles of $\mathcal{C}$ and $\mathcal{X}$ in the above argument to see that also $\delta_{B,\mathcal{T}_{\mathbb{M}}^{\varphi}}^{\mathbb{M},\partial\mathbb{M}}(\mathcal{D}(\mathcal{X}),\mathcal{D}(\mathcal{C}))\leq \ep$. Since $\ep$ is arbitrarily close to $d_{\mathrm{int}}^{\mathcal{T}^{\varphi}}(\mathcal{X},\mathcal{C})$, it now follows from  Proposition \ref{symm} that
\[	d_{B,\mathcal{T}_{\mathbb{M}}^{\varphi}}^{\mathbb{M},\partial\mathbb{M}}(\mathcal{D}(\mathcal{X}),\mathcal{D}(\mathcal{C}))=\max\left\{\delta_{B,\mathcal{T}_{\mathbb{M}}^{\varphi}}^{\mathbb{M},\partial\mathbb{M}}(\mathcal{D}(\mathcal{X}),\mathcal{D}(\mathcal{C})),\delta_{B,\mathcal{T}_{\mathbb{M}}^{\varphi}}^{\mathbb{M},\partial\mathbb{M}}(\mathcal{D}(\mathcal{C}),\mathcal{D}(\mathcal{X}))  \right\}\leq d_{\mathrm{int}}^{\mathcal{T}^{\varphi}}(\mathcal{X},\mathcal{C}).\]
\end{proof}

\begin{remark}
If one ignores the last paragraph of the above proof, it also serves as a proof that, for finite-type $\Lambda$-bounded filtered cospans $\mathcal{C}$ and $\mathcal{X}$, \[ \delta_{B,\mathcal{T}_{\mathbb{M}}^{\varphi}}^{\mathbb{M},\partial\mathbb{M}}(\mathcal{D}(\mathcal{C}),\mathcal{D}(\mathcal{X}))\leq \delta_{\mathrm{int}}^{\mathcal{T}^{\varphi}}(\mathcal{C},\mathcal{X}):=\inf\{\ep|\exists\mbox{ an $\ep$-implantation of }\mathcal{C}\mbox{ into }\mathcal{X}\}.\]  Thus, roughly speaking, if $\mathcal{C}$ is approximately a retract of $\mathcal{X}$ then the persistence diagram of $\mathcal{C}$ is approximately a subset of that of $\mathcal{X}$.
\end{remark}

The rest of this subsection is devoted to the proof of Lemma \ref{interp}. Let \\ $\mathcal{C}=\left((C_{\uparrow *},\partial_{\uparrow}^C,\ell_{\uparrow}^C),(C^{\downarrow}_{*},\partial^{\downarrow}_C,\ell^{\downarrow}_C),(D_*,\partial_D),\psi_{\uparrow},\psi^{\downarrow}\right)$ and $\mathcal{X}=\left((X_{\uparrow *},\partial_{\uparrow}^X,\ell_{\uparrow}^X),(X^{\downarrow}_{*},\partial^{\downarrow}_X,\ell^{\downarrow}_X),(Y_*,\partial_Y),\phi_{\uparrow},\phi^{\downarrow}\right)$ admit an $\ep$-implantation of $\mathcal{C}$ into $\mathcal{X}$ in $\mathsf{HCO}^{\Lambda}_{\kappa}$.  Thus there are morphisms \[ \frak{a}=(\alpha^{\downarrow},\alpha_{\uparrow},\alpha,K^{\downarrow},K_{\uparrow})\co \mathcal{C}\to \mathcal{T}^{\ep}\mathcal{X} \] and  \[ \frak{b}=(\beta^{\downarrow},\beta_{\uparrow},\beta,L^{\downarrow},L_{\uparrow})\co\mathcal{T}^{\ep}\mathcal{X}\to\mathcal{T}^{2\ep}\mathcal{C} \] such that $\frak{b}\circ\frak{a}$ is equal in $\mathsf{HCO}^{\Lambda}(\kappa)$ to the map $(1_{C^{\downarrow}_{*}},1_{C_{\uparrow *}},1_{D_*},0,0)\co \mathcal{C}\to\mathcal{T}^{2\ep}\mathcal{C}$. Thus in particular the chain maps $\alpha^{\downarrow}\co \co C^{\downarrow}_{*}\to X^{\downarrow}_{*}$, $\alpha_{\uparrow}\co C_{\uparrow *}\to X_{\uparrow *}$, $\beta^{\downarrow}\co X^{\downarrow}_{*}\to C^{\downarrow}_{*}$, and $\beta_{\uparrow}\co X_{\uparrow *}\to C_{\uparrow *}$ satisfy \[ \ell_{X}^{\downarrow}\circ\alpha^{\downarrow}\geq \rho_{-\ep}^{\varphi}\circ \ell_{C}^{\downarrow},\quad \ell^{X}_{\uparrow}\circ\alpha_{\uparrow}\leq \rho_{\ep}^{\varphi}\circ \ell^{C}_{\uparrow},\quad \ell_{C}^{\downarrow}\circ\beta^{\downarrow}\geq \rho_{-\ep}^{\varphi}\circ \ell_{X}^{\downarrow},\quad \ell^{C}_{\uparrow}\circ\beta_{\uparrow}\leq \rho_{\ep}^{\varphi}\circ \ell^{X}_{\uparrow},\]
and
there are maps $H^{\downarrow}\co C^{\downarrow}_{ *}\to C^{\downarrow}_{*+1}$, $H_{\uparrow}\co C_{\uparrow *}\to C_{\uparrow *+1}$
such that\[ 1_{C^{\downarrow}_{*}}-\beta^{\downarrow}\alpha^{\downarrow}=\partial^{\downarrow}_{C}H^{\downarrow}+H^{\downarrow}\partial^{\downarrow}_{C},\quad 1_{C_{\uparrow*}}-\beta_{\uparrow}\alpha_{\uparrow}=\partial_{\uparrow}^{C}H_{\uparrow}+H_{\uparrow}\partial_{\uparrow}^{C},  \] and \[ \ell^{\downarrow}_{C}\circ H^{\downarrow}\geq \rho_{-2\ep}^{\varphi}\circ \ell^{\downarrow}_{C},\quad \ell_{\uparrow}^{C}\circ H_{\uparrow}\leq \rho_{2\ep}^{\varphi}\circ\ell^{C}_{\uparrow}.\]

We will construct the interpolating family $\mathcal{Y}_t$ out of the \textbf{mapping cylinders} of $\alpha^{\downarrow}$ and $\alpha_{\uparrow}$. In general, if 
$f\co (A_*,\partial_A)\to (B_*,\partial_B)$ is a chain map between two chain complexes, the mapping cylinder of $f$ is defined to be the chain complex
 $(\mathrm{Cyl}(f)_*,\partial_{\mathrm{Cyl}})$ with $ \mathrm{Cyl}(f)_k=A_{k-1}\oplus B_k\oplus A_{k+1}$ and with $\partial_{\mathrm{Cyl}}$ given in block form with respect to this direct sum decomposition by
	 \[ \partial_{\mathrm{Cyl}}=\left(\begin{array}{ccc} -\partial_{A} & 0 & 0 \\ f & \partial_B & 0 \\ -1_A & 0 & \partial_A    \end{array}\right). \] 
	  The inclusions  $j_B\co B_*\to \mathrm{Cyl}(f)_*$ and $j_A\co A_*\to\mathrm{Cyl}(f)_*$ as the second and third summands are thus chain maps.  Moreover the map
	   \begin{equation}  \label{pbdef} p_B\co \mathrm{Cyl}(f)_*\to B_* \qquad p_B(a,b,c)=b+f(c) \end{equation}
	    is a chain homotopy inverse to $j_B$: we clearly have $p_Bj_B=1_{B_*}$, while the map $G\co \mathrm{Cyl}(f)_*\to \mathrm{Cyl}(f)_{*+1}$ defined by
	 \begin{equation}\label{cylhtopy} G(a,b,c)=(-c,0,0)\mbox{ satisfies }1_{B_*}-j_Bp_B=\partial_{\mathrm{Cyl}}G+G\partial_{\mathrm{Cyl}}.\end{equation}

Each member of our family $\mathcal{Y}_t$ will have as its underlying diagram of chain complexes the data \begin{equation}\label{cyldiag} \xymatrix{ \mathrm{Cyl}(\alpha^{\downarrow})_* \ar[dr]^{\phi^{\downarrow}\circ p_{X^{\downarrow}}}  & \\ & Y_* \\ \mathrm{Cyl}(\alpha_{\uparrow})_* \ar[ur]_{\phi_{\uparrow}\circ p_{X_{\uparrow}}} } \end{equation} Given this, to complete the specification of the $\mathcal{Y}_t$ we need to define ascending and descending filtration functions
$\ell^{t}_{\uparrow}\co \mathrm{Cyl}(\alpha_{\uparrow})_*\to\R\cup\{-\infty\}$ and $\ell^{\downarrow}_{t}\co \mathrm{Cyl}(\alpha^{\downarrow})_*\to\R\cup\{\infty\}$ for $t\in [0,1]$.  We define these to map $0$ to $-\infty$ and $\infty$, respectively, and to act on $\mathrm{Cyl}(\alpha_{\uparrow})_*\setminus\{0\}$ and $\mathrm{Cyl}(\alpha^{\downarrow})_*\setminus\{0\}$ by (recalling our increasing homeomorphism $\varphi\co [-\infty,\infty]\to [-\Lambda,\Lambda]$)
	 \begin{equation}\label{ltdef} \ell_{\uparrow}^{t}=\varphi\circ\left((1-t)(\varphi^{-1}\circ\ell_{\uparrow}^{0})+t(\varphi^{-1}\circ\ell_{\uparrow}^{1})\right),\qquad \ell^{\downarrow}_{t}=\varphi\circ\left((1-t)(\varphi^{-1}\circ\ell^{\downarrow}_{0})+t\varphi^{-1}\circ\ell^{\downarrow}_{1}\right)\end{equation} where
	 \[ \ell_{\uparrow}^{0}(a,b,c)=\max\left\{\rho_{\ep}(\ell^{C}_{\uparrow}(a)), \ell^{X}_{\uparrow}(b) , \rho_{\ep}^{\varphi}(\ell^{C}_{\uparrow}(c)) \right\}, \qquad  \ell_{\uparrow}^{1}(a,b,c)=\max\left\{\rho_{2\ep}(\ell^{C}_{\uparrow}(a)), \rho_{\ep}(\ell^{X}_{\uparrow}(b)) , \ell^{C}_{\uparrow}(c) \right\}  \] 
	 for $a\in C_{\uparrow\,k-1}$, $b\in X_{\uparrow k}$, $c\in C_{\uparrow k}$, and similarly 
	 \[ \ell^{\downarrow}_{0}(a,b,c)=\min\left\{\rho_{-\ep}(\ell_{C}^{\downarrow}(a)), \ell_{X}^{\downarrow}(b) , \rho_{-\ep}^{\varphi}(\ell_{C}^{\downarrow}(c)) \right\}, \qquad  \ell^{\downarrow}_{1}(a,b,c)=\min\left\{\rho_{-2\ep}(\ell_{C}^{\downarrow}(a)), \rho_{-\ep}(\ell_{X}^{\downarrow}(b)) , \ell_{C}^{\downarrow}(c) \right\} \] 
	 for $a\in C^{\downarrow}_{k-1},b\in X^{\downarrow}_{k},c\in C_{k}^{\downarrow}$.

	The facts that $\ell^{X}_{\uparrow}\circ\alpha_{\uparrow}\leq \rho_{\ep}^{\varphi}\circ\ell_{\uparrow}^{C}$ and $\ell^{\downarrow}_X\circ\alpha^{\downarrow}\geq \rho_{-\ep}^{\varphi}\circ\ell_{C}^{\downarrow}$
	imply that, for both $i=0,1$, we have $\ell_{\uparrow}^i\circ\partial_{\mathrm{Cyl}}\leq \ell_{\uparrow}^{i}$ on $\mathrm{Cyl}(\alpha_{\uparrow})$ and $\ell^{\downarrow}_{i}\circ\partial_{\mathrm{Cyl}}\geq \ell^{\downarrow}_{i}$ on $\mathrm{Cyl}(\alpha^{\downarrow})$.  Hence the same relations hold with the $\ell_{\uparrow}^{i}$  and $\ell^{\downarrow}_i$ replaced by the conjugated versions of their convex combinations as in (\ref{ltdef}).  So for each $t\in [0,1]$
	we have an ascending chain complex $(\mathrm{Cyl}(\alpha_{\uparrow})_*,\ell_{\uparrow}^{t})$ and a descending chain complex $(\mathrm{Cyl}(\alpha^{\downarrow})_*,\ell^{\downarrow}_t)$; the filtered cospan $\mathcal{Y}_t$ is defined to be given by the diagram (\ref{cyldiag}) with these ascending and descending chain complex structures on 
	$\mathrm{Cyl}(\alpha_{\uparrow})_*$ and $\mathrm{Cyl}(\alpha^{\downarrow})_*$.
	
	By construction, we have \[ \rho^{\varphi}_{-\ep}\circ \ell_{\uparrow}^{0}\leq \ell_{\uparrow}^{1}\leq \rho_{\ep}^{\varphi}\circ\ell_{\uparrow}^{0},\qquad \rho_{\ep}^{\varphi}\circ\ell^{\downarrow}_{0}\geq \ell^{\downarrow}_{1}\geq \rho_{-\ep}^{\varphi}\circ\ell^{\downarrow}_{0}; \] equivalently \[ -\ep+\varphi^{-1}\circ\ell_{\uparrow}^{0}\leq \varphi^{-1}\circ\ell_{\uparrow}^{1}\leq \ep+\varphi^{-1}\circ\ell_{\uparrow}^{0},\qquad \ep+\varphi^{-1}\circ\ell^{\downarrow}_{0}\geq \varphi^{-1}\circ\ell^{\downarrow}_{1}\geq -\ep+\varphi^{-1}\circ\ell^{\downarrow}_{0}.   \]
	If $s,t\in [0,1]$ with $s\leq t$, we have \[ \left|\varphi^{-1}\circ \ell_{\uparrow}^{t}-\varphi^{-1}\circ \ell_{\uparrow}^{s}\right|=\left|(t-s)(\varphi^{-1}\circ \ell_{\uparrow}^{1}-\varphi^{-1}\circ \ell_{\uparrow}^{0})\right|\leq \ep|t-s|.  \] Hence 
	 \[ \varphi^{-1}\circ \ell_{\uparrow}^{s}\geq \ep(s-t)+\varphi^{-1}\circ\ell_{\uparrow}^{t}, \qquad \varphi^{-1}\circ\ell_{\uparrow}^{t}\geq \ep(s-t)+\varphi^{-1}\circ\ell_{\uparrow}^{s}.\]  Then postcomposing each side of these inequalities with the increasing map $\varphi$ yields \[ \ell_{\uparrow}^{s}\geq \rho_{\ep(s-t)}^{\varphi}\circ\ell_{\uparrow}^{t},\qquad \ell_{\uparrow}^{t}\geq \rho_{\ep(s-t)}^{\varphi}\circ\ell_{\uparrow}^{s}.  \]  An identical analysis yields, likewise, that \[ \ell^{\downarrow}_{s}\leq \rho_{\ep(t-s)}^{\varphi}\circ\ell^{\downarrow}_{t},\qquad \ell^{\downarrow}_{t}\leq \rho_{\ep(t-s)}^{\varphi}\circ\ell^{\downarrow}_{s}\]
	It follows from this that the identity maps on $\mathrm{Cyl}(\alpha_{\uparrow})_*,\mathrm{Cyl}(\alpha^{\downarrow})_*,Y_*$  provide morphisms in $\mathsf{HCO}^{\Lambda}(\kappa)$ from $\mathcal{Y}_s$ to $\mathcal{T}_{\ep(t-s)}\mathcal{Y}_t$ and from  from $\mathcal{Y}_t$ to $\mathcal{T}_{\ep(t-s)}\mathcal{Y}_s$.  This evidently provides an $\ep(t-s)$-interleaving between $\mathcal{Y}_s$ and $\mathcal{Y}_t$, proving property (i) in Lemma \ref{interp}.
	
	We now prove property (ii), that $\mathcal{Y}_0\cong \mathcal{X}$.  We have inclusions $j_{X_{\uparrow}}\co X_{\uparrow *}\to \mathrm{Cyl}(\alpha_{\uparrow})_*$ and $j_{X^{\downarrow}}\co \mathrm{Cyl}(\alpha^{\downarrow})_*$ and retractions $p_{X_{\uparrow}}\co \mathrm{Cyl}(\alpha_{\uparrow})_*\to X_{\uparrow *}$ and $p_{X^{\downarrow}}\co \mathrm{Cyl}(\alpha^{\downarrow})_*\to X^{\downarrow}_{*}$ as in (\ref{pbdef}); these resepct the ascending and descending filtrations given by $\ell_{\uparrow}^{0}$ and $\ell^{\downarrow}_{0}$ because $\ell_{\uparrow}^{X}\circ\alpha_{\uparrow}\leq \rho_{\ep}^{\varphi}\circ\ell_{\uparrow}^{C}$ and $\ell^{\downarrow}_{X}\circ\alpha^{\downarrow}\geq \rho_{-\ep}^{\varphi}\circ\ell^{\downarrow}_{C}$.  Moreover $p_{X_{\uparrow}}\circ j_{X_{\uparrow}}$ and $p_{X^{\downarrow}}\circ j_{X^{\downarrow}}$ are the respective identities, and $j_{X_{\uparrow}}\circ p_{X_{\uparrow}}$ and $j^{X^{\downarrow}}\circ p_{X^{\downarrow}}$ are homotopic to the respective identities by maps as in (\ref{cylhtopy}) that also respect the filtrations given by $\ell_{\uparrow}^{0}$ and $\ell^{\downarrow}_{0}$.  Hence, with $\mathrm{Cyl}(\alpha_{\uparrow})_*$ and $\mathrm{Cyl}(\alpha^{\downarrow})_*$ viewed as ascending and descending chain complexes using  $\ell_{\uparrow}^{0}$ and $\ell^{\downarrow}_{0}$, the maps $j_{X_{\uparrow}}$ and $j_{X^{\downarrow}}$ are filtered homotopy equivalences.  Thus we obtain an isomorphism in $\mathsf{HCO}^{\Lambda}(\kappa)$, \[
	\xymatrix{ X^{\downarrow}_{*} \ar[rd]_{\phi^{\downarrow}}\ar[rr]^{j_{X^{\downarrow}}} & & 
	 \mathrm{Cyl}(\alpha^{\downarrow})_* \ar[dr]^{\phi^{\downarrow}\circ p_{X^{\downarrow}}}  & \\  & Y_*\ar@{=}[rr] &  & Y_* \\ X_{\uparrow *} \ar[ru]^{\phi_{\uparrow}} \ar[rr]^{j_{X_{\uparrow}}} &&\mathrm{Cyl}(\alpha_{\uparrow})_* \ar[ur]_{\phi_{\uparrow}\circ p_{X_{\uparrow}}}
	} 
\]	between $\mathcal{X}$ and $\mathcal{Y}_0$, proving (ii).
	
The filtered cospan $\mathcal{B}$ in part (iii) is given by \[  \xymatrix{ \mathrm{Cone}(\alpha^{\downarrow})_*\ar[rd] & \\ & 0 \\ \mathrm{Cone}(\alpha_{\uparrow})_*\ar[ru] &   }  \] with the filtration functions \[ \ell_{\uparrow}(c,x)=\max\{\rho_{2\ep}(\ell_{\uparrow}^{C}(x)),\rho_{\ep}(\ell_{\uparrow}^{X}(x))\}, \qquad  \ell^{\downarrow}(c,x)=\max\{\rho_{-2\ep}(\ell^{\downarrow}_{C}(x)),\rho_{-\ep}(\ell^{\downarrow}_{X}(x))\} \] on $\mathrm{Cone}(\alpha_{\uparrow})_*$ and $\mathrm{Cone}(\alpha^{\downarrow})_*$, respectively.  Thus, as a filtered vector space, $\mathrm{Cone}(\alpha_{\uparrow})_*\oplus C_{\uparrow *}$ is identical to $\mathrm{Cyl}(\alpha)_*$ with the filtration $\ell_{\uparrow}^{1}$, but the boundary operator on $\mathrm{Cone}(\alpha_{\uparrow})\oplus C_{\uparrow *}$ is given in block form by $\left(\begin{array}{ccc}-\partial_{\uparrow}^{C} & 0 & 0 \\ \alpha_{\uparrow} & \partial_{\uparrow}^{X} & 0 \\ 0 & 0 & \partial_{\uparrow}^{C}\end{array}\right)$ while that on 
$\mathrm{Cyl}(\alpha)_*$ is given by 	$\left(\begin{array}{ccc}-\partial_{\uparrow}^{C} & 0 & 0 \\ \alpha_{\uparrow} & \partial_{\uparrow}^{X} & 0 \\ -1_{C_{\uparrow *}} & 0 & \partial_{\uparrow}^{C}\end{array}\right)$.  A similar remark applies to $\mathrm{Cone}(\alpha^{\downarrow})\oplus C^{\downarrow}_{*}$ and $\mathrm{Cyl}(\alpha^{\downarrow})_*$.    

Because the homotopy $H_{\uparrow}$ between $\beta_{\uparrow}\circ\alpha_{\uparrow}$ and $1_{C_{\uparrow *}}$ satisfies $\ell_{\uparrow}^{C}\circ H_{\uparrow}\leq \rho_{2\ep}^{\varphi}\circ\ell_{\uparrow}^{C}$, and because $\ell_{\uparrow}^{C}\circ\beta_{\uparrow}\leq \rho_{\ep}^{\varphi}\circ\ell_{\uparrow}^{X}$, the map  given in block form by \[ A_{\uparrow}=\left(\begin{array}{ccc} 1_{C_{\uparrow *-1}} &  0 & 0 \\ 0 & 1_{X_{\uparrow *}} & 0 \\ -H_{\uparrow} & \beta_{\uparrow} & 1_{C_{\uparrow *}}   \end{array}\right) \] respects the filtration given by $\ell_{\uparrow}^{1}$, and one sees that it intertwines the boundary operators to give a chain map $\mathrm{Cyl}(\alpha_{\uparrow})_*\to \mathrm{Cone}(\alpha_{\uparrow})_*\oplus C_{\uparrow *}$.  Moreover $A_{\uparrow}$ is invertible, and its inverse $A_{\uparrow}^{-1}=\left(\begin{array}{ccc} 1_{C_{\uparrow *-1}} &  0 & 0 \\ 0 & 1_{X_{\uparrow *}} & 0 \\ H_{\uparrow} & -\beta_{\uparrow} & 1_{C_{\uparrow *}}   \end{array}\right)$ likewise gives a filtered chain map $\mathrm{Cone}(\alpha_{\uparrow})_*\oplus C_{\uparrow *}\to \mathrm{Cyl}(\alpha_{\uparrow})_*$.    So $A_{\uparrow}\co \mathrm{Cyl}(\alpha_{\uparrow})_*\to \mathrm{Cone}(\alpha_{\uparrow})_*\oplus C_{\uparrow *}$ is a filtered chain isomorphism with respect to the filtration $\ell_{\uparrow}^{1}$ on $\mathrm{Cyl}(\alpha_{\uparrow *})$ and the above filtration $\ell_{\uparrow}$ on $\mathrm{Cone}(\alpha_{\uparrow})_*$.  Likewise one has a filtered chain isomorphism \[  A^{\downarrow}= \left(\begin{array}{ccc}1_{C^{\downarrow}_{*+1}} & 0 & 0 \\ 0 & 1_{X^{\downarrow}_*} & 0 \\ -H^{\downarrow} & \beta^{\downarrow} & 1_{C^{\downarrow}_*}\end{array}\right)\co \mathrm{Cyl}(\alpha^{\downarrow})_*\to \mathrm{Cone}(\alpha^{\downarrow})_*\oplus C^{\downarrow}_{*}.    \]  Let $\pi^{\downarrow}\co \mathrm{Cone}(\alpha^{\downarrow})_*\oplus C^{\downarrow}_{*}\to C^{\downarrow}_{*}$ and $\pi_{\uparrow}\co \mathrm{Cone}(\alpha_{\uparrow})_{*}\oplus C_{\uparrow *}\to C_{\uparrow *}$ denote the projections, so that the filtered cospan $\mathcal{B}\oplus \mathcal{C}$ is given by \[ \xymatrix{
\mathrm{Cone}(\alpha^{\downarrow})_*\oplus C^{\downarrow}_{*} \ar[rd]^>>>>>>>>>{\psi^{\downarrow}\circ\pi^{\downarrow}} & \\ & D_* \\  \mathrm{Cone}(\alpha_{\uparrow})_*\oplus C_{\uparrow *} \ar[ru]_>>>>>>>>>{\psi_{\uparrow}\circ\pi_{\uparrow}} & }  \]  We claim that the diagram \begin{equation}\label{lastiso} \xymatrix{ \mathrm{Cyl}(\alpha^{\downarrow})_*\ar[rd]_>>>>>>>>>{\phi^{\downarrow}\circ p_{X^{\downarrow}}} \ar[rr]^{A^{\downarrow}} & &  \mathrm{Cone}(\alpha^{\downarrow})_*\oplus C^{\downarrow}_{*} \ar[rd]^>>>>>>>>>{\psi^{\downarrow}\circ\pi^{\downarrow}} & \\ & Y_*\ar[rr]^{\beta} & & D_* \\ \mathrm{Cyl}(\alpha_{\uparrow})_*\ar[ru]^>>>>>>>>>>{\phi_{\uparrow}\circ p_{X_{\uparrow}}} \ar[rr]_{A_{\uparrow}} & &  \mathrm{Cone}(\alpha_{\uparrow})_*\oplus C_{\uparrow *} \ar[ru]_>>>>>>>>>{\psi_{\uparrow}\circ\pi_{\uparrow}} & }      \end{equation}
  extends to an isomorphism $\mathcal{Y}_1\to \mathcal{B}\oplus \mathcal{C}$.  Since $A^{\downarrow}$ and $A_{\uparrow}$ are filtered homotopy equivalences and $\beta$ is a homotopy equivalence, Corollary \ref{hcoiso} shows that this holds provided that the diagram commutes up to (unfiltered) chain homotopy.  To see this, observe first that the map $K_{\uparrow}\co \mathrm{Cyl}(\alpha_{\uparrow})_*\to C_{\uparrow *+1}$ defined by $K_{\uparrow}(a,b,c)=H_{\uparrow}c$ satisfies $\partial_{\uparrow}^{C}K_{\uparrow}+K_{\uparrow}\partial_{\mathrm{Cyl}}=\pi_{\uparrow}\circ A_{\uparrow}-\beta_{\uparrow}\circ p_{X_{\uparrow}}$.  So $\psi_{\uparrow}\circ\pi_{\uparrow}\circ A_{\uparrow}$ is chain homotopic to $\psi_{\uparrow}\circ \beta_{\uparrow}\circ p_{X_{\uparrow}}$. But $\psi_{\uparrow}\circ\beta_{\uparrow}$ is chain homotopic to $\beta\circ\phi_{\uparrow}$ due to $\mathfrak{b}=(\beta^{\downarrow},\beta_{\uparrow},\beta,L^{\downarrow},L_{\uparrow})$ being a morphism between $\mathcal{T}^{\ep}\mathcal{X}$ and $\mathcal{T}^{2\ep}\mathcal{C}$. Thus $\psi_{\uparrow}\circ\pi_{\uparrow}\circ A_{\uparrow}$ is chain homotopic to $\beta\circ\phi_{\uparrow}\circ p_{X_{\uparrow}}$, and the lower parallelogram of (\ref{lastiso}) commutes up to homotopy.  The upper parallelogram commutes up to homotopy by the same argument.  So Corollary \ref{hcoiso} applies to give the desired isomorphism $\mathcal{Y}_1\cong \mathcal{B}\oplus\mathcal{C}$, completing the proof of Lemma \ref{interp} and hence of Theorem \ref{mainstab}.
  \begin{flushright} $\square$ \end{flushright}

\appendix

\section{The Morse-singular equivalence}\label{app}

The purpose of this appendix is to establish Lemma \ref{Efiltered}, concerning the behavior with respect to the natural filtrations of the homotopy equivalence from \cite{Paj} between the Morse and singular chain complexes. This first requires a detailed review of how this homotopy equivalence is constructed. 

Given a field  $\kappa$ and a compact smooth manifold $\mathbb{Y}$ with boundary, let $h\co \mathbb{Y}\to [a,b]$ be a Morse function such that $a$ and $b$ are regular values, and let $w$ be a gradient-like vector field for $h$ whose flow is Morse-Smale.  For the sake of determining signs in the Morse complex, one also chooses orientations of the descending disks $d_p(w)$ of the critical points of $h$.  \cite{Paj} provides a homotopy equivalence $\mathcal{E}(h,w)\co CM_{*}(h,v;\kappa)\to S_*(\mathbb{Y},h^{-1}(\{a\});\kappa)$; there are various auxiliary choices involved in the construction of $\mathcal{E}(h,w)$ that we will review, and it is shown in \cite{Paj} that different choices lead to chain homotopic maps.  In the context of Example \ref{morseex} where one has a Morse function $f\co\mathbb{X}\to [-\Lambda,\Lambda]$, we will apply this both with $h$ equal to $f$ and with $h$ equal to the restrictions of $f$ to various sublevel sets; the main subtlety will be in ensuring that the auxiliary choices can be made in such a way that the  various versions of $\mathcal{E}(h,w)$ are consistent with each other (not merely up to homotopy).

A first choice that is made in the construction of $\mathcal{E}(h,w)$ is that of a \emph{Morse-Smale filtration} $\mathcal{W}=\left(W^{(-1)}\subset W^{(0)}\subset \ldots\subset W^{(n)}\right)$ as in \cite[Section 4.3.8]{Paj}.  For this one uses a Morse function $\phi\co \mathbb{Y}\to \R$ for  which $w$ is still gradient-like but which, typically unlike $h$, is ``ordered'' in the sense that if $p$ and $q$ are critical points of $\phi$ with $\mathrm{ind}(p)<\mathrm{ind}(q)$ then $\phi(p)<\phi(q)$. One then takes regular values $a=a_0<a_1<\cdots<a_n=b$, where $n=\dim \mathbb{Y}$, such that, for $k=0,\ldots,n$, the index-$k$ critical points of $h$ (equivalently, of $\phi$) all lie in $\phi^{-1}((a_k,a_{k+1}))$, and defines $W^{(k)}=\phi^{-1}([a,a_{k+1}])$.  Thus $W^{(-1)}=h^{-1}(\{a\})$, $W^{(n)}=\mathbb{Y}$, and $H_i(W^{(k)},W^{(k-1)};\kappa))=0$ if $i\neq k$ while $H_k(W^{(k)},W^{(k-1)};\kappa)$ is a  $\kappa$-vector space with basis given by the relative homology classes $[d_p(w)]$ of the descending disks  for the index $k$-critical points $p$ of $h$, oriented via the arbitrarily-chosen orientations that are part of the input to the Morse complex.

Let us write $P_*(\mathcal{W})$ for the chain complex\footnote{\cite{Paj} would denote this complex as $C_*(\phi,w)$; for our purposes later it seems worthwhile to use notation emphasizing the role of the filtration $\mathcal{W}$.  For $i\notin \{0,1,\ldots,n\}$, $P_i(\mathcal{W})$ should be interpreted as zero and, consistently with this,  $W^{(i)}$ should be interpreted as equal to $W^{(-1)}$ for $i<-1$ and to $W^{(n)}$ for $i>n$.} (similar to the cellular chain complex of a CW complex) given as a graded $\kappa$-vector space by $P_k(\mathcal{W})=H_k(W^{(k)},W^{(k-1)};\kappa)$ and with differential $\delta\co P_{k+1}(\mathcal{W})\to P_k(\mathcal{W})$ given by the connecting homomorphism for the long exact sequence of the triple $(W^{(k+1)},W^{(k)},W^{(k-1)})$.   There is then an obvious isomorphism of graded vector spaces $I_{h,\mathcal{W}}\co CM_*(h,w;\kappa)\to P_*(\mathcal{W})$, sending a critical point $p$ to the relative homology class $[d_p(w)]$ of its descending disk, and \cite[Proposition 6.1.8]{Paj} asserts that this is an isomorphism of chain complexes.  Recalling that $W^{(-1)}=h^{-1}(\{a\})$ and $W^{(n)}=\mathbb{Y}$,  the map $\mathcal{E}(h,w)\co CM_*(h,w;\kappa)\to S_*(\mathbb{Y},h^{-1}(\{a\});\kappa)$ is the composition of this isomorphism $I_{h,\mathcal{W}}$ with a chain homotopy equivalence $\chi\co P_*(\mathcal{W})\to S_*(W^{(n)},W^{(-1)};\kappa)$ constructed  as follows, based on the proof of \cite[Theorem 6.3.6]{Paj}.

Given $k\geq 0$, we assume, inductively\footnote{As a base case, since $P_{i}(\mathcal{W})=S_i(W^{(n)},W^{(-1)};\kappa)=0$ for $i<0$, our conditions are trivially satisfied by taking $\chi_i=0$ for $i<0$.}, that for all $i<k$ we have constructed vector space homomorphisms $\chi_i\co P_i(\mathcal{W})=H_i(W^{(i)},W^{(i-1)};\kappa)\to S_i(W^{(n)},W^{(-1)};\kappa)$ in such a way that: 
\begin{description}[style=multiline, labelwidth=1.5cm]
	\item[\namedlabel{ind1}{($\chi.1$)}]  $\chi_i$ has image in the subspace $S_i(W^{(i)},W^{(-1)};\kappa)$ of $S_i(W^{(n)},W^{(-1)};\kappa)$.
	\item[\namedlabel{ind2}{($\chi.2$)}] $\partial_i\circ\chi_i=\chi_{i-1}\circ\delta_{i}$, where $\partial_i\co S_i(W^{(i)},W^{(-1)};\kappa)\to S_{i-1}(W^{{(i)}},W^{(-1)};\kappa)$ and $\delta_{i}\co P_{i}(\mathcal{W})\to P_{i-1}(\mathcal{W})$ are the differentials on the respective complexes;
	\item[\namedlabel{ind3}{($\chi.3$)}] Given any $x\in P_i(\mathcal{W})$, the element $\chi_i(x)\in S_i(W^{(i)},W^{(-1)};\kappa)$, which satisfies $\partial_i\chi_i(x)\in S_{i-1}(W^{(i-1)},W^{(-1)};\kappa)$ by (\ref{ind1}) and (\ref{ind2}), projects to $H_i\left(\frac{S_*(W^{(i)},W^{(-1)};\kappa)}{S_*(W^{(i-1)},W^{(-1)};\kappa)}\right)$ as a representative of the class $x\in H_i(W^{(i)},W^{(i-1)};\kappa)$. 
\end{description}

Under this inductive hypothesis, we show how to construct $\chi_k\co P_k(\mathcal{W})\to S_k(W^{(n)},W^{(-1)};\kappa)$ so that (\ref{ind1}), (\ref{ind2}), and (\ref{ind3}) continue to hold for $i=k$.   For each standard generator $[d_p(w)]$ of $P_k(\mathcal{W})$, we first choose a chain $c_p\in S_k(W^{(k)},W^{(-1)};\kappa)$ which projects to a relative cycle in $S_k(W^{(k)},W^{(k-1)};\kappa)$ representing $[d_p(w)]$ in homology. Then $\partial_kc_p\in S_{k-1}(W^{(k-1)},W^{(-1)};\kappa)$ is a cycle which projects to a representative of the class $\delta_k[d_p(w)]\in H_{k-1}(W^{(k-1)},W^{(k-2)};\kappa)$.  So by the inductive hypothesis, $\chi_{k-1}\delta_k[d_p(w)]$ and $\partial_kc_p$ project to $S_{k-1}(W^{(k-1)},W^{(k-2)};\kappa)$ as representatives of the same relative homology class, whence there are $b_p\in S_{k}(W^{(k-1)},W^{(-1)};\kappa)$ and $a_p\in S_{k-1}(W^{(k-2)},W^{(-1)};\kappa)$ such that \[ \chi_{k-1}\delta_k[d_p(w)]=\partial_kc_p+\partial_kb_p+a_p. \]  Now $\partial_{k-1}\chi_{k-1}\delta_k[d_p(w)]=\chi_{k-2}\delta_{k-1}\delta_k[d_p(w)]=0$ by induction, so it follows that $\partial_{k-1}a_p=0$.  But $H_{k-1}(W^{(k-2)},W^{(-1)};\kappa)=0$ (this follows by an easy induction from the fact that $H_i(W^{(j)},W^{(j-1)};\kappa)=0$ whenever $i>j$), so it follows that $a_p=\partial_kb'_p$ for some $b'_p\in S_k(W^{(k-2)},W^{(-1)};\kappa)$.  So if we set $\chi_k[d_p(w)]=c_p+b_p+b'_p\in S_k(W^{(k)},W^{(-1)};\kappa)$, then $\chi_k[d_p(w)]$ differs from $c_p$ by  chains contained in $W^{(k-1)}$ and so, like  $c_p$, represents $[d_p(w)]$ in $H_k(W^{(k)},W^{(k-1)};\kappa)$.  Moreover, by construction, $\chi_{k-1}\partial_k[d_p(w)]=\partial_k\chi_k[d_p(w)]$. 

Repeating this construction for each index-$k$ critical point $p$ and then extending linearly to the span $P_k(\mathcal{W})$ of the various $[d_p(w)]$, we obtain a map $\chi_k\co P_k(\mathcal{W}))\to S_k(W^{(k)},W^{(-1)};\kappa)\subset S_k(W^{(n)},W^{(-1)};\kappa)$ such that (\ref{ind1}), (\ref{ind2}), and (\ref{ind3}) now hold for all $i\leq k$, completing the inductive construction of the chain map $\chi\co P_*(\mathcal{W})\to S_*(W^{(n)},W^{(-1)};\kappa)$.

\begin{remark} The choices of $c_p$, $b_p$, and $b'_p$ (and thus the definition of $\chi_k[d_p(w)]$) are made independently for the different critical points of index $k$; the only prior choices that impact the definition of $\chi_k[d_p(w)]$ are those made in the course of constructing $\chi_{k-1}\delta_{k}[d_p(w)]$.  
\end{remark}

\begin{remark}\label{deftocrit} If  $z$ is the largest critical value of the function $h\co \mathbb{Y}\to[a,b]$, then at each stage the chains $\chi_k[d_p(w)]$ can be taken to lie in the subspace $S_k(h^{-1}([a,z])\cap W^{(k)},W^{(-1)};\kappa)$ of $S_k(W^{(k)},W^{(-1)};\kappa)$; this follows readily from the fact that the flow of $-w$ can be used to construct a deformation retraction of  $W^{(k)}$ onto $h^{-1}([a,z])\cap W^{(k)}$ (\cite[Remark 3.4]{Mil}), which can be used to modify the chains $c_p$, $b_p$, and $b'_p$ to have image in $h^{-1}([a,z])$.
\end{remark}

\begin{remark}Define ascending $\N$-filtrations on $P_*(\mathcal{W})$ and on $S_*(W^{(n)},W^{(-1)};\kappa)$ by setting, for $m\in \N$, \[ P^{(m)}_{k}=\left\{\begin{array}{ll}P_{k}(\mathcal{W}) & \mbox{if }k\leq m\\ 0 &\mbox{otherwise} \end{array}\right. \quad\mbox{and}\quad S^{(m)}_{k}=S_k(W^{(m)},W^{(-1)};\kappa).  \]  Then (\ref{ind1}) shows that $\chi$ is a map of $\N$-filtered chain complexes and (\ref{ind3}) shows that $\chi$ induces isomorphisms $H_*(P^{(0)}_{*})\to H_*(S^{(0)}_{*})$  and $H_*\left(\frac{P^{(m+1)}_*}{P^{(m)}_{*}}\right)\to H_*\left(\frac{S^{(m+1)}_*}{S^{(m)}_{*}}\right)$. As in \cite[Proof of Corollary 6.3.8]{Paj}, from this one infers inductively from the five-lemma that $\chi$ is a quasi-isomorphism, and hence a homotopy equivalence since $P_*(\mathcal{W})$ and $S_*(W^{(n)},W^{(-1)};\kappa)$ are bounded-below chain complexes of $\kappa$-vector spaces. Thus any map $\mathcal{E}(h,w)=\chi\circ I_{h,\mathcal{W}}\co CM_*(h,w)\to S_*(\mathbb{Y},h^{-1};\{a\};\kappa)$ constructed by the above prescription is always a homotopy equivalence. \end{remark}

With this preparation, we shift notation to that used in  Example \ref{morseex} and Lemma \ref{Efiltered}: let $\kappa$ be a field and let $\mathbb{X}$ be a compact smooth manifold with boundary, and let $f\co \mathbb{X}\to [-\Lambda,\Lambda]$ be a Morse function for which $-\Lambda$ and $\Lambda$ are regular values, with preimages $\partial^{\pm}\mathbb{X}=f^{-1}(\{\pm\Lambda\})$.  One then has an ascending chain complex, the Morse complex $CM_{*}(f,v;\kappa)$, defined with the assistance of a gradient-like vector field $v$ for $f$ having Morse-Smale flow and of arbitrarily-chosen orientations on the descending disks of the critical points, such that for each $t\in [-\Lambda,\Lambda)$ the filtered subcomplex $CM_{*}^{\leq t}(f,v;\kappa)$ is generated in degree $k$ by those critical points $p$ of $f$ with index $k$ such that $f(p)\leq t$, with boundary operator constructed from counts of gradient flowlines for $-v$ connecting critical points that differ in index by $1$.  

At the same time, the relative singular chain complex $S_*(\mathbb{X}\setminus\partial^+\mathbb{X},\partial^-\mathbb{X};\kappa)$ is an ascending complex, with $S_{*}^{\leq t}(\mathbb{X},\partial^-\mathbb{X};\kappa)=S_*(f^{-1}([-\Lambda,t]),\partial^-\mathbb{X};\kappa)$, and Lemma \ref{Efiltered} asserts that the choices involved in the construction of the map $\mathcal{E}(f,v)\co CM(f,v;\kappa)\to S_*(\mathbb{X}\setminus\partial^+\mathbb{X},\partial^-\mathbb{X};\kappa)$ can be made in such a way that $\mathcal{E}(f,v)$ is a filtered quasi-isomorphism.

\begin{proof}[Proof of Lemma \ref{Efiltered}]
	
	Let the critical values of $f$, all lying in the interval $(-\Lambda,\Lambda)$ be $t_1<t_2<\cdots <t_r$, and choose (necessarily regular) values $a_0<a_1<\cdots<a_r$ with $a_0=-\Lambda$, $a_r=\Lambda$, and $t_i<a_i<t_{i+1}$ for $i\in\{1,\ldots,r-1\}$.  For $m\in\{1,\ldots,r\}$ let us define $\mathbb{X}_m=f^{-1}([a_0,a_m])$.  Note that $CM_{*}(f|_{\mathbb{X}_m},v|_{\mathbb{X}_m})=CM_{*}^{\leq t_m}(f,v;\kappa)$.  So for any $m\in \{1,\ldots,r\}$, the construction recalled above gives rise to a homotopy equivalence $\mathcal{E}(f|_{\mathbb{X}_m},v|_{\mathbb{X}_m})\co CM_{*}^{\leq t_m}(f,v;\kappa)\to S_*(\mathbb{X}_m,\partial^-\mathbb{X})$; in particular, the case that $m=r$ gives a version of $\mathcal{E}(f,v)$ from the statement of the proposition.  The main point will be to ensure that, if suitable choices are made in the constructions, the restriction of $\mathcal{E}(f,v)$ to $CM_{*}^{\leq t_m}(f,v;\kappa)$ is equal to $\mathcal{E}(f|_{\mathbb{X}_m},v|_{\mathbb{X}_m})$.
	
	The construction of each $\mathcal{E}(f|_{\mathbb{X}_m},v|_{\mathbb{X}_m})$ involves the choice of a Morse-Smale filtration \[ \mathcal{W}_m=\left(\partial^-\mathbb{X}=W^{(-1)}_{m}\subset W^{(0)}_{m}\subset \cdots\subset W^{(n)}_{m}=\mathbb{X}_m\right) \] with $n=\dim\mathbb{X}_m=\dim\mathbb{X}$, giving rise to the complex $P_*(\mathcal{W}_m)$ with $P_k(\mathcal{W}_m)=H_k(W^{(k)}_{m},W^{(k-1)}_{m};\kappa)$ as described at the start of this section.  Using \cite[Theorem 4.3.58]{Paj}, we may, and do, choose these filtrations in such a way that $W^{(k)}_{m}\subset W^{(k)}_{m+1}$ for all $k$ and $m$.  Hence, if $1\leq m<m'\leq r$, there is a chain map $
	i_{m,m'}\co P_*(\mathcal{W}_m)\to P_*(\mathcal{W}_{m'})$ induced by the inclusions $W^{(k)}_{m}\hookrightarrow W^{(k)}_{m'}$.  As in the proofs of \cite[Propositions 6.2.4 and 6.2.1]{Paj}, one obtains a commutative diagram \begin{equation}\label{cmpdiag} \xymatrix{ CM_{*}^{\leq t_m}(f,v;\kappa)\ar[rr]^{I_{f|_{\mathbb{X}_m},\mathcal{W}_m}} \ar[d] & &  P_*(\mathcal{W}_m) \ar[d]^{i_{m,m'}} \\ CM_*^{\leq t_{m'}}(f,v;\kappa) \ar[rr]_{I_{f|_{\mathbb{X}_{m'}},\mathcal{W}_{m'}}} &  & P_*(\mathcal{W}_{m'}) }\end{equation} where the left vertical map is the inclusion.
	In particular, if $p$ is a critical point of $f$ with $f(p)<t_m$, then the inclusion-induced map $i_{m,m'}\co H_k(W^{(k)}_{m},W^{(k-1)}_{m};\kappa)\to H_k(W^{(k)}_{m'},W^{(k-1)}_{m'};\kappa)$ sends the standard basis element $[d_p(v|_{\mathbb{X}_m})]=I_{f|_{\mathbb{X}_m},\mathcal{W}_m}(p)$ for $P_k(\mathcal{W}_m)=H_k(W^{(k)}_{m},W^{(k-1)}_{m};\kappa)$, corresponding to the descending disk for $p$, to the standard basis element $[d_p(v|_{\mathbb{X}_{m'}})]$ for $P_k(\mathcal{W}_{m'})$.
	
	We now claim that we can construct the homotopy equivalences $\chi^{m}\co P_*(\mathcal{W}_m)\to S_*(\mathbb{X}_m,\partial^-\mathbb{X};\kappa)$ for $m=1,\ldots,r$, consistently with the description at the start of this section, in such a way that, for each $m$, we have a commutative diagram \begin{equation}\label{mrdiag} \xymatrix{ P_*(\mathcal{W}_m) \ar[d]^{i_{m,r}}\ar[r]^<<<<<{\chi^{m}} & S_*(\mathbb{X}_m,\partial^-\mathbb{X};\kappa) \ar[d] \\ P_*(\mathcal{W}_r) \ar[r]_<<<<<{\chi^{r}} & S_*(\mathbb{X},\partial^-\mathbb{X};\kappa) }  \end{equation}	where the right vertical map is the inclusion.  
	
	To achieve this, we apply the inductive procedure described earlier simultaneously for $m=1,\ldots,r$, giving, for each index $k$, commutative diagrams \begin{equation}\label{mm'diag} \xymatrix{ P_k(\mathcal{W}_m) \ar[d]^{i_{m,m'}}\ar[r]^<<<<<{\chi^{m}_k} & S_k(W_{m}^{(k)},\partial^-\mathbb{X};\kappa) \ar[d] \\ P_k(\mathcal{W}_{m'}) \ar[r]_<<<<<{\chi^{r}_k} & S_k(W_{m'}^{(k)},\partial^-\mathbb{X};\kappa) } \end{equation} whenever $m<m'$. (Here the induction is on $k$.)  In the inductive step, if $p$ is an index-$k$ critical point, we let $m$ be the minimal number such that $p\in \mathbb{X}_m$ and choose 
	the chain $\chi_{k}^{m}[d_p(w|_{\mathbb{X}_m})]\in S_k(W_{m}^{(k)},W_{m}^{(-1)};\kappa)$ as required for the inductive construction of $\chi^{m}\co P_*(\mathcal{W}_m)\to S_*(\mathbb{X}_m,\partial^-\mathbb{X};\kappa)$.  (Thus $\partial_k\chi_{k}^{m}[d_p(w|_{\mathbb{X}_m})]$ coincides with $\chi_{k-1}^{m}\delta_k[d_p(w|_{\mathbb{X}_m})]$ which has been constructed earlier in the induction, and $\chi_{k}^{m}[d_p(w|_{\mathbb{X}_m})]$ represents $[d_p(w|_{\mathbb{X}_m})]$ in $H_k(W_{m}^{(k)},W_{m}^{(k-1)};\kappa)$.)   
	Because, for $m<m'$, the inclusion-induced maps $i_{m,m'}\co P_*(\mathcal{W}_m)\to\mathcal{P}_*(\mathcal{W}_{m'})$ are chain maps which respect the bases given by the descending disk classes $[d_p(w)]$, the image of $\chi_{k}^{m}[d_p(w|_{\mathbb{X}_m})]$ under the inclusion $W_{m}^{(k)}\hookrightarrow W_{m'}^{(k)}$   will satisfy the properties required in the definition of $\chi_{k}^{m'}[d_p(w|_{\mathbb{X}_{m'}})]$.  In this way we obtain definitions of $\chi_{k}^{m}[d_p(w|_{\mathbb{X}_m})]$ for all choices of $m$ with the property that $p\in \mathbb{X}_m$. Extending linearly over all index-$k$ critical points gives maps $\chi_{k}^{m}\co P_k(\mathcal{W}_m)\to S_{k}(W_{m}^{(k)};\partial^-\mathbb{X};\kappa)$ such that (\ref{mm'diag}) commutes by definition and such the requisite properties (\ref{ind1}),(\ref{ind2}),(\ref{ind3}) are satisfied.  
	Induction on $k$ then gives homotopy equivalences $\chi^{m}$ consistent with the construction in \cite{Paj} such that (\ref{mrdiag}) commutes.
	
	Moreover, consistently with Remark \ref{deftocrit}, if $p$ is a critical point of $h$ and $m$ is the minimal value such that $p\in \mathbb{X}_m$, then  $\chi^{m}[d_p(w|_{\mathbb{X}_m})]$ can be taken to be supported in $f^{-1}([-\Lambda,t_m])$, as $t_m$ is the largest critical value of $f|_{\mathbb{X}_m}$.  So we can arrange moreover that the homotopy equivalences $\chi^{m}$ factor as \[ \chi^{m}\co P_*(\mathcal{W}_m)\to   S_{*}(f^{-1}([-\Lambda,t_m]),\partial^-\mathbb{X})\to S_{*}(\mathbb{X}_m,\partial^-\mathbb{X};\kappa)  \] with the second map given by inclusion.  Since both the composition and the second map above are homotopy equivalences, so is the first map.  So composing with the isomorphisms $I_{f|_{\mathbb{X}_m},\mathcal{W}_m}$ from (\ref{cmpdiag}), we deduce that the homotopy equivalence \[ \mathcal{E}(f,v)=\chi^{r}\circ I_{f,\mathcal{W}_r}\co CM_*(f,v;\kappa)\to S_*(\mathbb{X},\partial^-\mathbb{X};\kappa) \] has image contained in $S_*(f^{-1}([-\Lambda,t_r];\partial^-\mathbb{X};\kappa)\subset S_*(\mathbb{X}\setminus\partial^+\mathbb{X},\partial^-\mathbb{X};\kappa)$ and restricts, for each critical value $t_m$ of $f$, to a homotopy equivalence \[ CM_{*}^{\leq t_m}(f,v;\kappa)\to S_{*}(f^{-1}([-\Lambda,t_m]),\partial^-\mathbb{X};\kappa).\]
	
	We are to show that, for every $t\in [-\Lambda,\Lambda)$, $\mathcal{E}(f,v)$ restricts as a quasi-isomorphism $CM_{*}^{\leq t}(f,v;\kappa)\to S_{*}(f^{-1}([-\Lambda,t]),\partial^-\mathbb{X};\kappa)$; we have just established this in the case that $t$ is one of the critical values $t_m$.  But this quickly implies the desired result for all $t$, since if $t<t_1$ then $CM_{*}^{\leq t}(f,v;\kappa)=0$ while $f^{-1}([-\Lambda,t])$ deformation retracts to $\partial^-\mathbb{X}$, and otherwise if $m\in\{1,\ldots,r\}$ is the largest value such that $t_m\leq t$ then $CM_{*}^{\leq t}(f,v;\kappa)=CM_{*}^{\leq t_m}(f,v;\kappa)$ and $f^{-1}([-\Lambda,t])$ deformation retracts to $f^{-1}([-\Lambda,t_m])$. 
\end{proof}


\begin{thebibliography}{99} 
	\bibitem[Ba94]{Bar} S. Barannikov. \emph{The framed Morse complex and its invariants}. Singularities and bifurcations, 93--115, Adv. Soviet Math., \textbf{21}, Amer. Math. Soc., Providence, RI, 1994.
	\bibitem[BBF24]{BBF20} U. Bauer, M. Botnan, and B. Fluhr. \emph{Universal distances for extended persistence}.  J. Appl. Comput. Topol \textbf{8} (2024), no. 3, 475--530.
	\bibitem[BBF21]{BBF21} U. Bauer, M. Botnan, and B. Fluhr. \emph{Structure and interleavings of relative interlevel set cohomology}. arXiv:2108.09298.
	\bibitem[BF22]{BF22} U. Bauer and B. Fluhr. \emph{Relative Interlevel Set Cohomology Categorifies Extended Persistence Diagrams}. arXiv:2205.15275.
	\bibitem[BL15]{BL} U. Bauer and M. Lesnick. \emph{Induced matchings and the algebraic stability of persistence modules}. J. Comput. Geom. \textbf{6} (2015), no. 2, 162--191.
	\bibitem[BEMP13]{BEMP} P. Bendich, H. Edelsbrunner, D. Morozov, and A. Patel. \emph{Homology and robustness of level and interlevel sets}. \emph{Homology Homotopy Appl.} \textbf{15} (2013), no. 1, 51--72.
	\bibitem[BGO19]{BGO} N. Berkouk, G. Ginot, and S. Oudot. \emph{Level--sets persistence and sheaf theory}. arXiv:1907.09759
	\bibitem[BCZ23]{BCZ} P. Biran, O. Cornea, and J. Zhang. \emph{Triangulation, persistence, and Fukaya categories}. arXiv:2304.01785.
	\bibitem[BK91]{BK} A. Bondal and M. Kapranov. \emph{Framed triangulated categories}. Mat. Sb. \textbf{181} (1990), no. 5, 669-683.
	\bibitem[Bre93]{Bre} G. Bredon. \emph{Topology and Geometry}. Grad. Texts Math. \textbf{139}, Springer, Berlin, 1993.
	\bibitem[BH17]{BH} D. Burghelea and S. Haller. \emph{Topology of angle valued maps, bar codes and Jordan blocks}. J. Appl. Comput. Topol. \textbf{1} (2017), no. 1, 121--197.
	\bibitem[CdSM09]{CDM} G. Carlsson, V. de Silva, and D. Morozov. \emph{Zigzag Persistent Homology and Real-valued Functions}, in  \emph{Proceedings 25th ACM Symposium on Computational Geometry (SoCG)}, 2009, pp. 247--256.
	\bibitem[CdSGO16]{CDGO} F. Chazal, V. de Silva, M. Glisse, S. Oudot. \emph{The structure and stability of persistence modules},  SpringerBriefs in Mathematics. Springer, 2016.
	\bibitem[CGGM24]{CGGM} M. Che,   F. Galaz-Garc\'ia, L. Guijarro, and I. Membrillo Solis. \emph{Metric geometry of spaces of persistence diagrams}. J. Appl. Comput. Topol \textbf{8} (2024), 2197--2246.
	\bibitem[dSMS18]{dSMS} V. de Silva, E. Munch, and A. Stefanou. \emph{Theory of interleavings on categories with a flow}. Theory Appl. Categ. \textbf{33} (2018), no. 21,  583--607.
	\bibitem[CEH07]{CEH07} D. Cohen-Steiner, H. Edelsbrunner, and J. Harer. \emph{Stability of persistence diagrams}. Discrete Comput. Geom. \textbf{37} (2007), 103--120.
	\bibitem[CEH09]{CEH09} D. Cohen-Steiner, H. Edelsbrunner, and J. Harer. \emph{Extending persistence using Poincar\'e and Lefschetz duality}. Found. Comput. Math. \textbf{9} (2009), no. 1, 79-–103. Erratum in Found. Comput. Math. \textbf{9} (2009), no. 1, 133--134.
	\bibitem[C-B15]{CB} W. Crawley-Boevey. \emph{Decomposition of pointwise finite-dimensional persistence modules}. J. Algebra Appl. \textbf{14} (2015), no. 5, 1550066.
	\bibitem[Dri04]{Drin} V. Drinfeld. \emph{
		DG quotients of DG categories}. J. Algebra. \textbf{272} (2004), no. 2, 643--691.
	\bibitem[ELZ02]{ELZ} H. Edelsbrunner, D. Letscher, and A. Zomorodian. \emph{Topological persistence and simplification}. Discrete Comput. Geom. \textbf{28} (2002), no. 4, 511--533.
	\bibitem[EH10]{EH} H. Edelsbrunner and J. Harer.  \emph{Computational topology. An introduction}. AMS, Providence, 2010. 
	\bibitem[EZ50]{EZ} S. Eilenberg and J. Zilber. \emph{Semi-simplicial complexes and singular homology}. Ann. Math. \textbf{51} (1950), no. 3, 499--513.
	\bibitem[GM03]{GM} S. Gelfand and Y. Manin. \emph{Methods of homological algebra}, 2nd ed. Springer, Berlin, 2003.
	\bibitem[Ha02]{Ha} A. Hatcher. \emph{Algebraic Topology}. Cambridge University Press, 2002.
	\bibitem[Mil63]{Mil} J. Milnor. \emph{Morse theory}.  Annals of Mathematics Studies, \textbf{51}, Princeton University Press, Princeton, N.J., 1963. 
	\bibitem[Mil65]{Milhcob} J. Milnor. \emph{Lectures on the $h$-cobordism theorem}. Princeton University Press, Princeton, N.J., 1965.
	\bibitem[P06]{Paj} A. V. Pajitnov. \emph{Circle-valued Morse theory}. de Gruyter Studies in Mathematics \textbf{32}, Berlin, 2006.
	\bibitem[Rol76]{Rol} D. Rolfsen. \emph{Knots and Links}. AMS Chelsea Publishing \textbf{346}, 1976.
	\bibitem[tD08]{tD} T. tom Dieck. \emph{Algebraic Topology}. EMS, Z\"urich, 2008.
	\bibitem[U22]{U22} M. Usher. \emph{Symplectic Banach-Mazur distances between open subsets of $\mathbb{C}^n$}. J. Topol. Anal. \textbf{14} (2022), no. 1, 231--286.
	\bibitem[U23]{U23} M. Usher. \emph{Abstract interlevel persistence for Morse--Novikov and Floer theory}. arXiv:2302.14342.
	\bibitem[UZ16]{UZ} M. Usher and J. Zhang. \emph{Persistent homology and Floer-Novikov theory}. Geom. Topol. \textbf{20} (2016), 3333--3430.
	\bibitem[We94]{Wei} C. Weibel. \emph{An introduction to homological algebra}. Cambridge University Press, 1994.
\bibitem[ZC05]{ZC} A. Zomorodian and G. Carlsson. \emph{Computing persistent homology}. Discrete Comput. Geom. \textbf{33} (2005), 249--274.
	\end{thebibliography}
\end{document}